\documentclass[11pt]{article}

\usepackage[english]{babel}

\usepackage[T1]{fontenc}
\usepackage{lmodern}

\usepackage[utf8]{inputenc}
\usepackage{stmaryrd}
\usepackage[a4paper,margin=2.5cm]{geometry}
\usepackage[dvipsnames]{xcolor}
\usepackage{eufrak}
\usepackage{chemist}
\usepackage{amsmath,amsfonts,amsthm,amssymb}

\DeclareMathOperator{\spanvector}{Span}
\DeclareMathOperator{\Poiss}{Poiss}

\DeclareMathOperator{\Hess}{Hess}
\DeclareMathOperator{\tr}{tr}
\usepackage{blindtext}
\usepackage{mathtools}
\usepackage{mathrsfs}
\newcommand{\defeq}{\mathrel{\mathop:}=}

\usepackage{booktabs,tabularx}
\usepackage{apxproof}
\usepackage{environ}
\usepackage{dsfont}
\usepackage{etoolbox}
\usepackage{fancyvrb}
\usepackage{ifthen}
\usepackage{kvoptions}
\usepackage{bibunits}
\usepackage{comment}
\usepackage{chemformula}
\usepackage{mathrsfs}
\usepackage{float}
\usepackage{xcolor}
\usepackage{enumerate}
\usepackage[toc,page]{appendix}
\usepackage{graphicx,color}
\usepackage{hyperref}
\usepackage{mathtools}
\usepackage{setspace}
\usepackage{cite}

\makeatletter
\renewcommand*\env@matrix[1][\arraystretch]{%
  \edef\arraystretch{#1}%
  \hskip -\arraycolsep
  \let\@ifnextchar\new@ifnextchar
  \array{*\c@MaxMatrixCols c}}
\makeatother

\newtheorem{theorem}{Theorem}[section]
\newtheorem{lemma}[theorem]{Lemma}
\newtheorem{corollary}[theorem]{Corollary}
\newtheorem{assumption}{Assumption}

\newtheorem{proposition}[theorem]{Proposition}

\numberwithin{equation}{section}

\newtheorem*{remark}{Remark}
\newtheorem*{theorem*}{Theorem}
\newtheorem*{proposition*}{Proposition}
\newtheorem*{corollary*}{Corollary}
\newtheorem*{lemma*}{Lemma}

\DeclareMathOperator*{\spec}{spec}

\newcommand{\ignore}[1]{}

\usepackage{amsmath}
\usepackage{graphicx}
\title{How seed banks evolve in plants: a stochastic dynamical system
subject to a strong drift}

\author{Alison Etheridge\thanks{Department of Statistics, University of Oxford, UK} \and João Luiz de Oliveira Madeira\footnotemark[1]}

\begin{document}

\maketitle

\begin{center}{{\small\textbf{Abstract}}}
\end{center}

{\small We study how changes in population size and 
fluctuating environmental conditions influence the 
establishment of seed banks in plants. Our model is a modification of the 
Wright–Fisher model with seed bank, introduced by 
Kaj, Krone and Lascoux~\cite{kaj2001coalescent}. We distinguish between 
wild type individuals, producing only nondormant seeds, and mutants, 
producing seeds with dormancy. To understand how changing population
size shapes the 
establishment of seed banks, we analyse the process under a diffusive scaling. 
The results support the biological insight that seed banks are favoured in
a declining population, and disfavoured if population size is constant or
increasing. The surprise is that this is true even when population sizes
are changing very slowly - over the timescales of evolution. 
We also investigate the influence of short-term fluctuations, such
as annual variations in rainfall or temperature.
%under 
%adverse and fluctuating environments, and disfavoured under constant or
%improving environments. 

Mathematically, our analysis reduces to a stochastic dynamical system forced 
onto a manifold by a large drift, which converges under scaling to a 
diffusion on the manifold. Inspired by the Lyapunov–Schmidt reduction, we 
derive an explicit formula for the limiting diffusion coefficients by 
projecting the system onto its linear counterpart. This provides a general 
framework for deriving diffusion approximations in models with strong drift 
and nonlinear constraints.}

\medskip


\section{Introduction} \label{Paper03_introduction}

In this paper, we study, from a theoretical perspective, how and why seed 
banks evolve in plants. A seed bank consists of seeds that can germinate 
after remaining dormant for several generations, acting as a reservoir of 
genetic material from the past that competes with the genetic contribution 
of the current generation. We suppose that a mutation arises in a population 
with no pre-existing seed bank, conferring the ability to form a seed bank. 
We also assume that the distribution of the total number of seeds produced 
by a plant carrying this mutation is the same as the number
produced in a single generation by a wild type individual, 
and that when they germinate
mutant and wild-type juveniles are equally fit, so the mutation provides no 
obvious 
intrinsic advantage to individual carriers. We investigate the probability 
that the mutation invades the population, and compare its fixation probability 
to that of a newly arriving neutral mutation without a seed bank. Our results 
show that seed banks significantly affect fixation probability, in a way that 
depends on the environmental conditions. Specifically, we show that in a 
constant population, or when conditions favour population growth over
evolutionary timescales, the 
presence of a seed bank is disadvantageous. In contrast, 
environmental conditions that cause random fluctuations in population size 
or population decline (again on evolutionary timescales) can favour seed banks.
We also consider the impact of short term fluctuations in the ability of 
plants to produce seed and see that these too can favour the development
of seed banks.
Our study was initially motivated by a question posed by Diala Abu Awad 
and Amandine Véber, whose ongoing work addresses the biological 
ramifications of our results.

A major challenge in the mathematical analysis of our model is that we must 
keep track of multiple timescales, since the finite number of generations 
that a seed can remain dormant will be negligible with respect to the 
timescale of evolutionary processes. The main mathematical contribution of this
paper is that to overcome this challenge, in 
Section~\ref{Paper03_general_result_derivatives_manifold}
we develop a general formula for 
the quadratic approximation of the stable manifold around the asymptotic 
attractor of an ordinary differential equation (ODE) flow. In 
combination with results from~\cite{katzenberger1991solutions} 
and~\cite{parsons2017dimension}, this formula 
can be used to compute the diffusion 
approximation of large-scale stochastic dynamical systems with multiple 
timescales arising from theoretical biology and applied probability. 

Before formally defining the interacting particle system in Section~\ref{Paper03_main_results_section}, we outline the biological motivation for the model.

\subsection{Biological Motivation} 
\label{Paper03_subsection_biological_motivation}

Seeds are the basic reproductive units of gymnosperms and angiosperms, 
the most widespread plants on the planet. From a physiological point of view, 
each seed contains all the necessary structures to germinate, i.e.~to grow 
into a new mature plant. In this context, \emph{dormancy} is biologically 
defined as the failure of a viable seed to complete germination given all 
necessary environmental 
conditions~\cite{bewley1997seed, baskin2004classification}. Therefore, seed 
banks can be thought of as reservoirs of genetic material from past 
generations: a seed produced as an offspring of the current generation can 
skip some generations, and germinate in the future. We say that a plant 
has a \emph{seed bank} if its seeds exhibit dormancy. The presence 
(and length) of dormancy has a strong hereditary component 
(see~\cite[Chapter~8]{baskin1998seeds}), and it is largely determined by the 
mother plant. Interestingly, for some species, seeds of the same individual 
can have different germination times — in this case, seed dormancy is said to 
exhibit a \emph{bet-hedging strategy}~\cite{abley2024bet}

In evolutionary biology, mutations are typically classified into three 
categories: beneficial (which increase fitness), neutral (which have no impact 
on fitness), and deleterious (which decrease fitness). Natural selection 
favours beneficial mutations and their establishment. Nevertheless, neutral 
mutations can also become fixed even without natural selection, through 
mechanisms such as genetic drift driven by demographic 
fluctuations~\cite{kimura1985neutral, nei2010neutral}.

From a genetic point of view, dormancy is a highly variable trait. For 
instance, seeds of different lineages of the species \textit{Avena fatua}, 
also known as wild oat, may exhibit different periods of dormancy, or even no 
dormancy~\cite{simons2011modes, abley2024bet}. It is also remarkable that 
dormant and nondormant seeds can be found as traits over the entire 
phylogenetic tree of gymnosperms and angiosperms, indicating that the seed 
dormancy trait must have been lost and evolved multiple times during 
evolution~\cite{baskin1998seeds, linkies2010evolution,dayrell2017phylogeny}. 
Although highly variable, there is also a correlation between different 
ecosystems and the prevalence of seed 
dormancy~\cite[Chapter~12]{baskin1998seeds}. For example, in tropical regions, 
there is a higher prevalence of species of plants without seed banks, whilst 
in colder and drier regions, the presence of seed dormancy predominates 
(see~\cite[Chapter~12]{baskin1998seeds} and~\cite{dayrell2017phylogeny}). 
Nonetheless, even in stable environments, the prevalence of plant populations 
with seed banks is comparable in magnitude to the prevalence of populations 
without seed banks~\cite{dayrell2017phylogeny}. Moreover, experimental data 
suggest that dormancy diversity along a phylogenetic tree is highly skewed 
towards the root of the tree~\cite{dayrell2017phylogeny}. Taking into account 
these observations, the following questions naturally arise:
\begin{enumerate}[(1)]
    \item Is seed dormancy a beneficial trait in environments with harsher 
conditions (e.g.~lower levels of water, extreme temperatures) or less stable
conditions?
    \item Can seed banks evolve and establish as a result of genetic drift 
in stable environments?
    \item If the dormancy trait is prevalent in a population living in a 
stable environment, how likely is it to be lost?
\end{enumerate}

To the best of our knowledge, there are no studies using individual-based 
models to address these questions. In this paper, we use a variant of the 
Wright–Fisher model introduced by 
Kaj, Krone, and Lascoux in~\cite{kaj2001coalescent}. 
In contrast to previous mathematical work, we are interested in finding 
conditions under which seed banks might become prevalent, rather than
the implications for genetic variation of an established seed bank.
We provide partial answers to the questions above by 
investigating the probability of establishment of a mutation that
confers the ability to produce seeds with dormancy under
three different scenarios:
\begin{enumerate}[(i)]
    \item \textbf{Constant environment}: we show that the presence of a seed 
bank is disfavoured compared to the evolution of a neutral allele without 
dormancy (see Theorem~\ref{Paper03_bound_fixation_probability_any_K}). 
Nonetheless, fixation of the seed bank trait can still occur due to genetic 
drift, and its fixation probability is of the same order of magnitude as that 
of a nondormant neutral trait. Furthermore, in a population predominantly 
carrying the dormancy trait, it is unlikely that this trait will be lost 
(see Theorem~\ref{Paper03_weak_convergence_diffusion_constant_environment} and 
Proposition~\ref{Paper03_bound_drift_general_K_proposition}). In particular, 
our results are consistent with biological data from populations living in 
stable environments~\cite{dayrell2017phylogeny}.
    
    \item \textbf{Slowly-varying environment}: in this setting, we 
suppose that as a result of changing environmental conditions, 
population size is changing on the timescale of evolution. Adverse conditions 
(e.g.~lack of water, extreme temperatures) can be expected to result
in population 
decline, while favourable conditions (e.g.~adequate water supply) lead 
to population growth. In 
Theorem~\ref{Paper03_weak_convergence_WF_model_slow_env} and 
Corollary~\ref{Paper03_corollary_SDE_proportion_slowly_varying_environment}, 
we show that adverse environmental conditions favour the seed bank trait, 
while the nondormancy trait is favoured when conditions promote population 
growth. Moreover, random fluctuations in population size on the evolutionary 
timescale tend to favour an intermediate state in which both dormant and 
nondormant individuals coexist. 
We emphasise that this is in spite of the fact that in our model, seeds
will only remain dormant for times that are negligible on the timescale 
of evolution. 
Our results help to explain biological data 
comparing tropical and temperate habitats~\cite[Chapter~12]{baskin1998seeds}.

\item \textbf{Fast–changing environmental conditions}: 
in this setting, we model short-term environmental conditions—such as 
annual variations in rainfall or temperature—through fluctuations 
in seed production, while holding the population size constant. 
Adverse conditions correspond to a decrease in the overall seed output 
of all individuals in a given generation, regardless of their genetic type, 
while favourable conditions correspond to an increase. In 
Theorem~\ref{Paper03_thm_convergence_diffusion_fast_changing_environment} 
we show that such short-lived fluctuations can favour the seed–bank trait 
for a class of germination–time distributions. Our findings help to explain 
biological patterns in different habitats~\cite[Chapter~12]{baskin1998seeds}, 
and they extend classical results such as those 
of~\cite{cohen1966optimizing}. Crucially, in contrast to previous work,
we incorporate the effects of
the randomness inherent in reproduction in a finite population (so-called
genetic drift).
\end{enumerate}

Before precisely stating our model and our results, we briefly review 
some earlier mathematical work on the evolution of seed banks.

\subsection{Related Work} \label{}

Cohen~\cite{cohen1966optimizing} introduced the first model of which we are 
aware to study under which conditions seed banks can be beneficial. He 
considers an annual plant that reproduces only once during its lifetime. In 
his model, in each year, a fixed fraction $g \in (0,1)$ of seeds germinate, 
and a fixed fraction $d \in (0,1)$ of the remaining seeds dies out. Moreover, 
in each year, the average number of seeds produced per mature plant is 
described by a random variable $\xi(t)$, where $t \geq 0$ indicates the year. 
According to Cohen's model, if $\zeta(t)$ denotes the number of seeds in the seed 
bank in year $t \geq 0$, we have the recursion
\begin{equation} \label{Paper03_definition_cohen_model}
    \zeta(t+1) = (1 - g)\zeta(t) - d(1-g)\zeta(t) + g\xi(t)\zeta(t), \; \forall t \in \mathbb{N}_{0}.
\end{equation}
The random process $(\xi(t))_{t \geq 0}$ captures the environmental 
fluctuations, and by finding $g \in (0,1)$ that maximises the number of 
surviving seeds, Cohen concludes that in a randomly varying environment the 
presence of seed banks is favoured. 
In contrast to our model, in Cohen's model, the seed bank is assumed to be 
established, and so his analysis does not take into account the competition between 
plants with and without seed banks. The possibility of invasion of a mutation
with a different germination fraction is considered 
in~\cite{ellner1987competition}, and
further modifications of the model were 
implemented in~\cite{bulmer1984delayed,abley2024bet} to 
study the effect of competitive pressure and other environmental conditions. 
The recursion~\eqref{Paper03_definition_cohen_model} implicitly assumes a 
geometric distribution for the germination time of seeds. For plant 
populations, however, the distribution of germination time is often complex, 
and it is rarely given by the geometric distribution 
(see~\cite[Table~12.2]{baskin1998seeds}).
Moreover, none of these works include the impact of genetic drift. 

Some work has also been done on the establishment of dormancy in microbial 
populations.  In this context, Blath and 
T\'{o}bi\'{a}s~\cite{blath2020invasion} study a model in which an individual 
can switch between active and inactive states in a population under competitive
pressure. They conclude that if the switching into the dormant state is induced 
by competitive pressure, then dormancy is beneficial and can establish with 
positive probability when the population size is sufficiently large.

Blath, Hermann and Slowik~\cite{blath2021branching} also studied the beneficial 
aspects of dormancy for microbial populations. The authors study a branching 
process model in which particles can switch between an active and a dormant 
state. In this model, the authors include a cost for dormancy, which captures 
the idea that microbes need to invest energy and resources into the machinery 
required to allow them to remain dormant for long periods of time. In their 
framework, the benefits of dormancy do not overcome its cost in a constant 
environment. However, when the environment fluctuates randomly, dormancy
%the presence 
%of a seed bank 
may be beneficial, and the choice between responsive and 
stochastic switching may have a great impact on population growth.

Although the results found in~\cite{blath2020invasion, blath2021branching} are 
applicable to microbial populations under different hypotheses, their 
conclusions cannot be directly extended to plants. First, in contrast to 
microbial populations, plant seed banks are produced as a fraction of the 
offspring of active (or mature) plants, rather than by switching the metabolic 
state of an adult plant. Moreover, in plants, even dormant seeds generally 
cannot skip more than a few 
generations~\cite{tellier2019persistent, barrett2005temporal,lennon2021principles}, 
which limits the time span of the genetic reservoir, 
whereas the models of~\cite{blath2020invasion, blath2021branching} assume 
that microbial dormancy can persist for periods of time comparable to the 
timescale of evolutionary change. Furthermore, some of the molecular mechanisms 
that regulate dormancy are common among seed plants~\cite{willis2014evolution},
and so the cost of dormancy in plants should not be as high as the 
corresponding cost to microbes.

Since seed banks act as reservoirs of genetic material, they have an impact on 
how a sample of individuals from a population are genetically related, 
i.e.~they influence the \emph{genealogy} of the population; this influence 
depends on the duration of dormancy, and it has been extensively studied from 
the point of view of stochastic 
analysis~\cite{lennon2021principles, kaj2001coalescent,blath2013ancestral,blath2016new}. 
We will not focus on this aspect of seed banks in this paper. The difference 
in the length of the dormancy period motivated Blath and 
co-authors~\cite{blath2013ancestral,blath2016new} to classify seed banks into 
two categories: \emph{weak} (referring to seed banks with finite mean 
germination time) and \emph{strong} (describing seed banks whose dormancy 
length is comparable to the timescale of evolution). The term weak reflects 
the fact that seed banks with finite mean germination time do not substantially 
alter the genealogical structure (at least over long timescales). 
Since our focus is on 
plants, and we show that seeds with dormancy lasting only up to a finite 
number of generations can still substantially affect establishment probability,
we will not adopt this weak/strong classification in the present work.

\subsection{Phases of fixation of a mutation} 
\label{Paper03_phases_invasion_section}

For an advantageous mutation that arises in an established 
population and sweeps to fixation, it is customary to divide
the invasion into four phases~\cite{barton1998effect}:
\begin{enumerate}[(i)]
    \item \emph{Branching phase}: In this phase, the proportion of mutants in 
the population remains sufficiently small that competition between them can be 
neglected. Hence, the behaviour of the number of individuals can be 
mathematically modelled as a branching process.
    \item \emph{Early or lag phase}: The number of mutant individuals is 
neither too small to neglect competition nor sufficiently large to be 
considered a significant proportion of the total population. In this case, 
since the average per capita growth rate of the mutant is bigger than that of 
individuals of the `wild type' (those not carrying the mutation), the 
behaviour of the proportion of mutants is well approximated by the solution of 
an ordinary differential equation (ODE) with positive growth 
rate~\cite{cuthbertson2012fixation}.
    \item \emph{Diffusion phase}: The mutants reach a significant 
proportion of the total population, and this proportion changes more slowly
as a result of mutants competing with one another. On a longer timescale, 
we can approximate its evolution by a stochastic differential equation (SDE).
    \item \emph{Fixation phase}: The number of wild type individuals is 
negligible compared to the number of mutants, and one can approximate the 
dynamics of wild type individuals by a subcritical branching process. 
\end{enumerate}

The challenge in applying this framework to model the fixation of neutral 
variants is the analysis of the early phase, since in the neutral case the 
average per capita growth rate of the mutant population is the same as that of
the wild type, and therefore genetic drift 
dominates~\cite{kimura1985neutral, nei2010neutral}. 
For this reason, in the biological literature, the invasion of a neutral 
mutation is often modelled as though
the diffusion phase starts immediately after the branching phase.

Since plants harbouring seed dormancy do not have any clear advantage or 
disadvantage, our study will be more closely related to the study of fixation 
of a neutral mutation than a beneficial one. It is however beyond the scope of 
this paper to provide a formulation of early phase in the neutral case, and 
we will follow Kimura~\cite{kimura1962probability,kimura1985neutral} and 
restrict our attention to the branching and diffusion phases only, 
i.e.~we will ignore the lag and the fixation phases. We will explain, after 
stating Theorem~\ref{Paper03_branching_approximation_constant_environment} in 
Section~\ref{Paper03_BP_constant_environment_description}, how our choice of 
initial condition for the diffusion phase will be informed by the outcome
of the branching phase.

\medskip

\noindent{\textbf{Structure of the paper:}} 
In Section~\ref{Paper03_main_results_section} we present the main model and 
results of this paper. We will discuss three different scenarios: the case of 
constant environment will be described in 
Section~\ref{Paper03_WF_model_const_environ_description}, that of 
a population whose size changes on the same timescale 
as the evolutionary process 
in Section~\ref{Paper03_subsection_WF_slow_environment}, 
and the case in which environmental conditions affecting seed
production fluctuate on the
timescale of a single generation in Section~\ref{subsection:WF_fast_env_model}.

In Sections~\ref{Paper03_subsection_katzenberger_overview} 
and~\ref{Paper03_subsection_computing_derivatives}, we provide a brief 
overview of stochastic dynamical systems forced onto a manifold by a large 
drift. Section~\ref{Paper03_general_result_derivatives_manifold} states a 
general formula to compute the quadratic approximation of the flow of a 
dynamical system near an attractor manifold; 
to improve readability, the proof of this result is 
postponed to Section~\ref{Paper03_section_proofs_dynamical_systems}.
Section~\ref{Paper03_derivatives_specific_ode} illustrates the 
use of this formula through an application to a specific ODE. Finally, proofs 
of the main results begin in 
Section~\ref{Paper03_section_proof_results_branching}, where we establish the 
result concerning the branching phase. In 
Section~\ref{Paper03_proof_results_concerning_diffusion} we combine our 
dynamical systems results with theorems in stochastic analysis to derive the 
limiting stochastic differential equation that approximates our model in 
the diffusion phase.
The remaining proofs are in 
Section~\ref{Paper03_section_proof_auxilliary_lemmas}.

\medskip

\noindent \textbf{Notation:} In what follows, we denote the set of strictly 
positive integers by $\mathbb{N}$, and let 
$\mathbb{N}_{0} \defeq  \mathbb{N} \cup \{0\}$. Moreover, for any 
$K \in \mathbb{N}$, we let $\llbracket K \rrbracket \defeq \{1,2,\ldots, K\}$, 
and $\llbracket K \rrbracket_{0} \defeq \llbracket K \rrbracket \cup \{0\}$. 
Since we will work with $(K+1)$-dimensional vectors, it will be convenient to 
use bold letters for vectors and matrices, and plain letters for scalar 
quantities. We let $\mathbf{e}_{0}, \ldots, \mathbf{e}_{K}$ be the canonical 
basis of $\mathbb{R}^{K+1}$, and we denote the inner product between vectors 
$\boldsymbol{u}, \boldsymbol{v} \in \mathbb{R}^{K+1}$ 
by $\langle \boldsymbol{u}, \boldsymbol{v} \rangle$.

For any random variable $Y$, let $\sigma(Y)$ denote the $\sigma$-algebra generated 
by $Y$.
Moreover, if $\sigma$ and $\tilde{\sigma}$ are two $\sigma$-algebras, 
let $\sigma \times \tilde{\sigma}$ denote the corresponding product 
$\sigma$-algebra.
For a Borel subset $I \subset \mathbb{R}^{K+1}$, we let 
$\mathscr{D}([0, +\infty), I)$ denote the set of $I$-valued càdlàg paths.
For $T > 0$ and a given càdlàg finite variation process 
$\mathbf{A}: [0, + \infty) \rightarrow \mathbb{R}^{K+1}$, 
let $\mathfrak{T}(\mathbf{A}, T)$ denote the total variation of $\mathbf{A}$ on the time interval $[0,T]$.
Moreover, for a real càdlàg martingale 
$({M}(t))_{t \geq 0}$, we use $([{M}](t))_{t \geq 0}$ and 
$(\langle M \rangle(t))_{t \geq 0}$ for, respectively, the quadratic variation and the predictable bracket associated to ${M}$ 
(see~\cite[Theorem~4.52]{jacod2013limit} for the difference between 
these processes).

It will also be convenient to introduce some notation regarding comparisons 
between functions.
For a set $\mathcal{S}$, let $f,g: \mathcal{S} \rightarrow \mathbb{R}_{+}$ be 
non-negative real-valued functions on $\mathcal{S}$. We write $f \lesssim g$ 
to indicate that there exists $C > 0$ such that for every $x \in \mathcal{S}$, 
$f(x) \leq Cg(x)$. If we wish to emphasise that the constant $C$ depends on a 
parameter $p$, then we write $f \lesssim_p g$. Moreover, for 
$f: [0, \infty) \rightarrow \mathbb{R}$, $N \in \mathbb{N}$ and $T \geq 0$, 
we let
\begin{equation*}
    \int_0^T f(t) \; d\lfloor Nt \rfloor \defeq \sum_{t = 0}^{\lfloor NT \rfloor - 1} f\left(\frac{t}{N}\right).
\end{equation*}

\section{Model description and main results} 
\label{Paper03_main_results_section}

In this section, we introduce the model that will be the principal focus 
of this paper (and some of its modifications), and state our main results. 
Our model can be viewed as an 
extension of the Wright-Fisher model with seed bank introduced by 
Kaj, Krone and Lascoux in~\cite{kaj2001coalescent}. 
It is appropriate for a population of annual plants. An entirely analogous
analysis could be pursued for plants that survive for multiple generations, 
but the equations would become considerably more complicated and harder 
to interpret.
To improve readability, we shall divide the description according to whether
we are considering a constant or varying environment.

\subsection{Wright-Fisher model in constant environment} 
\label{Paper03_WF_model_const_environ_description}

Let $N \in \mathbb{N}$ be a scaling parameter. In the case of constant 
environment, $N$ will be the population size. We consider a population 
evolving in discrete generations and composed of two types, which differ in 
their ability to produce seed banks, but not in their overall fitness. 
Wild type individuals have no seed bank, so that they can contribute to the 
genetic material of the next generation only. On the other hand, mutant 
individuals produce seeds that can last up to $K \in \mathbb{N}$ generations. 
The distribution of the germination time of seeds of a mutant plant is 
given by a vector $\boldsymbol{b} = (b_i)_{i \in \llbracket K \rrbracket_0}$, 
satisfying
\begin{equation} 
\label{Paper03_assumption_probability_distribution_bi}
        \sum_{i = 0}^{K} b_{i} = 1 \quad \textrm{and} \quad b_0 > 0.
    \end{equation}
In particular, $\mathbf{b}$ is an element of the $K$-dimensional simplex
\begin{equation} \label{Paper03_K_dimensional_simplex}
    \mathbb{S}_{K}^* \defeq \left\{\boldsymbol{b} = (b_i)_{i \in \llbracket K \rrbracket_{0}} \in [0,1]^{K+1}: \; \sum_{i  = 0}^K b_i = 1, \; b_0 > 0\right\}.
\end{equation}
We explain below 
Equation~(\ref{Paper03_binomial_distribution_WF_constant_environment})
why it is convenient to impose the condition that $b_0 > 0$. 
We define the \emph{mean germination time} as
\begin{equation} 
\label{Paper03_mean_age_germination}
    B \defeq \sum_{i = 1}^{K} ib_{i}.
\end{equation}
Intuitively, $b_{i}$ represents the proportion of seeds of a mutant 
individual that spend exactly $i$ generations in the seed bank, 
for $i \in \llbracket K \rrbracket_{0}$. 
For any $t \in \mathbb{N}_{0}$, let $X^{N}_{i}(t)$ be the number of 
mature mutant plants that lived in generation $t - i$. 
Let $M > 0$ be some large parameter and let $t \in \mathbb{N}_0$ 
satisfy $t \geq K-1$. Each wild type individual in generation $t$ contributes 
$\Poiss(M)$ seeds, and each mature mutant in generation 
$t - i \, (i \in  \llbracket K \rrbracket_0)$, contributes 
$\Poiss(Mb_i)$ seeds to the pool of seeds that are ready to germinate in the
current generation. From this pool, we sample $N$ seeds to germinate and 
form the mature plants in generation $t + 1$. 
The parameter $M$ is assumed to be large enough that 
we can ignore the possibility that there are fewer than $N$ seeds from which to
sample, and indeed, as usual for Wright-Fisher models, we shall suppose that
the pool of seeds is large enough that we can ignore the difference between
sampling with and without replacement. 

More formally, we define our Markov process 
$\left(\boldsymbol{X}^{N}(t)\right)_{t \in \mathbb{N}_{0}} = \left(X_{0}^{N}(t),
 X_{1}^{N}(t), \ldots, X_{K}^{N}(t) \right)_{t \geq K-1}$ 
in such a way that for any $t \in \mathbb{N}_{0}$, conditional on 
$\boldsymbol{X}^{N}(t)$, 
the distribution of 
$X^{N}_{0}(t+1)$ is given by
\begin{equation} 
\label{Paper03_binomial_distribution_WF_constant_environment}
    X_{0}^{N}(t+1)\big\vert \boldsymbol{X}^{N}(t) 
\sim \textrm{Bin}\left(N, \; 
\frac{\sum_{i = 0}^{K} b_{i}X^{N}_{i}(t)}{(N - X^{N}_{0}(t)) 
+ \sum_{i = 0}^{K} b_{i}X^{N}_{i}(t)}\right),
\end{equation}
and $X^{N}_{i}(t+1) \defeq X^{N}_{i-1}(t)$ for 
$i \in \llbracket K \rrbracket$. 

\begin{remark}
The condition $b_0 > 0$ 
in~\eqref{Paper03_assumption_probability_distribution_bi} ensures that even
if the population of mature plants consists entirely of mutant individuals, 
there will always be sufficient seeds ready to germinate to form the next 
generation. Our main results can still be proved in the case $b_0=0$
by writing the dynamics of the Wright-Fisher model in such a way that if the 
mature plants in generation $t$ are all mutant individuals, then so are
those in 
generation $t+1$. We decided that including this level of generality led to
additional complications in notation, but no
new mathematical or biological insights.
\end{remark}

We are interested in the case in which one mature mutant conferring the 
dormancy trait appeared in generation $0$, 
i.e.~when $\boldsymbol{X}(0) = (1,0,0, \cdots,0)$. As explained in 
Section~\ref{Paper03_phases_invasion_section}, we start by considering the 
branching phase.

\subsubsection{Branching phase in constant environment} 
\label{Paper03_BP_constant_environment_description}

It will be convenient to count (for mutant individuals) the number of seeds 
produced that will germinate $i \in \llbracket K \rrbracket_{0}$ generations 
in the future rather than keeping track of the number of mutants in the 
previous generations. Hence, let $M > 1$ be some large positive number. 
Let the $(K+1)$-type branching process 
$\left(\mathbf{Z}(t)\right)_{t \in \mathbf{N}_{0}} 
= \left({Z}_{0}(t), \ldots, Z_{K}(t)\right)_{t \in \mathbf{N}_{0}}$ 
be a multitype branching process with dynamics given as follows.
\begin{enumerate}[(i)]
    \item In each generation, each type $0$ particle gives birth to 
$\Poiss(b_{0})$ type $0$ particles and $\Poiss(Mb_{i})$ type $i$ particles 
for any $i \in \llbracket K \rrbracket$.
    \item For $i > 1$, a particle of type $i$ matures into a particle of 
type $i -1$ with probability $1$.
    \item Each particle of type $1$ germinates with probability $1/M$ into 
a particle of type $0$, and dies with probability $ 1 - 1/M$.
\end{enumerate}
The type $0$ particles represent the mature mutants in the current generation, 
whilst for $i \in \llbracket K \rrbracket$, the type $i$ particles represent 
the seeds in the seed bank that will be ready to 
germinate exactly $i$ generations in the 
future. Therefore, $\mathbf{Z}$ is not an approximation of the vector 
process $\boldsymbol{X}^{N}$, but it follows from its definition (and from 
the Poisson approximation to the binomial) that $Z_{0}$ approximates the 
behaviour of $X^{N}_{0}$. Note that the generator of $\mathbf{Z}$ does not 
depend on the scaling parameter $N$, which reflects the fact that this 
approximation is only valid when the number of mutant individuals is 
negligible with respect to the total population size $N$. Also note that 
the mean matrix $\mathbf{M}$ of $\mathbf{Z}$ is given by
\begin{equation} \label{Paper03_mean_matrix_branching_constant_environment}
        \mathbf{M} = \left(\begin{array}{cccccc}
            b_{0} & Mb_{1} & Mb_{2} & Mb_{3} & \ldots  & Mb_{K} \\
           1/M & 0 & 0 & 0 & \ldots & 0 \\ 0 & 1 & 0 & 0 & \ldots & 0 \\ 0 & 0 & 1 & 0 & \ldots & 0\\ \vdots & \vdots & \vdots & \vdots & \ddots & \vdots \\ 0 & 0 & 0 & 0 & \ldots & 0
        \end{array}\right).
    \end{equation}
Since we are not assuming that the mutant individuals harbour any clear 
advantage or disadvantage, we expect the branching process approximating 
$(X^{N}_{0}(t))_{t \in \mathbb{N}_{0}}$ to be critical. 
Recall the definition of $\boldsymbol{e}_0$ from the end of 
Section~\ref{Paper03_phases_invasion_section}, and the definition of the 
mean germination time $B$ from~\eqref{Paper03_mean_age_germination}.

\begin{theorem} 
\label{Paper03_branching_approximation_constant_environment}
The multitype branching process $\mathbf{Z}$ is critical, and its survival 
probability satisfies
\begin{equation*}
\lim_{t \rightarrow \infty} t\mathbb{P}_{\mathbf{e}_{0}}
\left(\mathbf{Z}(t) \neq 0\right) = 2(B+1).
\end{equation*}
Moreover, conditional on its survival, the limiting distribution of the 
scaled number of mature individuals satisfies, for any $y \in [0, \infty)$,
\begin{equation*}
\lim_{t \rightarrow \infty} 
\mathbb{P}_{\mathbf{e}_{0}}\left(\frac{Z_{0}(t)}{t} \geq y \, 
\Big \vert \, \mathbf{Z}(t) \neq 0\right) = \exp\left(-2y(B+1)^2\right).
\end{equation*}
\end{theorem}

\medskip
\noindent
{\bf Choosing the initial condition for the diffusion phase:}
Theorem~\ref{Paper03_branching_approximation_constant_environment} will be 
proved in Section~\ref{Paper03_section_proof_results_branching}. 
We should like to compare these quantities to the survival probability and 
expected numbers conditional on survival for a neutral mutation without 
seed bank arising in the wild type population. 
The evolution of such a neutral mutation is approximated in the branching 
phase by a (one-type) critical branching process $(Y(t))_{t \geq 0}$ with 
the variance of the offspring distribution equal to one. In this case we have 
(see for instance~\cite[Theorems~I.9.1 and~I.9.2]{athreya2004branching})
\begin{equation} \label{Paper03_critical_branching_phase_monotype}
    \lim_{t \rightarrow \infty} t\mathbb{P}\left(Y(t) \neq 0\right) = 2 
\quad  \textrm{ and } \quad \lim_{t \rightarrow \infty} 
\mathbb{P}\left(\frac{Y(t)}{t} \geq y \, \Big \vert \, Y(t) \neq 0\right) 
= \exp(-2y), \; \forall \, y \in [0, \infty).
\end{equation}
In particular, for $B > 0$, 
Theorem~\ref{Paper03_branching_approximation_constant_environment} 
and~\eqref{Paper03_critical_branching_phase_monotype} imply that
\begin{equation*}
    \lim_{t \rightarrow \infty} t\mathbb{P}_{\mathbf{e}_{0}}\left(\mathbf{Z}(t) \neq 0\right) = 2(B+1) > 2 = \lim_{t \rightarrow \infty} t\mathbb{P}\left(Y(t) \neq 0\right),
\end{equation*}
which means that the presence of seed bank increases the survival probability 
during the branching phase. However, the expected number of mature individuals 
by the end of the branching phase (conditional on survival) is lower than 
that at the end of the branching phase (again conditional on survival) 
for a neutral mutation without dormancy, since 
Theorem~\ref{Paper03_branching_approximation_constant_environment} 
and~\eqref{Paper03_critical_branching_phase_monotype} imply that
\begin{equation*}
    \lim_{t \rightarrow \infty} \mathbb{E}_{\mathbf{e}_{0}}\left[\frac{Z_{0}(t)}{t} \, \Big \vert \, \mathbf{Z}(t) \neq 0\right] = \frac{1}{2(B+1)^2} < \frac{1}{2} = \lim_{t \rightarrow \infty} \mathbb{E}\left[\frac{Y(t)}{t} \, \Big \vert \, Y(t) \neq 0\right]. 
\end{equation*}
The mean is, however, a crude measure of how seed banks affect the outcome 
of the branching phase. As explained in 
Section~\ref{Paper03_phases_invasion_section}, we shall make the 
approximation that the diffusion process begins as soon as the branching
phase ends. As partial compensation for this crude assumption,
we shall adopt a slightly more
refined approach than is usual to splice together the final state of 
the branching phase and the initial condition of the diffusion phase.
%end of the 
%branching phase as a proxy for the initial condition of the diffusion phase. 
%To this end, we adopt a more refined approach to compare the branching phase outcomes of a neutral allele carrying the dormancy trait with those of a nondormant allele.

We have already seen in~\eqref{Paper03_critical_branching_phase_monotype}
that the probability of survival up to time $t$ under
the branching approximation is increased by the seed bank property, at the 
expense of a smaller expected population size conditional on survival. We 
look for a correspondence between the initial values for the diffusion 
approximation for the two scenarios (a mutation that confers a 
seed bank and a purely neutral mutation) that captures these two effects.   
The branching approximation is valid only when the proportion of 
mutants in the mature population is negligible compared to the total 
population size. Let $\varepsilon \in (0, 1)$ be sufficiently small that the 
branching approximation is adequate while the number of mutant individuals 
in the mature population is at most $\varepsilon N$. 
For $B \geq 0$, we look for a map 
$\psi_B: [0, \varepsilon] \rightarrow [0, \infty)$ such that $\psi_B(0) = 0$ 
and for any $0 \leq y^{(1)} < y^{(2)} \leq \varepsilon$, 
for sufficiently large $t$,
\begin{equation} 
\label{Paper03_correspondence_between_branching_phase_monotype_seed_bank}
    \mathbb{P}_{\mathbf{e}_{0}}\left(\psi_B(y^{(1)}) \leq \frac{Z_0(t)}{t} \leq \psi_B(y^{(2)})\right) = \mathbb{P}\left(y^{(1)} \leq \frac{Y(t)}{t} \leq y^{(2)}\right) + o(t^{-1}).
\end{equation}
Note that the probabilities on each side of this equation are for the 
\emph{unconditioned} processes. 
Roughly, the probability that a neutral mutation survives and reaches
frequency $y$ is the same as the probability that the seed bank mutation
survives and reaches $\psi_B(y)$.
Applying  Theorem~\ref{Paper03_branching_approximation_constant_environment} 
and~\eqref{Paper03_critical_branching_phase_monotype} 
to~\eqref{Paper03_correspondence_between_branching_phase_monotype_seed_bank}, 
we conclude that the map $\psi_B$ must satisfy, 
for any $0 \leq y^{(1)} < y^{(2)} \leq \varepsilon$,
\begin{equation} 
\label{Paper03_correspondence_between_branching_phase_monotype_seed_bank_ii}
    (B+1)\left(\exp\left(-2\psi_B(y^{(1)})(B+1)^2\right) - \exp\left(-2\psi_B(y^{(2)})(B+1)^2\right)\right) = \exp(-2y^{(1)}) - \exp(-2y^{(2)}). 
\end{equation}
Our next result follows directly 
from~\eqref{Paper03_correspondence_between_branching_phase_monotype_seed_bank_ii}.

\begin{lemma} 
\label{Paper03_lemma_correspondence_value_branching_phase}
For any $\varepsilon \in (0, 1]$ and $B \geq 0$, there is a unique map 
$\psi_{B}: [0, \varepsilon] \rightarrow [0, \infty)$ 
satisfying~\eqref{Paper03_correspondence_between_branching_phase_monotype_seed_bank_ii} 
and $\psi_B(0) = 0$, and for any $y \in [0, \varepsilon]$, it is given by
    \begin{equation*}
        \psi_B(y) = \frac{1}{2(B+1)^2} \log\left(\frac{B+1}{B + e^{-2y}}\right).
    \end{equation*}
    Moreover, $\psi_B$ can be continuously extended to $[0,1]$, and satisfies the following properties:
    \begin{enumerate}[(i)]
        \item For any fixed $y > 0$, the map $B \mapsto \psi_B(y)$ is strictly decreasing.\item  For any fixed $B \in [0, \infty)$, the map $y \mapsto \psi_B(y)$ is strictly increasing.
    \end{enumerate}
\end{lemma}
As an illustration of 
Lemma~\ref{Paper03_lemma_correspondence_value_branching_phase} 
we plot $\psi_B$ for some different values of $B$ in 
Figure~\ref{Paper03_figure_different_corresponding_functios_branching}. 
Observe that property~(i) of 
Lemma~\ref{Paper03_lemma_correspondence_value_branching_phase} reflects 
our expectation that, by the end of the branching phase, in the case of 
constant environment, there should be a lower proportion of mutant individuals 
in the case of a seed bank than in the case of a neutral mutation that
is not associated with a seed bank. 

The next step will be to compare the diffusion approximation for the
mutation with seed bank started from $\psi_B(y)$ to that of a neutral 
mutation without seed bank started from $y$.

\begin{figure}[t]
\centering
\includegraphics[width=0.5\linewidth,trim=0cm 0cm 0cm 0cm,clip=true]{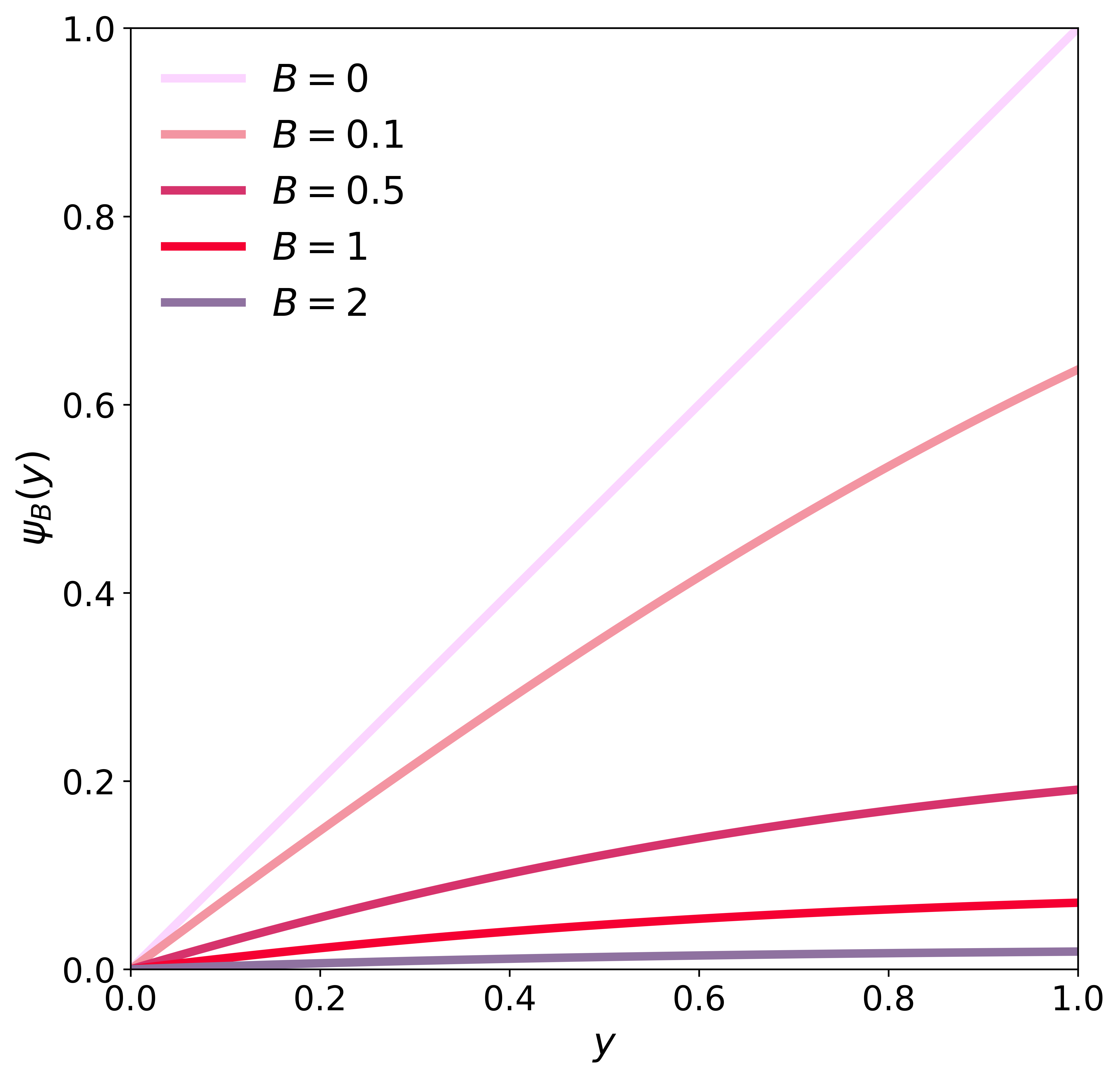}
\caption{For $B \in \{0, 0.1, 0.5, 1, 2\}$, we 
plot the function $\psi_B$ that determines, for sufficiently
small $y$, the initial condition for our diffusion approximation with 
seed bank that we argue in the main text provides an appropriate
comparison to the case
without seed bank started from initial condition $y$.
%For a full explanation see the main text.  
Observe that, as expected 
(c.f.~Lemma~\ref{Paper03_lemma_correspondence_value_branching_phase}), 
$\psi_B$ is a decreasing function of the mean germination time $B$. 
%Biologically, this means that by the end of the branching phase, we expect a lower proportion of mature mutant individuals carrying the seed bank bank trait when compared to the end of the branching phase of a mutant neutral allele with no seed banks.}
}
\label{Paper03_figure_different_corresponding_functios_branching}
\end{figure}

\subsubsection{Diffusion phase in constant environment} 
\label{Paper03_diffusion_phase_constant_environment}

Suppose now that mutant individuals survive the branching phase to become 
a non-negligible fraction of the mature population. For the next phase, 
one expects that for sufficiently large $N$ the dynamics of the population 
will be well approximated by an SDE. Recall that for $N \in \mathbb{N}$, 
$i \in \llbracket K \rrbracket_0$ and $t \in \mathbb{N}_0$, $X^{N}_{i}(t)$ 
indicates the number of mutant individuals in generation $t - i$. Then, 
following the usual scaling in Wright-Fisher models, we define for any 
$N \in \mathbb{N}$ and $t \in \mathbb{R}_{+}$, the scaled process
\begin{equation} 
\label{Paper03_scaled_process_WF_model}
\boldsymbol{x}^{N}(t) \defeq \frac{\boldsymbol{X}^{N}(\lfloor Nt \rfloor)}{N},
\end{equation}
where $\boldsymbol{x}^{N}=(x^{N}_{0}, x^{N}_{1}, \ldots, x^{N}_{K})$.

For the usual Wright-Fisher model with two or more genetic types, the 
weak convergence of the scaled process to the Wright-Fisher diffusion as 
$N \rightarrow \infty$ is classically proved by directly showing the 
convergence of the action of the infinitesimal generator of the scaled 
discrete process to the action of the SDE generator on twice continuously 
differentiable functions (see~\cite[Theorem~10.1.1]{ethier2009markov}). 
Unfortunately, this approach is not applicable to the analysis of the process 
$(\boldsymbol{x}^{N}(t))_{t \geq 0}$. In order to understand why, we compute 
the expected increment in the number of mutant individuals over a single
generation.

Let $\{\mathcal{F}^N_{t}\}_{t \geq 0}$ be the natural filtration of the process
$(\boldsymbol{x}^N(t))_{t \geq 0}$. Then, 
by~\eqref{Paper03_scaled_process_WF_model} 
and~\eqref{Paper03_binomial_distribution_WF_constant_environment}, 
for any $t \in \mathbb{N}_{0}$, conditioning on the information up to 
time $t$, we have
\begin{eqnarray} 
\mathbb{E}\left[x^{N}_{0}\left(t + \tfrac{1}{N}\right) - x^{N}_{0}(t) 
\Big\vert \mathcal{F}^N_t\right] 
& = &\frac{\sum_{i = 0}^{K} b_{i}X^{N}_{i}(\lfloor Nt\rfloor)}{(N - X^{N}_{0}
(\lfloor Nt \rfloor)) 
+ \sum_{i = 0}^{K} b_{i}X^{N}_{i}(\lfloor Nt \rfloor)} - x^{N}_{0}(t) 
\nonumber
\\[2ex] 
& = &\frac{\sum_{i = 0}^{K} b_{i}x^{N}_{i}(t)}{(1 - x^{N}_{0}(t)) 
+ \sum_{i = 0}^{K} b_{i}x^{N}_{i}(t)} - x^{N}_{0}(t) 
\nonumber
\\[2ex] 
& = &\frac{(1 - x^{N}_{0}(t))\left(\sum_{i = 1}^{K} b_{i}x^{N}_{i}(t) 
- (1 - b_{0})x^{N}_{0}(t)\right)}{(1 - x^{N}_{0}(t)) 
+ \sum_{i = 0}^{K} b_{i}x^{N}_{i}(t)}.
\label{Paper03_source_ODE_first_coordinate}
\end{eqnarray}
Moreover, for $i \in \llbracket K \rrbracket$,  
from~\eqref{Paper03_scaled_process_WF_model}
\begin{equation} \label{Paper03_source_ODE_other_coordinates}
x^{N}_{i}\left(t + \tfrac{1}{N}\right) = x^{N}_{i-1}(t).
\end{equation}
In contrast to the classical Wright-Fisher model, for which the corresponding
expected increment would be $\mathcal{O}(N^{-1})$, when taking the limit as 
$N \rightarrow \infty$, it is not even immediate that the right-hand side 
of~\eqref{Paper03_source_ODE_first_coordinate} converges to $0$. 
Nevertheless, if $(\boldsymbol{x}^{N}(t))_{t \geq 0}$ does converge 
weakly to a $(K+1)$-dimensional diffusion process 
$(\boldsymbol{x}(t))_{t \geq 0}$ on $[0,1]^{K+1}$ 
as $N \rightarrow \infty$, then since $K$ is finite, we do not expect to see 
the change in the number of mutant individuals over~$K$ consecutive 
generations in the limit. 
%In other words, if $\left((\boldsymbol{x}^{N}(t))_{t \geq 0}\right)_{N \in \mathbb{N}}$ weakly converges to a $(K+1)$-dimensional diffusion process 
That is, we expect that
\begin{equation*}
x_{0}(t) \equiv x_{1}(t) \equiv \cdots \equiv x_{K}(t).
\end{equation*}
Also observe that expressions~\eqref{Paper03_source_ODE_first_coordinate} 
and~\eqref{Paper03_source_ODE_other_coordinates} vanish when 
$x^{N}_{0} = x^{N}_{1} = \cdots = x^{N}_{K}$.

To prove convergence to a diffusion in our setting we follow a method 
introduced by Katzenberger in~\cite{katzenberger1991solutions}. We shall 
discuss Katzenberger's results in more detail in 
Section~\ref{Paper03_subsection_katzenberger_overview}, so here we restrict 
ourselves to a heuristic explanation of his approach. 
We view the stochastic process $\boldsymbol{x}^{N}$ as a noise-driven process 
that fluctuates around a deterministic dynamical system given 
by~\eqref{Paper03_source_ODE_first_coordinate} 
and~\eqref{Paper03_source_ODE_other_coordinates}. 
Consider then the vector-valued map 
$\mathbf{F} = (F_{i})_{i = 0}^{K}: [0,1]^{K+1} \rightarrow \mathbb{R}^{K+1}$ 
given for all $\boldsymbol{x} \in [0,1]^{K+1}$ by
\begin{equation}
\label{Paper03_flow_definition_WF_model}
    \left\{
    \arraycolsep=1.3pt\def\arraystretch{2.2}
    \begin{array}{cl} F_{0}(\boldsymbol{x}) 
& \defeq \displaystyle{\frac{(1 - x_{0})
\left(\sum_{i = 1}^{K} b_{i}x_{i}- (1 - b_{0})x_{0}\right)}{(1 - x_{0}) 
+ \sum_{i = 0}^{K} b_{i}x_{i}}}, \\
    F_{i}(\boldsymbol{x}) 
& \defeq x_{i - 1} - x_{i} \quad \forall \, i \in \llbracket K \rrbracket.
    \end{array}\right.
\end{equation}
The component $F_{0}$ corresponds to the expected difference between the number 
of mutant individuals in consecutive generations, i.e.~to 
Equation~\eqref{Paper03_source_ODE_first_coordinate}, while for 
$i \in \llbracket K \rrbracket$, $F_i$~corresponds to the ageing of seeds, 
i.e.~to Equation~\eqref{Paper03_source_ODE_other_coordinates}. Let
\begin{equation} 
\label{Paper03_ODE_describing_expected_increments_number_mutants}
    \frac{d \boldsymbol{x}}{dt} = \mathbf{F}(\boldsymbol{x}),
\end{equation}
and note that $\boldsymbol{x}(0) \in [0,1]^{K+1}$ implies 
$\boldsymbol{x}(t) \in [0,1]^{K+1}$ for all $t \geq 0$. 
On the hypercube $[0,1]^{K+1}$, the map $\mathbf{F}$ only vanishes on the 
diagonal~$\Gamma = \left\{\boldsymbol{x} = (x, \ldots, x) \vert \; 
x \in [0,1]\right\}$. 

We now consider the semimartingale decomposition of $\boldsymbol{x}^{N}$, 
i.e.~for $T \geq 0$, we write
\begin{equation} 
\label{Paper03_semimartingale_formulation_discrete_process_first_coordenate}
\begin{aligned}
    x^{N}_{0}(T) = x^{N}_{0}(0) + M^{N}_{0}(T) + \int_{0}^{T} F_{0}(\boldsymbol{x}^{N}) \, d \lfloor Nt \rfloor,
\end{aligned}
\end{equation}
where the martingale $M^{N}_{0}$ arises from the binomial sampling determining 
the number of mature mutant individuals. Conditioned on 
$\boldsymbol{x}^{N}(t)$, $x^{N}_{i}\left(t + \frac{1}{N}\right)$ is 
deterministic for $i \in \llbracket K \rrbracket$, so the only noise term 
arising from the dynamics is the one due to the sampling of current generation 
and for $i \in \llbracket K \rrbracket$,
\begin{equation} 
\label{Paper03_semimartingale_formulation_discrete_process_other_coordenates}
\begin{aligned}
    x^{N}_{i}(T) = x^{N}_{i}(0) + \int_{0}^{T} F_{i}(\boldsymbol{x}^{N}) \, d \lfloor Nt \rfloor.
\end{aligned}
\end{equation}

Since, by construction, $\boldsymbol{x}^{N}$ is accelerated in time, we expect 
the weak limit of the sequence 
$\left((\boldsymbol{x}^{N}(t))_{t \geq 0}\right)_{N \in \mathbb{N}}$ to take 
values in the diagonal and so we focus on identifying the first component. 
First, for any $\boldsymbol{x} \in [0,1]^{K+1}$, we define the flow
\begin{equation*}
\begin{aligned}
    \boldsymbol{\phi}(\boldsymbol{x}, \cdot): & \, [0, +\infty) \rightarrow [0,1]^{K+1} \\
    & \, T \mapsto \boldsymbol{\phi}(\boldsymbol{x}, T) = \boldsymbol{x} + \int_{0}^{T} \mathbf{F}(\boldsymbol{\phi}(\boldsymbol{x}, t)) \, dt.
\end{aligned}
\end{equation*}
We show in Lemma~\ref{Paper03_eigenvalue_condition_constant_environment} that 
solutions to the 
ODE~\eqref{Paper03_ODE_describing_expected_increments_number_mutants} 
asymptotically converge to the diagonal $\Gamma$, and therefore for any 
$\boldsymbol{x} \in [0,1]^{K+1}$, we can define the map
\begin{equation} 
\label{Paper03_definition_projection_map}
\begin{aligned}
    \boldsymbol{\Phi}: & \, [0,1]^{K+1} \rightarrow \Gamma \\
    & \, \boldsymbol{x} \mapsto \boldsymbol{\Phi}(\boldsymbol{x}) 
= \lim_{T \rightarrow \infty} \boldsymbol{\phi}(\boldsymbol{x},T).
\end{aligned}
\end{equation}
Notice that since, for any $t>0$,
$\boldsymbol{\Phi}(\boldsymbol{x}(t))=\boldsymbol{\Phi}(\boldsymbol{x}(0))$,
an application of the chain rule gives
\begin{equation}
\label{Paper03_definition_projection_mapA}
    \sum_{i = 0}^{K} 
\frac{\partial \Phi_{0}}{\partial x_{i}}F_{i}\left(\boldsymbol{x}\right) = 0 
\qquad\quad \forall \boldsymbol{x} \in [0,1]^{K+1}.
\end{equation}
Using this observation and 
applying Itô's formula to $\Phi_{0}(\boldsymbol{x}^{N})$, we obtain, 
for all $T \geq 0$,
\begin{equation} 
\label{Paper03_Ito_formula_discrete_constant_environment}
\begin{aligned}
    & \Phi_{0}(\boldsymbol{x}^{N}(T)) - \Phi_{0}(\boldsymbol{x}^{N}(0)) \\ 
& \quad \quad = \int_{0}^{T} \frac{\partial \Phi_{0}}{\partial x_{0}} 
\, dM^{N}_{0}(t) + \frac{1}{2 }\int_{0}^{T} 
\frac{\partial^{2}\Phi_{0}}{\partial x_{0}^{2}} d\left[M^{N}_{0}\right](t) 
\, + \int_{0}^{T} \sum_{i = 0}^{K} 
\frac{\partial \Phi_{0}}{\partial x_{i}}F_{i}(\boldsymbol{x}^{N}) 
\, d\lfloor Nt \rfloor  
+ \epsilon^{N}(T) \\ 
& \quad \quad = \int_{0}^{T} \frac{\partial \Phi_{0}}{\partial x_{0}} \, dM^{N}_{0}(t) 
+ \frac{1}{2 }\int_{0}^{T} \frac{\partial^{2}\Phi_{0}}{\partial x_{0}^{2}} 
d\left[M^{N}_{0}\right](t) \, + \epsilon^{N}(T),
\end{aligned}
\end{equation}
where $\epsilon^{N}(T)$ captures the error arising from the jumps. 

As usual for Wright-Fisher models, the limiting martingale $M_0$ will be such 
that 
\[d\langle M_0 \rangle(t)  = x_0(1- x_0) \, dt,
\]
and the error term $\epsilon^N \rightarrow 0$ as $N \rightarrow \infty$. 
Our next result, which will be proved in 
Section~\ref{Paper03_subsection_diffusion_WF_model_constant_environment}, 
makes this rigorous.

\begin{theorem} 
\label{Paper03_weak_convergence_diffusion_constant_environment}
Suppose that $(\boldsymbol{x}^{N}(0))$ converges to 
$(\boldsymbol{x}(0)) \in \Gamma$ as $N \rightarrow \infty$. Then the sequence 
of processes 
$\left((\boldsymbol{x}^{N}(t))_{t \geq 0}\right)_{N \in \mathbb{N}}$ 
converges in distribution in the space of càdlàg functions 
$\mathscr{D}\left([0, \infty), [0,1]^{K+1} \right)$ to a diffusion process 
$(\boldsymbol{x}(t))_{t \geq 0}$ with sample paths in the diagonal $\Gamma$, 
and such that $(x_{0}(t))_{t \geq 0}$ satisfies the stochastic differential 
equation
\begin{equation} 
\label{Paper03_SDE_WF_constant_environment}
x_{0}(T)= \; x_{0}(0) + \frac{1}{2} \int_{0}^{T} 
\left(x_{0}(1 - x_{0}) 
\frac{\partial^{2}\Phi_{0}}{\partial x_{0}^{2}}(\boldsymbol{x})\right) \, dt 
+ \int_{0}^{T} \frac{\sqrt{x_{0}(1 - x_{0})}}{(B(1 - x_{0}) + 1)} \, dW_{0}(t),
\end{equation}
where $W_{0}$ is a standard Brownian motion, $B$ is the mean germination time 
given by~\eqref{Paper03_mean_age_germination} and $\boldsymbol{\Phi}$ is the 
projection map defined in~\eqref{Paper03_definition_projection_map}.
\end{theorem}

A major challenge in analysing the SDE of 
Theorem~\ref{Paper03_weak_convergence_diffusion_constant_environment} is the 
need to compute explicitly the first and second order derivatives of the 
projection map~$\boldsymbol{\Phi}$. If the underlying ODE admits an explicit 
solution, this computation can be carried out directly. However, nonlinear 
ODEs rarely admit explicit solutions, and in our case we have only been
able to solve~\eqref{Paper03_flow_definition_WF_model} explicitly (via the 
method of characteristics) when $K = 1$ 
(see Section~\ref{Paper03_explicit_solution_small_K}). 
To overcome this problem, 
in~\cite{parsons2017dimension}
Parsons and Rogers derive a method of computing the first and second 
order derivatives of $\boldsymbol{\Phi}$ by analysing the underlying 
dynamical system, as well as the linear centre manifold and the curvature of 
the nonlinear stable manifold around~$\Gamma$. %~\cite{parsons2017dimension}. 
We will discuss their approach in detail in 
Section~\ref{Paper03_subsection_computing_derivatives}. By replicating their 
method, we are able to compute the first order derivatives, for all 
$K \in \mathbb{N}$,
\begin{equation} 
\label{Paper03_first_order_derivatives_projection}
\frac{\partial \Phi_{0}}{\partial x_{0}} (\boldsymbol{x}) 
= \frac{1}{(B(1-x_{0})+1)} \qquad\quad \forall \boldsymbol{x} \in \Gamma,
\end{equation}
leading to the martingale term in~\eqref{Paper03_SDE_WF_constant_environment}.
Their formula for second order derivatives requires the computation of the 
curvature of the flow field around~$\Gamma$. To compute this curvature, they 
suggest computing the eigenvalues and eigenvectors of the Jacobian of the 
flow field $\mathbf{F}$ at $\Gamma$, and then performing a change of basis 
to study the behaviour of the dynamical system in the linearly stable 
directions. In our case, this approach only works for small values of 
$K \in \mathbb{N}$, since it requires explicit expressions for the roots of 
a polynomial of degree $K$.

In Section~\ref{Paper03_general_result_derivatives_manifold} we will avoid 
this problem by directly projecting the dynamical system onto its linearly 
stable manifold around $\Gamma$. This method, usually known as the 
Lyapunov-Schmidt reduction in the dynamical systems literature 
(see~\cite[Section~I.3]{golubitsky2012singularities}), is very useful 
for studying the behaviour of singularities in finite and infinite 
dynamical systems. As we explain in 
Section~\ref{Paper03_general_result_derivatives_manifold}, instead of the 
more usual application to classify bifurcations, we use this approach
to understand the behaviour of the non-linear stable manifold, 
i.e.~to study how the flow behaves near~$\Gamma$. By following this programme, 
in Section~\ref{Paper03_general_result_derivatives_manifold} 
we obtain the curvature as the unique solution of a linear system 
(see Theorem~\ref{Paper03_prop_uniqueness_quadratic_term_nonlinear_manifold}).

We can then compute explicitly the value of 
$\frac{\partial^{2} \Phi_{0}}{\partial x_{0}^{2}}$ for any fixed 
$K \in \mathbb{N}$. For instance, for $K = 1$, i.e.~when seeds are
dormant for at most one generation, we have for any 
$\boldsymbol{x} \in \Gamma$,
\begin{equation} 
\label{Paper03_drift_K_1}
\frac{\partial^{2} \Phi_{0}^{(1)}}{\partial x_{0}^{2}}(\boldsymbol{x}) 
= \frac{(1-b_{0})(2 - (1-b_{0})^{2}(1-x_{0})^{2})}{((1-b_{0})(1-x_{0}) + 1)^{3}} > 0, \qquad\; \forall x_{0} \, \in [0,1],
\end{equation}
where we have used the superscript~$(1)$ to emphasise that this formula 
is valid when $K = 1$. For $K = 2$, we obtain
\begin{equation} 
\label{Paper03_drift_K_2}
\begin{aligned}
  &\frac{\partial^{2} \Phi_{0}^{(2)}}{\partial x_{0}^{2}}(\boldsymbol{x}) \\ 
& \quad = \frac{(1-x_{0})\Big[B(1-b_{0})(1 - (1-b_{0})(1-x_{0}))
(B(1-x_{0})+2) + b_{2}\left(2b_{1}+3b_{2}\right)\Big] 
+ B(4 - B^{2})}{(B(1-x_{0})+1)^{3}((1-b_{0})(1-x_{0})+2)},
\end{aligned}
\end{equation}
which is also strictly positive since $x_{0} \in [0,1]$, $b_{0} \in [0,1]$ 
and the mean germination time satisfies 
$B = b_{1} + 2b_{2} \leq 2(1-b_{0}) \leq 2$.

In these cases, we see from~\eqref{Paper03_drift_K_1} 
and~\eqref{Paper03_drift_K_2} that the SDE in 
Theorem~\ref{Paper03_weak_convergence_diffusion_constant_environment} 
has a positive drift, indicating that in this phase, conferring a seed bank
gives mutants an advantage over the wild type, in contrast to a neutral 
mutation not associated with a seed bank. Recall however, that in order to
compare the outcome of this phase to that in the case without a seed bank, 
we must compare the diffusion determined by
Equation~\eqref{Paper03_SDE_WF_constant_environment}
started from $\psi_B(y)$ to the 
neutral Wright-Fisher diffusion started from $y$ (for small $y$)
and $\psi_B(y)<y$. A natural 
question is whether there exists $K\in \mathbb{N}$ such that the drift in
the diffusion with seed bank is sufficiently positive 
to compensate for this disadvantage.

In order to answer this question, we could aim to find a general formula for 
the second order derivative $\frac{\partial^{2} \Phi_{0}}{\partial x_{0}^{2}}$ 
which holds for any $K \in \mathbb{N}$. 
Since in nature the majority of dormant seeds do not last more than a few 
generations (see~\cite[Section~7.IV]{baskin1998seeds}), and we obtain a 
linear  system that could be used to compute an explicit expression for any 
fixed $K$, it initially seems tractable to do this at least
for some biologically
relevant examples. 
However, 
as we already illustrated in Equation~\eqref{Paper03_drift_K_2}, the 
expressions rapidly become complicated and are not obviously informative
even for small values of $K$. 

To overcome this problem, we shall find an explicit bound on 
$\frac{\partial^{2} \Phi_{0}}{\partial x_{0}^{2}}$ that will hold for any 
$K \in \mathbb{N}$ and allow us to bound the fixation probability of the 
solution to~\eqref{Paper03_SDE_WF_constant_environment}. To better understand 
the heuristics behind our next results, we perform a Taylor expansion 
of the first coordinate $F_{0}$ of the flow field 
$\mathbf{F}$ in~\eqref{Paper03_flow_definition_WF_model} near the diagonal 
$\Gamma$. For sufficiently small $\varepsilon > 0$, and 
$\boldsymbol{x}$ within distance $\varepsilon$ of $\Gamma$,
\begin{equation} \label{Paper03_taylor_expansion_original_flow}
\begin{aligned}
    F_{0}(\boldsymbol{x}) & = \frac{(1 - x_{0})
\left(\sum_{i = 1}^{K} b_{i}x_{i}- (1 - b_{0})x_{0}\right)}{\Big((1 - x_{0}) 
+ \sum_{i = 0}^{K} b_{i}x_{i}\Big)} \\ 
&  =  \frac{(1 - x_{0})\left(\sum_{i = 1}^{K} b_{i}x_{i}- (1 - b_{0})x_{0}
\right)}{1 + \left(\sum_{i = 1}^{K} b_{i}x_{i}- (1 - b_{0})x_{0}\right)} \\ 
& \approx (1 - x_{0})\Bigg(\left(\sum_{i = 1}^{K} b_{i}x_{i}- (1 - b_{0})x_{0}
\right) - \left(\sum_{i = 1}^{K} b_{i}x_{i}- (1 - b_{0})x_{0}\right)^{2} 
+ \mathcal{O}(\varepsilon^{3}) \Bigg).
\end{aligned}
\end{equation}
Ignoring the quadratic and higher order terms that appear in the Taylor 
expansion of $F_{0}$, we can then define a flow field 
$\mathbf{G}: [0,1]^{K+1} \rightarrow \mathbb{R}^{K+1}$ by
\begin{equation}
\label{Paper03_modified_flow_definition_WF_model}
    \left\{
    \arraycolsep=1.4pt\def\arraystretch{2.2}
    \begin{array}{cl} G_{0}(\boldsymbol{x}) & \defeq (1 - x_{0})\left(\sum_{i = 1}^{K} b_{i}x_{i}- (1 - b_{0})x_{0}\right), \\
    G_{i}(\boldsymbol{x}) & \defeq x_{i - 1} - x_{i} \quad \forall \, i \in \llbracket K \rrbracket.
    \end{array}\right.
\end{equation}
Note that if solutions to the ODE with flow field $\mathbf{F}$ converge to 
$\Gamma$, then so too do the solutions to the ODE with flow field given by 
$\mathbf{G}$. Let $\boldsymbol{\Phi}^{G}$ be the associated projection map, 
defined in a way analogous to the map~$\boldsymbol{\Phi}$ 
in~\eqref{Paper03_definition_projection_map}. 
On $\Gamma$, 
\[
\frac{\partial \Phi_0}{\partial x_0}
= \frac{\partial \Phi^{G}_0}{\partial x_0}.
\]
Since the quadratic term that 
appears within the bracket on the right hand side of the
last line in~\eqref{Paper03_taylor_expansion_original_flow} is negative, 
close to $\Gamma$, $F_0(\boldsymbol{x})<G_0(\boldsymbol{x})$, and so we expect
the drift in the flow corresponding to $\mathbf{F}$ to be dominated
by that in the flow corresponding to $\mathbf{G}$. Comparing the 
corresponding expressions in the 
SDE~\eqref{Paper03_SDE_WF_constant_environment}, this suggests that
\begin{equation*}
    \frac{\partial^{2} \Phi_{0}}{\partial x_{0}^{2}}(\boldsymbol{x}) 
\leq \frac{\partial^{2} \Phi^{G}_{0}}{\partial x_{0}^{2}}(\boldsymbol{x}) 
\qquad\quad \forall \, \boldsymbol{x} \in \Gamma \textrm{ and } 
\forall \, K \in \mathbb{N}.
\end{equation*}
In Section~\ref{Paper03_derivatives_specific_ode}, we find an explicit 
expression for $\frac{\partial^{2} \Phi^{G}_{0}}{\partial x_{0}^{2}}$ and 
analyse the eigenvalues of the solution of a matrix Lyapunov equation in
order to confirm our intuition and establish our next proposition. 
Recall the definition of the mean germination time $B$ 
from~\eqref{Paper03_mean_age_germination}. 
\begin{proposition} 
\label{Paper03_bound_drift_general_K_proposition}
For any $K \in \mathbb{N}$ and 
$\mathbf{b} = (b_{i})_{i \in \llbracket K \rrbracket_{0}}$ 
satisfying~\eqref{Paper03_assumption_probability_distribution_bi}, 
the second order derivative $\frac{\partial^{2} \Phi_{0}}{\partial x_{0}^{2}}$ 
of the projection map~$\boldsymbol{\Phi}$ satisfies the following conditions.
    \begin{enumerate}[(i)]
        \item The map 
$x \in [0,1] \mapsto \displaystyle 
\frac{\partial^{2}\Phi_{0}}{\partial x_{0}^{2}}(x, x, \cdots, x)$ 
is strictly increasing.
        \item For any $\boldsymbol{x} \in \Gamma$, we have
         \begin{equation*}
          \frac{\partial^{2} \Phi_{0}}{\partial x_{0}^{2}}(\boldsymbol{x}) 
\leq \frac{B(B(1-x_{0})+2)}{(B(1-x_{0})+1)^{3}}.
    \end{equation*}
    \item For $\boldsymbol{1} \defeq (1,1, \cdots,1)$, we have 
$\displaystyle 
\frac{\partial^{2} \Phi_{0}}{\partial x_{0}^{2}}(\boldsymbol{1}) = 2B$.
    \end{enumerate}   
\end{proposition}

Recall that we are going to approximate the probability of fixation of the
mutation by the probability that the 
diffusion~\eqref{Paper03_SDE_WF_constant_environment} 
hits $1$ before it hits $0$, and that we defined the map
$y\mapsto \psi_b(y)$ of 
Lemma~\ref{Paper03_lemma_correspondence_value_branching_phase}
in order to make a meaningful comparison to
the fixation probability of a neutral mutation that does not confer the 
seed bank property.
%and that 
%we take as initial condition $x_0(0)=\psi_B(y)$ for some $y\in (0,1)$.
%in order to compare to 
%a neutral mutation that does not confer the seed bank property. 
%. 
%We will approximate the overall probability of fixation by that of the diffusion phase. Recall that for any mean germination time $B \geq 0$, the map $\psi_B: [0, 1] \rightarrow [0, \infty)$ defined in Lemma~\ref{Paper03_lemma_correspondence_value_branching_phase} is such that for sufficiently small $\varepsilon \in (0,1)$, the probability that by the end of the branching phase the mutant with seed banks reaches a proportion $\psi_B(y)$ of the total population, for some $y \in (0, \varepsilon]$, is close to the probability that a neutral allele without the seed bank trait reaches a proportion $y$ of the total population by the end of its corresponding branching phase. To compare the fixation probability of the mutant with the seed bank trait with the fixation probability of a neutral allele that does not carry the seed bank trait, we will start the diffusion phase from the initial condition~$x_{0}(0) = \psi_B(y)$, for some $y \in (0,1)$. We emphasize that this choice of initial condition is illustrative as we are ignoring the ‘lag' phase that separates the branching phase from the diffusive or extinction of the mutant type. 

Let $(x_{0}(t))_{t \geq 0}$ be the solution of 
SDE~\eqref{Paper03_SDE_WF_constant_environment}. 
For any $a \in [0,1]$, we define the stopping time
\begin{equation} 
\label{Paper03_stopping_time_border}
    T_{a} \defeq \inf\left\{t \geq 0: \; x_{0}(t) = a\right\}.
\end{equation}
We seek the probability that
%Then, the mutant with the seed bank property fixes in the population if and 
%only if 
$T_{1} < T_{0}$ when $x_{0}(0)=\psi_B(y)$. 
To introduce our result, we define the map 
%\textcolor{red}{I am not convinced that the notation $\Psi$ is great - on a 
%poor printout it looks a lot like $\psi$.}
$\Psi: [0, \infty) \times [0, 1) \rightarrow [0,1]$ given by
\begin{equation} 
\label{Paper03_identity_fixation_probability_bound}
    \Psi(B,y) \defeq  1 - e^{-B\psi_B(y)} + \psi_B(y)e^{-B\psi_B(y)}.
\end{equation}

\begin{theorem} \label{Paper03_bound_fixation_probability_any_K}
    For any $y \in (0,1)$, the map $B\in [0, \infty) \mapsto \Psi(B,y)$ is strictly decreasing, and it satisfies the following identities:
    \begin{equation*}
        \lim_{B \rightarrow 0} \Psi(B,y) = y \quad \textrm{and} 
\quad \lim_{B \rightarrow \infty} \Psi(B,y) = 0.
    \end{equation*}
Moreover, for any $y \in [0,1)$, $K \in \mathbb{N}$ and 
$\mathbf{b} \in [0,1]^{K+1}$ 
satisfying~\eqref{Paper03_assumption_probability_distribution_bi},
we have
    \begin{equation*}
        \mathbb{P}\left(T_{1} < T_{0} \Big\vert \, x_{0}(0) = \psi_B(y)\right) 
\leq \Psi(B,y).
    \end{equation*}
\end{theorem}
Recalling our specification of $\psi_B(y)$, and that for the neutral 
Wright-Fisher diffusion started from $y$ the probability of fixation is
exactly $y$, 
Theorem~\ref{Paper03_bound_fixation_probability_any_K} indicates that in a 
constant environment, the fixation probability of a mutation conferring 
a seed bank is lower than that of a neutral allele with no seed bank. 
In Figure~\ref{Paper03_figure_K_2} we compare the estimate $\Psi(B,y)$ 
to the actual probability of fixation for the case $K=2$. 
Although the SDE~\eqref{Paper03_SDE_WF_constant_environment} has a positive 
drift, it is insufficient to compensate for the lower proportion of mutant 
individuals at the end of the branching phase.
The proof of 
Theorem~\ref{Paper03_bound_fixation_probability_any_K} is
in Section~\ref{Paper03_subsection_diffusion_WF_model_constant_environment}. 

In view of the prevalence of seed banks in nature, it is natural to 
investigate the impact of a varying environment. 
This will be the theme of the next two sections.

\begin{figure}[t]
\centering
\includegraphics[width=\linewidth,trim=0cm 0cm 0cm 0cm,clip=true]{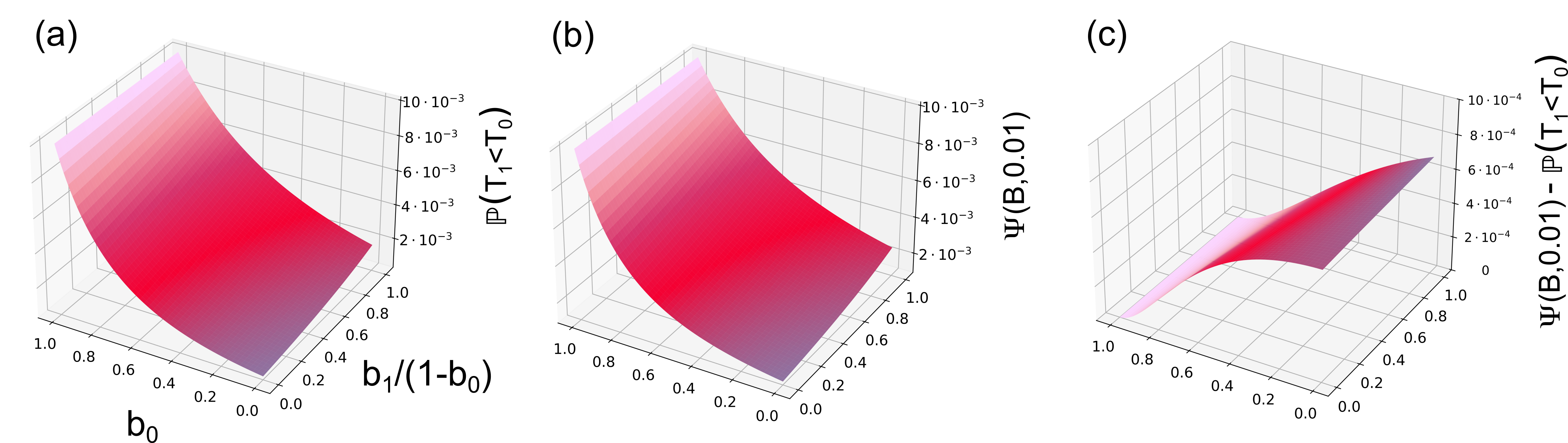}
\caption{
Comparison of the fixation probability $\mathbb P(T_1 < T_0 )$ of the 
SDE~\eqref{Paper03_SDE_WF_constant_environment} in~(a) with the upper
bound provided by the  
map $\Psi(B,0.01)$ in~(b) for the case $K = 2$, as a function of the 
parameters $b_0$ and $b_1/(1 - b_0)$. The fixation probability  was 
computed via the scale function of the the SDE using the
formula~\eqref{Paper03_drift_K_2}, and starting the diffusion phase 
from $x_0 = \psi_B(0.01)$, while $\Psi(B,0.01)$ was computed using the
identity~\eqref{Paper03_identity_fixation_probability_bound}.
Plot~(c) shows the difference $\Psi(B,0.01) - \mathbb P(T_1 < T_0)$. Note that the vertical axis in (c) has a different scale from those in (a) and (b).}
\label{Paper03_figure_K_2}
\end{figure}

\subsection{Wright-Fisher model in a slowly changing environment} 
\label{Paper03_subsection_WF_slow_environment}

We are going to suppose that as a result of changes in the environment,
population size is changing. 
We  modify the model introduced in 
Section~\ref{Paper03_WF_model_const_environ_description} 
to incorporate a population size that changes 
on the timescale of evolution.
For $N \in \mathbb{N}$ and $t \in \mathbb{N}_{0}$, the total number of 
individuals living at time $t$ will be given by 
$\lfloor \Xi^{N}(t)N \rfloor$, where 
$(\Xi^{N}(t))_{t \in \mathbb{N}_{0}}$ is a random process 
taking values in $\left[\xi_{\min}, \xi_{\max}\right]$, and 
$\xi_{\min}$ and $\xi_{\max}$ are strictly positive real numbers. 
We shall choose the dynamics of $(\Xi^N(t))$ in such a way 
that under 
our scaling of time, it will converge to the solution of an SDE.
Roughly speaking, $(\Xi^{N}(t))_{t \in \mathbb{N}_{0}}$ 
is supposed to capture changes in the environment that happen on the 
timescale of evolution, such as slow changes in the climate, and host-parasite 
interactions~\cite{verin2018host, davis2005evolutionary}. 

Let $M > 0$ be 
some large parameter. Informally, the population will evolve as follows.  
Each wild type individual in generation $t$ contributes $\Poiss (M)$ seeds, 
and for $i\in \llbracket K \rrbracket_0$ each mutant individual in 
generation $t - i$ contributes 
$\Poiss (Mb_i)$ seeds
to a pool that is ready to germinate in the current generation.
To determine the genetic composition of generation $t + 1 > K$, we 
first sample $\Xi^{N}(t+1)$, and then 
sample $\lfloor\Xi^{N}(t+1)N \rfloor$ seeds from the pool 
to germinate and grow into mature plants.

%
%In contrast to the model of Section~\ref{Paper03_WF_model_const_environ_description}, the size of the population of mature plants changes across time. 
We must also describe the evolution of the process 
$(\Xi^{N}(t))_{t \in \mathbb{N}_{0}}$. We assume that the composition of 
the population in generation $t$ does not affect the total population size 
in generation $t+1$. Moreover, since we are trying to model changes in 
population size that happen on the timescale of evolution, we assume that 
(after proper rescaling of time) as $N \rightarrow \infty$, 
$(\Xi^N)_{N \in \mathbb{N}}$ converges weakly to a diffusion.

\begin{assumption}[Diffusion limit of the slowly changing environment] 
\label{Paper03_diffusion_limit_slow_environment}
There exists a random variable $\xi_{0}$ taking values in 
$[\xi_{\min}, \xi_{\max}]$ such that $\Xi^{N}(0)$ converges weakly to 
$\xi_{0}$ as $N \rightarrow \infty$. Moreover, there exist $\delta > 0$ and 
continuous functions 
$\alpha, \eta: [\xi_{\min}, \xi_{\max}] \rightarrow \mathbb{R}$ satisfying
    \begin{enumerate}[(i)]
        \item $\alpha(\xi_{\min}) \geq 0$ and $\alpha(\xi_{\max}) \leq 0$,
        \item $\eta(\xi_{\min}) = \eta(\xi_{\max}) = 0$,
    \end{enumerate}
    such that for every $N \in \mathbb{N}$, there exists functions 
$w^N_1,w^N_2,w^N_3: [\xi_{\min}, \xi_{\max}] \rightarrow \mathbb{R}$ 
such that 
    \begin{equation*}
        \sup_{\xi \in [\xi_{\min},\xi_{\max}]} 
\; \left(\vert w^N_1(\xi) \vert \vee \vert w^N_2(\xi) \vert \right) 
= \mathcal{O}(N^{-1-\delta}), \qquad %\textrm{and} \quad 
\sup_{\xi \in [\xi_{\min}, \xi_{\max}]} \, \vert w^N_3(\xi) \vert = \mathcal{O}(N^{-3/2}),
        \end{equation*}
    and for all $t \in \mathbb{N}_{0}$ and $\xi \in [\xi_{\min}, \xi_{\max}]$, 
the process $\Xi^{N}$ satisfies the following identities:
    \begin{equation*}
    \begin{aligned}
    \mathbb{E}\left[\Xi^{N}(t + 1) - \Xi^{N}(t) \, \vert \, \Xi^{N}(t) 
= \xi\right] & = \frac{\alpha(\xi)}{N} + w^N_1(\xi), \\
\mathbb{E}\left[(\Xi^{N}(t + 1) - \Xi^{N}(t))^{2} \, \vert \, \Xi^{N}(t) 
= \xi\right] & = \frac{\eta^{2}(\xi)}{N} + w^N_2(\xi), \\ 
\textrm{and} 
\quad \mathbb{E}\left[\vert\Xi^{N}(t + 1) - \Xi^{N}(t)\vert^{3} \, 
\vert \, \Xi^{N}(t) = \xi\right] & = w^N_3(\xi).
    \end{aligned}
    \end{equation*}
\end{assumption}

\begin{remark}
%In Assumption~\ref{Paper03_diffusion_limit_slow_environment}, we suppose that
%the process 
By assuming that $(\Xi^N(t))_{t \in \mathbb{N}_0}$ is uniformly bounded away 
from $0$, we avoid the possibility of dramatic 
bottlenecks (see e.g.~\cite{pra2025multi} for a Wright-Fisher model with 
bottlenecks). We suppose $(\Xi^N(t))_{t \in \mathbb{N}_0}$ is uniformly 
bounded from above to avoid explosions in population size. In particular, 
Assumption~\ref{Paper03_diffusion_limit_slow_environment} implies that 
$(\Xi^N)_{N \in \mathbb{N}}$ converges (after time rescaling) to a diffusion 
on $[\xi_{\min}, \xi_{\max}]$
(that may be absorbed at either endpoint).
%either at $\xi_{\min}$ or $\xi_{\max}$.
\end{remark}

Our formal description of the evolution of 
$(\boldsymbol{X}^{N}(t), \Xi^{N}(t))_{t \in \mathbb{N}_{0}}$ mirrors
that in the case of a constant population size.
As before, $X^{N}_{i}(t)$ will denote the number of mutant individuals in 
generation~$t - i$ for $i \in \llbracket K \rrbracket_0$. 
We define the Markov process 
$(\boldsymbol{X}^{N}(t), \Xi^{N}(t))_{t \in \mathbb{N}_{0}}$
as follows. 
For any $t \in \mathbb{N}_{0}$, and for any given configuration 
$(\boldsymbol{X}^{N}(t), \Xi^{N}(t))$: 
\begin{enumerate}[(i)]
    \item The conditional distribution of $\Xi^{N}(t+1)$ given 
$(\boldsymbol{X}^{N}(t), \Xi^{N}(t))$ is the same as the 
conditional distribution of $\Xi^{N}(t+1)$ given $\Xi^{N}(t)$, 
and it satisfies Assumption~\ref{Paper03_diffusion_limit_slow_environment};
    \item The conditional distribution of $X^{N}_{0}(t + 1)$ given 
$(\boldsymbol{X}^{N}(t), \Xi^{N}(t), \Xi^{N}(t + 1))$ 
is given by
    \begin{equation} 
\label{Paper03_binomial_distribution_slowly_changing_environment}
        \textrm{Bin}\left(\lfloor \Xi^{N}(t+1)N \rfloor, \, 
\frac{\sum_{i = 0}^{K}b_{i}X^{N}_{i}(t)}{(\lfloor \Xi^{N}(t)N 
\rfloor - X^{N}_{0}(t)) 
+ \sum_{i = 0}^{K}b_{i}X^{N}_{i}(t)}\right).
    \end{equation}
\end{enumerate}
In~\eqref{Paper03_binomial_distribution_slowly_changing_environment}, 
$\lfloor \Xi^{N}(t+1)N \rfloor$ is the number of mature individuals in 
generation $t+1$, 
%THIS IS WRONG$\sum_{i=0}^{K} b_i X_i^N(t)$ is the proportion of 
%germinating seeds 
%produced by mutants, 
and $\lfloor \Xi^{N}(t)N \rfloor - X_0^N(t)$ is
the number of mature wild type individuals in generation $t$.
%, or, 
%equivalently, the proportion of germinating seeds produced by the wild type. 
For each $N \in \mathbb{N}$, we consider the scaled process 
$(\boldsymbol{x}^{N}(t), \xi^{N}(t))_{t \geq 0}$ given by
\begin{equation} 
\label{Paper03_rescaling_processes_slowly_changing_environment}
    \boldsymbol{x}^{N}(t) \defeq 
\frac{\boldsymbol{X}^{N}(\lfloor Nt \rfloor)}{N}, \quad 
\textrm{ and } \quad \xi^{N}(t) \defeq \Xi^{N}(\lfloor Nt \rfloor).
\end{equation}
We will determine the behaviour of the sequence of processes 
$(\boldsymbol{x}^{N}, \xi^{N})_{N \in \mathbb{N}}$ as $N \rightarrow \infty$. 
Note that by Assumption~\ref{Paper03_diffusion_limit_slow_environment} and 
classical arguments (e.g.~\cite[Theorem~10.1.1]{ethier2009markov}), the 
sequence of processes $(\xi^{N})_{N \in \mathbb{N}}$ converges in 
distribution in 
$\mathscr{D}\left(\mathbb{R}_{+}, [\xi_{\min}, \xi_{\max}]\right)$ 
to the solution $(\xi(t))_{t \geq 0}$ of the SDE
\begin{equation} 
\label{Paper03_SDE_slow_environment}
    d\xi = \alpha(\xi) \, dt + \eta(\xi) \, dW_{\textrm{env}}(t),
\end{equation}
where $(W_{\textrm{env}}(t))_{t \geq 0}$ is a standard Brownian motion. 
We would then expect that the sequence of stochastic processes 
$(\boldsymbol{x}^{N}, \xi^{N})_{N \in \mathbb{N}}$ should 
converge as $N \rightarrow \infty$ to a coupled system of SDEs. 
Before stating a formal result, we briefly explain some of the heuristics. 
Let $\{\mathcal{F}^N_t\}_{t \geq 0}$ be the natural filtration of the process 
$(\boldsymbol{x}^N(t), \xi^N(t))_{t \geq 0}$. Observe that by the dynamics 
of $(\boldsymbol{x}^{N}, \xi^{N})_{t \geq 0}$ and 
Assumption~\ref{Paper03_diffusion_limit_slow_environment}, we have, 
for all $N \in \mathbb{N}$ and $t \in [0, + \infty)$,
\begin{equation} 
\label{Paper03_source_ODE_slow_environment_env_coordinate}
\begin{aligned}
    \mathbb{E}\left[\xi^{N}\left(t + \tfrac{1}{N}\right)- \xi^{N}(t) 
\Big\vert \mathcal{F}^N_t\right] = \frac{\alpha(\xi^{N}(t))}{N} 
+ \mathcal{O}(N^{-1-\delta}).
\end{aligned}
\end{equation}
The expected increment in the number of mutants living in the 
mature population is given by
\begin{equation} 
\label{Paper03_source_ODE_slow_environment_first_coordinate}
\begin{aligned}
    & \mathbb{E}\left[x^{N}_{0}\left(t + \tfrac{1}{N}\right) - x^{N}_{0}(t) 
\Big\vert \, \sigma\left(\mathcal{F}^N_t, \xi^{N}\left(t + \tfrac{1}{N}\right) 
\right) \right]  \\[2ex] 
& \quad  = \frac{\xi^{N}\left(t + \frac{1}{N}\right)\sum_{i = 0}^{K} 
b_{i}x^{N}_{i}(t)}{\left(\xi^{N}(t) - x^{N}_{0}(t)\right) 
+ \sum_{i = 0}^{K} b_{i}x^{N}_{i}(t)} - x^{N}_{0}(t) \\[2ex] 
& \quad = \frac{\left[\xi^{N}\left(t + \frac{1}{N}\right)- \xi^{N}(t)\right]
\sum_{i = 0}^{K} b_{i}x^{N}_{i}(t)}{\left(\xi^{N}(t) - x^{N}_{0}(t)\right) 
+ \sum_{i = 0}^{K} b_{i}x^{N}_{i}(t)} 
+ \frac{\left(\xi^{N}(t) - x^{N}_{0}(t)\right)
\left(\sum_{i = 1}^{K} b_{i}x^{N}_{i}(t) - (1 - b_{0})x^{N}_{0}(t)\right)}
{\left(\xi^{N}(t) - x^{N}_{0}(t)\right) + \sum_{i = 0}^{K} b_{i}x^{N}_{i}(t)},
\end{aligned}
\end{equation}
while $x^{N}_{i}$ is given 
by~\eqref{Paper03_source_ODE_other_coordinates} for 
$i \in \llbracket K \rrbracket$. 

As in the case of constant environment, we are going to see a separation of 
timescales. The stochastic process $\boldsymbol{x}^N(t)$
will fluctuate around the deterministic flow determined by 
the ODE arising as a limit 
of~\eqref{Paper03_source_ODE_slow_environment_first_coordinate} with
time accelerated by the factor $N$.
Comparing to~\eqref{Paper03_source_ODE_first_coordinate}, 
the expression~\eqref{Paper03_source_ODE_slow_environment_first_coordinate}
has an extra term, corresponding to the increment in the population size over a 
single generation. To handle this we will need a more general
version of Katzenberger's Theorem.
Equation~\eqref{Paper03_source_ODE_slow_environment_env_coordinate} %, which 
tells us that this increment %in the population size over a single 
%generation 
is order $N^{-1}$, and looking ahead 
to~\eqref{Paper03_stochastic_process_general_discrete_N}
we see that this can 
be incorporated into the `well-behaved' semimartingale $z^N$ in
Katzenberger's method, and
%allows us to infer that in the fast timescale 
%the increment in population size over a single generation will be neglible.
%As a result,
only the second term on the right-hand side 
of~\eqref{Paper03_source_ODE_slow_environment_first_coordinate} should be 
included in the (accelerated) ODE. 
In other words, 
with a slowly changing environment, we should consider the flow field 
$\mathbf{F}^{\textrm{sl,env}} 
= (F^{\textrm{sl,env}}_0, \cdots, 
F^{\textrm{sl,env}}_K, F^{\textrm{sl,env}}_\textrm{env}) : 
[0, \xi_{\max}]^{K+1} \times [\xi_{\min}, \xi_{\max}] \rightarrow \mathbb{R}^{K+2}$ 
given, for all 
$(\boldsymbol{x}, \xi) \in [0, \xi_{\max}]^{K+1} 
\times [\xi_{\min}, \xi_{\max}]$ with $x_{0} \leq \xi$, by
\begin{equation}
\label{Paper03_flow_definition_WF_model_slowly_changing_environment}
    \left\{
    \arraycolsep=1.4pt\def\arraystretch{2.2}
    \begin{array}{cl} 
    F^{\textrm{sl,env}}_{0}( \boldsymbol{x}, \xi) 
& \defeq \displaystyle{\frac{(\xi - x_{0})
\left(\sum_{i = 1}^{K} b_{i}x_{i}- (1 - b_{0})x_{0}\right)}{(\xi - x_{0}) 
+ \sum_{i = 0}^{K} b_{i}x_{i}}}, \\
    F^{\textrm{sl,env}}_{i}(\boldsymbol{x}, \xi) & \defeq x_{i - 1} - x_{i} 
\quad \forall \, i \in \llbracket K \rrbracket, \\ 
F^{\textrm{sl,env}}_{\textrm{env}}(\boldsymbol{x}, \xi) & \defeq 0.
    \end{array}\right.
\end{equation}
Comparing~\eqref{Paper03_flow_definition_WF_model_slowly_changing_environment} 
to the flow field 
$\mathbf{F} = (F_i)_{i \in \llbracket K \rrbracket_0}: [0,1]^{K+1} 
\rightarrow \mathbb{R}$ 
given by~\eqref{Paper03_flow_definition_WF_model}, 
which defines the ODE arising from the Wright-Fisher model with a constant 
population size, 
we observe 
that for any $i \in \llbracket K \rrbracket_{0}$,
\begin{equation} 
\label{Paper03_scaling_time_change_slow_environment}
    F^{\textrm{sl,env}}_{i}(\xi, \boldsymbol{x}) = \xi F_{i}\left(\frac{\boldsymbol{x}}{\xi}\right),
\end{equation}
and, therefore, when viewed as acting on $\boldsymbol{x}/\xi$,
the flow given 
in~\eqref{Paper03_flow_definition_WF_model_slowly_changing_environment} is 
a time change of that in the constant environment case.
Recall that $\Gamma$ is the diagonal of the $(K+1)$-dimensional hypercube, 
and the attractor manifold associated to the ODE with flow given by 
$\mathbf{F}$. For any $\xi \in [\xi_{\min}, \xi_{\max}]$, 
let $\xi \Gamma \defeq \{\xi \boldsymbol{x}: \; 
\boldsymbol{x} \in \Gamma\}$, and set
\begin{equation} 
\label{Paper03_attractor_manifold_slow_changing_environment}
    \Gamma^{\; \textrm{sl,env}} \defeq \left\{(\boldsymbol{x}, \xi) 
\in [0, \xi_{\max}]^{K+1} \times [\xi_{\min}, \xi_{\max}]: \; 
\boldsymbol{x} \in \xi \Gamma\right\}.
\end{equation}
Analogous to the definition of $\boldsymbol{\Phi}$ 
in~\eqref{Paper03_definition_projection_map},
we define
\[
\boldsymbol{\Phi}^{\textrm{sl,env}}= 
(\Phi^{\textrm{sl,env}}_0, \cdots, \Phi^{\textrm{sl,env}}_K, 
\Phi^{\textrm{sl,env}}_\textrm{env}): \Big\{(\boldsymbol{x}, \xi) 
\in [0, \xi_{\max}]^{K+1} \times [\xi_{\min}, \xi_{\max}]: \; 
x_0 \leq \xi\Big\} \rightarrow \Gamma^{\textrm{sl,env}}
\] 
to be the projection map associated with the ODE corresponding to the 
flow field $\mathbf{F}^{\textrm{sl,env}}$. 
Since the population size changes occur on the timescale of evolution, 
they are not ‘felt' by the accelerated ODE, so
%i.e.~with respect to the ODE associated to the flow $\mathbf{F}^{\textrm{sl,env}}$ defined in~\eqref{Paper03_flow_definition_WF_model_slowly_changing_environment},we conclude that
\begin{equation} 
\label{Paper03_bijection_projection_maps_DEs_environment_coordinate}
    \Phi^{\textrm{sl,env}}_{\textrm{env}}(\boldsymbol{x}, \xi) 
= \xi + \int_{0}^{\infty} F^{\textrm{sl,env}}_{\textrm{env}}
(\boldsymbol{\phi}(t)) \, dt = \xi,
\end{equation}
where $(\boldsymbol{\phi}(t))_{t \geq 0}$ is the flow started 
from~$(\boldsymbol{x}, \xi)$.
Combining~\eqref{Paper03_scaling_time_change_slow_environment} 
and~\eqref{Paper03_definition_projection_map}, we conclude that 
for all $i \in \llbracket K \rrbracket_{0}$,
\begin{equation} 
\label{Paper03_bijection_projection_maps_DEs}
    \Phi^{\textrm{sl,env}}_{i}(\boldsymbol{x}, \xi) 
= \xi \Phi_{i}\left(\frac{\boldsymbol{x}}{\xi}\right).
\end{equation}
In particular, the ODE flow converges to $\Gamma^{\; \textrm{sl,env}}$.
We shall %use identities~\eqref{Paper03_bijection_projection_maps_DEs} 
use~\eqref{Paper03_bijection_projection_maps_DEs_environment_coordinate}
% in our computations regarding $\boldsymbol{\Phi}$, 
and Katzenberger's method (explained in  
Section~\ref{Paper03_subsection_katzenberger_overview}) 
to derive the limiting SDE for the size of the mutant population. 
Let $\boldsymbol{\rho}: \Gamma^{\; \textrm{sl,env}} \rightarrow [0,1]^{K+1}$ 
%be the map 
represent 
the vector of proportions of mutant individuals in the population, i.e.~for
%the map given by, for 
any $(\boldsymbol{x}, \xi) \in \Gamma^{\; \textrm{sl,env}}$,
\begin{equation} 
\label{Paper03_definition_proportion_mutations}
    \boldsymbol{\rho}(\boldsymbol{x}, \xi) 
\defeq \displaystyle{\frac{\boldsymbol{x}}{\xi}}.
\end{equation}

\begin{theorem} 
\label{Paper03_weak_convergence_WF_model_slow_env}
    Suppose that $(\boldsymbol{x}^{N}(0),\xi^{N}(0))$ converges to 
$(x(0), \xi(0)) \in \Gamma^{\; \textrm{sl,env}}$ as $N \rightarrow \infty$, 
and that Assumption~\ref{Paper03_diffusion_limit_slow_environment} holds. 
Then the sequence of processes 
$\left((\boldsymbol{x}^{N}(t), \xi^{N}(t))_{t \geq 0}\right)_{N \in \mathbb{N}}$ 
converges in distribution in 
$\mathscr{D}\left([0, \infty), [0, \xi_{\max}]^{K+1} 
\times [\xi_{\min}, \xi_{\max}] \right)$ 
to a diffusion process $(x(t), \xi(t))_{t \geq 0}$ with sample paths 
in $\Gamma^{\; \textrm{sl,env}}$ and such that $(x_{0}(t))_{t \geq 0}$ 
satisfies the coupled system of stochastic differential equations
    \begin{equation} 
\label{Paper03_SDE_slow_environment_formal}
        \xi(T) = \xi(0) + \int_{0}^{T} \alpha(\xi) \, dt 
+ \int_{0}^{T} \eta(\xi) \, dW_{\textrm{env}}(t),
    \end{equation}
    and
    \begin{equation} 
\label{Paper03_SDE_slow_environment_number_mutants}
    \begin{aligned}
        x_{0}(T) & = \; x_{0}(0) + \frac{1}{2} \int_{0}^{T} 
\frac{\partial^{2} \Phi_{0}}{\partial \rho_{0}^{2}}
\left(\boldsymbol{\rho}\right) \left(\frac{x_{0}(\xi - x_{0}) 
+ \eta^{2}(\xi)x_{0}^{2}}{\xi^{2}} \right) \, dt \\[2ex] 
& \quad \quad  +
       \int^{T}_{0} \frac{x_{0}\alpha(\xi)}{(B(\xi - x_{0}) + \xi)} \, dt
        - \int^{T}_{0} \frac{Bx^{2}_{0}\eta^{2}(\xi)}{\xi(B(\xi - x_{0})
+\xi)^{2}} \, dt \\[2ex] 
& \quad \quad + \int_{0}^{T} 
\frac{\sqrt{\xi x_{0}(\xi - x_{0})}}{(B(\xi - x_{0}) + \xi)} \, dW_{0}(t) 
+ \int_{0}^{T} \frac{x_{0}\eta(\xi)}{(B(\xi - x_{0}) + \xi)} \, 
dW_{\textrm{env}}(t),
    \end{aligned}
    \end{equation}
where $W_{0}$ and $W_{\textrm{env}}$ are two independent 
standard Brownian motions.
\end{theorem}

In the model with constant population size, it is natural to work with the number
of mutants and to consider derivatives with respect to $x_0$.
When the population size varies, however, the relevant state variable is the
proportion of mutants.
Accordingly, in~\eqref{Paper03_SDE_slow_environment_number_mutants}, the quantity
\[
\frac{\partial^{2} \Phi_{0}}{\partial \rho_{0}^{2}}(\boldsymbol{\rho})
\]
denotes the second-order partial derivative of $\Phi_{0}$ with respect to the
first coordinate $\rho_0$ of the vector $\boldsymbol{\rho} \in \Gamma$.
With this interpretation,
$\frac{\partial^{2} \Phi_{0}}{\partial \rho_{0}^{2}}(\boldsymbol{\rho})$
satisfies the assumptions of
Proposition~\ref{Paper03_bound_drift_general_K_proposition}.
Note also that the independent Brownian motions $W_{0}$ and $W_{\textrm{env}}$ 
in~\eqref{Paper03_SDE_slow_environment_formal} 
and~\eqref{Paper03_SDE_slow_environment_number_mutants} reflect the different 
sources of noise: $W_{0}$ arises from genetic drift, whilst 
$W_{\textrm{env}}$ represents the noise arising from the evolution of the 
population size.
The proof of Theorem~\ref{Paper03_weak_convergence_WF_model_slow_env} will be 
given in Section~\ref{Paper03_subsection_diffusion_WF_model_slow_environment}. 

To understand how fluctuations in the population size affect the probability of 
fixation of the mutant type, it is convenient to express the SDE 
in~\eqref{Paper03_SDE_slow_environment_number_mutants} in terms of 
${\rho}_0$, rather than in terms of $x_{0}$. Since $(\xi(t))_{t \geq 0}$ is 
%almost surely 
bounded from below by $\xi_{\min} > 0$, we can obtain an SDE for 
$\boldsymbol{\rho}$ from an application of Itô's formula.
% to the definition of $\boldsymbol{\rho}$ (see for instance~\cite[Theorem~3.6]{karatzas1998brownian}) and derive an SDE for the proportion of mutants in the population. This is our next result.

\begin{corollary} 
\label{Paper03_corollary_SDE_proportion_slowly_varying_environment}
Let $(\xi(t),\boldsymbol{x}(t))_{t \geq 0}$ be a diffusion process 
in $\Gamma^{\textrm{sl,env}}$ satisfying the coupled system of 
SDEs~\eqref{Paper03_SDE_slow_environment_formal} 
and~\eqref{Paper03_SDE_slow_environment_number_mutants}, and for 
$T\geq 0$ let 
$\boldsymbol{\rho}(T) \defeq \boldsymbol{\rho}(\boldsymbol{x}(T), \xi(T))$ be given as in~\eqref{Paper03_definition_proportion_mutations}.
Then, the pair of stochastic processes 
$(\xi(t), \boldsymbol{\rho}(t))_{t \geq 0}$ satisfies the coupled system of 
stochastic differential equations
     \begin{equation*}
        \xi(T) = \xi(0) + \int_{0}^{T} \alpha(\xi) \, dt 
+ \int_{0}^{T} \eta(\xi) \, dW_{\textrm{env}}(t),
    \end{equation*}
    and
    \begin{equation} 
\label{Paper03_SDE_proportion_mut_WF_fluctuating_slow_env}
    \begin{aligned}
        \rho_{0}(T) & = \; \rho_{0}(0) + \frac{1}{2} \int_{0}^{T} 
\frac{\partial^{2} \Phi_{0}}{\partial \rho_{0}^{2}}
\left(\boldsymbol{\rho}\right) 
\left(\frac{\rho_{0}(1 - \rho_{0})}{\xi(t)} \right) \, dt -
     \int^{T}_{0} 
\frac{B\rho_{0}(1 - \rho_{0})\alpha(\xi(t))}{(B(1 - \rho_{0}) + 1)\xi(t)} \, 
dt \\[2ex] 
& \quad \quad  
        + \int^{T}_{0} 
\left[\frac{B \rho_{0}(1- \rho_{0}) \eta^{2}(\xi(t))}
{(B(1 - \rho_{0})+1)\xi^{2}(t)}
+\frac{\rho^{2}_{0}\eta^{2}(\xi(t))}{2\xi(t)}\left(\frac{\partial^{2} 
\Phi_{0}}{\partial \rho_{0}^{2}}\left(\boldsymbol{\rho}\right)  
- \frac{2B}{(B(1 - \rho_{0})+1)}\right)\right] \, dt \\[2ex] 
& \quad \quad + \int_{0}^{T} 
\frac{\sqrt{\rho_{0}(1 - \rho_{0})}}{(B(1 - \rho_{0}) + 1)\sqrt{\xi(t)}} \, 
dW_{0}(t) 
- \int_{0}^{T} \frac{B\rho_{0}(1-\rho_{0})\eta(\xi(t))}
{(B(1 - \rho_{0}) + 1)\xi(t)} \, dW_{\textrm{env}}(t),
    \end{aligned}
    \end{equation}
where $W_{0}$ and $W_{\textrm{env}}$ are two independent standard 
Brownian motions.
\end{corollary}

We would like to describe the behaviour of the coupled stochastic 
differential equations given by 
Theorem~\ref{Paper03_weak_convergence_WF_model_slow_env} and 
Corollary~\ref{Paper03_corollary_SDE_proportion_slowly_varying_environment}. 
The analysis is not trivial, since there are unknown functions $\alpha$ and 
$\eta$ whose behaviour will influence the dynamics of the diffusion process. 
In order to glean some information about the impact of changes in 
the population size on the establishment of seed banks
%from a biological perspective, 
it is convenient to consider separately the cases in which the 
population size changes deterministically ($\eta\equiv 0$), and that in 
which it is stochastic ($\eta\not\equiv 0$).
%\begin{enumerate}[(i)]
%    \item The case when the population size changes in a deterministic 
%manner, i.e.~when $\eta \equiv 0$;
%    \item The case when the population size fluctuates on 
%the timescale of evolution, i.e.~$\eta \not\equiv 0$.
%\end{enumerate}

%We discuss this below.

\subsubsection{Case~$(i)$: $\eta \equiv 0$}

Suppose $\eta \equiv 0$, i.e.~that the process $\left(\xi(t)\right)_{t \geq 0}$ evolves according to the ODE
\begin{equation*}
    \frac{d \xi}{dt} = \alpha(\xi).
\end{equation*}
By Corollary~\ref{Paper03_corollary_SDE_proportion_slowly_varying_environment},
the evolution of the proportion $\rho_{0}$ of mutants
in the population is given by
\begin{equation} 
\label{Paper03_SDE_number_mut_deterministic_env}
    \begin{aligned}
        \rho_{0}(T) 
& = \; \rho_{0}(0) + \frac{1}{2} \int_{0}^{T} 
\frac{\partial^{2} \Phi_{0}}{\partial \rho_{0}^{2}}
\left(\boldsymbol{\rho}\right) \left(\frac{\rho_{0}(1 - \rho_{0})}{\xi} \right)
 \, dt 
- \int^{T}_{0} \frac{B\rho_{0}(1 - \rho_{0})\alpha(\xi)}
{(B(1 - \rho_{0}) + 1)\xi} \, dt \\[2ex] 
& \quad \quad + \int_{0}^{T} \frac{\sqrt{\rho_{0}(1 - \rho_{0})}}
{(B(1 - \rho_{0}) + 1)\sqrt{\xi}} \, dW_{0}(t).
    \end{aligned}
    \end{equation}
Equation~\eqref{Paper03_SDE_number_mut_deterministic_env} suggests a 
non-trivial impact of the evolution of the carrying capacity on the 
probability of fixation of a mutation that produces a seed bank. 
The principal difference between the 
SDE~\eqref{Paper03_SDE_number_mut_deterministic_env} 
and that describing the dynamics of the proportion of mutants when the
population size is
constant is the second integral on the right-hand side 
of~\eqref{Paper03_SDE_number_mut_deterministic_env}. When $\alpha < 0$, 
i.e.~when the population size is declining, the presence of a 
seed bank becomes more advantageous. Heuristically, this is explained by 
the fact that if the population size is decreasing, then previous generations 
of mutant plants may contribute a higher proportion of seeds to the pool
that is ready to germinate in a given generation, due to the 
lower number of wild type individuals in the current generation with which
they are competing. We see the opposite effect when $\alpha > 0$, 
i.e.~when the population size is growing. Under this scenario, previous 
generations of mutant plants contribute a lower proportion of 
the seeds from which the next generation is sampled than they would 
if the population size were constant.
In other words, our results suggest 
that when the population size declines over time, for example due to 
adverse environmental conditions, the presence of (finite-age) seed banks 
provides an evolutionary advantage.
Since we are supposing that the change in population size over the 
time period that a seed spends in the seed bank is $\mathcal{O}(1/N)$, it is
at first sight somewhat surprising that this effect is noticeable, but
the significant impact of the accumulation of many small changes over
a large number of generations is of course familiar from population
genetics.

To better illustrate these ideas, we consider the following toy model.
Assume that $K=1$, that is, seeds can persist for at most one generation
in the seed bank, so that $\frac{\partial^2 \Phi_0}{\partial \rho^2}(\boldsymbol{\rho})$ is given by~\eqref{Paper03_drift_K_1}. Let $b_0 \in (0,1]$ denote the proportion of seeds that are not dormant. Moreover, suppose the dynamics of the (scaled) population size $(\xi(t))_{t\ge0}$
satisfies the logistic growth equation
\begin{equation} \label{Paper03:logistic_growth}
\frac{d\xi}{dt} = r\,\xi(\xi_\infty - \xi),
\end{equation}
where $r,\xi_\infty>0$. We further assume that $\xi(0)=1$.
Biologically, this model represents a scenario in which the environment is
constant during the branching phase, but the carrying capacity of the
environment (reflected by $\xi_\infty$) changes at the end of this phase
due to external factors (for example, climate change). Let $(\rho_0(t))_{t \geq 0}$ be the diffusion associated to~\eqref{Paper03_SDE_number_mut_deterministic_env}. As in~\eqref{Paper03_stopping_time_border}, for any $a \in [0,1]$, let $T_{a}$ be the stopping time at which $(\rho_0(t))_{t \geq 0}$ reaches level $a$. Recall from Lemma~\ref{Paper03_lemma_correspondence_value_branching_phase} that larger values of the mean germination time $B = 1 - b_0$ are associated to a lower proportion of mature mutant individuals by the end of the branching phase. We are interested in understanding the impact of different choices of $\xi_{\infty}$ and $b_0$ on the fixation probability of the seed bank trait, i.e.~the probability of the event $\{T_1 < T_0 \}$ (see Figure~\ref{Paper03_fig:fixation_prob_environment_slow}). Since the dynamics of $(\rho_0(t))_{t \geq 0}$ is not autonomous, we compute the fixation probability by numerically solving a backward Kolmogorov equation (see e.g.~\cite[Section~4.2]{ewens2004mathematical}). As discussed after~\eqref{Paper03_SDE_number_mut_deterministic_env}, our simulations suggest the following statements:
\begin{enumerate}[(1)]
    \item For $\xi_\infty > 1$, i.e.~when the population size exhibits growth, the presence of the seed bank is deleterious (see Figure~\ref{Paper03_fig:fixation_prob_environment_slow}).
    \item For $\xi_\infty < 1$, there is population decline, which favours the presence of dormancy. This effect, however, may not be sufficiently large to compensate the smaller proportion of mature mutant individuals by the end of the branching phase (see e.g.~the plot of the fixation probability as a function of $b_0$ when $\xi_\infty= 0.9$ in Figure~\ref{Paper03_fig:fixation_prob_environment_slow}). Nonetheless, if $\xi_{\infty}$ is sufficiently small, then we see an advantage for the seed bank trait (with respect to the evolution of a neutral trait without dormancy). Even in this scenario, the fixation probability does not necessarily decrease monotonically with $b_0$ (see the plot for $\xi_\infty= 0.8$ in Figure~\ref{Paper03_fig:fixation_prob_environment_slow}), which reflects the smaller proportion of mature mutant individuals by the end of the branching phase (see Lemma~\ref{Paper03_lemma_correspondence_value_branching_phase}). In this case, there exists an optimal value of $b_0 \in (0,1)$ that maximises the fixation probability.
\end{enumerate}

\begin{figure}[t]
\centering
\includegraphics[width=0.75\linewidth, trim=0cm 0cm 0cm 0cm,clip=true]{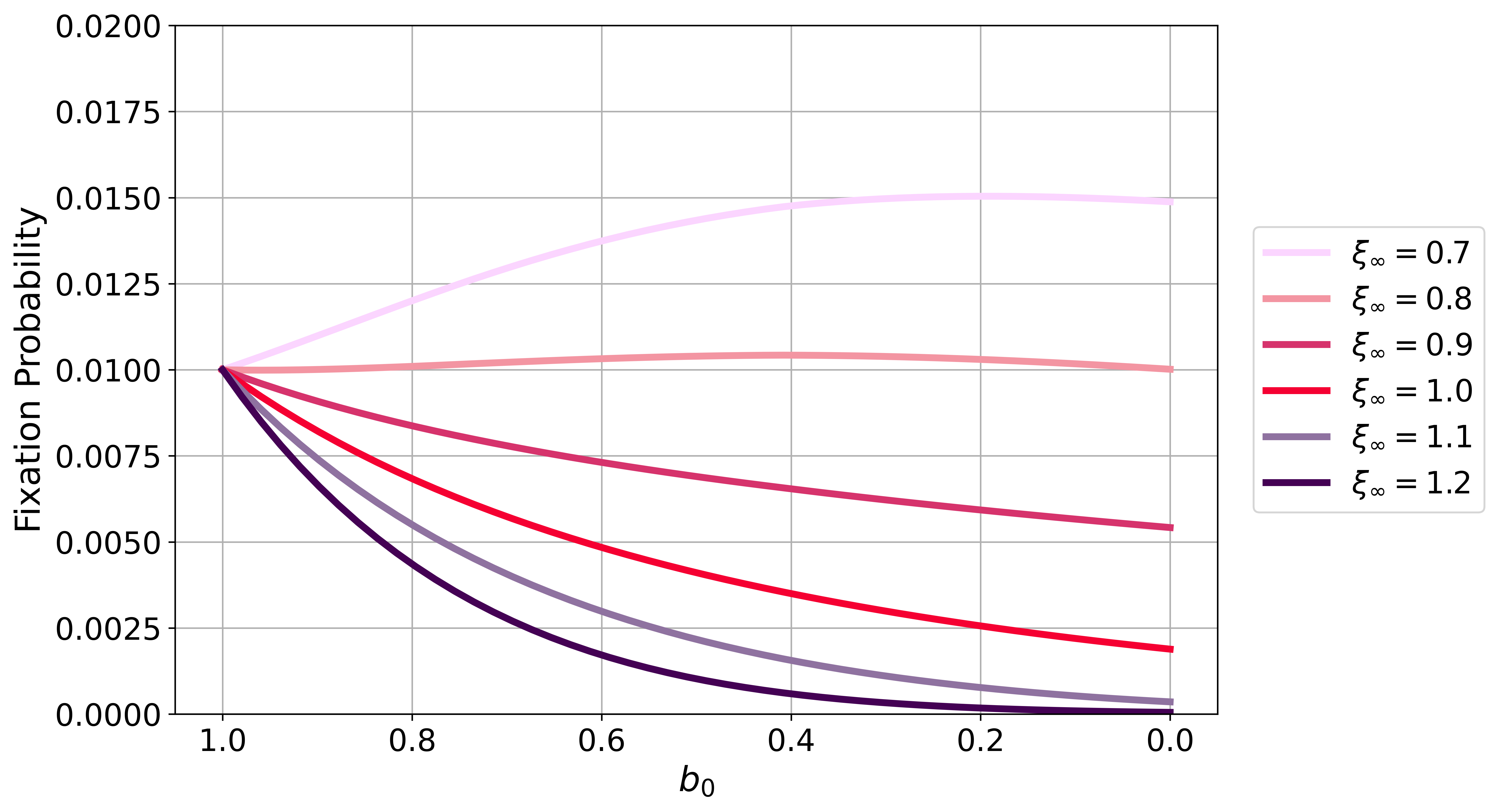}
\caption{\label{Paper03_fig:fixation_prob_environment_slow}~Fixation probability of $(\rho_0(t))_{t \geq 0}$ for different values of $\xi_{\infty} \in \{0.7, 0.8, 0.9, 1.0, 1.1, 1.2\}$ and $b_0 \in (0,1]$. We assume the dynamics of $(\xi(t))_{t \geq 0}$ is described by~\eqref{Paper03:logistic_growth} with $r = 20$, and that $(\rho_0(t))_{t \geq 0}$ solves the SDE~\eqref{Paper03_SDE_number_mut_deterministic_env} with $K = 1$, so that $\frac{\partial^2 \Phi_0}{\partial \rho^2}(\boldsymbol{\rho})$ is given by~\eqref{Paper03_drift_K_1}. To take into account the impact of dormancy on the branching phase (see Lemma~\ref{Paper03_lemma_correspondence_value_branching_phase}), we computed the fixation probability assuming that $\rho_0(0) = \psi_B(0.01)$ for all $B = 1 - b_0$, where $\psi_B$ is the function defined in Lemma~\ref{Paper03_lemma_correspondence_value_branching_phase}. Observe that in the absence of dormancy, i.e.~when $b_0 = 1$, the fixation probability of a neutral trait is not affected by the changes in population size. Our numerical simulations suggest that for sufficiently small values of $\xi_{\infty}$, i.e.~for sufficiently large declines in population size, the evolution of the seed bank trait is favoured with respect to the evolution of a neutral trait without dormancy.}
\end{figure}

\subsubsection{Case~$(ii)$: $\eta \neq 0$}

We now suppose that the limiting population size $(\xi(t))_{t \geq 0}$, 
is a diffusion process, satisfying %which reflects the carrying capacity of the environment, satisfies the SDE
\begin{equation*}
      d{\xi}= \alpha(\xi) dt +  \eta(\xi) dW_{\textrm{env}}(t).
\end{equation*}
%i.e.~the case when the population size is a diffusion process. 
We again consider the system of SDEs corresponding to the coupled 
evolution of the carrying capacity $\xi$ and the proportion of mutants 
in the population $\rho_{0}$ from 
Corollary~\ref{Paper03_corollary_SDE_proportion_slowly_varying_environment}.

To analyse the impact of dormancy in this setting, recall the 
definition of $\mathbb{S}^*_K$ from~\eqref{Paper03_K_dimensional_simplex}. 
For $K\in \mathbb{N}$, we define the real-valued function 
\[
g: \mathbb{S}^*_K \times [0,1] \times [\xi_{\min}, \xi_{\max}] 
\rightarrow \mathbb{R}
\]
by
\begin{equation} 
\label{Paper03_integrand_understand_impact_slow_fluctuations_environment}
\begin{aligned}
    g\Big((b_i)_{i \in \llbracket K \rrbracket_0}, \rho_0, \xi\Big) 
\defeq \rho_0 \left[\frac{B(1- \rho_{0})}{(B(1 - \rho_{0})+1)\xi}
+\frac{\rho_{0}}{2}\left(\frac{\partial^{2} \Phi_{0}}{\partial \rho_{0}^{2}}
\left(\boldsymbol{\rho}\right)  
- \frac{2B}{(B(1 - \rho_{0})+1)}\right)\right],
\end{aligned}
\end{equation}
where the parameter $B$ is the mean time spent in the seed bank 
(see~\eqref{Paper03_mean_age_germination}). Observe that the integrand 
inside the square brackets in the second line 
of~\eqref{Paper03_SDE_proportion_mut_WF_fluctuating_slow_env} is given by 
$\eta^2(\xi)\xi^{-1}g$. For positive values of $g$, fluctuations in 
population size favour the dormancy trait, and for negative values 
they disfavour it. The behaviour of the function $g$ defined 
in~\eqref{Paper03_integrand_understand_impact_slow_fluctuations_environment} 
is not trivial since it depends on at least three parameters 
(see Figure~\ref{Paper03_fig:impact_fluctuating_environment_slow}). 

\begin{figure}[!ht]
\centering
\includegraphics[width=\linewidth, trim=0cm 0cm 0cm 0cm,clip=true]{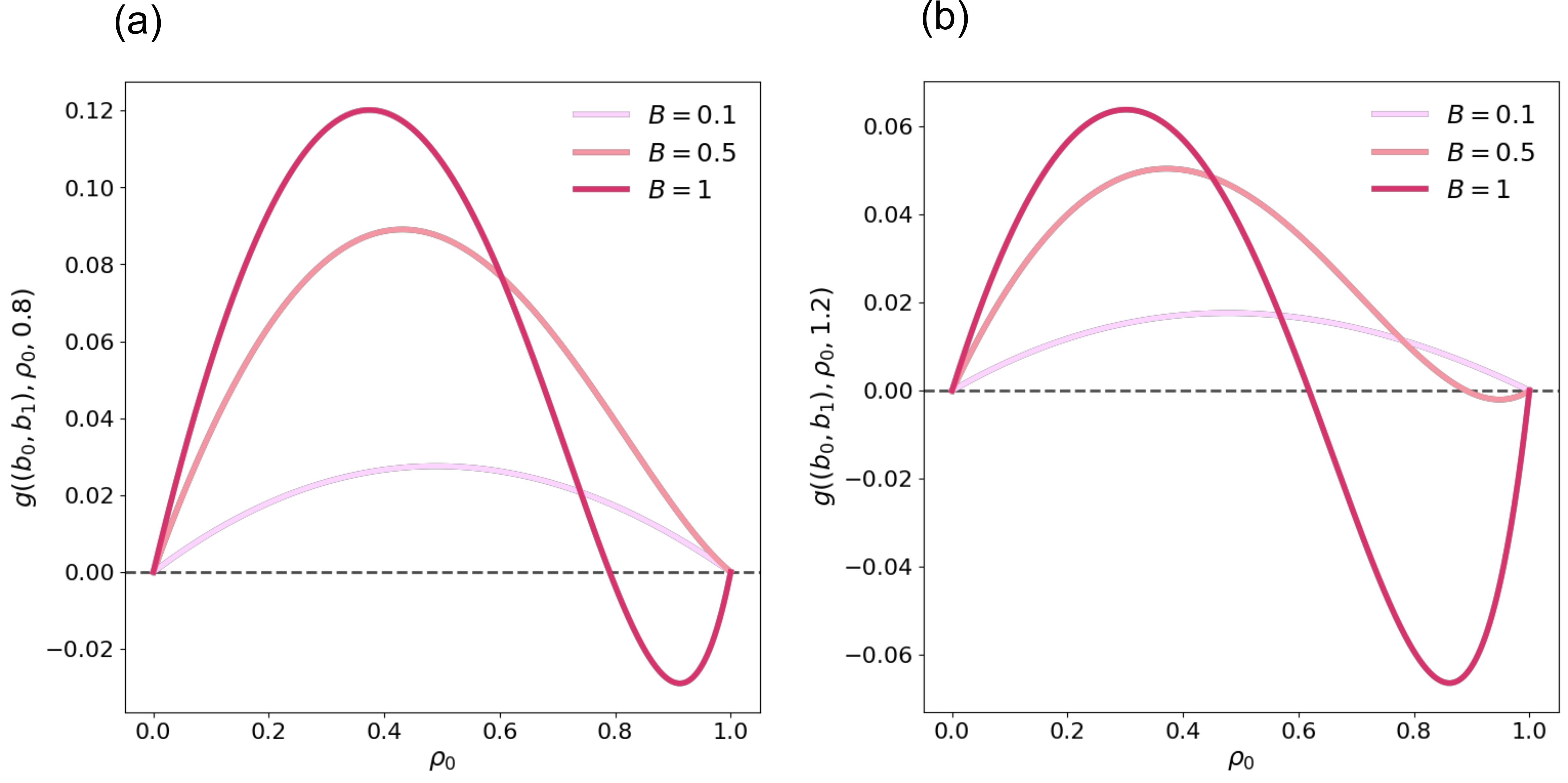}
\caption{\label{Paper03_fig:impact_fluctuating_environment_slow}~Plot of 
the function $g$ defined 
in~\eqref{Paper03_integrand_understand_impact_slow_fluctuations_environment} 
for $K = 1$ and $B = 1 - b_0 \in \{0.1,\, 0.5,\, 1.0\}$. 
The values of $g$ were computed by applying~\eqref{Paper03_drift_K_1} 
to~\eqref{Paper03_integrand_understand_impact_slow_fluctuations_environment}, 
fixing different values of $\xi$ in panels~(a) ($\xi = 0.8$) and~(b) ($\xi = 1.2$).
For large values of $B$, the function $g$ becomes negative when the 
proportion of mutants $\rho_0$ is sufficiently high. 
The threshold value of $B$ for which this occurs depends on the 
population size (parameter $\xi$); 
indeed, in contrast to panel~(a), in panel~(b) the function $g$ is 
negative for $B = 0.5$ and $\rho_0 > 0.9$.}
\end{figure}

\begin{proposition}
\label{Paper03_interpretation_impact_dormancy_slow_fluctuations_population_size}
    Let $0 < \xi_{\min} < \xi_{\max}$. 
The map $g: \mathbb{S}^*_K \times [0,1] 
\times [\xi_{\min}, \xi_{\max}] \rightarrow \mathbb{R}$ 
in~\eqref{Paper03_integrand_understand_impact_slow_fluctuations_environment} 
satisfies the following conditions:
    \begin{enumerate}[(i)]
        \item For any $K \in \mathbb{N}$, 
$\mathbf{b} = (b_{i})_{i \in \llbracket K \rrbracket_0} \in \mathbb{S}^*_K$ 
and $\xi \in [\xi_{\min}, \xi_{\max}]$, we have
        \begin{equation*}
            g(\boldsymbol{b}, 0, \xi) = g(\boldsymbol{b}, 1, \xi) = 0.
        \end{equation*}
        \item For any $K \in \mathbb{N}$, 
$\mathbf{b} = (b_{i})_{i \in \llbracket K \rrbracket_0} \in \mathbb{S}^*_K$ 
and $0 < \xi_{\min} < \xi_{\max}$, there exists 
$\rho_c = \rho_c(K, \boldsymbol{b}, \xi_{\max}) \in (0,1)$ such that for 
all $\rho_0 \in (0,\rho_c)$ and $\xi \in [\xi_{\min}, \xi_{\max}]$, we have 
$g(\boldsymbol{b}, \rho_0,\xi) > 0$.
        \item There exists a decreasing function 
$\xi \in (0, \infty) \mapsto B_c(\xi) > 0$ satisfying 
        \begin{equation} 
\label{Paper03_limit_population_size_impact_beneficial_events_dormancy_fluctuations}
            \lim_{\xi \rightarrow \infty} B_c(\xi) = 0
        \end{equation}
        such that if $K \in \mathbb{N}$ and 
$\mathbf{b} = (b_{i})_{i \in \llbracket K \rrbracket_0} \in \mathbb{S}^*_K$ 
are chosen in such way that $B > B_c(\xi_{\min})$, then there is 
$\rho_c = \rho_c(\xi_{\min}) \in (0,1)$ such that for any 
$\xi \in [\xi_{\min}, \xi_{\max}]$, we have 
$g(\boldsymbol{b}, \rho_0, \xi) < 0$ for $\rho_0 \in (\rho_c, 1)$. 
    \end{enumerate}
\end{proposition}

Biologically, 
Proposition~\ref{Paper03_interpretation_impact_dormancy_slow_fluctuations_population_size} can be interpreted as follows:
\begin{enumerate}[(1)]
    \item By assertion~(ii), the fluctuations in the population size favour 
the growth of the proportion of mutants carrying the dormancy trait when this 
proportion is sufficiently low.
    \item By assertion~(iii), if the mean time to germination is too large, 
the fluctuations in the population size favour the coexistence of the mutant 
and wild types. This phenomenon resembles the case of fluctuating selection, 
where heterozygotes are favoured when compared to homozygotes 
(see e.g.~\cite{haldane1963polymorphism}).
\end{enumerate}

Our findings suggest that, at least within a certain range of parameters, 
fluctuations in population size favour the dormancy trait. However, there 
exists a critical age such that, if the mean germination 
time exceeds this critical value, the wild type is favoured when a large 
proportion of mutants is present, which is in agreement with the predictions 
made by Cohen in~\cite{cohen1966optimizing}. 
%To better understand how environmental variability influences the establishment of seed banks, we next consider a regime in which environmental conditions fluctuate on a timescale much faster than the evolutionary dynamics.

\subsection{Wright-Fisher model in a fast-changing environment} \label{subsection:WF_fast_env_model}

In our final regime, we consider the impact of very rapid fluctuations in the 
environment. We are going to assume that the total population size is
fixed, but that these fluctuations affect the ability of mature plants
to produce seed. 
%Both mutant and non-mutant plants will be affected in the
%same way. 
%In contrast to the slowly varying model introduced in Section~\ref{Paper03_subsection_WF_slow_environment}, 
%we now assume that the population size is fixed and equal to $N \in \mathbb{N}$ at all times. 
Let $M>0$ be a large parameter. %To incorporate fast environmental fluctuations,
We modify the model of 
Section~\ref{Paper03_WF_model_const_environ_description} as follows. 
Let $(\Upsilon_0^N(t))_{t \in \mathbb{N}_0}$ be a sequence of 
i.i.d.~random variables taking values in $\{-1,0,1\}$, 
representing annual environmental variations (such as rainfall and 
temperature fluctuations~\cite{davis2005evolutionary}). In generation 
$t \in \mathbb{N}_0$, each wild type individual contributes
\[
\Poiss\!\left( M \, (1 + s_N \, \Upsilon_0^N(t+1)) \right)
\]
seeds, while for $i \in \llbracket K \rrbracket_0$, each mutant in 
generation $t-i$ contributes
\[
\Poiss\!\left( M b_i \, (1 + s_N \, \Upsilon_0^N(t+1-i)) \right)
\]
to the pool of seeds ready to germinate in generation $t + 1$. Here, 
$s_N>0$ controls the magnitude of environmental variation. Observe that this 
model is analogous to the Wright-Fisher model with fluctuating selection 
(see~\cite[Section~3]{biswas2021spatial}), with the distinction that in our 
model, the total production of seeds of both wild type and mutant individuals 
living in the same generation is equally affected by the environment. Thus,
as before, there is no intrinsic advantage to either carrying or not
carrying the seed bank mutation.
In the classical case of the Wright-Fisher model with fluctuating selection, 
in order to obtain a diffusive limit, 
the selection parameter (which is equivalent to the parameter $s_N$ in our 
model) is usually assumed to be of order $N^{-1/2}$ as $N \rightarrow \infty$. 
Following 
this strategy, we next impose analogous conditions on $s_N$ and the sequence of 
random variables  $(\Upsilon^N_0(t))_{t \in \mathbb{N}_0}$ that will allow 
us to establish convergence of our model (after scaling) to an SDE.

\begin{assumption}[Dynamics of the fast-changing environment] 
\label{Paper03_assumption_fast_changing_environment} 
The sequence $(s_N)_{N \in \mathbb{N}}$ is such that for every 
$N \in \mathbb{N}$, $s_N \in [0,1)$ and there exists $s \in (0,\infty)$ 
such that
\begin{equation} \label{Paper03_limit_scaling_fluctuating_fitness_parameter_assumption}
    \lim_{N \rightarrow \infty} N^{1/2}s_N = s.
\end{equation}
Moreover, there exists $p \in [0,1/2]$ such that for every 
$N \in \mathbb{N}$, $(\Upsilon_0^N(t))_{t \in \mathbb{N}_0}$ is a sequence 
of i.i.d.~random variables satisfying
\begin{equation} 
\label{Paper03_probability_distribution_fast_environment_assumption}
    \mathbb{P}(\Upsilon_0^N(0) = -1) = \mathbb{P}(\Upsilon_0^N(0) = 1) = \frac{1 - \mathbb{P}(\Upsilon_0^N(0) = 0)}{2} = p. 
\end{equation}
\end{assumption}

As before, $X^N_i(t)$ will denote the number of mutant individuals in 
generation $t-i$ for $i \in \llbracket K \rrbracket_0$. We also write 
$\boldsymbol{\Upsilon}^N(t) 
\defeq (\Upsilon^N_i(t))_{i \in \llbracket K-1 \rrbracket_0}$, 
where $\Upsilon^N_i(t) \defeq \Upsilon^N_0(t-i)$ for 
$i \in \llbracket K-1 \rrbracket_0$. Similarly to the case of a constant 
environment, we define the Markov process 
$(\boldsymbol{X}^N(t), 
\boldsymbol{\Upsilon}^N(t))_{t \in \mathbb{N}_0}$ as follows. 
For any $t \in \mathbb{N}_0$, %and for any given 
%configuration $(\boldsymbol{X}^N(t),\boldsymbol{\Upsilon}^N(t))$:
\begin{enumerate}[(i)]
    \item Given $(\boldsymbol{X}^N(t),\boldsymbol{\Upsilon}^N(t))$, 
the conditional distribution of $\Upsilon^N_0(t+1)$ 
follows~\eqref{Paper03_probability_distribution_fast_environment_assumption}.
    \item The conditional distribution of $X^N_0(t+1)$ given 
$(\boldsymbol{X}^N(t),\boldsymbol{\Upsilon}^N(t), 
\Upsilon^N_0(t+1))$ is
    \begin{equation} 
\label{Paper03_binomial_WF_fast_fluctuating_environment}
    \textrm{Bin}\left(N, \, 
\frac{b_0(1 + s_N\Upsilon^N_0(t+1))X^N_0(t) 
+ \sum_{i = 1}^{K}b_{i}(1 + s_N\Upsilon^N_{i-1}(t)) 
X^{N}_{i}(t)}{\left(N - X^{N}_{0}(t) + b_0X^N_0(t)\right) 
(1 + s_N\Upsilon^N_0(t+1)) 
+ \sum_{i = 1}^{K}b_{i} 
(1 + s_N\Upsilon^N_{i-1}(t))X^{N}_{i}(t)}\right).
    \end{equation}
\end{enumerate}
In~\eqref{Paper03_binomial_WF_fast_fluctuating_environment}, 
$(1 + s_N\Upsilon^N_0(t+1))$ reflects the environmental influence 
on the production of seeds by individuals in generation $t$, and 
$(1 + s_N\Upsilon^N_{i-1}(t))$ reflects the  environmental 
influence on seed production in generation $t - i$, for 
$i \in \llbracket K \rrbracket$.

We will define a suitably scaled process and show that it converges to an 
SDE as $N \to \infty$. The main difficulty here is the presence of three 
contributions to the dynamics, corresponding respectively to evolutionary 
change due to genetic drift, dormancy in the 
seed bank, and the fast environmental fluctuations. The usual approach to
identifying a diffusion limit for 
the Wright-Fisher model with fluctuating selection is based on an averaging 
of the action of the infinitesimal generator over the environmental 
fluctuations (see~\cite{kurtz1991averaging,biswas2021spatial}). The 
difficulty of applying this technique to our setting is that, as explained 
in Sections~\ref{Paper03_diffusion_phase_constant_environment} 
and~\ref{Paper03_subsection_WF_slow_environment}, it is not possible to 
establish the diffusive limit in the Wright-Fisher model with seed bank 
through a generator computation. However, as we shall see, Katzenberger's
result will still apply, essentially because, much as in the case of 
fluctuating selection, we choose $s_N$ in such a way that
the amplitude of the 
fast environmental fluctuations tends to zero and the 
corresponding quadratic variation process 
converges over evolutionary timescales.
%We will avoid this technical challenge by 
%defining the scaled process in such way that the environmental fluctuations 
%will happen on the same timescale as dormancy. 
%\textcolor{red}{I don't understand what you mean here.}
Formally, for each 
$N \in \mathbb{N}$, we define the process 
$(\boldsymbol{x}^N(t), \boldsymbol{\upsilon}^N(t))_{t \geq 0}$ by
\begin{equation} \label{Paper03_scaling_fast_environment}
    \boldsymbol{x}^N(t) \defeq \frac{\boldsymbol{X}^{N}(\lfloor Nt \rfloor)}{N} \quad 
\textrm{ and } \quad \boldsymbol{\upsilon}^{N}(t) 
\defeq s_N\boldsymbol{\Upsilon}^{N}(\lfloor Nt \rfloor).
\end{equation}
Let $\{\mathcal{F}^N_t\}_{t \geq 0}$ be the natural filtration of the 
process $(\boldsymbol{x}^N(t), \boldsymbol{\upsilon}^N(t))_{t \geq 0}$. 
Before stating the main convergence theorem of this section, we record the 
conditional expectations of the componentwise increments of the process 
$(\boldsymbol{x}^N(t), \boldsymbol{\upsilon}^N(t))_{t \geq 0}$.

\begin{lemma} 
\label{Paper03_lemma_computation_conditional_increments_fast_environment}
    For all $N \in \mathbb{N}$ and $t \in [0, +\infty)$, we have
    \begin{equation} \label{Paper03_conditional_increment_x_0_fast_environment}
    \begin{aligned}
        & \mathbb{E}\left[\left. x^N_0\left(t + \tfrac{1}{N}\right) - x^N_0(t) 
\right\vert \mathcal{F}^N_t\right] \\ 
& \quad = \frac{2ps_N^2(1 -x^N_0(t))\Big((1 -x^N_0(t)) 
+ b_0x^N_0(t)\Big)\left(\sum_{i = 1}^K 
b_i(1 + \upsilon^N_{i-1}(t))x^N_{i}(t)\right)}{\left[h\left(\boldsymbol{x}^N(t),
\boldsymbol{\upsilon}^N(t)\right)^2 - s_N^2(1 - (1-b_0)x^N_0(t))^2
\right]\cdot h\left(\boldsymbol{x}^N(t),\boldsymbol{\upsilon}^N(t)\right)} 
\\ 
& \quad \quad \quad + 
\frac{(1 - x^N_0(t))\left(\sum_{i = 1}^K 
b_i(1 + \upsilon^N_{i-1}(t))x^N_i(t) - (1 - b_0)x^N_0(t)\right)}{h\left(
\boldsymbol{x}^N(t),\boldsymbol{\upsilon}^N(t)\right)},
    \end{aligned}
    \end{equation}
where, for any $(\boldsymbol{x},\boldsymbol{\upsilon}) 
\in [0,1]^{K+1} \times (-1,1)^{K}$,
    \begin{equation} 
\label{Paper03_simplified_notation_fast_environment_drift}
        h(\boldsymbol{x},\boldsymbol{\upsilon}) \defeq \left(1 - x_0\right) 
+ b_0x_0 + \sum_{i = 1}^K b_i(1 + \upsilon_{i-1}(t))x_i.
    \end{equation}
Moreover, for $i \in \llbracket K \rrbracket$,
    \begin{equation} \label{Paper03_conditional_increment_x_i_fast_environment}
        \mathbb{E}\Bigg[x^N_i\left(t + \tfrac{1}{N}\right) - x^N_i(t) \Bigg\vert \mathcal{F}^N_t\Bigg] = x^N_{i -1}(t) - x^N_{i}(t),
    \end{equation}
    and for $i \in \llbracket K -1 \rrbracket_{0}$,
    \begin{equation} 
\label{Paper03_conditional_increment_upsilon_i_fast_environment}
        \mathbb{E}\Bigg[\upsilon^N_{i}\left(t + \tfrac{1}{N}\right) 
- \upsilon^N_{i}(t) \Bigg\vert \mathcal{F}^N_t\Bigg] 
= \mathds{1}_{\{i \geq 1\}}\upsilon^N_{i -1}(t) - \upsilon^N_{i}(t).
    \end{equation}
\end{lemma}

Lemma~\ref{Paper03_lemma_computation_conditional_increments_fast_environment} 
will be proved in 
Section~\ref{Paper03_subsection_diffusion_proofs_fast_environment}. As in the 
cases of constant and slowly changing population sizes studied so far, we 
are going to see a separation of timescales. The stochastic process 
$(\boldsymbol{x}^N(t),\boldsymbol{\upsilon}^N(t))_{t \geq 0}$ will fluctuate 
around the deterministic flow determined by the ODE that arises as a limit 
of~\eqref{Paper03_conditional_increment_x_0_fast_environment},~\eqref{Paper03_conditional_increment_x_i_fast_environment} 
and~\eqref{Paper03_conditional_increment_upsilon_i_fast_environment} with 
time accelerated by a factor $N$. Just as in the 
case of a slowly varying population size we were able to discard the
$\mathcal{O}(1/N)$ increments of population size in writing down the 
deterministic flow,  
here~\eqref{Paper03_limit_scaling_fluctuating_fitness_parameter_assumption}, 
which tells us that the fluctuation of the seed production over a single 
generation is of order $N^{-1/2}$, will imply that in the fast 
timescale of the ODE flow the random fluctuation of seed production over 
a single 
generation will be negligible, although we still must record the 
fluctuations that happened over the past $K$ generations. 
As a result, the (accelerated) ODE will be determined by the second term on 
the right-hand side 
of~\eqref{Paper03_conditional_increment_x_0_fast_environment}, and by the 
right-hand sides of~\eqref{Paper03_conditional_increment_x_i_fast_environment} 
and~\eqref{Paper03_conditional_increment_upsilon_i_fast_environment}. 
In other words, with a fast-changing environment, we should consider the flow 
field 
$\boldsymbol{F}^{\textrm{fast}} 
= (F^{\textrm{fast}}_0, \cdots, F^{\textrm{fast}}_K, 
F^{\textrm{fast}}_{\textrm{env},0}, \cdots, 
F^{\textrm{fast}}_{\textrm{env},(K-1)}): [0,1]^{K+1} \times [-1,1]^K 
\rightarrow \mathbb{R}^{2K +1}$ given by
\begin{equation}
\label{Paper03_flow_definition_WF_model_fast_changing_environment}
    \left\{
    \arraycolsep=1.4pt\def\arraystretch{2.2}
    \begin{array}{cl} 
    F^{\textrm{fast}}_{0}( \boldsymbol{x}, \boldsymbol{\upsilon}) 
& \defeq \displaystyle{\frac{\left(1 - x_0\right)\left(\sum_{i = 1}^K b_i(1 + \upsilon_{i-1}(t))x_i - (1 - b_0)x_0\right)}{(1 -x_0) + b_0x_0 + \sum_{i = 1}^K b_i(1 + \upsilon_{i-1})x_i}}, \\
    F^{\textrm{fast}}_{i}(\boldsymbol{x}, \boldsymbol{\upsilon}) & \defeq x_{i - 1} - x_{i} 
\quad \forall \, i \in \llbracket K \rrbracket, \\ 
F^{\textrm{fast}}_{\textrm{env},i}(\boldsymbol{x}, \boldsymbol{\upsilon}) & \defeq \mathds{1}_{\{i \geq 1\}}\upsilon_{i -1} - \upsilon_{i} \quad \forall \, i \in \llbracket K -1 \rrbracket_0.
    \end{array}\right.
\end{equation}
Comparing~\eqref{Paper03_flow_definition_WF_model_fast_changing_environment} 
to the flow field $\boldsymbol{F} 
= (F_i)_{i \in \llbracket K \rrbracket_0}: [0,1]^{K+1} \rightarrow \mathbb{R}$ 
given by~\eqref{Paper03_flow_definition_WF_model}, which defines the ODE 
arising from the Wright-Fisher model in a constant environment, we observe 
that unlike the case of a slowly varyng population size, 
the fast environmental fluctuations impact the flow in a way that cannot 
be encoded as a simple time change.
%, as in the case of the ODE arising from 
%the Wright-Fisher model with slowly varying populaion size
%(see the discussion 
%after~\eqref{Paper03_flow_definition_WF_model_slowly_changing_environment}). 
Nevertheless,~\eqref{Paper03_flow_definition_WF_model_fast_changing_environment}
indicates that $\boldsymbol \upsilon^N$ should converge to $0$ as $N \rightarrow \infty$, 
so that we should still be able to apply Katzenberger's machinery in this 
context. Recall that $\Gamma$ is the diagonal of the $(K+1)$-dimensional 
hypercube, and the attractor manifold associated to the ODE with flow given 
by $\boldsymbol{F}$. Set
\begin{equation} 
\label{Paper03_attractor_manifold_fast_environment}
    \Gamma^{\, \textrm{fast}} \defeq \Big\{(\boldsymbol{x},
\boldsymbol{\upsilon}) \in [0,1]^{K+1} \times [-1,1]^{K}: \, 
\boldsymbol{F}^{\, \textrm{fast}}(\boldsymbol{x},\boldsymbol{\upsilon}) 
= 0\Big\} = \Gamma \times \{0\}^{K}.
\end{equation}
Analogous to the definition of $\boldsymbol{\Phi}$ in~\eqref{Paper03_definition_projection_map}, we define
\begin{equation*}
    \boldsymbol{\Phi}^{\textrm{fast}} = (\Phi^{\textrm{fast}}_0, \cdots, \Phi^{\textrm{fast}}_K, 0, \cdots, 0): [0,1]^{K+1} \times [-1,1]^{K} \rightarrow \Gamma^{\, \textrm{fast}}
\end{equation*}
to be the projection map associated with the ODE corresponding to the flow 
field $\boldsymbol{F}^{\textrm{fast}}$. % Once again, we shall use 
%Katzenberger's method 
%(explained in Section~\ref{Paper03_subsection_katzenberger_overview}) 
%to derive the limiting SDE for the proportion of mutants in the population. 
In what follows, %to avoid introducing excessive notation, we also let 
$\boldsymbol{0}$ denotes the $K$-dimensional null vector. 

\begin{theorem} 
\label{Paper03_thm_convergence_diffusion_fast_changing_environment}
    Suppose that $(\boldsymbol{x}^N(0), \boldsymbol{\upsilon}^N(0))$ converges 
to $(x(0), 0) \in \Gamma^{\, \textrm{fast}}$ as $N \rightarrow \infty$, and that 
Assumption~\ref{Paper03_assumption_fast_changing_environment} holds. Then the 
sequence of processes 
$\left((\boldsymbol{x}^{N}(t), \boldsymbol{\upsilon}^{N}(t))_{t \geq 0}\right)_{N \in \mathbb{N}}$ 
converges in distribution in 
$\mathscr{D}\left([0, \infty), [0, 1]^{K+1} 
\times [-1,1]^K \right)$ 
to a diffusion process with sample paths 
in $\Gamma^{\; \textrm{fast}}$, 
i.e.~$(\boldsymbol{\upsilon}^N)_{N \in \mathbb{N}}$ converges to 
$\boldsymbol{0}$, and $(\boldsymbol{x}^N)_{N \in \mathbb{N}}$ converges to 
a diffusion $(\boldsymbol{x}(t))_{t \geq 0}$ in $\Gamma$ such that 
$(x_{0}(t))_{t \geq 0}$ 
satisfies the stochastic differential equation
\begin{equation} \label{Paper03_limiting_SDE_fast_env}
\begin{aligned}
    x_0(T) & = \int_{0}^T \left[\frac{x_0(1-x_0)}{2}
\frac{\partial^2 \Phi_0}{\partial x_0^2}(\boldsymbol{x}) 
\left(1 + 2ps^2(1-b_0)^2x_0(1-x_0)\right) + 
ps^2 \frac{\partial^2 \Phi^{\, \textrm{fast}}_0}{\partial \upsilon_0^2} 
(\boldsymbol{x}, \boldsymbol{0}) \right] dt 
\\ 
& \quad \quad + 2ps^2(1-b_0) \int_{0}^T {x_0(1-x_0)} 
\Bigg[\frac{((1-x_0) + b_0x_0)}{(B(1-x_0)+1)} 
- \frac{\partial^2 \Phi^{\, \textrm{fast}}_0}{\partial x_0 \partial \upsilon_0}
 (\boldsymbol{x}, \boldsymbol{0}) \Bigg] dt 
\\ 
& \quad \quad + \int_0^T \frac{\sqrt{x_{0}(1 - x_{0})}}{(B(1 - x_{0}) + 1)} 
\, dW_{0}(t),
\end{aligned}
\end{equation}
where $(W_0(t))_{t \geq 0}$ is a standard Brownian motion.
\end{theorem}

Comparing~\eqref{Paper03_limiting_SDE_fast_env} with the diffusion limit 
of the Wright--Fisher model in a constant 
environment~\eqref{Paper03_SDE_WF_constant_environment}, we see that the 
noise term in~\eqref{Paper03_limiting_SDE_fast_env} arises solely from 
genetic sampling. In the fast–changing regime the environmental fluctuation 
acts uniformly on all individuals of a generation, and in the diffusion 
limit this common factor averages out. This differs from classical 
fluctuating selection models, in which individuals of different genetic types
are affected differently by a given environment, and the environmental noise 
survives in the limit.

Nevertheless, the fast environment influences the limiting SDE through 
additional drift terms, since rapid fluctuations alter the forcing flow of 
the Wright-Fisher system. Although the %resulting 
flow~$\boldsymbol{F}^{\,\textrm{fast}}$ is too complex to analyse explicitly 
for general $K$, our methods allow us to compute the second-order derivatives 
of the projection map~$\boldsymbol{\Phi}^{\,\textrm{fast}}_{0}$ that 
appear on the right-hand side of~\eqref{Paper03_limiting_SDE_fast_env} for 
every specific $K \in \mathbb{N}$. The expressions become very complex and
so for biological interpretation, we 
focus on the case $K=1$.

\subsubsection{Biological interpretation for the case $K = 1$} 
\label{Paper03_sec_particular_analysis_fast_env_K_1}

%To better understand the impact of fast environmental fluctuations on the 
%fixation of seed banks in plants, we will restrict our attention to the 
%case $K=1$. 
Recall the formula for $\frac{\partial^2 \Phi_0}{\partial x_0^2}$ 
stated in~\eqref{Paper03_drift_K_1} for $K=1$. By using Parsons and Rogers 
representation of the second order derivatives of the projection map introduced
in~\cite{parsons2017dimension}, and our formula for the quadratic approximation
of the non-linear stable manifold 
(see Theorem~\ref{Paper03_prop_uniqueness_quadratic_term_nonlinear_manifold}), 
in Section~\ref{Paper03_sec_comp_lemmas_fast_env}
we will establish the
following result.

\begin{lemma} \label{Paper03_lem:explicit_drifdt_fast_env}
    Under the assumption that $K=1$, the following formulae hold for 
any $\boldsymbol{x} \in \Gamma$:
    \begin{equation*}
    \begin{aligned}
        \frac{\partial^2 \Phi^{\, \textrm{fast}}_0}{\partial x_0 \partial \upsilon_0} (\boldsymbol{x}, \boldsymbol{0}) & = \frac{(1-b_0)^2(1-x_0)+(1-b_0)(1-2x)}{((1-b_0)(1-x_0) + 1)^3} \\ & \quad \; - \frac{(1-b_0)^2(1-x_0)^2\Big[(1-b_0)^2x_0^2-(1-b_0)(4-b_0)x_0+(2-b_0)\Big]}{((1-b_0)(1-x_0) + 1)^3((1-b_0)(1-x_0)+2)},
    \end{aligned}
    \end{equation*}
    and
    \begin{equation*}
    \begin{aligned}
        \frac{\partial^2 \Phi^{\, \textrm{fast}}_0}{\partial \upsilon_0^2} (\boldsymbol{x}, \boldsymbol{0}) & = \frac{2(1-b_0)^2x_0(1-x_0)\left[(1-b_0)(1-x_0)^2+(1-2x_0)\right]}{((1-b_0)(1-x_0) + 1)^3} \\ & \quad \; + \frac{(1-b_0)^2x_0(1-x_0)^2\Big[(1-b_0)^2x_0^2-(1-b_0)(5-b_0)x_0+(4-b_0)\Big]}{((1-b_0)(1-x_0) + 1)^3((1-b_0)(1-x_0)+2)}.
    \end{aligned}
    \end{equation*}
\end{lemma}

Still under the assumption $K=1$, we 
examine how the parameter $b_0$ influences 
the additional drift terms that appear in the limiting SDE in the 
fast–changing environment (see~\eqref{Paper03_limiting_SDE_fast_env}) but 
are absent from the corresponding SDE in the constant–environment 
case~\eqref{Paper03_SDE_WF_constant_environment}. To this end, we define 
the map $h:[0,1]^2 \rightarrow \mathbb{R}$ by, for any $(x_0,b_0) \in [0,1]$,
\begin{equation}\label{Paper03_functionh_fast_env}
\begin{aligned}
    h(x_0,b_0) & \defeq (1-b_0)^2x_0(1-x_0)\frac{\partial^2 \Phi_0}{\partial x_0^2}(\boldsymbol{x}) +  \frac{1}{x_0(1-x_0)}\frac{\partial^2 \Phi^{\, \textrm{fast}}_0}{\partial \upsilon_0^2} (\boldsymbol{x}, \boldsymbol{0}) - 2(1-b_0)  \frac{\partial^2 \Phi^{\, \textrm{fast}}_0}{\partial x_0 \partial \upsilon_0} (\boldsymbol{x}, \boldsymbol{0}) \\ & \quad \quad + \frac{2(1-b_0)((1-x_0)+b_0x_0)}{((1-b_0)(1-x_0)+1)}.
\end{aligned}
\end{equation}
Comparing~\eqref{Paper03_limiting_SDE_fast_env} 
with~\eqref{Paper03_SDE_WF_constant_environment}, we see that the 
fast-changing environment contributes an additional drift term equal to
\begin{equation} \label{Paper03_extra_drift_fast_env}
    \int_0^T p s^2 x_0(1-x_0) \, h(x_0,b_0)\, dt.
\end{equation}
We plot the values of the function $h$ for different values of 
$(x_0,b_0) \in [0,1]^2$ in 
Figure~\ref{Paper03_fig:impact_fluctuating_environment_fast}.
\begin{figure}[t]
\centering
\includegraphics[width=0.6\linewidth, trim=0cm 0cm 0cm 0cm,clip=true]{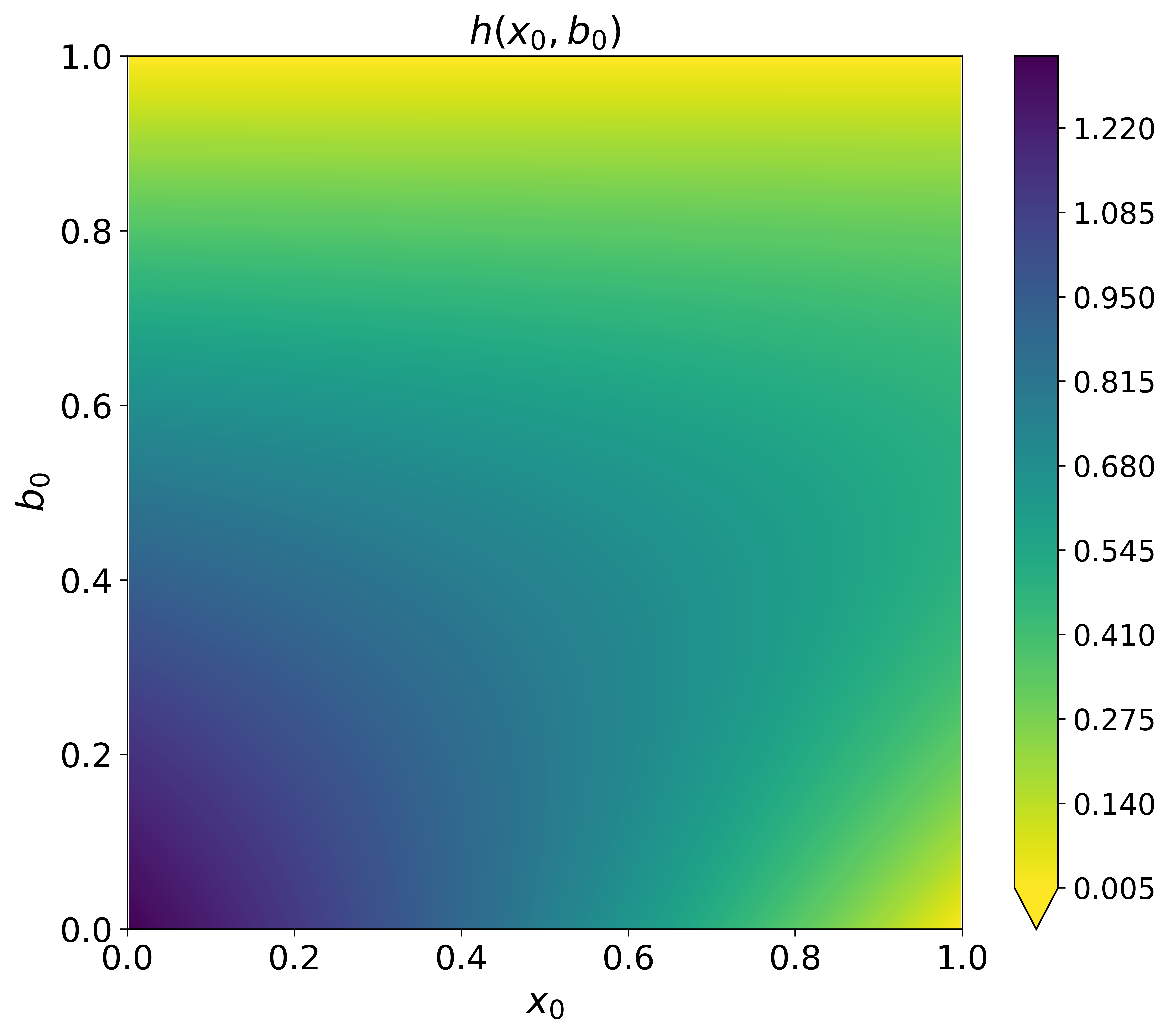}
\caption{\label{Paper03_fig:impact_fluctuating_environment_fast}~Contour plot 
of the function $h(x_0,b_0)$ of~\eqref{Paper03_functionh_fast_env}
on the domain $[0,1]^2$. Darker colours are 
associated with larger values of $h$, while lighter colours are associated 
with smaller values. For $K = 1$, $h$ is strictly non-negative, 
its maximum value is $4/3$ and is achieved at $(0,0)$. The plot suggests 
that in the fast-changing environment, the seed bank trait is favoured, 
as expected.}
\end{figure}
Observe that the map~$h$ is non-negative on its entire domain, 
suggesting that the fast-changing environment favours the seed bank trait.
Biologically, this reflects the buffering effect of dormancy against short 
environmental events. Moreover, since we may take the values of $s$ to be 
arbitrarily large in~\eqref{Paper03_extra_drift_fast_env}, 
the resulting contribution to the drift can exert a substantial influence on 
the dynamics.

We also note from Figure~\ref{Paper03_fig:impact_fluctuating_environment_fast} that the maximum value of \(h\) is attained at \((0,0)\), where it equals \(4/3\). Biologically, this corresponds to a situation in which the density of mature mutants is extremely low and all mutant seeds skip one generation before maturing. This suggests that the seed bank trait may be particularly effective in protecting a small mutant population in the presence of fast environmental fluctuations.

A limitation of this analysis is that we do not incorporate the impact of the 
fast fluctuations on the branching phase. The 
analysis of multi-type branching processes in random environments is complex, 
and it can rarely be done rigorously for models such as ours
with multiple parameters (see~\cite[Section~4]{athreya1971branching}). 
Nonetheless, biologically, it could be the case that a constant environment 
becomes a fast-changing one after some perturbation. In this case, if the 
seed bank trait survives the branching phase and enters the diffusion phase, 
the fluctuations can confer an evolutionary advantage to the dormancy trait.

\section{Dimension reduction for stochastic differential equations 
forced onto a manifold by large drift} 
\label{Paper03_mathematical_background}

In this section, we shall explain, review, and extend some results regarding 
the weak convergence of interacting particle systems to SDEs which are forced 
onto a manifold by a large drift, enabling a dimension reduction of the 
original problem. 
%The results reviewed and developed in this section are 
%essential for the study of our problem, since as explained in 
%Section~\ref{Paper03_main_results_section}, the presence of seed banks breaks 
%down the usual argument to justify the derivation of Wright-Fisher diffusion 
%through a generator computation. We organise this section as follows. 
In Section~\ref{Paper03_subsection_katzenberger_overview}, we explain 
Katzenberger's results concerning the convergence of stochastic dynamical 
systems in the presence of a large drift~\cite{katzenberger1991solutions}. 
We will then present, 
in Section~\ref{Paper03_subsection_computing_derivatives}, 
a method due to Parsons and Rogers~\cite{parsons2017dimension} for computing 
the derivatives of the projection map associated to an ODE whose attractor 
manifold is one-dimensional. %~\cite{parsons2017dimension}. 
In 
Section~\ref{Paper03_general_result_derivatives_manifold}, we apply the 
Lyapunov-Schmidt reduction method to develop a general formula for one of 
the parameters needed for the Parsons and Rogers method. We illustrate in 
Section~\ref{Paper03_derivatives_specific_ode} how to apply this formula in 
the case of the ODE with flow field given 
by~\eqref{Paper03_flow_definition_WF_model} arising from the analysis of 
the Wright-Fisher model in constant environment.

\subsection{Overview of Katzenberger results} 
\label{Paper03_subsection_katzenberger_overview}

In~\cite{katzenberger1991solutions}, Katzenberger considered a sequence 
$(\boldsymbol{x}^{N})_{N \in \mathbb{N}}$ of càdlàg 
$\mathbb{R}^{K+1}$-valued stochastic processes of the form
\begin{equation} \label{Paper03_stochastic_process_general_discrete_N}
\begin{aligned}
    \boldsymbol{x}^{N}(T) = \boldsymbol{x}^{N}(0) 
+ \int_{0}^{T} U^{N}(\boldsymbol{x}^{N}(t-)) \, d\mathbf{z}^{N}(t) 
+ \int_{0}^{T} \mathbf{F}(\boldsymbol{x}^{N}(t-)) 
\, d\lfloor Nt \rfloor + \epsilon^N,
\end{aligned}
\end{equation}
where $(\mathbf{z}^{N})_{N \in \mathbb{N}}$ is a sequence of well-behaved 
càdlàg $\mathbb{R}^{K+1}$-semimartingales, $(U^{N})_{N \in \mathbb{N}}$ is 
a sequence of continuous $(K+1) \times (K+1)$-matrix valued functions,
$\mathbf{F}$ is a vector field whose associated ODE has an asymptotic 
attractor manifold $\Gamma$ of dimension strictly less than $K+1$, 
and $\epsilon^N$ captures the jumps associated with the discrete 
approximation of the underlying dynamics. More generally, the integral with 
integrand $\mathbf{F}$ can be taken with respect to any finite variation 
process which asymptotically places infinite mass on compact time intervals.

Assume that for each $N \in \mathbb{N}$, $\boldsymbol{x}^{N}$ has sample 
paths in some connected subset $\mathscr{U} \subset \mathbb{R}^{K+1}$ such 
that $\mathscr{U}$ is contained in the basin of attraction of the manifold 
$\Gamma \subset \mathscr{U}$. In this case, since we are accelerating the 
timescale of the integral of the vector field $\mathbf{F}$, as 
$N \rightarrow \infty$, if the integrals of $U^{N}$ behave sufficiently well, 
the distance between the sample paths of $\boldsymbol{x}^{N}$ and $\Gamma$ 
should converge to $0$. If $\boldsymbol{\Phi}: \mathscr{U} \rightarrow \Gamma$ 
is the projection map associated to the ODE with flow field given by 
$\mathbf{F}$, then Itô's formula yields
\begin{equation} 
\label{Paper03_intermediate_step_convergence_integrals_katzenberger}
    \boldsymbol{\Phi}(\boldsymbol{x}^{N}(T)) 
= \boldsymbol{\Phi}(\boldsymbol{x}^{N}(0)) 
+ \int_{0}^{T} \mathcal{J}\boldsymbol{\Phi}U^{N}\, d\mathbf{z}^{N}(t) 
+ \frac{1}{2} \sum_{hijl} \int_{0}^{T} 
\frac{\partial^{2}\boldsymbol{\Phi}}{\partial x_{i} \partial x_{j}} 
U^{N}_{ih}U^{N}_{jl} \, d[z^{N}_{h}, z^{N}_{l}](t) + \tilde{\epsilon}^N,
\end{equation}
where $\mathcal{J}\boldsymbol{\Phi}$ denotes the Jacobian matrix of 
$\boldsymbol{\Phi}$ and $\tilde{\epsilon}^N$ captures the corrections due to 
jumps of the process. Here, we used the fact that by the definition of the 
projection map $\boldsymbol{\Phi}$, we have 
$\mathcal{J}\boldsymbol{\Phi}_{\boldsymbol{x}} \mathbf{F} (\boldsymbol{x}) 
\equiv 0$ for all $\boldsymbol{x} \in \mathscr{U}$, 
c.f.~\eqref{Paper03_definition_projection_mapA}.
If the sequence of semimartingales is sufficiently regular and the sample 
paths of $\boldsymbol{x}^{N}$ converge to a process in $\Gamma$, then 
$\boldsymbol{x}^{N} - \boldsymbol{\Phi}(\boldsymbol{x}^{N})$ must converge 
to zero, %which means that in the limit, $\boldsymbol{x}^{N}$ must 
%satisfy
so using~\eqref{Paper03_intermediate_step_convergence_integrals_katzenberger} 
in the limit we must obtain
\begin{equation} 
\label{Paper03_limiting_SDE_katzenberger_general}
    \boldsymbol{x}(T) = \boldsymbol{x}(0) 
+ \int_{0}^{T} \mathcal{J}\boldsymbol{\Phi}U \, d\mathbf{z}(t) 
+ \frac{1}{2} \sum_{hijl} \int_{0}^{T} 
\frac{\partial^{2}\boldsymbol{\Phi}}{\partial x_{i} \partial x_{j}} 
U_{ih}U_{jl} \, d[z_{h}, z_{l}](t),
\end{equation}
provided $\mathbf{z}^{N} \xrightarrow{\mathcal{D}} \mathbf{z}$, $U^{N} 
\rightarrow U$, and $\epsilon^{N} \xrightarrow{\mathcal{D}} 0$ 
as $N \rightarrow \infty$. 
Observe that the limiting SDE has sample paths in~$\Gamma$.

Not surprisingly, the proof of convergence requires some regularity 
conditions to be satisfied by the sequence 
$(\boldsymbol{x}^{N})_{N \in \mathbb{N}}$. For ease of reference, we recall 
these conditions and the results that we are going to use
from~\cite{katzenberger1991solutions}. 
First we introduce some notation. Define the 
open disc 
$D(1) \defeq \{\lambda \in \mathbb{C}: \; \vert \lambda + 1 \vert < 1 \}$, 
and let $\Gamma \subset \mathscr{U}$ be the $d$-dimensional connected 
attractor manifold of the ODE associated to the vector field~$\mathbf{F}$. 
%We require two assumptions on the attractor manifold.

\begin{assumption}[Attractor manifold]
\label{Paper03_assumption_manifold}
    We assume
    \begin{equation*}
        \Gamma \defeq \left\{\boldsymbol{x} \in \mathscr{U}: \; 
\mathbf{F}(\boldsymbol{x})=0\right\}
    \end{equation*}
    to be a twice-continuously differentiable connected $d$-dimensional 
manifold, and $\mathbf{F}$ to be twice-continuously differentiable. Moreover, 
for any $\boldsymbol{x} \in \Gamma$, $\mathcal{J}\mathbf{F}_{\boldsymbol{x}}$ 
has exactly (K+1-d) eigenvalues in~$D(1)$.
\end{assumption}

We now turn to the conditions required of the stochastic process. 
We denote the filtration of the process $\boldsymbol{x}^{N}$ by 
$\{\mathcal{F}^{N}_{t}\}_{t \geq 0}$, and assume that it satisfies the usual 
conditions. Since we will be applying Katzenberger's results to the analysis of
a modified Wright-Fisher model, we can and will improve readability by
making stronger assumptions than required by Katzenberger.
We assume that for each $N \in \mathbb{N}$, the semimartingale 
$\mathbf{z}^{N}$ jumps only when $(\lfloor Nt \rfloor)_{t \geq 0}$ jumps, 
i.e.~only at times in the set $\{\frac{i}{N}: \; i \in \mathbb{N}\}$. 
We also assume 
$\mathscr{U}$ to be compact (since this will be true in our case). 

For a finite 
variation process $\boldsymbol{A}: [0, \infty) \rightarrow \mathbb{R}^{K+1}$ 
and $T \geq 0$, let $\mathfrak{T}(\boldsymbol{A}^N, T)$ be the total 
variation of $\boldsymbol{A}$ on the time interval $[0,T]$.

\begin{assumption}[Uniform integrability and tightness] 
\label{Paper03_assumption_katzenberger_stochastic_process}
For $N \geq 1$, let $\mathbf{z}^{N}$ be an 
$\{\mathcal{F}^{N}_{t}\}$- semimartingale with sample paths in 
$\mathscr{D}([0, +\infty), \mathbb{R}^{K+1})$. Let 
$\mathbf{z}^{N} = \mathbf{M}^{N} + \mathbf{A}^{N}$, where $\mathbf{M}^{N}$ is 
a local martingale and $\mathbf{A}^{N}$ a finite variation process. We shall 
assume: % the conditions below.
    \begin{enumerate}[(i)]
        \item For any $T \geq 0$, 
$\left\{\left[\mathbf{M}^{N}\right](T) 
+ \mathfrak{T}(\mathbf{A}^{N}, T)\right\}_{N \in \mathbb{N}}$ 
is uniformly integrable;
        \item For any $T \geq 0$, $\varepsilon > 0$, we have
        \begin{equation*}
            \lim_{\delta \rightarrow 0^{+}} \, 
\limsup_{N \rightarrow \infty} \, 
\mathbb{P} \left(\sup_{t \leq T} \, \left(\mathfrak{T}(\mathbf{A}^{N}, 
t + \delta) - \mathfrak{T}(\mathbf{A}^{N}, t)\right) 
> \varepsilon \right) = 0.
        \end{equation*}
    \end{enumerate}
\end{assumption}

We now state Katzenberger's theorem, which is a simplified version of Theorem~7.2 in~\cite{katzenberger1991solutions}.

\begin{theorem}[Katzenberger~\cite{katzenberger1991solutions}] 
\label{Paper03_major_tool_katzenberger}
    Let the sequence of processes $(\boldsymbol{x}^{N})_{N \in \mathbb{N}}$ be 
given by~\eqref{Paper03_stochastic_process_general_discrete_N}. Under 
Assumptions~\ref{Paper03_assumption_manifold} 
and~\ref{Paper03_assumption_katzenberger_stochastic_process}, the sequence 
$(\boldsymbol{x}^{N}, \mathbf{z}^{N})_{N \in \mathbb{N}}$ is relatively compact
in the space of càdlàg functions 
$\mathscr{D}\left([0, + \infty), 
\mathbb{R}^{K+1} \times \mathbb{R}^{K+1}\right)$. 
Moreover, any limit point $(\boldsymbol{x}, \mathbf{z})$ is a continuous 
semimartingale such that $\boldsymbol{x}(t) \in \Gamma$ for every $t$ almost 
surely, and $(\boldsymbol{x}, \mathbf{z})$ satisfies the 
SDE~\eqref{Paper03_limiting_SDE_katzenberger_general}.
\end{theorem}

A major challenge in the practical application of 
Theorem~\ref{Paper03_major_tool_katzenberger} 
%to the analysis of stochastic models arising from population genetics and 
%other fields 
is the computation of the derivatives of the projection map 
$\boldsymbol{\Phi}$. When the associated ODE has an explicit solution, or 
when it admits some characteristic curve, one can compute the derivatives 
directly. However, in practice, the ODEs arising in applications are 
usually nonlinear, and generally lack an explicit solution. In the next 
section, we explain a formula developed by Parsons and Rogers 
in~\cite{parsons2017dimension} to overcome this problem.

\subsection{Parsons and Rogers approach to the computation of 
derivatives of projection map} 
\label{Paper03_subsection_computing_derivatives}

Parsons and Rogers~\cite{parsons2017dimension} observed that, when the 
attractor manifold~$\Gamma$ is one-dimensional, the computation of the 
derivatives can be achieved through the analysis of key features of the 
underlying dynamical system, concretely the tangent vector to $\Gamma$, the 
perpendicular vector to the flow field, the curvature of $\Gamma$, and the 
curvature of the nonlinear stable manifold near $\Gamma$. To understand 
these ideas, we shall briefly recall some concepts regarding dynamical systems.

Let $\mathscr{U} \subset \mathbb{R}^{K + 1}$ be a connected and compact set, and consider the ODE
\begin{equation} 
\label{Paper03_deterministic_flow}
\begin{aligned}
    \frac{d\boldsymbol{x}}{dt} & = \mathbf{F}(\boldsymbol{x}),
\end{aligned}
\end{equation}
where $\mathbf{F}: \mathscr{U} \rightarrow \mathbb{R}^{K+1}$ is a smooth map 
such that if $\boldsymbol{x}(0) \in \mathscr{U}$, then the flow defined 
by~\eqref{Paper03_deterministic_flow} never leaves $\mathscr{U}$, 
i.e.~$\boldsymbol{x}(0) \in \mathscr{U} \Rightarrow \boldsymbol{x}(t) 
\in \mathscr{U}$ for all $t \geq 0$. We define the smooth manifold
\begin{equation} 
\label{Paper03_definition_manifold}
    \Gamma \defeq \{\boldsymbol{x} = (x_{0}, \ldots, x_{K}) 
\in \mathscr{U}: \, \mathbf{F}(\boldsymbol{x}) = 0\}.
\end{equation}
From now on, we make the following assumptions on $\Gamma$ and $\mathbf{F}$.

\begin{assumption}[Geometric properties of $\Gamma$] 
\label{Paper03_geometric_properties_Gamma}
 Equation~\eqref{Paper03_definition_manifold} defines a 
one-dimensional manifold $\Gamma$, and both $\Gamma$ and the vector field 
$\mathbf{F}$ satisfy Assumption~\ref{Paper03_assumption_manifold}. 
Moreover, $\Gamma$ is a curve parametrised by the first coordinate of the 
system, i.e.~there exists a connected subset 
$\mathcal{I}_{\Gamma} \subset \mathbb{R}$ such that 
$\boldsymbol{x} \in \Gamma \Rightarrow x_{0} \in \mathcal{I}_{\Gamma}$, and 
a twice-continuously differentiable function 
$\boldsymbol{\gamma}: \mathcal{I}_{\Gamma} \rightarrow \Gamma$ such that
\begin{equation*}
 \boldsymbol{x} \in \Gamma \Leftrightarrow \boldsymbol{\gamma}(x_{0}) 
= \boldsymbol{x}.
    \end{equation*}
\end{assumption}

As observed in~\cite[Subsection~2.1]{parsons2017dimension}, the 
one-dimensional parametrisation can be relaxed to a local parametrisation, 
so that Assumption~\ref{Paper03_geometric_properties_Gamma} is not restrictive.
By classical results of dynamical systems theory~\cite{carr2012applications}, 
Assumption~\ref{Paper03_assumption_manifold} implies that $\Gamma$ is an 
attractor manifold for the ODE~\eqref{Paper03_deterministic_flow}. 
To formalise this idea, for any $\boldsymbol{x} \in \mathscr{U}$, we define 
the flow
\begin{equation*}
\begin{aligned}
    \boldsymbol{\phi}(\boldsymbol{x}, \cdot): [0, +\infty) & \rightarrow \mathscr{U} \\
    T & \mapsto \phi(\boldsymbol{x}, T) = \boldsymbol{x} + \int_{0}^{T} F(\boldsymbol{\phi}(\boldsymbol{x}, t)) \, dt.
\end{aligned}
\end{equation*}
Then, for any $\boldsymbol{x} \in \mathscr{U}$, we can define the map
\begin{equation*}
\begin{aligned}
    \boldsymbol{\Phi}: \mathscr{U} & \rightarrow \Gamma \\
    \boldsymbol{x} & \mapsto \boldsymbol{\Phi}(\boldsymbol{x}) = \lim_{T \rightarrow \infty} \boldsymbol{\phi}(\boldsymbol{x},T).
\end{aligned}
\end{equation*}
By Assumption~\ref{Paper03_assumption_manifold}, the map $\boldsymbol{\Phi}$ 
is well defined. Moreover, since by 
Assumption~\ref{Paper03_geometric_properties_Gamma}, the manifold $\Gamma$ 
can be parametrised in terms of the first coordinate, in order to obtain the 
first and second order derivatives of $\boldsymbol{\Phi}$ on $\Gamma$, it 
will suffice to compute the first and second order derivatives of $\Phi_{0}$ 
on $\Gamma$.

By Assumption~\ref{Paper03_assumption_manifold}, for any fixed 
$\boldsymbol{x} \in \Gamma$, $\mathcal{J}\mathbf{F}_{\boldsymbol{x}}$ has 
exactly one eigenvalue $0$ and $K$ eigenvalues with negative real part. 
We recall the standard nomenclature from the dynamical systems literature.  
Denote by $E_{c} \equiv E_{c}(\boldsymbol{x})$ the linear space spanned by 
the right eigenvector of $\mathcal{J}\mathbf{F}_{\boldsymbol{x}}$ associated 
with the eigenvalue~$0$, and $E_{s} \equiv E_{s}(\boldsymbol{x})$ the linear 
space spanned by the right eigenvectors associated with the other eigenvalues 
of $\mathcal{J}\mathbf{F}_{\boldsymbol{x}}$. We call $E_{c}$ the 
\emph{linear centre manifold}, and $E_{s}$ the \emph{linear stable manifold}. 
For any $\boldsymbol{x} \in \Gamma$, observe that $\Gamma$ is the 
non-linear centre manifold around $\boldsymbol{x}$ and by the Centre 
Manifold Theorem (see~\cite[Theorem~3.2.1]{guckenheimer02}), $E_c$ is tangent 
to $\Gamma$ at $\boldsymbol{x}$. Also by the Centre Manifold Theorem, 
there exists a unique invariant $K$-dimensional manifold 
$\mathfrak{W}_{s} \equiv \mathfrak{W}_{s}(\boldsymbol{x})$, tangent to 
$E_{s}$ at $\boldsymbol{x}$, such that for all 
$\boldsymbol{y} \in \mathfrak{W}_{s}$, we have 
$\boldsymbol{\Phi}(\boldsymbol{y}) = \boldsymbol{x}$.

For any $\boldsymbol{x} \in \Gamma$, 
let $\boldsymbol{v} \equiv \boldsymbol{v}(\boldsymbol{x}) \in \mathbb{R}^{K+1}$
be the left eigenvector of $\mathcal{J}\mathbf{F}_{\boldsymbol{x}}$ associated 
with the eigenvalue~$0$, 
i.e.~$\mathcal{J}\mathbf{F}_{\boldsymbol{x}}^{*}\boldsymbol{v}(\boldsymbol{x}) 
= 0$, with $\boldsymbol{v}(\boldsymbol{x}) \neq 0$. 
By Assumption~\ref{Paper03_geometric_properties_Gamma}, $\Gamma$ is a 
twice-continuously differentiable curve of the first coordinate, and so
without loss of generality we use the parametrisation
$\boldsymbol{v}(\boldsymbol{x}) \equiv \boldsymbol{v}(x_{0})$.
%, as a parametrisation of the left eigenvector $\boldsymbol{v}$ in terms of the first coordinate $x_{0}$. 
Parsons and Rogers observe (see Equation~(12) in~\cite{parsons2017dimension}) 
that for any $\boldsymbol{x} \in \Gamma$, the set of points 
$\boldsymbol{y} \in \mathbf{R}^{K+1} \cap \mathcal{B}(\boldsymbol{x}, 
\varepsilon)$ such that $\boldsymbol{y} \in \mathfrak{W}_{s}(\boldsymbol{x})$ 
is approximated up to second order by
\begin{equation} 
\label{Paper03_Parsons_Rogers_description_flow}
    \Big\langle \boldsymbol{v}(x_{0}), \boldsymbol{y} - 
\boldsymbol{\gamma}(x_{0}) \Big\rangle 
- \frac{1}{2} \Big\langle \boldsymbol{y} - \boldsymbol{\gamma}(x_{0}), 
\boldsymbol{\Theta}(x_{0})(\boldsymbol{y} - \boldsymbol{\gamma}(x_{0})) 
\Big\rangle + \mathcal{O}(\varepsilon^{3}) = 0,
\end{equation}
where $\boldsymbol{\Theta}(x_{0})$ is a symmetric matrix describing the 
curvature of~$ \mathfrak{W}_{s}(\boldsymbol{x})$.

\begin{remark}
    We emphasise that Equation~\eqref{Paper03_Parsons_Rogers_description_flow} 
differs from Equation~(12) in~\cite{parsons2017dimension} by the 
factor $\frac{1}{2}$ in front of the second term on the left-hand side.
% of~\eqref{Paper03_Parsons_Rogers_description_flow}, which implies 
As a result, the matrix $\boldsymbol{\Theta}(x_0)$ 
in~\eqref{Paper03_Parsons_Rogers_description_flow} is exactly 2 times 
the corresponding matrix $\Theta$ %that appears in Equation~(12) 
in~\cite{parsons2017dimension}. Although this results in a minor 
inconvenience when directly comparing the two papers, the normalisatation 
in~\eqref{Paper03_Parsons_Rogers_description_flow} will be convenient for 
our quadratic approximation of the nonlinear stable manifold 
$\mathfrak{W}_s$ in Section~\ref{Paper03_general_result_derivatives_manifold}.
\end{remark}

By taking a Taylor expansion up to second order of the left-hand side 
of~\eqref{Paper03_Parsons_Rogers_description_flow}, Parsons and Rogers 
are able to derive explicit expressions for the derivatives of $\Phi_{0}$. 
For the first order derivatives, for any $\boldsymbol{x} \in \Gamma$,
\begin{equation} 
\label{Paper03_formula_Parsons_Rogers_first_order_derivatives}
    \frac{\partial \Phi_{0}}{\partial x_{i}}(\boldsymbol{x}) 
= \frac{v_{i}}{\sum_{l = 0}^{K} v_{l} \gamma_{l}'} 
\quad \forall i \in \llbracket K \rrbracket_{0},
\end{equation}
and for the second order derivatives,
\begin{equation} 
\label{Paper03_formula_Parsons_Rogers_second_order_derivatives}
   \frac{\partial^{2} \Phi_{0}}{\partial x_{i} \partial x_{j}}(\boldsymbol{x}) 
= \frac{1}{\sum_{l = 0}^{K} v_{l} \gamma_{l}'} 
\left(v_{i}' \frac{\partial \Phi_{0}}{\partial x_{j}} 
+ v_{j}' \frac{\partial \Phi_{0}}{\partial x_{i}} - \theta_{ij} 
- \frac{\partial \Phi_{0}}{\partial x_{i}}
\frac{\partial \Phi_{0}}{\partial x_{j}} 
\sum_{l = 0}^{K}(2v_{l}'\gamma_{l}' + v_{l}\gamma_{l}'')\right) 
\quad \forall i,j \in \llbracket K \rrbracket_{0}.
\end{equation}

Observe that in 
both~\eqref{Paper03_formula_Parsons_Rogers_first_order_derivatives} 
and~\eqref{Paper03_formula_Parsons_Rogers_second_order_derivatives}, 
the derivatives of the coordinates of $\boldsymbol{v}$ and 
$\boldsymbol{\gamma}$ are taken with respect to~$x_{0}$. Hence, 
Parsons and Rogers reduce the problem of computing the first and second order 
derivatives of $\Phi_{0}$ to a problem of characterising the dynamical system 
near the attractor manifold. Moreover, neither the parametrisation of 
$\Gamma$ given by $\boldsymbol{\gamma}$ nor the left eigenvector 
$\boldsymbol{v}$ should be difficult to compute, since they follow from 
standard computations concerning the Jacobian of the vector field $\mathbf{F}$.

The challenge is then to compute the matrix $\boldsymbol{\Theta}$. 
In~\cite[Appendix~B]{parsons2017dimension}, Parsons and Rogers propose 
performing this 
computation by finding an explicit basis of generalised eigenvectors of 
$\mathcal{J}\mathbf{F}_{\boldsymbol{x}}$ and changing the coordinates near 
a point $\boldsymbol{x} \in \Gamma$. Although this is the usual approach
in the literature on dynamical systems, it is not feasible for our 
application for large values of $K$, since for $K \geq 5$, the characteristic 
polynomial of $\mathcal{J}\mathbf{F}_{\boldsymbol{x}}$ may not have explicit 
roots. In the next section, we find an alternative formula for 
$\boldsymbol{\Theta}$ which sidesteps the explicit computation of 
the eigenvalues of $\mathcal{J}\mathbf{F}_{\boldsymbol{x}}$.

\subsection{Lyapunov-Schmidt method for the approximation of the non-linear stable manifold}
\label{Paper03_general_result_derivatives_manifold}

In this section, we present an approach to computing the curvature of 
the nonlinear stable manifold that does not rely on the computation of a 
basis of generalised eigenvectors of the Jacobian matrix of the flow. 
We emphasise that although our primary aim is to apply this approach to 
the ODE arising from our model, the results of this section are applicable 
to any dynamical system satisfying 
Assumptions~\ref{Paper03_assumption_manifold} 
and~\ref{Paper03_geometric_properties_Gamma}.

We continue to work in the setting of 
Section~\ref{Paper03_subsection_computing_derivatives}, 
i.e.~assume we are studying the ODE~\eqref{Paper03_deterministic_flow} 
which has an invariant attractor manifold $\Gamma$ defined 
by~\eqref{Paper03_definition_manifold}, and  
Assumptions~\ref{Paper03_assumption_manifold} 
and~\ref{Paper03_geometric_properties_Gamma} are both satisfied. 
For any $\boldsymbol{x} \in \Gamma$, let 
$\boldsymbol{u} \equiv \boldsymbol{u}(\boldsymbol{x})$ be a right eigenvector 
of $\mathcal{J}\mathbf{F}_{\boldsymbol{x}}$ associated with the eigenvalue~$0$, 
and let $\boldsymbol{v} \equiv \boldsymbol{v}(\boldsymbol{x})$ be a left 
eigenvector of $\mathcal{J}\mathbf{F}_{\boldsymbol{x}}$ associated with the 
eigenvalue~$0$. We choose the pair of vectors $\boldsymbol{u}$ and 
$\boldsymbol{v}$ so that $\langle \boldsymbol{u}, \boldsymbol{v} \rangle = 1$. 
%Without loss of generality, we 
We assume that both $\boldsymbol{u}$ and $\boldsymbol{v}$ are known. 
Observe that for each $\boldsymbol{x} \in \Gamma$, the linear centre manifold 
$E_{c} \equiv E_{c}(\boldsymbol{x})$ is given by 
$E_{c} = \spanvector(\boldsymbol{u})$. 
Let $E_{s} \equiv E_{s}(\boldsymbol{x})$ be the linear stable manifold spanned 
by the right eigenvectors of $\mathcal{J}\mathbf{F}_{\boldsymbol{x}}$ 
associated with 
the eigenvalues of $\mathcal{J}\mathbf{F}_{\boldsymbol{x}}$ with 
negative real part. Recall that $\Gamma$ is a non-linear centre manifold, and 
let $\mathfrak{W}_{s} \equiv \mathfrak{W}_{s}(\boldsymbol{x})$ be the 
non-linear stable manifold tangent to $E_{s}$ at $\boldsymbol{x}$.

We will apply the Lyapunov-Schmidt reduction method to approximate 
$\mathfrak{W}_{s}$ up to second order. See for instance Section~I.3 and 
Chapter~VII in~\cite{golubitsky1985singularities}, or 
Chapter~5 in~\cite{kuznetsov1998elements} for applications 
of this method to the study of bifurcations.

Define the linear maps
\begin{equation} 
\label{Paper03_projection_linear_center_manifold}
\begin{aligned}
    \mathcal{P}_{c}: \mathbb{R}^{K+1} & \rightarrow E_{c} \\
    \boldsymbol{y} & \mapsto \mathcal{P}_{c}(\boldsymbol{y}) 
= \langle \boldsymbol{y}, \boldsymbol{v} \rangle \boldsymbol{u},
\end{aligned}
\end{equation}
and
\begin{equation} 
\label{Paper03_projection_linear_stable_manifold}
\begin{aligned}
    \mathcal{P}_{s}: \mathbb{R}^{K+1} & \rightarrow E_{s} \\
    \boldsymbol{y} & \mapsto \mathcal{P}_{s}(\boldsymbol{y}) 
= \boldsymbol{y} - \langle \boldsymbol{y}, \boldsymbol{v} \rangle 
\boldsymbol{u}.
\end{aligned}
\end{equation}
Note that we can also write $\mathcal{P}_{s} = \mathbf{I} - \mathcal{P}_{c}$, 
where $\mathbf{I}: \mathbb{R}^{K+1} \rightarrow \mathbb{R}^{K+1}$ is the 
identity map. Hence, $\mathcal{P}_{c}$ is the map which projects 
$\mathbb{R}^{K+1}$ onto the linear centre manifold $E_{c}(\boldsymbol{x})$, 
while $\mathcal{P}_{s}$ projects $\mathbb{R}^{K+1}$ onto the linear stable 
manifold $E_{s}(\boldsymbol{x})$.

Our next result shows that the matrix $\boldsymbol{\Theta}$ can be computed 
in terms of $\boldsymbol{u}$, $\boldsymbol{v}$, $\mathcal{P}_s$, and using 
the Jacobian and the Hessians of the flow map at $\boldsymbol{x}$.

\begin{theorem}
\label{Paper03_prop_uniqueness_quadratic_term_nonlinear_manifold}
    The matrix $\boldsymbol{\Theta}$ 
in~\eqref{Paper03_Parsons_Rogers_description_flow} is the unique symmetric 
square matrix of order $K+1$ satisfying the following identities
    \begin{align}
        \label{Paper03_lyapunov_equation}
       \mathcal{J}\mathbf{F}^{*}_{\boldsymbol{x}} \boldsymbol{\Theta} 
+ \boldsymbol{\Theta} \mathcal{J}\mathbf{F}_{\boldsymbol{x}} 
& = \mathcal{P}_{s}^{*} \left(\sum_{i = 0}^{K} v_{i} 
\Hess_{\boldsymbol{x}}(F_{i})\right) \mathcal{P}_{s}, \\ 
\boldsymbol{\Theta} \boldsymbol{u} & = 0 
\label{Paper03_restriction_theta_u_in_nullspace}. 
    \end{align}
\end{theorem}

Theorem~\ref{Paper03_prop_uniqueness_quadratic_term_nonlinear_manifold}, 
combined with Katzenberger's result 
(Theorem~\ref{Paper03_major_tool_katzenberger}) and 
the formulae~\eqref{Paper03_formula_Parsons_Rogers_first_order_derivatives} 
and~\eqref{Paper03_formula_Parsons_Rogers_second_order_derivatives} of 
Parsons and Rogers, allows us to explicitly characterise the limiting 
diffusion process for finite-dimensional stochastic dynamical systems of 
arbitrarily large order, provided that 
Assumptions~\ref{Paper03_assumption_manifold},~\ref{Paper03_assumption_katzenberger_stochastic_process} 
and~\ref{Paper03_geometric_properties_Gamma} are satisfied.  
The proof of 
Theorem~\ref{Paper03_prop_uniqueness_quadratic_term_nonlinear_manifold} 
relies on an application of the Centre Manifold Theorem. We 
postpone it until Section~\ref{Paper03_section_proofs_dynamical_systems}.

Note that Equation~\eqref{Paper03_lyapunov_equation} is a Lyapunov equation 
(also known as a Silvester equation), of semi-stable type. The term 
\emph{semi-stable} refers to the fact that the Jacobian matrix 
$\mathcal{J}\boldsymbol{F_x}$ has one eigenvalue $0$ and $K$ eigenvalues with 
negative real part.
%, which follows from Assumptions~\ref{Paper03_assumption_manifold} and~\ref{Paper03_geometric_properties_Gamma}. In this regard, s
Stable Lyapunov equations are matrix equations of the form
\begin{equation*}
    \boldsymbol{A}^*\boldsymbol{Y} + \boldsymbol{Y} \boldsymbol{A} = \boldsymbol{C},
\end{equation*}
in which $\boldsymbol{A}$, $\boldsymbol{C}$ and $\boldsymbol{Y}$ are square 
matrices of order $K+1$, all the eigenvalues of $\boldsymbol{A}$ have 
negative real part, and $\boldsymbol{C}$ is symmetric. Stable Lyapunov 
equations are often used in control theory to study the asymptotic behaviour 
of linear systems (see for instance~\cite[Chapter~5]{sontag2013mathematical}).

Although from an abstract perspective, 
Theorem~\ref{Paper03_prop_uniqueness_quadratic_term_nonlinear_manifold} 
completely characterises the matrix $\boldsymbol{\Theta}$, the linear 
Lyapunov equation in~\eqref{Paper03_lyapunov_equation} can be 
computationally expensive. In our setting, when $K$ is large, solving 
Equation~\eqref{Paper03_lyapunov_equation} 
analytically is not feasible. Instead, in order to understand the 
behaviour of the matrix $\boldsymbol{\Theta}$, 
we shall adapt a strategy 
borrowed from the theory of stable Lyapunov equations, 
%when the linear system is too large to be analytically solved, a common strategy is to try to 
and directly classify the solution matrix as positive or negative 
(semi)definite .

Recall that a real symmetric matrix $\mathbf{M}$ of order $K+1$ is called 
positive definite (resp.~positive semidefinite) if and only if 
$\langle \boldsymbol{y}, \mathbf{M} \boldsymbol{y} \rangle > 0$ for all 
$\boldsymbol{y} \in \mathbb{R}^{K+1}\setminus \{0\}$, 
(resp.~if and only if 
$\langle \boldsymbol{y}, \mathbf{M} \boldsymbol{y} \rangle \geq 0$ 
for all $\boldsymbol{y} \in \mathbb{R}^{K+1}$). 
Similarly, $\mathbf{M}$ is negative definite (resp.~negative semidefinite) 
if and only if $\langle \boldsymbol{y}, \mathbf{M} \boldsymbol{y} \rangle < 0$ 
for all $\boldsymbol{y} \in \mathbb{R}^{K+1}\setminus \{0\}$ 
(resp.~if and only if 
$\langle \boldsymbol{y}, \mathbf{M} \boldsymbol{y} \rangle \leq 0$ 
for all $\boldsymbol{y} \in \mathbb{R}^{K+1}$). 
A real symmetric matrix is positive semidefinite (resp.~negative semidefinite) 
if and only if all its eigenvalues are nonnegative (resp.~not positive),
see e.g.~\cite[Chapter~1]{bhatia2009positive} for a review.
% of positive and negative matrices). Consider 
We define a partial order $\succeq$ on the space of real symmetric
matrices by $\mathbf{M} \succeq \mathbf{Q}$ if and only if 
$\mathbf{M} - \mathbf{Q}$ is positive semidefinite. 
We also write $\mathbf{M} \succeq 0$ (resp.~$\preceq 0$) if and only if 
$\mathbf{M}$ is positive semidefinite (resp.~negative semidefinite). 
We are ready to state a formula for the solution of 
Equation~\eqref{Paper03_lyapunov_equation} that will give us information 
about the signs of the real part of the eigenvalues of $\boldsymbol{\Theta}$.

\begin{theorem} 
\label{Paper03_explicit_formula_solution_lyapunov_equation}
    Let $\boldsymbol{\Theta}$ be the unique symmetric matrix satisfying 
both~\eqref{Paper03_lyapunov_equation} 
and~\eqref{Paper03_restriction_theta_u_in_nullspace}. 
Then, $\boldsymbol{\Theta}$ satisfies the formula
    \begin{equation*}
        \boldsymbol{\Theta} 
= - \int_{0}^{+ \infty} 
\exp({\mathcal{J}\mathbf{F}^{*}_{\boldsymbol{x}}t})\mathcal{P}_{s}^{*} 
\left(\sum_{i = 0}^{K} v_{i} \Hess_{\boldsymbol{x}}(F_{i})\right) 
\mathcal{P}_{s}\exp({\mathcal{J}\mathbf{F}_{\boldsymbol{x}}t}) \, dt.
    \end{equation*}
    Moreover, if 
$\mathcal{P}_{s}^{*} \left(\sum_{i = 0}^{K} v_{i} 
\Hess_{\boldsymbol{x}}(F_{i})\right) \mathcal{P}_{s} \succeq 0$ 
(resp.~$\preceq 0$), then $\boldsymbol{\Theta} \preceq 0$ (resp.~$\succeq 0$).
\end{theorem}

The proof of Theorem~\ref{Paper03_explicit_formula_solution_lyapunov_equation} 
is also postponed until Section~\ref{Paper03_section_proofs_dynamical_systems}.
We observe that the presence of the projection matrices 
$\mathcal{P}_{s}$ and $\mathcal{P}^{*}_{s}$ inside the integral formula is 
pivotal for the proof of 
Theorem~\ref{Paper03_explicit_formula_solution_lyapunov_equation}. Indeed, 
if we do not include these matrices, the integral is not well defined. 
In particular, this construction with projection matrices could also be 
applied to other semistable Lyapunov equations. This would correct a small 
mistake made by Parsons and Rogers in Equation~(C.9) 
in~\cite{parsons2017dimension}, in which they represent the solution of a 
semi-stable Lyapunov equation in integral form without explicitly inserting 
the projection matrices inside the integrand.

\subsection{Application to a specific ODE} 
\label{Paper03_derivatives_specific_ode}

In this section, we will apply 
Theorems~\ref{Paper03_prop_uniqueness_quadratic_term_nonlinear_manifold} 
and~\ref{Paper03_explicit_formula_solution_lyapunov_equation} to study 
the particular ODE arising from our Wright-Fisher model with seed bank.

Recall the flow field $\mathbf{F}: [0,1]^{K+1} \rightarrow \mathbb{R}^{K+1}$ 
given by~\eqref{Paper03_flow_definition_WF_model}:
\begin{equation*}
     \left\{
    \arraycolsep=1.4pt\def\arraystretch{2.2}
    \begin{array}{cl} F_{0}(\boldsymbol{x}) 
& \defeq \displaystyle{\frac{(1 - x_{0})\left(\sum_{i = 1}^{K} 
b_{i}x_{i}- (1 - b_{0})x_{0}\right)}{(1 - x_{0}) 
+ \sum_{i = 0}^{K} b_{i}x_{i}}}, \\
    F_{i}(\boldsymbol{x}) & \defeq x_{i - 1} - x_{i} \quad 
\forall \, i \in \llbracket K \rrbracket.
    \end{array}\right.
\end{equation*}
As explained in Section~\ref{Paper03_main_results_section}, for each 
$i \in \llbracket K \rrbracket_{0}$, $b_{i}$ represents the probability of a 
seed from a mutant individual being ready to germinate in exactly $i$ 
generations. We assume that the vector 
$\mathbf{b} = (b_{i})_{i \in \llbracket K \rrbracket_0}$ 
satisfies~\eqref{Paper03_assumption_probability_distribution_bi}. 
Recall also that the non-linear centre manifold 
$\Gamma = \{\boldsymbol{x} \in [0,1]^{K+1}: \, 
\mathbf{F}(\boldsymbol{x}) = 0\}$ in this example
is the diagonal of the $(K+1)$-dimensional 
hypercube:
\begin{equation*}
 \Gamma \defeq \Big\{\boldsymbol{x} \in [0,1]^{K+1}: \, x_{0} = x_{1} 
= \ldots = x_{K}\Big\}.
\end{equation*}
The parametrisation of $\Gamma$ is then trivially
given by the map $\boldsymbol{\gamma} : [0,1] \rightarrow \Gamma$, where, 
for all $x \in [0,1]$,
\begin{equation} 
\label{Paper03_parametrisation_attractor_manifold}
    \boldsymbol{\gamma}(x) \defeq \left(\begin{array}{c} x \\ x \\ x \\ 
\vdots \\ x
        \end{array}\right).
\end{equation}
We wish to study the behaviour of $\mathbf{F}$ near $\Gamma$. 
For any $\boldsymbol{x} = (x, \ldots, x) \in \Gamma$, 
the Jacobian of $\mathbf{F}$ evaluated at $\boldsymbol{x}$ is given by
    \begin{equation} 
\label{Paper03_formula_jacobian_manifold}
    \begin{aligned}
        \mathcal{J}\mathbf{F}_{\boldsymbol{x}} = \left(\begin{array}{ccccc}
           - (1 - b_{0})(1-x) & b_{1}(1-x) & b_{2}(1-x) & \cdots & b_{K}(1-x) \\
          1 & -1 & 0 & \cdots & 0 \\ 0 & 1 & -1 & \cdots & 0 \\ \vdots & \vdots & \vdots & \ddots & \vdots \\ 0 & 0 & 0 & \cdots & -1
        \end{array}\right).
    \end{aligned}
    \end{equation}
We first check that $\Gamma$ satisfies 
Assumption~\ref{Paper03_assumption_manifold} with respect to the 
map $\mathbf{F}$.

\begin{lemma} 
\label{Paper03_eigenvalue_condition_constant_environment}
For every $\boldsymbol{x} \in \Gamma$, the Jacobian matrix of the flow 
$\mathbf{F}$ evaluated at $\boldsymbol{x}$ has exactly one eigenvalue equal 
to $0$ and $K$ eigenvalues in the open disk 
$D(1) = \{\lambda \in \mathbb{C}: \; \vert \lambda + 1 \vert < 1 \}$.
\end{lemma}

\begin{proof}
    Observe that $\lambda \in \mathbb{C}$ is an eigenvalue of 
$\mathcal{J}\mathbf{F}_{\boldsymbol{x}}$ if and only $\lambda + 1$ is an 
eigenvalue of $\mathcal{J}\mathbf{F}_{\boldsymbol{x}} + \mathbf{I}_{K+1}$, 
where $\mathcal{J}\mathbf{F}_{\boldsymbol{x}} + \mathbf{I}_{K+1}$ is the 
square matrix of order $K + 1$ given by
    \begin{equation*}
        \mathcal{J}\mathbf{F}_{\boldsymbol{x}} + \mathbf{I}_{K+1} 
= \left(\begin{array}{ccccc}
       - (1 - b_{0})(1-x) + 1 & b_{1}(1-x) & b_{2}(1-x) & \cdots & b_{K}(1-x) \\
            1 & 0 & 0 & \cdots & 0 \\ 0 & 1 & 0 & \cdots & 0 \\ \vdots & \vdots
 & \vdots & \ddots & \vdots \\ 0 & 0 & 0 & \cdots & 0
        \end{array}\right).
    \end{equation*}
By directly computing the characteristic polynomial $\tilde{p}$ of 
$\mathcal{J}\mathbf{F}_{\boldsymbol{x}} + \mathbf{I}_{K+1}$, we obtain that 
for all $\tilde{\lambda} \in \mathbb{C}$,
   \begin{equation} \label{Paper03_characteristic_polynomial_J_+_I_const_env}
    \tilde{p}(\tilde{\lambda}) 
= (-1)^{K}\tilde{\lambda}^{K}({- (1 - b_{0})(1-x)} + 1 - 
\tilde{\lambda}) 
+ (-1)^{K}{(1 - x)}\sum_{i = 1}^{K}b_{i}\tilde{\lambda}^{K - i}.
   \end{equation}
To prove our claim, it suffices to establish that $\tilde{p}$ has exactly $K$ 
roots satisfying $\vert \tilde{\lambda} \vert < 1$. To see that this holds, 
note that since $ \mathcal{J}\mathbf{F}_{\boldsymbol{x}} + \mathbf{I}_{K+1}$ 
is a strictly positive matrix, by the Perron-Frobenius Theorem,
$ \mathcal{J}\mathbf{F}_{\boldsymbol{x}} + \mathbf{I}_{K+1}$ has a simple 
eigenvalue 
$\tilde{\lambda}_{ \left(\mathcal{J}\mathbf{F}_{\boldsymbol{x}} 
+ \mathbf{I}_{K+1}\right)}$ which is a real positive number, 
and which is strictly greater than the modulus of all other
eigenvalues.
%such that all other eigenvalues have modulus strict less than 
%$\tilde{\lambda}_{ \left(\mathcal{J}\mathbf{F}_{\boldsymbol{x}} 
%+ \mathbf{I}_{K+1}\right)}$. 
   
We claim that 
$\tilde{\lambda}_{ \left(\mathcal{J}\mathbf{F}_{\boldsymbol{x}} 
+ \mathbf{I}_{K+1}\right)} = 1$. Indeed, since 
$(1 - b_{0}) = \sum_{i = 1}^{K} b_{i}$, $\tilde{\lambda} = 1$ is a root 
of $\tilde{p}$. Suppose now that 
$\tilde{\lambda}_{ \left(\mathcal{J}\mathbf{F}_{\boldsymbol{x}} 
+ \mathbf{I}_{K+1}\right)} > 1$. 
In this case, by dividing the characteristic polynomial $\tilde{p}$ by 
$(\tilde{\lambda}_{ \left(\mathcal{J}\mathbf{F}_{\boldsymbol{x}} 
+ \mathbf{I}_{K+1}\right)})^K$, we obtain
   \begin{equation*}
    \begin{aligned}
        0 & = {- (1 - b_{0})(1- x)} + 1 
- \tilde{\lambda}_{ \left(\mathcal{J}\mathbf{F}_{\boldsymbol{x}} 
+ \mathbf{I}_{K+1}\right)} + {(1- x)} 
\sum_{i = 1}^{K}b_{i}\tilde{\lambda}_{ \left(\mathcal{J}
\mathbf{F}_{\boldsymbol{x}} + \mathbf{I}_{K+1}\right)}^{- i} \\ 
& < {- (1 - b_{0})(1- x)} + 1 
- \tilde{\lambda}_{ \left(\mathcal{J}\mathbf{F}_{\boldsymbol{x}} 
+ \mathbf{I}_{K+1}\right)} + {(1 - b_{0})(1- x)} \\ & < 0,
    \end{aligned}
   \end{equation*}
where we used that $1 - b_{0} = \sum_{i = 1}^{K}b_{i}$ to derive the second 
inequality, and the assumption 
$\tilde{\lambda}_{ \left(\mathcal{J}\mathbf{F}_{\boldsymbol{x}} 
+ \mathbf{I}_{K+1}\right)} > 1$ to obtain the final one. 
Since this is a contradiction, we conclude that 
$\tilde{\lambda}_{ \left(\mathcal{J}\mathbf{F}_{\boldsymbol{x}} 
+ \mathbf{I}_{K+1}\right)} = 1$, and the result follows.
\end{proof}

Lemma~\ref{Paper03_eigenvalue_condition_constant_environment} implies that 
the diagonal $\Gamma$ is an attractor manifold, and therefore we can define 
the projection map $\boldsymbol{\Phi}: [0,1]^{K+1} \rightarrow \Gamma$ as 
in~\eqref{Paper03_definition_projection_map}. We would like to apply 
the formulae~\eqref{Paper03_formula_Parsons_Rogers_first_order_derivatives} 
and~\eqref{Paper03_formula_Parsons_Rogers_second_order_derivatives} of
Parsons and Rogers to compute the first and second order derivatives of 
$\Phi_{0}$. With this in mind, let 
$\boldsymbol{u} \equiv \boldsymbol{u}(\boldsymbol{x}) \in \mathbb{R}^{K+1}$ 
be a right eigenvector of $\mathcal{J}\mathbf{F}_{\boldsymbol{x}}$ associated
with the eigenvalue $0$. To simplify computations, we set
\begin{equation} 
\label{Paper03_right_eigenvector_Jacobian}
    \boldsymbol{u} = \left(\begin{array}{c} 1 \\ 1 \\ 1 \\ \vdots \\ 1
        \end{array}\right).
\end{equation}
We want to take $\boldsymbol{v} \equiv \boldsymbol{v}({\boldsymbol{x}})$ 
to be the left 
eigenvector of $\mathcal{J}\mathbf{F}_{\boldsymbol{x}}$ associated with the 
eigenvalue $0$ 
(i.e.~$\mathcal{J}\mathbf{F}_{\boldsymbol{x}}^{*} \boldsymbol{v} = 0$)
for which $\left\langle \boldsymbol{v}, \boldsymbol{u}\right\rangle = 1$. 
This gives
\begin{equation} 
\label{Paper03_left_eigenvector_jacobian}
    \boldsymbol{v} 
= \frac{1}{(B(1 - x) + 1)}\left(\begin{array}{c} 1 \\ (1 - b_{0})(1 - x) \\ 
\left(\sum_{i = 2}^{K} b_{i}\right)(1 - x) \\ 
\vdots \\ b_{K} (1 - x)
        \end{array}\right),
\end{equation}
where as usual $B = \sum_{i=1}^{K}ib_{i}$ represents the mean germination time.
% of germination (see~\eqref{Paper03_mean_age_germination}). 
We are now in a position to use
Equation~\eqref{Paper03_formula_Parsons_Rogers_first_order_derivatives}
%, we are already in condition 
to compute the first order derivates of $\Phi_{0}$ and 
prove identity~\eqref{Paper03_first_order_derivatives_projection}.

\begin{lemma} 
\label{Paper03_lemma_first_order_derivatives_projection_map}
 For any $K \in \mathbb{N}$, $\mathbf{b} = (b_{i})_{i = 0}^{K}$ satisfying 
Assumption~\ref{Paper03_assumption_probability_distribution_bi}, and $\boldsymbol{x} \in \Gamma$, we have
 \begin{equation*}
 \frac{\partial \Phi_{0}}{\partial x_{0}} (\boldsymbol{x}) 
= \frac{1}{(B(1-x_{0})+1)} \quad \textrm{ and } \quad 
\frac{\partial \Phi_{0}}{\partial x_{i}} (\boldsymbol{x}) 
= \frac{\left(\sum_{j = i}^{K}b_{j}\right)(1-x_{0})}{(B(1-x_{0})+1)} \; 
\quad\forall i \in \llbracket K \rrbracket.
    \end{equation*}
\end{lemma}

\begin{proof}
    From~\eqref{Paper03_parametrisation_attractor_manifold}, the 
parametrisation curve~$\boldsymbol{\gamma}$ satisfies
    \begin{equation*}
        \gamma'(x_{0}) = \boldsymbol{1} = (1, 1, 1, \ldots, 1)^* \quad \textrm{and} \quad \gamma''(x_{0}) = 0, 
    \end{equation*}
    so that 
    \begin{equation*}
        \sum_{l = 0}^{K} v_{l} \gamma'_{l} \equiv 1.
    \end{equation*}
Substituting in~\eqref{Paper03_formula_Parsons_Rogers_first_order_derivatives} 
(using~\eqref{Paper03_left_eigenvector_jacobian}) completes the proof.
\end{proof}

We now turn to properties of the second order derivatives. 
Since our aim is to derive these properties for general $K \in \mathbb{N}$, 
including in particular $K > 5$, we 
%cannot follow the approach suggested by Parsons and Rogers to compute the curvature of the non-linear stable manifold (observe that as shown in the proof of Lemma~\ref{Paper03_eigenvalue_condition_constant_environment}, it is not possible to determine for any $K \in \mathbb{N}$ the roots of the characteristic polynomial of $\mathcal{J}\mathbf{F}_{\boldsymbol{x}}$). We shall
apply the Lyapunov-Schmidt method described in the previous section. 
We start by observing that the linear space spanned by $\boldsymbol{u}$ is the 
linear centre manifold $E_{c} \equiv E_{c}(\boldsymbol{x})$ at 
$\boldsymbol{x}$. By~\eqref{Paper03_projection_linear_center_manifold}, the 
projection $\mathcal{P}_{c}$ onto $E_{c}$ is given by 
$\mathcal{P}_{c}(\boldsymbol{y}) 
= \langle \boldsymbol{y}, \boldsymbol{v} \rangle \boldsymbol{u}$. 
It will be convenient to represent $\mathcal{P}_{c}$ as a matrix
\begin{equation} 
\label{Paper03_computation_projection_matrix_center_linear_manifold}
    \mathcal{P}_{c} = \frac{1}{(B(1 - x) + 1)}\left(\begin{array}{ccccc}
        1 & (1-b_{0})(1 - x) & \left(\sum_{i = 2}^{K} b_{i}\right)(1 - x) 
& \cdots & b_{K} (1 - x) \\
          1 & (1-b_{0})(1 - x) & \left(\sum_{i = 2}^{K} b_{i}\right)(1 - x) 
& \cdots & b_{K} (1 - x) \\ 1 & (1-b_{0})(1 - x) & 
\left(\sum_{i = 2}^{K} b_{i}\right)(1 - x) & \cdots & b_{K} (1 - x) \\ 
\vdots & \vdots & \vdots & \ddots & \vdots \\ 
1 & (1-b_{0})(1 - x) & \left(\sum_{i = 2}^{K} b_{i}\right)(1 - x) 
& \cdots & b_{K} (1 - x)
    \end{array}\right).
\end{equation}
To simplify notation, for all $j \in \llbracket K \rrbracket$, we define
\begin{equation} 
\label{Paper03_auxiliary_B_j}
    B_{j} \defeq B - \sum_{q = j}^{K} b_{q}.
\end{equation}
From identity~\eqref{Paper03_projection_linear_stable_manifold}, 
the projection matrix $\mathcal{P}_{s}$ onto the linear stable manifold 
$E_{s}$ is given by $\mathcal{P}_{s} = \mathbf{I}_{K+1} - \mathcal{P}_{c}$. 
Substituting the 
expression~\eqref{Paper03_computation_projection_matrix_center_linear_manifold},
we find
\begin{equation} 
\label{Paper03_computation_projection_matrix_stable_linear_manifold}
\begin{aligned}
    & (B(1 - x) + 1)\mathcal{P}_{s} \\ & \,  = \left(\begin{array}{ccccc}
        B(1 - x) & - (1-b_{0})(1 - x) & - \left(\sum_{i = 2}^{K} b_{i}\right)(1 - x) & \cdots & - b_{K} (1 - x) \\
          -1 & B_{1}(1 - x) + 1 & -\left(\sum_{i = 2}^{K} b_{i}\right)(1 - x) 
& \cdots & -b_{K} (1 - x) \\ -1 & - (1-b_{0})(1 - x)& B_{2}(1 - x) + 1 
& \cdots & -b_{K} (1 - x) \\ 
\vdots & \vdots & \vdots & \ddots & \vdots \\ 
-1 & -(1-b_{0})(1 - x) & -\left(\sum_{i = 2}^{K} b_{i}\right)(1 - x) 
& \cdots & B_{K} (1 - x) + 1
    \end{array}\right),
\end{aligned}
\end{equation}
with transpose $\mathcal{P}_{s}^{*}$
\begin{equation} 
\label{Paper03_computation_transpose_projection_matrix_stable_linear_manifold}
\begin{aligned}
    & (B(1 - x) + 1)\mathcal{P}_{s}^{*} \\ & \,  
= \left(\begin{array}{ccccc}
        B(1 - x) & - 1 & - 1 & \cdots & - 1 \\
          - (1-b_{0})(1 - x) & B_{1}(1 - x) + 1 & - (1-b_{0})(1 - x) 
& \cdots & - (1-b_{0})(1 - x) \\ 
-\left(\sum_{i = 2}^{K} b_{i}\right)(1 - x) & 
-\left(\sum_{i = 2}^{K} b_{i}\right)(1 - x) & B_{2}(1 - x) + 1 
& \cdots & -\left(\sum_{i = 2}^{K} b_{i}\right)(1 - x) \\ 
\vdots & \vdots & \vdots & \ddots & \vdots \\ 
-b_{K} (1 - x) & -b_{K} (1 - x) & -b_{K} (1 - x) & \cdots & B_{K} (1 - x) + 1
    \end{array}\right).
\end{aligned}
\end{equation}
Let $\boldsymbol{\Theta}$ be the matrix describing the curvature of the 
non-linear stable manifold $\mathfrak{W}_{s} \equiv \mathfrak{W}_{s}(\boldsymbol{x})$. We apply Theorem~\ref{Paper03_prop_uniqueness_quadratic_term_nonlinear_manifold}. 
In particular, $\boldsymbol{\Theta}$ is a symmetric square matrix of 
order $K+1$
%, i.e.
%\begin{equation*}
%    \boldsymbol{\Theta} = (\theta_{ij})_{i,j \in \llbracket K \rrbracket_{0}} = \left(\begin{array}{ccccc}
%        \theta_{00} & \theta_{01} & \theta_{02} & \cdots & \theta_{0K} \\
%        \theta_{01} & \theta_{11} & \theta_{12} & \cdots & \theta_{1K} \\ \theta_{02} & \theta_{12} & \theta_{22} & \cdots & \theta_{2K} \\ \vdots & \vdots & \vdots & \ddots & \vdots \\ \theta_{0K} & \theta_{1K} & \theta_{2K} & \cdots & \theta_{KK}
%    \end{array}\right),
%\end{equation*}
satisfying $\boldsymbol{\Theta} \boldsymbol{u} = 0$. By the definition of 
the right eigenvector $\boldsymbol{u}$ 
in~\eqref{Paper03_right_eigenvector_Jacobian} and the fact that
$\boldsymbol{\Theta}$ is symmetric, this last condition can be written as
\begin{equation} 
\label{Paper03_null_space_theta}
    \left\{\begin{array}{ccc}
         \theta_{00} + \theta_{01} + \theta_{02} + \cdots + \theta_{0K} & = & 0 \\
        \theta_{01} + \theta_{11} + \theta_{12} + \cdots + \theta_{1K} & 
= & 0 \\ 
\theta_{02} + \theta_{12} + \theta_{22} + \cdots + \theta_{2K} & 
= & 0 \\ 
\vdots & \vdots & \vdots \\ 
\theta_{0K} + \theta_{1K} + \theta_{2K} + \cdots + \theta_{KK} & = & 0
    \end{array}\right. .
\end{equation}
Moreover, %By Lemma~\ref{Paper03_computation_quadratic_term_nonlinear_manifold_first_step},
$\boldsymbol{\Theta}$ satisfies the semistable Lyapunov equation
\begin{equation} 
\label{Paper03_linear_system_original_ODE}
       \mathcal{J}\mathbf{F}^{*}_{\boldsymbol{x}} \boldsymbol{\Theta} 
+ \boldsymbol{\Theta} \mathcal{J}\mathbf{F}_{\boldsymbol{x}} 
= \mathcal{P}_{s}^{*} \left(\sum_{i = 0}^{K} v_{i} 
\Hess_{\boldsymbol{x}}(F_{i})\right) \mathcal{P}_{s}.
    \end{equation}
%Corollary~\ref{Paper03_computation_symmetric_matrix_easier_linear_system} 
%implies that in order to compute $\boldsymbol{\Theta}$, it is enough to 
We wish then to solve~\eqref{Paper03_linear_system_original_ODE} 
and~\eqref{Paper03_null_space_theta}. First we compute the terms that 
appear on each side of~\eqref{Paper03_linear_system_original_ODE}. Starting 
with the left-hand side, note that by~\eqref{Paper03_formula_jacobian_manifold},
the transpose of the Jacobian of $\mathbf{F}$ is given by
\begin{equation} 
\label{Paper03_formula_transpose_jacobian}
\mathcal{J}\mathbf{F}^{*}_{\boldsymbol{x}} = \left(\begin{array}{ccccc}
    - (1 - b_{0})(1 - x) & 1 & 0 & \cdots & 0 \\
    b_{1}(1 - x) & -1 & 1 & \cdots & 0 \\  b_{2}(1 - x) & 0 & -1 & \cdots & 0 \\ \vdots & \vdots & \vdots & \ddots & \vdots \\ b_{K}(1 - x) & 0 & 0 & \cdots & -1
\end{array}\right).
\end{equation}
Combining~\eqref{Paper03_formula_jacobian_manifold} 
and~\eqref{Paper03_formula_transpose_jacobian}, we have that for all 
$i,j \in \llbracket K \rrbracket$,
\begin{equation} 
\label{Paper03_left_hand_side_partial_terms}
    \left\{\begin{array}{ccl}
        \left(\mathcal{J}\mathbf{F}^{*}_{\boldsymbol{x}} \boldsymbol{\Theta} 
+ \boldsymbol{\Theta} \mathcal{J}\mathbf{F}_{\boldsymbol{x}}\right)_{00} & 
= & -2 (1 - b_{0})(1 - x) \theta_{00} + 2\theta_{01}, \\
         \left(\mathcal{J}\mathbf{F}^{*}_{\boldsymbol{x}} \boldsymbol{\Theta} 
+ \boldsymbol{\Theta} \mathcal{J}\mathbf{F}_{\boldsymbol{x}}\right)_{0j} & 
= & b_{j}(1- x) \theta_{00} + (-(1 - b_{0})(1 - x) - 1)\theta_{0j} + 
\mathds{1}_{\{j < K\}}\theta_{0,(j+1)} + \theta_{1j}, \\ 
\left(\mathcal{J}\mathbf{F}^{*}_{\boldsymbol{x}} \boldsymbol{\Theta} 
+ \boldsymbol{\Theta} \mathcal{J}\mathbf{F}_{\boldsymbol{x}}\right)_{ij} 
& = & b_{i}(1 - x)\theta_{0j} + b_{j} (1 - x) \theta_{0i} - 2 \theta_{ij} 
+ \mathds{1}_{\{i < K\}} \theta_{i+1, j} 
+ \mathds{1}_{\{j < K\}} \theta_{i, j+1}.
    \end{array}\right.
\end{equation}

We now compute the right-hand side 
of~\eqref{Paper03_linear_system_original_ODE}. 
Observe that from the definition of the 
flow~\eqref{Paper03_flow_definition_WF_model}, for all 
$i \in \llbracket K \rrbracket$ we have 
$\Hess_{\boldsymbol{x}}(F_{i}) \equiv 0$. Therefore, we conclude that
\begin{equation*}
    \mathcal{P}_{s}^{*} \left(\sum_{i = 0}^{K} v_{i} 
\Hess_{\boldsymbol{x}}(F_{i})\right) \mathcal{P}_{s} 
= v_{0} \mathcal{P}_{s}^{*}\Hess_{\boldsymbol{x}}(F_{0}) \mathcal{P}_{s}.
\end{equation*}
Recalling the definition of $\mathbf{F}$ 
in~\eqref{Paper03_flow_definition_WF_model}, we can compute the second order 
derivatives of $F_{0}$ on $\Gamma$, obtaining for all 
$\boldsymbol{x} = (x, \ldots, x) \in \Gamma$,
\begin{equation}  
\label{Paper03_computation_hessian_manifold}
\left\{\arraycolsep=1.4pt\def\arraystretch{2.2}\begin{array}{ccc}
    \displaystyle{\frac{\partial^{2}F_{0}}{\partial x_{0}^{2}}(\boldsymbol{x})}
 & = & 2(1 - b_{0}) - 2(1-b_{0})^{2}(1 - x), \\
    \displaystyle{\frac{\partial^{2}F_{0}}{\partial x_{0} \partial x_{j}}
(\boldsymbol{x})} 
& = & 2b_{j}(1-b_{0})(1 - x) - b_{j} \quad 
\forall j \in \llbracket K \rrbracket, \\ 
\displaystyle{\frac{\partial^{2}F_{0}}{\partial x_{i} \partial x_{j}}
(\boldsymbol{x})} & = & -2b_{i}b_{j}(1 - x) \quad \forall i,j 
\in \llbracket K \rrbracket.
\end{array}\right.
\end{equation}
It will be convenient to decompose $ \Hess_{\boldsymbol{x}}(F_{0})$ as
\begin{equation} 
\label{Paper03_decomposition_hessian_original_flow}
  \Hess_{\boldsymbol{x}}(F_{0}) = \Hess_{\boldsymbol{x}}(G_{0}) + 2(1-x)\Delta, 
\end{equation}
where
\begin{equation} 
\label{Paper03_hessian_alternative_ODE}
    \Hess_{\boldsymbol{x}}(G_{0}) \defeq \left(\begin{array}{ccccc}
    2(1 - b_{0}) & -b_{1} & -b_{2} & \cdots & -b_{K} \\
  -b_{1} & 0 & 0 & \cdots & 0 \\  -b_{2} & 0 & 0 & \cdots & 0 \\ 
\vdots & \vdots & \vdots & \ddots & \vdots \\ -b_{K} & 0 & 0 & \cdots & 0
\end{array}\right),
\end{equation}
and
\begin{equation} 
\label{Paper03_difference_matrix_between_hessians}
    \Delta \defeq \left(\begin{array}{ccccc}
  -(1 - b_{0})^{2} & b_{1}(1-b_{0}) & b_{2}(1-b_{0}) 
& \cdots & b_{K}(1 - b_{0}) \\
    b_{1}(1-b_{0}) & -b_{1}^{2} & -b_{1}b_{2} & \cdots & -b_{1}b_{K} \\  
b_{2}(1-b_{0}) & -b_{1}b_{2} & -b_{2}^{2} & \cdots & -b_{2}b_{K} \\ 
\vdots & \vdots & \vdots & \ddots & \vdots \\ 
b_{K}(1-b_{0}) & -b_{1}b_{K} & -b_{2}b_{K} & \cdots & -b_{K}^{2}
\end{array}\right).
\end{equation}
Observe from~\eqref{Paper03_hessian_alternative_ODE} that 
$\Hess_{\boldsymbol{x}}(G_{0})$ is the Hessian of the function 
$G_{0}:[0,1]^{K+1} \rightarrow \mathbb{R}$ given 
in~\eqref{Paper03_modified_flow_definition_WF_model}.
% and which corresponds 
%from a biological point of view to the dynamics of mature mutant individuals 
%in a population in which seeds would thrive without competition between them. 
Since the Lyapunov equation is linear, it is enough to find symmetric matrices 
$\boldsymbol{\Theta}^{G}$ and $\boldsymbol{\Theta}^{\Delta}$ 
solving Equations~\eqref{Paper03_null_space_theta} 
and~\eqref{Paper03_linear_system_original_ODE}, where in the last equation 
we replace $\Hess_{\boldsymbol{x}}(F_{0})$ by $\Hess_{\boldsymbol{x}}(G_{0})$ 
and $2(1-x)\Delta$, respectively. There are no straightforward techniques for 
finding solutions to the systems of linear equations that are valid for any 
$K \in \mathbb{N}$. One possible strategy is to solve the linear system for 
small values of $K$ in order to get some intuition for what the solution 
should look like, and then prove directly that the candidate solution satisfies 
the linear system. If uniqueness holds, as it does in our case, this strategy, 
when applicable, allows us to rigorously establish the solution to the system. 
This is the strategy that we will follow to compute $\boldsymbol{\Theta}^{G}$.

\begin{lemma} \label{Paper03_explicit_formula_theta_G_component}
    For any $K \in \mathbb{N}$ and $\mathbf{b} = (b_{i})_{i = 0}^{K}$ 
satisfying Equation~\eqref{Paper03_assumption_probability_distribution_bi}, 
the symmetric matrix 
$\boldsymbol{\Theta}^{G} \defeq (\theta^{G}_{ij})_{i,j\in 
\llbracket K \rrbracket_{0}}$ is such that for any 
$i,j \in \llbracket K \rrbracket$, the following formulae hold
    \begin{equation*}
    \begin{aligned}
      \theta^{G}_{00} = -\frac{B^{2}(1-x)}{(B(1-x)+1)^{3}}, \; \theta^{G}_{0i} 
= \frac{B\left(\sum_{l = i}^{K}b_{l}\right)(1-x)}{(B(1-x)+1)^{3}} \; 
\textrm{ and } \theta^{G}_{ij} 
= - \frac{\left(\sum_{l = i}^{K}b_{l}\right)
\left(\sum_{q = j}^{K}b_{q}\right)(1-x)}{(B(1-x)+1)^{3}}. 
    \end{aligned}
    \end{equation*}
\end{lemma}

\begin{proof}
%    Since we must prove our claim for any $K \in \mathbb{N}$, it is 
%easier to prove Lemma~\ref{Paper03_explicit_formula_theta_G_component} by directly 

We proceed directly by plugging the given expressions into 
Equations~\eqref{Paper03_null_space_theta} 
and~\eqref{Paper03_linear_system_original_ODE}. First, since we can write
    \begin{equation*}
        B = \sum_{i = 1}^{K}ib_{i} = \sum_{i = 1}^{K} \sum_{l = i}^{K} b_{l},
    \end{equation*}
it follows immediately thta
%    we conclude that $\boldsymbol{\Theta}^{G}$ solves~\eqref{Paper03_null_space_theta}, i.e.~that 
$\boldsymbol{\Theta}^{G}\boldsymbol{u} = 0$. 
We must check that it %$\boldsymbol{\Theta}^{G}$ 
also satisfies~\eqref{Paper03_linear_system_original_ODE}. %By plugging the formulas of this lemma 
Substituting into~\eqref{Paper03_left_hand_side_partial_terms}, we conclude 
that for all $i,j \in \llbracket K \rrbracket$,
    \begin{equation} 
\label{Paper03_testing_theta_alternative_flow}
    \begin{aligned}
      \left(\mathcal{J}\mathbf{F}^{*}_{\boldsymbol{x}} \boldsymbol{\Theta}^{G} 
+ \boldsymbol{\Theta}^{G} \mathcal{J}\mathbf{F}_{\boldsymbol{x}}\right)_{00} 
& = \frac{2B(1-b_{0})(1-x)}{(B(1-x)+1)^{2}}, \\[2ex] 
\left(\mathcal{J}\mathbf{F}^{*}_{\boldsymbol{x}} \boldsymbol{\Theta}^{G} 
+ \boldsymbol{\Theta}^{G} \mathcal{J}\mathbf{F}_{\boldsymbol{x}}\right)_{0j} 
& = - \frac{\left(B b_{j} + (1-b_{0})\left(\sum_{q = j}^{K} b_{q}\right)\right)
(1-x)}{(B(1-x)+1)^{2}}, \\[2ex] 
\left(\mathcal{J}\mathbf{F}^{*}_{\boldsymbol{x}} \boldsymbol{\Theta}^{G} 
+ \boldsymbol{\Theta}^{G} \mathcal{J}\mathbf{F}_{\boldsymbol{x}}\right)_{ij} 
& = \frac{\left(b_{i}\left(\sum_{l = j}^{K} b_{l}\right) 
+ b_{j}\left(\sum_{l = i}^{K} b_{l}\right)\right)(1-x)}{(B(1-x)+1)^{2}}.
    \end{aligned}
    \end{equation}
To conclude the proof, we compute the terms 
of~$v_{0} \mathcal{P}_{s}^{*}\Hess_{\boldsymbol{x}}(G_{0}) \mathcal{P}_{s}$. 
We start by computing $v_{0} \Hess_{\boldsymbol{x}}(G_{0}) \mathcal{P}_{s}$. 
Observe that 
combining~\eqref{Paper03_left_eigenvector_jacobian},~\eqref{Paper03_computation_projection_matrix_stable_linear_manifold},~\eqref{Paper03_hessian_alternative_ODE} 
and the fact that $\sum_{j = 1}^{K} b_{j} = 1 - b_{0}$, we have
\begin{equation} 
\label{Paper03_coordinate_00_intermediate_matrix_step_i}
\begin{aligned}
    \Big(v_{0}\Hess_{\boldsymbol{x}}(G_{0})\mathcal{P}_{s}\Big)_{00} 
& = \frac{1}{(B(1-x)+1)^{2}} 
\left(B(1-x)\frac{\partial^{2}G_{0}}{\partial x_{0}^{2}}(\boldsymbol{x}) 
- \sum_{j = 1}^{K}\frac{\partial^{2}G_{0}}{\partial x_{0} \partial x_{j}}
(\boldsymbol{x})\right) \\[2ex] 
& = \frac{2B(1-b_{0})(1-x) + (1- b_{0})}{(B(1-x)+1)^{2}} \\[2ex] 
& = \frac{(1-b_{0})(2B(1- x)+1)}{(B(1 - x) + 1)^{2}}.
\end{aligned}
\end{equation}
Moreover, for all $j \in \llbracket K \rrbracket$, we have
\begin{equation} 
\label{Paper03_coordinates_first_row_intermediate_matrix_step_i}
\begin{aligned}
    & \Big(v_{0}\Hess_{\boldsymbol{x}}(G_{0})\mathcal{P}_{s}\Big)_{0j} \\ 
& \quad = \frac{1}{(B(1 - x) + 1)^{2}} 
\Biggl[- \Biggl(\sum_{q = j}^{K} b_{q}\Biggl)(1 - x) 
\Biggl(\sum_{\substack{l = 0, \\ 
l \neq j}}^{K} \frac{\partial^{2}G_{0}}{\partial x_{0} \partial x_{l}}\Biggl) 
+ (B_{j}(1- x) + 1)
\frac{\partial^{2}G_{0}}{\partial x_{0} \partial x_{j}}\Biggl] \\[2ex] 
& \quad = -\left(\frac{(1-b_{0})\left(\sum_{q = j}^{K} b_{q}\right)(1-x) 
+ (B(1 - x) + 1)b_{j}}{(B(1 - x) + 1)^{2}}\right),
\end{aligned}
\end{equation}
where for the second equality we used that, 
by~\eqref{Paper03_auxiliary_B_j}, $B_j = B - \sum_{q = j}^{K} b_{q}$. 
By the definition of $\mathcal{P}_{s}$ given 
in~\eqref{Paper03_computation_projection_matrix_stable_linear_manifold} and
the computation of the Hessian~\eqref{Paper03_hessian_alternative_ODE}, 
we also have for all $i \in \llbracket K \rrbracket$,
\begin{equation} 
\label{Paper03_coordinates_first_column_intermediate_matrix_step_i}
\begin{aligned}
    \Big(v_{0}\Hess_{\boldsymbol{x}}(G_{0})\mathcal{P}_{s}\Big)_{i0} 
& = \frac{1}{(B(1 - x) + 1)^{2}} 
\Biggl(B(1 - x)\frac{\partial^{2} G_{0}}{\partial x_{0} \partial x_{i}}
(\boldsymbol{x}) 
- \sum_{j = 1}^{K} 
\frac{\partial^{2} G_{0}}{\partial x_{j} \partial x_{i}}\Biggl) \\ 
&  \quad = \frac{-Bb_{i}(1-x)}{(B(1 - x) + 1)^{2}}.
\end{aligned}
\end{equation}
Returning to the computation of 
$v_{0} \Hess_{\boldsymbol{x}}(G_{0}) \mathcal{P}_{s}$, for all 
$i,j \in \llbracket K \rrbracket$, we have, since 
$\frac{\partial^{2} G_{0}}{\partial x_{i} \partial x_{j}} \equiv 0$ for 
$i,j \in \llbracket K \rrbracket$,
\begin{equation} 
\label{Paper03_coordinates_general_term_intermediate_matrix_step_i}
\begin{aligned}
    \Big(v_{0}\Hess_{\boldsymbol{x}}(G_{0})\mathcal{P}_{s}\Big)_{ij} 
& = v_{0}\frac{\partial^{2} G_{0}}{\partial x_{i} \partial x_{0}} 
(\boldsymbol{x}) (\mathcal{P}_{s})_{0j} 
+ v_{0}\frac{\partial^{2} G_{0}}{\partial x_{i} \partial x_{j}} 
(\boldsymbol{x}) (\mathcal{P}_{s})_{jj} + v_{0} \sum_{\substack{l = 1, \\ 
l \neq j}}^{K} \frac{\partial^{2} G_{0}}{\partial x_{i} \partial x_{l}} 
(\mathcal{P}_{s})_{lj} \\ 
& = \frac{b_{i}\Big(\sum_{q = j}^{K} b_q\Big)(1 - x)}{(B(1 - x) + 1)^{2}}.
\end{aligned}
\end{equation}
We now compute 
$v_{0} \mathcal{P}_{s}^{*}\Hess_{\boldsymbol{x}}(G_{0}) \mathcal{P}_{s}$. 
Using the expression~\eqref{Paper03_computation_transpose_projection_matrix_stable_linear_manifold} 
for $\mathcal{P}_{s}^{*}$ 
and the identities~\eqref{Paper03_coordinate_00_intermediate_matrix_step_i},~\eqref{Paper03_coordinates_first_column_intermediate_matrix_step_i} 
and~\eqref{Paper03_testing_theta_alternative_flow}, we have
\begin{equation*}
\begin{aligned}
  \Big(v_{0} \mathcal{P}_{s}^{*}\Hess_{\boldsymbol{x}}(G_{0}) 
\mathcal{P}_{s}\Big)_{00} 
& = \frac{B(1 - x)\Big(v_0\Hess_{\boldsymbol{x}}(G_{0}) 
\mathcal{P}_{s}\Big)_{00} 
- \sum_{i = 1}^{K} \Big(v_0\Hess_{\boldsymbol{x}}(G_{0}) 
\mathcal{P}_{s}\Big)_{i0}}{(B(1 - x) + 1)} \\[1ex] 
& = \frac{2B(1-b_{0})(1-x)}{(B(1-x)+1)^{2}} \\[1ex]
& = \left(\mathcal{J}\mathbf{F}^{*}_{\boldsymbol{x}} \boldsymbol{\Theta}^{G} 
+ \boldsymbol{\Theta}^{G} \mathcal{J}\mathbf{F}_{\boldsymbol{x}}\right)_{00},
\end{aligned}
\end{equation*}
as desired. 
Moreover by identities~\eqref{Paper03_computation_transpose_projection_matrix_stable_linear_manifold},~\eqref{Paper03_coordinates_first_row_intermediate_matrix_step_i},~\eqref{Paper03_coordinates_general_term_intermediate_matrix_step_i} and~\eqref{Paper03_testing_theta_alternative_flow}, 
we conclude that for any $j \in \llbracket K \rrbracket$,
\begin{equation*}
\begin{aligned}
    \Big(v_{0} \mathcal{P}_{s}^{*}\Hess_{\boldsymbol{x}}(G_{0}) 
\mathcal{P}_{s}\Big)_{0j} 
& = \frac{B(1 - x)\Big(v_0\Hess_{\boldsymbol{x}}(G_{0}) 
\mathcal{P}_{s}\Big)_{0j} 
- \sum_{i = 1}^{K} \Big(v_0\Hess_{\boldsymbol{x}}(G_{0}) 
\mathcal{P}_{s}\Big)_{ij}}{(B(1 - x) + 1)} \\[2ex] 
& = - \frac{\left(B b_{j} + (1-b_{0})
\left(\sum_{q = j}^{K} b_{q}\right)\right)(1-x)}{(B(1-x)+1)^{2}} \\[2ex] 
& = \left(\mathcal{J}\mathbf{F}^{*}_{\boldsymbol{x}} \boldsymbol{\Theta}^{G} 
+ \boldsymbol{\Theta}^{G} \mathcal{J}\mathbf{F}_{\boldsymbol{x}}\right)_{0j}.
\end{aligned}
\end{equation*}
It remains then to compute the terms that do not belong to the first row
of the matrix. 
By identities~\eqref{Paper03_testing_theta_alternative_flow},~\eqref{Paper03_computation_transpose_projection_matrix_stable_linear_manifold}, ~\eqref{Paper03_coordinates_first_row_intermediate_matrix_step_i}, and~\eqref{Paper03_coordinates_general_term_intermediate_matrix_step_i}, 
we have for $i,j \in \llbracket K \rrbracket$, 
\begin{equation*}
\begin{aligned}
     & \Big(v_{0} \mathcal{P}_{s}^{*}\Hess_{\boldsymbol{x}}(F_{0}) 
\mathcal{P}_{s}\Big)_{ij} \\ 
& \quad  = \Big(\mathcal{P}_{s}^{*}\Big)_{i0}
\Big(v_0\Hess_{\boldsymbol{x}}(G_{0}) \mathcal{P}_{s}\Big)_{0j} 
+ \Big(\mathcal{P}_{s}^{*}\Big)_{ii}\Big(v_0\Hess_{\boldsymbol{x}}(G_{0}) 
\mathcal{P}_{s}\Big)_{ij} + \sum_{\substack{r = 1, \\ 
r \neq i}}^{K} \Big(\mathcal{P}_{s}^{*}\Big)_{ir}
\Big(v_0\Hess_{\boldsymbol{x}}(G_{0}) \mathcal{P}_{s}\Big)_{rj} \\[1ex]  
& \quad = \frac{\Big(\sum_{l = i}^{K} b_{l}\Big)
\Big(\sum_{q = j}^{K} b_{q}\Big)(1 - b_{0})(1 - x)^{2} 
+ b_{j}
\Big(\sum_{l = i}^{K} b_{l}\Big)(1 - x)(B(1 - x) + 1)}{(B(1 - x) + 1)^{3}} \\ 
& \quad \quad \quad 
+ \frac{ b_{i}\Big(\sum_{q = j}^{K} b_{q}\Big)(1 - x)(B_{i}(1 - x) + 1) 
- \Big(\sum_{l = i}^{K} b_{l}\Big) \Big(\sum_{q = j}^{K} b_{q}\Big)
\Big(\sum_{r = 1, \, r \neq i}^{K}b_{r}\Big)(1 - x)^{2}}{(B(1 - x) + 1)^{3}} 
\\[1ex] 
&  \quad = \frac{\left(b_{i}\left(\sum_{l = j}^{K} b_{l}\right) 
+ b_{j}\left(\sum_{l = i}^{K} b_{l}\right)\right)(1-x)}{(B(1-x)+1)^{2}} \\[1ex]
 & \quad 
= \left(\mathcal{J}\mathbf{F}^{*}_{\boldsymbol{x}} \boldsymbol{\Theta}^{G} 
+ \boldsymbol{\Theta}^{G} \mathcal{J}\mathbf{F}_{\boldsymbol{x}}\right)_{ij},
\end{aligned}
\end{equation*}
where we used $B_{i} \defeq B - \sum_{l = i}^{K} b_{l}$ and 
$\sum_{r = 1, \, r \neq i}^{K}b_{r} = 1 - b_{0} - b_{i}$ in the third equality.
Since all the terms agree, we conclude $\boldsymbol{\Theta}^{G}$ is indeed the 
desired matrix.
\end{proof}

It remains to analyse $\boldsymbol{\Theta}^{\Delta}$. Unfortunately, here we 
have not been able to derive a general formula which is valid for any 
$K \in \mathbb{N}$. Clearly, if one has some specific value of $K$ in mind, 
then our approach allows the computation of $\boldsymbol{\Theta}^{\Delta}$ 
for that particular case. We would like however to derive properties of the 
system that will hold for at least a large set of possible values of $K$. 
We will then use 
Theorem~\ref{Paper03_explicit_formula_solution_lyapunov_equation} to obtain 
some qualitative behaviour of $\boldsymbol{\Theta}$. To follow this strategy, 
we observe that by~\eqref{Paper03_difference_matrix_between_hessians}, 
$\Delta$ is a real symmetric matrix such that $\Delta \boldsymbol{u} = 0$. 
Therefore, $\mathcal{P}_{s}^{*}\Delta \mathcal{P}_{s} = \Delta$. 
This observation simplifies our analysis.

\begin{lemma} 
\label{Paper03_auxiliar_analysis_difference_matrix}
    Let $\Delta$ be the square matrix given 
by~\eqref{Paper03_difference_matrix_between_hessians}. Then, for any 
fixed $K \in \mathbb{N}$, the characteristic polynomial of $\Delta$ is 
given by
    \begin{equation*}
        p_{\Delta}(\lambda) = (-1)^{K+1}\lambda^{K}\left(\lambda + (1-b_{0})^{2} + \sum_{i = 1}^{K}b_{i}^{2}\right).
    \end{equation*}
    In particular, $\Delta$ is negative semidefinite.
\end{lemma}

\begin{proof}
    It is convenient to express the characteristic polynomial in terms of 
the principal minors of $\Delta$. Recall that a submatrix 
$\mathbf{M}$ of $\Delta$ is called a principal submatrix of order $n$ if 
there exist $0 \leq i_{1} < i_{2} < \ldots < i_{n} \leq K$ such that
    \begin{equation*}
        \mathbf{M} = \left(\begin{array}{cccc}
    \Delta_{i_{1}i_{1}} & \Delta_{i_{1}i_{2}}& \cdots &  \Delta_{i_{1}i_{n}} \\
    \Delta_{i_{1}i_{2}} &  \Delta_{i_{2}i_{2}} & \cdots 
&  \Delta_{i_{2}i_{n}} \\ \vdots & \vdots & \ddots & \vdots \\  
\Delta_{i_{1}i_{n}} &  \Delta_{i_{2}i_{n}} & \cdots &  \Delta_{i_{n}i_{n}}
\end{array}\right).
    \end{equation*}
Denoting by $c_{i}$ the sum of the determinants of the principal submatrices 
of order $i$ of $\Delta$, we can express the characteristic polynomial of 
$\Delta$ as (see for instance~\cite[Theorem~4.3.1]{brualdi2008combinatorial})
    \begin{equation} 
\label{Paper03_equation_characteristic_minors_delta_difference_hessians}
    p_{\Delta}(\lambda) 
= (-1)^{K+1}\lambda^{K+1} + \sum_{i = 0}^{K}(-1)^{i}\lambda^{i}c_{K+ 1- i}.
    \end{equation}
To prove our claim, it will be enough to compute the sum of the minors. 
For $i = K$, we have 
$c_{1} = \tr(\Delta) = - (1-b_{0})^{2} - \sum_{i = 1}^{K} b_{i}^{2}$. 
We now claim that $c_{i} = 0$ for $i > 1$. In fact, we shall verify that 
the determinant of any principal submatrix of $\Delta$ with order higher 
than $1$ is zero. Indeed, writing $\delta_{0} = (1-b_{0})$ and 
$\delta_{i} = - b_{i}$ for $i \in \llbracket K \rrbracket$, by the 
definition of $\Delta$
in~\eqref{Paper03_difference_matrix_between_hessians}, 
the determinant of any principal minor $\mathbf{M}$ of order $n$ of $\Delta$ 
can be written as
    \begin{equation*}
    \det(\mathbf{M}) = (-1)^{n}\left\vert\begin{array}{cccc}
    \delta_{i_{1}} \delta_{i_{1}} & \delta_{i_{1}} \delta_{i_{2}}& \cdots 
&  \delta_{i_{1}} \delta_{i_{n}} \\
    \delta_{i_{1}} \delta_{i_{2}} &  \delta_{i_{2}} \delta_{i_{2}} & \cdots 
&  \delta_{i_{2}} \delta_{i_{n}} \\ 
\vdots & \vdots & \ddots & \vdots \\  
\delta_{i_{1}} \delta_{i_{n}} &  \delta_{i_{2}} \delta_{i_{n}} & \cdots 
&  \delta_{i_{n}} \delta_{i_{n}}
\end{array}\right\vert = 0,
    \end{equation*}
    since the rows are linearly dependent. Our claim now follows
from~\eqref{Paper03_equation_characteristic_minors_delta_difference_hessians}.
This implies that $\Delta$ has the eigenvalue $0$ with algebraic 
multiplicity $K$, and one negative eigenvalue equal to 
$- \left((1-b_{0})^{2} + \sum_{i = 1}^{K} b_{i}^{2}\right)$. 
In particular, $\Delta$ is negative semidefinite, which completes the proof.
\end{proof}

We are now in position to prove 
Proposition~\ref{Paper03_bound_drift_general_K_proposition}.

\begin{proof}[Proof of Proposition~\ref{Paper03_bound_drift_general_K_proposition}]
By the 
formula~\eqref{Paper03_formula_Parsons_Rogers_second_order_derivatives} of
Parsons and Rogers, for every $i,j \in \llbracket K \rrbracket_0$, we have
\begin{equation*}
    \begin{aligned}
      \frac{\partial^{2} \Phi_{0}}{\partial x_{i} \partial x_j}(\boldsymbol{x})
 = \frac{1}{\sum_{l = 0}^{K} v_{l} \gamma_{l}'} 
\left(v_{i}' \frac{\partial \Phi_{0}}{\partial x_{j}} 
+ v_{j}' \frac{\partial \Phi_{0}}{\partial x_{i}} - \theta^{G}_{ij} 
- \theta^{\Delta}_{ij}\right).
    \end{aligned}
    \end{equation*}
From the expression for the left eigenvector $v$  
in~\eqref{Paper03_left_eigenvector_jacobian}, and the fact that the 
diagonal $\Gamma$ has zero curvature, we have 
$\sum_{l = 0}^{K} v_{l} \gamma_{l}' = 1$, 
$\sum_{l = 0}^{K} v_{l}' \gamma_{l}' = 0$ and $\boldsymbol{\gamma}'' = 0$. 
Differentiating $\boldsymbol{v}$ %in~\eqref{Paper03_left_eigenvector_jacobian} 
with respect to $x_{0}$, recalling the formula for the first order derivative 
in Lemma~\ref{Paper03_lemma_first_order_derivatives_projection_map}, and 
applying Lemma~\ref{Paper03_explicit_formula_theta_G_component}, 
we conclude that
     \begin{equation} 
\label{Paper03_real_formula_second_order_derivative}
    \begin{aligned}
        \frac{\partial^{2} \Phi_{0}}{\partial x_{0}^{2}}(\boldsymbol{x}) 
= \frac{B(B(1-x)+2)}{(B(1-x)+1)^{3}} - \theta^{\Delta}_{00},
    \end{aligned}
    \end{equation}
    while for $i \in \llbracket K \rrbracket$,
     \begin{equation} 
\label{Paper03_real_formula_second_order_derivative_0_i}
    \begin{aligned}
      \frac{\partial^{2} \Phi_{0}}{\partial x_{0} \partial x_i}(\boldsymbol{x}) 
= - \frac{\sum_{l = i}^K b_l}{(B(1-x)+1)^{3}} - \theta^{\Delta}_{0i},
    \end{aligned}
    \end{equation}
    and finally for $i,j \in \llbracket K \rrbracket$,
     \begin{equation} 
\label{Paper03_real_formula_second_order_derivative_i_j}
    \begin{aligned}
     \frac{\partial^{2} \Phi_{0}}{\partial x_{i} \partial x_j}(\boldsymbol{x}) 
= - \frac{\left(\sum_{l = i}^K b_l\right)
\left(\sum_{q = j}^K b_q\right)(1 - x)}{(B(1-x)+1)^{3}} - \theta^{\Delta}_{ij}.
    \end{aligned}
    \end{equation}
    Now, by Lemma~\ref{Paper03_auxiliar_analysis_difference_matrix}, $\Delta$ is negative semidefinite, and therefore Proposition~\ref{Paper03_explicit_formula_solution_lyapunov_equation} implies $\boldsymbol{\Theta}^{\Delta}$ is positive semidefinite. In particular, the diagonal terms of $\boldsymbol{\Theta}^{\Delta}$ are nonnegative, and therefore by~\eqref{Paper03_real_formula_second_order_derivative}, we conclude that
     \begin{equation*}
    \begin{aligned}
        \frac{\partial^{2} \Phi_{0}}{\partial x_{0}^{2}}(\boldsymbol{x}) \leq \frac{B(B(1-x)+2)}{(B(1-x)+1)^{3}} \; \forall \, x \in [0,1].
    \end{aligned}
    \end{equation*}
    Moreover, when $x = 1$, the decomposition of $\Hess_{\boldsymbol{x}}(F_{0})$
in~\eqref{Paper03_decomposition_hessian_original_flow} yields 
$\Hess_{\boldsymbol{x}}(F_{0}) = \Hess_{\boldsymbol{x}}(G_{0})$.
% which means that for $x = 1$, $\boldsymbol{\Theta}^{\Delta} = 0$. 
Therefore,
    \begin{equation*}
          \frac{\partial^{2} \Phi_{0}}{\partial x_{0}^{2}}(\boldsymbol{1}) = 2B.
    \end{equation*}
    It remains then to verify the monotonicity property. To see this, observe 
that for $x \in [0,1]$,
    \begin{equation*}
        \frac{d}{dx} \left(\frac{B(B(1-x)+2)}{(B(1-x)+1)^{3}} \right)
= \frac{B^{2}(2B(1-x)+5)}{(B(1-x)+1)^{4}} > 0,
    \end{equation*}
    whenever $b_{0} < 1$ (see~\eqref{Paper03_mean_age_germination}). Moreover, since the function 
$x \mapsto \frac{(1-x)}{(B(1-x)+1)}$ is decreasing in $[0,1]$, by 
Lemma~\ref{Paper03_auxiliar_analysis_difference_matrix}, we have for any 
$0 \leq x \leq \tilde{x} \leq 1$,
    \begin{equation*}
        \frac{(1-x)}{(B(1-x)+1)} \Delta 
\preceq \frac{(1-\tilde{x})}{(B(1-\tilde{x})+1)}\Delta,
    \end{equation*}
    which implies (by 
Proposition~\ref{Paper03_explicit_formula_solution_lyapunov_equation}) 
that the function $x \mapsto \theta^{\Delta}_{00}(x)$ is monotonically 
non-increasing on $[0,1]$. 
From~\eqref{Paper03_real_formula_second_order_derivative}, we conclude 
that the function 
$x \mapsto \frac{\partial^{2} \Phi_{0}}{\partial x_{0}^{2}}(x,x,\ldots,x)$ 
is strictly increasing on $[0,1]$ as desired, and the proof is complete.
\end{proof}

\section{Analysis of the Branching Phase}
\label{Paper03_section_proof_results_branching}

In this section, we provide the (rather standard)
proof of Theorem~\ref{Paper03_branching_approximation_constant_environment}, 
concerning the multitype branching process approximation to the 
Wright-Fisher model in a constant environment, valid while the mature mutant
plants constitute a negligible proportion of the population. 
Recall from
Section~\ref{Paper03_WF_model_const_environ_description} 
that $(Z_{0}(t))_{t \in \mathbb{N}_{0}}$ represents the number of 
mature mutant individuals in the population and that the mean
matrix $\mathbf{M}$ 
of the process $\mathbf{Z}$ (which also records numbers and time until
germination of seeds in the seed bank) is given 
by~\eqref{Paper03_mean_matrix_branching_constant_environment}.

Since all entries of $\mathbf{M}$ are nonnegative, by the 
Perron-Frobenius Theorem, $\mathbf{M}$ has a maximal eigenvalue 
$\lambda_{\mathbf{M}}$ which is strictly positive, and such that all the 
other eigenvalues have modulus strictly less then $\lambda_{\mathbf{M}}$. 
To confirm that the branching process $\mathbf{Z}$ is critical, we compute 
the value of $\lambda_{\mathbf{M}}$ 
(see\cite[Section~V.7]{athreya2004branching}).

\begin{lemma} 
\label{Paper03_criticality_branching_process}
For any $K \in \mathbf{N}$ and $\mathbf{b} = (b_{i})_{i = 0}^{K}$, the 
Perron-Frobenius eigenvalue $\lambda_{\mathbf{M}}$ of $\mathbf{M}$ is 
equal to $1$, i.e.~the multitype branching process $\mathbf{Z}$ is critical.
\end{lemma}

\begin{proof}
    It suffices to note that the matrix $\mathbf{M}$ has the same 
characteristic polynomial as the matrix 
$\mathcal{J}\mathbf{F}_{\boldsymbol{0}} + \mathbf{I}_{K+1}$ that appears in 
the proof of Lemma~\ref{Paper03_eigenvalue_condition_constant_environment}, 
%where $\mathcal{J}\mathbf{F}_{\boldsymbol{0}}$ is the Jacobian matrix of the 
%vector field $\mathbf{F}$ of the ODE arising from the dynamics of the 
%Wright-Fisher model with seed banks evaluated at $\boldsymbol{x} = \boldsymbol{0} = (0, 0, \ldots, 0)$, i.e.
%    \begin{equation*}
%    {p}_{\mathbf{M}}({\lambda}) = (-1)^{K}{\lambda}^{K}\left({- (1 - b_{0})(1-x)} + 1 - {\lambda}\right) + (-1)^{K}{(1 - x)}\sum_{i = 1}^{K}b_{i}{\lambda}^{K - i}.
%   \end{equation*}
%    Therefore, by applying the same argument as in the proof of Lemma~\ref{Paper03_eigenvalue_condition_constant_environment}, we conclude
for which we already proved that $\lambda_{\mathbf{M}} = 1$, as desired.
\end{proof}
%It is not surprising that $\mathbf{M}$ and $\mathcal{J}\mathbf{F}_{\boldsymbol{0}} + \mathbf{I}_{K+1}$ share the same eigenvalues. Indeed, $\mathcal{J}\mathbf{F}_{\boldsymbol{0}}$ would be the infinitesimal generator of the branching process $\mathbf{Z}$ if the parameter $M$ was taken to be equal to $1$. This happens since the branching process $\mathbf{Z}$ approximates the dynamics when the number of mutants in the population is negligible when compared to the total number of individuals, and in this case the dynamics could also be approximated by the ODE flow near $\boldsymbol{0}$.
%We would like to understand the asymptotic behaviour of $\mathbf{Z}$. 
Since~$\mathbf{Z}$ is a critical multitype branching process, we can 
apply the results of Mullikin~\cite{mullikin1963limiting} to understand its 
asymptotic behaviour. For this, we must find the right and left eigenvectors 
$\boldsymbol{u}$ and $\boldsymbol{v}$ of the mean matrix $\mathbf{M}$, 
associated to the eigenvalue $1$, and normalised so that 
$\langle \boldsymbol{u}, \boldsymbol{1} \rangle = 
\langle \boldsymbol{u}, \boldsymbol{v} \rangle = 1$. 
By~\eqref{Paper03_mean_matrix_branching_constant_environment},
\begin{equation} 
\label{Paper03_frobenius_eigenvectors}
    \boldsymbol{u} = \frac{1}{(M+K)}\left(\begin{array}{c} M \\ 1 \\ 1 \\ 
\vdots \\ 1
        \end{array}\right) 
\quad \textrm{ and } \quad \boldsymbol{v} 
= \frac{(M+K)}{M(B+1)}\left(\begin{array}{c} 1 \\ M(1-b_{0}) \\ 
M\Big(\sum_{i = 2}^{K} b_{i}\Big) \\ \vdots \\ Mb_{K}
        \end{array}\right).
\end{equation}
In particular, 
\begin{equation} 
\label{Paper03_parameter_branching_phase}
    u_{0}v_{0} = \frac{1}{(B+1)},
\end{equation}
where $B$ is the mean germination time of the seed bank.
% given in~\eqref{Paper03_mean_age_germination}. We will see shortly that the expression in~\eqref{Paper03_parameter_branching_phase} will determine the values of the asymptotic survival probability and expected number of mature individuals conditioned on survival. We are in position of proving Theorem~\ref{Paper03_branching_approximation_constant_environment}.

\begin{proof}[Proof of Theorem~\ref{Paper03_branching_approximation_constant_environment}]
Let $Z^{(i)}_{j}$ be the number of particles of type $j$ generated by the 
branching of a particle of type $i$. For any 
$l \in \llbracket K \rrbracket_{0}$, we let
\begin{equation*}
    q^{(l)}(i,j) \defeq \mathbb{E}\left[Z^{(l)}_{i}Z^{(l)}_{j} 
- \delta_{ij}Z^{(l)}_{i}\right],
\end{equation*}
where $\delta_{ij}$ denotes the Kronecker delta. 
By direct computation and recalling that each type $0$ particle generates 
$\Poiss(b_{0})$ offspring of type $0$ and $\Poiss(Mb_{i})$ offspring of 
type $i$, for $i \in \llbracket K \rrbracket$, and that each non-mature 
seed gives birth to at most one offspring during a `branching' event, 
we conclude that
\begin{equation*}
    \left\{\begin{array}{ccll}
       q^{(0)}(0,0)  & = & b_{0}^{2}, & \\
        q^{(0)}(0,j) & = & M b_{0} b_{j} & \forall \, j 
\in \llbracket K \rrbracket, \\ 
q^{(0)}(i,j) & = & M^{2}b_{i}b_{j} &\forall \, i,j 
\in \llbracket K \rrbracket, \\ 
q^{(l)}(i,j) & = & 0 &\forall \, l 
\in \llbracket K \rrbracket \textrm{ and } 
\forall \, i,j \in \llbracket K \rrbracket_{0}.
    \end{array}\right.
\end{equation*}
For $l \in \llbracket K \rrbracket_{0}$, and 
$\boldsymbol{w} \in \mathbb{R}^{K+1}$, we define the quadratic form
\begin{equation*}
    Q^{(l)}(\boldsymbol{w}) \defeq \frac{1}{2} \sum_{i = 0}^{K} 
\sum_{j = 0}^{K} w_{i} q^{(l)}(i,j) w_{j}.
\end{equation*}
For any $\boldsymbol{w} \in \mathbb{R}^{K+1}$, let 
$\mathbf{Q}(\boldsymbol{w}) \in \mathbb{R}^{K+1}$ be given by 
$\mathbf{Q}(\boldsymbol{w}) \defeq ( Q^{(0)}(\boldsymbol{w}), 
\cdots, Q^{(K)}(\boldsymbol{w}))$. Then, by explicitly computing 
$\mathbf{Q}$, we conclude that for 
$\boldsymbol{w} \in \mathbb{R}^{K+1}$, we have
\begin{equation*}
Q^{(0)}(\boldsymbol{w}) = \frac{M^{2}}{2} \sum_{i = 1}^{K} \, 
\sum_{j = 1} b_{i}b_{j}w_{i}w_{j} + Mb_{0}w_{0} 
\sum_{j = 1}^{K}b_{j}w_{j} + \frac{b_{0}^{2}w_{0}^{2}}{2},
\end{equation*}
and $Q^{(l)}(\boldsymbol{w}) = 0$ for $l \in \llbracket K \rrbracket$. 
For the particular case $\boldsymbol{w} = \boldsymbol{u}$, where 
$\boldsymbol{u}$ is the right Frobenius eigenvector of $\mathbf{M}$, we 
conclude from~\eqref{Paper03_frobenius_eigenvectors} that
\begin{equation} 
\label{Paper03_quadratic_form_branching}
\begin{aligned}
    Q^{(0)}(\boldsymbol{u}) = \frac{u_{0}^{2}}{2}.
\end{aligned}
\end{equation}
Let $\mathbf{e}_{i}$ be the vector in $\mathbb{R}^{K+1}$ which has 
coordinate $1$ in position $i$ and $0$ otherwise. Then, using the
results of Mullikin~\cite{mullikin1963limiting}, 
if the branching process $\mathbf{Z}$ starts from one mature mutant plant, 
i.e.~if $Z_{0} = \mathbf{e}_{0}$, we have
\begin{equation*}
\begin{aligned}
    \lim_{t \rightarrow \infty} t \mathbb{P}_{\mathbf{e}_{0}}
\left(\mathbf{Z}(t)\neq 0\right) 
= \frac{\langle \mathbf{e}_{0},  \boldsymbol{u}\rangle}{\langle \boldsymbol{v}, \mathbf{Q}(\boldsymbol{u}) \rangle} = \frac{u_{0}}{v_{0}Q^{(0)}(\boldsymbol{u})}
 = \frac{2}{u_{0}v_{0}} = 2(B+1),
\end{aligned}
\end{equation*}
where, as usual, $B$ is the mean germination time of the seed bank, and
%, and it given by~\eqref{Paper03_mean_age_germination}. Observe that w
we substituted the expressions~\eqref{Paper03_quadratic_form_branching} 
and~\eqref{Paper03_parameter_branching_phase}. Moreover, 
%by applying Mullikin's results to characterise the limiting distribution of a linear functional of $\mathbf{Z}$ conditioned on the survival of the process (
using~\cite[Theorem~9]{mullikin1963limiting}, for any $y \geq 0$,
     \begin{equation*}
     \begin{aligned}
       \lim_{t \rightarrow \infty} \mathbb{P}_{\mathbf{e}_{0}}
\left(\frac{Z_{0}(t)}{t} \geq y \, \Big \vert \, \mathbf{Z}(t) \neq 0\right) 
& = 1 - \exp\left(\frac{y}{\langle \mathbf{e}_{0}, \boldsymbol{v} \rangle 
\cdot \langle \mathbf{Q}(\boldsymbol{u}), \boldsymbol{v} \rangle}\right) \\ 
& = 1 - \exp\left(-2y(B+1)^2\right),
    \end{aligned}
   \end{equation*}
where we used Equations~\eqref{Paper03_quadratic_form_branching} 
and~\eqref{Paper03_parameter_branching_phase} to obtain the last equality. 
Our claim then holds.
\end{proof}

We now proceed to the proof of 
Lemma~\ref{Paper03_lemma_correspondence_value_branching_phase}.

\begin{proof}[Proof of Lemma~\ref{Paper03_lemma_correspondence_value_branching_phase}]
    Since $\psi_B(0) = 0$, 
by~\eqref{Paper03_correspondence_between_branching_phase_monotype_seed_bank_ii},
 in order to compute an explicit formula for $\psi_B(y)$, it is enough to 
observe that for any $y \in [0, 1)$, we have
    \begin{equation*}
    \begin{aligned}
      (B+1)(1 - \exp(-2\psi_B(y)(B+1)^2)) = 1 - \exp(-2y),
    \end{aligned}
    \end{equation*}
    which implies that $\psi_B(y)$ must be given by
    \begin{equation*}
        \psi_B(y) = \frac{1}{2(B+1)^2} \log\left(\frac{B+1}{B + e^{-2y}}\right).
    \end{equation*}
    %To conclude the proof of the lemma, we observe that the formula for 
%$\psi_B(y)$ 
Differentiating with respect to $B$,
    \begin{equation} 
\label{Paper03_derivative_B_correspondence_branching_phase}
    \begin{aligned}
        \frac{\partial}{\partial B} \psi_B(y) 
= -\frac{1}{(B+1)^3} \log\left(\frac{B+1}{B+e^{-2y}}\right) 
- \frac{(1 - e^{-2y})}{2(B+1)^3(B+e^{-2y})} < 0 \, 
\forall (B,y) \in [0, \infty) \times [0,1),
    \end{aligned}
    \end{equation}
which establishes property~(i) of 
Lemma~\ref{Paper03_lemma_correspondence_value_branching_phase}. To establish 
property~(ii) we observe that %the 
%explicit formula for $\psi_B(y)$ yields
    \begin{equation*}
        \frac{\partial}{\partial y} \psi_B(y) 
= \frac{e^{-2y}}{(B+1)^2(B+ e^{-2y})} > 0, \, 
\qquad\forall (B,y) \in [0, \infty) \times [0,1),
    \end{equation*}
    which completes the proof.
\end{proof}

\section{Analysis of the diffusion phase} \label{Paper03_proof_results_concerning_diffusion}

In this section we turn to the proofs of our results for the diffusion phase.
%shall prove the theorems regarding the weak convergence 
%of the discrete Wright-Fisher model with seed banks to a diffusion. 
For didactic reasons, we shall divide them %the section 
according to the type 
of environment considered.

\subsection{Diffusion in constant environment} 
\label{Paper03_subsection_diffusion_WF_model_constant_environment}

In this section, we aim to prove 
Theorem~\ref{Paper03_weak_convergence_diffusion_constant_environment} 
regarding the weak convergence of the Wright-Fisher model in a constant 
environment described in 
Section~\ref{Paper03_WF_model_const_environ_description}, and 
Theorem~\ref{Paper03_bound_fixation_probability_any_K} regarding the 
fixation probability of the mutant. We begin with
the proof of %shall start by establishing convergence,
%i.e.~by proving 
Theorem~\ref{Paper03_weak_convergence_diffusion_constant_environment}. 
Although this theorem can be understood as a particular case of the 
corresponding result 
in a slowly changing environment, the proof is more digestible if 
presented in this setting first.
%we include the proof for 
%didactic reasons.

Recall that the discrete process is represented by the vector-valued process
\begin{equation*}
    (\boldsymbol{X}^{N}(t))_{t \in \mathbf{N}_{0}} = (X^{N}_{0}(t), X^{N}_{1}(t), \ldots, X^{N}_{K}(t))_{t \in \mathbb{N}_{0}},
\end{equation*}
where $X^{N}_{i}(t)$ indicates the number of mutant individuals living in generation $t-i$. As explained in  Section~\ref{Paper03_diffusion_phase_constant_environment}, we consider the scaled vector-valued process $(\boldsymbol{x}^N(t))_{t \geq 0}$ defined in~\eqref{Paper03_scaled_process_WF_model}.
Our aim is to write the dynamics of $(\boldsymbol{x}^{N}(t))_{t \geq 0}$ in such a way that we can apply Katzenberger's results explained in Section~\ref{Paper03_subsection_katzenberger_overview}. We will follow closely the ideas introduced by Katzenberger in~\cite[Section~8]{katzenberger1991solutions}. 

Observe that by construction the process $(\boldsymbol{x}^{N}(t))_{t \geq 0}$ has sample paths in $[0,1]^{K+1}$, for all $N \in \mathbb{N}$. For each $N \in \mathbb{N}$, let $\left\{ \zeta^{N, \tau}_{j} \, \vert \, \tau \in \mathbb{N}_{0}, \, j \in \mathbb{N}  \right\}$ be a set of i.i.d.~random variables taking values in the space of measurable functions from $[0,1]^{K+1}$ to $\{0,1\}$ such that for all $\boldsymbol{x} \in [0,1]^{K+1}$,
\begin{equation} \label{Paper03_definition_zeta_N_constant_environment}
    \mathbb{E}\left[\zeta^{N, \tau}_{j}(\boldsymbol{x})\right] = \frac{\sum_{i = 0}^{K}b_{i}x_{i}}{(1 -x_{0}) + \sum_{i = 0}^{K}b_{i}x_{i}}.
\end{equation}
For fixed $N \in \mathbb{N}$, $\tau \in \mathbb{N}_{0}$ and $j \in [N]$,
conditional on $\boldsymbol{x}^N(\tau)$, 
$\zeta^{N, \tau}_{j}(\boldsymbol{x}^{N}(\tau))$ is a Bernoulli random variable 
that is equal to $1$ if the $j$-th individual of generation $\tau + 1$ is a 
mutant. %, i.e.~if the $j$-th individual harbours the seed bank property. 
Therefore, for any $N \in \mathbb{N}$ and $\tau \in \mathbb{N}_{0}$, we 
can write $X^{N}_{0}(\tau + 1)$ as
\begin{equation*}
    X^{N}_{0}(\tau + 1) = \sum_{j = 1}^{N} \zeta^{N, \tau}_{j}\left(\frac{\boldsymbol{X}^{N}(\tau)}{N}\right).
\end{equation*}
In order to capture the fluctuations of the system, for each $N \in \mathbb{N}$, let $\left\{\theta^{N,\tau} \, \vert \, \tau \in \mathbb{N}_{0} \right\}$ be a set of i.i.d.~random variables taking values in the space of measurable functions from $[0,1]^{K+1}$ to $\mathbb{R}$ given by
\begin{equation} \label{Paper03_definition_theta_aux_rv_construction_martingale_constant_environment}
    \theta^{N,\tau} (\boldsymbol{x}) \defeq \frac{1}{\sqrt{N}} \sum_{j = 1}^{N} \left(\zeta^{N,\tau}_{j}(\boldsymbol{x}) - \frac{\sum_{i = 0}^{K} b_{i}x_{i}}{(1 - x_{0}) + \sum_{i = 0}^{K} b_{i}x_{i}}\right).
\end{equation}
Recall the map $\mathbf{F}: [0,1]^{K+1} \rightarrow \mathbb{R}^{K+1}$ defined in~\eqref{Paper03_flow_definition_WF_model}. We shall now describe the dynamics of $x^{N}_{0}$ in terms of $\theta^{N}$ and the map $\mathbf{F}$. By the definition of $(\boldsymbol{x}^N(t))_{t \geq 0}$ in~\eqref{Paper03_scaled_process_WF_model}, we have for all $T \in \mathbb{R}_{+}$,
\begin{equation} \label{Paper03_semimartingale_formulation_constant_environment_discrete}
\begin{aligned}
    x^{N}_{0}(T) & = x^{N}_{0}(0) + \sum_{t = 0}^{\lfloor NT \rfloor -1} \; \Bigg(\Bigg(\frac{1}{N} \sum_{j = 1}^{N} \zeta^{N,t}_{j}(\boldsymbol{x}^{N}(t/N))\Bigg) - x^{N}_{0}(t /N)\Bigg) \\
    & = x^{N}_{0}(0) + M^{N}_{0}(T) + \int_{0}^{T} F_{0}(\boldsymbol{x}^{N}) \, d \lfloor Nt \rfloor,
\end{aligned}
\end{equation}
where the càdlàg process $(M^{N}_{0}(t))_{t \in \mathbb{R}_{+}}$ is given by, for all $N \in \mathbb{N}$ and all $T \in \mathbb{R}_{+}$,
\begin{align}
    M^{N}_{0}(T) \defeq \frac{1}{\sqrt{N}} \sum_{t = 0}^{\lfloor NT \rfloor - 1} \theta^{N,t}(\boldsymbol{x}^{N}(t/N)) = \int_{0}^{\lfloor NT \rfloor - 1} \, \frac{1}{\sqrt{N}} \, d \theta^{N,t}(\boldsymbol{x}^{N}(t/N)). \label{Paper03_martingale_seeds_WF_constant_environment}
\end{align}
Heuristically, $M^{N}_{0}$ corresponds to the martingale term arising from 
random sampling of genetic types at each generation, while the integral of 
$F_{0}$ corresponds to the large drift that will force the process in the limit
to have sample paths in the diagonal $\Gamma$ of the $(K+1)$-dimensional 
hypercube~$[0,1]^{K+1}$. Observe also that, % for all 
%$i \in \llbracket K \rrbracket$, 
since $X^{N}_{i}(t) \defeq X^{N}_{0}(t-i)$, 
we can write, for all $i \in \llbracket K \rrbracket$ and all 
$T \in \mathbb{R}_+$,
\begin{equation} \label{Paper03_dynamics_seed_bank_WF_constant_environment}
    x^{N}_{i}(t) = x^{N}_{i}(0) + \int_{0}^{t} F_{i}(\boldsymbol{x}^{N}) \, d\lfloor N\tau \rfloor.
\end{equation}

As explained in Section~\ref{Paper03_subsection_katzenberger_overview}, in 
order to prove the desired convergence, we must verify two conditions: that 
the map~$\mathbf{F}$ is (sufficiently) asymptotically stable 
(i.e.~that~$\mathbf{F}$ satisfies Assumption~\ref{Paper03_assumption_manifold}),
and that~$M^{N}_{0}$ is sufficiently well-behaved 
(i.e.~$M^{N}_{0}$ must satisfy 
Assumption~\ref{Paper03_assumption_katzenberger_stochastic_process}). 
By Lemma~\ref{Paper03_eigenvalue_condition_constant_environment}, 
proved in Section~\ref{Paper03_derivatives_specific_ode}, the 
map~$\mathbf{F}$ satisfies Assumption~\ref{Paper03_assumption_manifold}, 
and the diagonal $\Gamma$ of the $(K+1)$-dimensional hypercube is an attractor 
manifold. It remains then to establish that $M^{N}_{0}$ satisfies 
Assumption~\ref{Paper03_assumption_katzenberger_stochastic_process}. Since we 
will need to repeatedly establish the uniform integrability of functions of 
purely discontinuous martingales, we state a lemma with conditions that 
are easily verified in our scenario.

\begin{lemma} \label{Paper03_auxiliary_lemma_uniform_integrability_martingales}
    Let $\Big((M^N(t))_{t \geq 0}\Big)_{N \in \mathbb{N}}$ be a sequence of real-valued càdlàg martingales satisfying the following conditions:
    \begin{enumerate}[(i)]
        \item For every $N \in \mathbb{N}$, jumps of $(M^N(t))_{t \geq 0}$ can only occur at times $\{jN^{-1}: \, j \in \mathbb{N}_0\}$, i.e.
        \begin{equation*}
            \vert M^N(t) - M^N(t-) \vert > 0 \Rightarrow t \in \{jN^{-1}: \, j \in \mathbb{N}_0\}.
        \end{equation*}
        \item For any $T \in [0,\infty)$, there exist $C_{T,1}, C_{T,2} > 0$ such that the following relations hold almost surely for every $N \in \mathbb{N}$
    \begin{equation*}
        \langle M^N \rangle(T) \leq C_{T,1} \quad \textrm{and} \quad \sup_{t \in [0,T]} \; \vert M^N(t + 1/N) - M^N(t) \vert \leq C_{T,2}.
    \end{equation*}
    \end{enumerate}
    Then the sequences of random variables $\Big((M^N(T))^2\Big)_{N \in \mathbb{N}}$ and $\Big(\left[M^N\right](T)\Big)_{N \in \mathbb{N}}$ are uniformly integrable, for any $T \in [0, \infty)$.
\end{lemma}

The proof of Lemma~\ref{Paper03_auxiliary_lemma_uniform_integrability_martingales} is standard, and we postpone it until Section~\ref{Paper03_section_proof_lemma_uniform_integrability_martingales}.

\begin{lemma} \label{Paper03_martingales_WF_model_constant_env}
     For all $N \in \mathbb{N}$, the process $M^{N}_{0}$ is a square-integrable martingale. Furthermore, for all $T \in \mathbb{R}_+$, the sequences of random variables $\left(\left[M^{N}_{0}\right](T)\right)_{N \in \mathbb{N}}$ and $\Big((M^{N}_{0}(T))^2\Big)_{N \in \mathbb{N}}$ are uniformly integrable.
\end{lemma}

\begin{proof}
By~\eqref{Paper03_definition_theta_aux_rv_construction_martingale_constant_environment}, 
for each $N \in \mathbb{N}$, 
$\left\{\theta^{N,\tau} \, \vert \, \tau \in \mathbb{N}_{0}\right\}$ is a set of 
i.i.d.~mean $0$ random variables, and therefore $M^{N}_{0}$ is a martingale 
with respect to the natural filtration of the process 
$\{\mathcal{F}^{N}_{T}\}_{T \geq 0}$ %, where for all $T \geq 0$, 
given by $\mathcal{F}^{N}_{T} \defeq 
\sigma(\boldsymbol{x}^{N}(t): \, t \leq T)$. Indeed, for any 
$0 \leq t_{1} \leq t_{2} -\frac{1}{N}$, we have, by the tower property, that
    \begin{equation*}
        \mathbb{E}\left[\theta^{N,t_{2}}(\boldsymbol{x}^{N}(t_{2})) \, \Big\vert \, \mathcal{F}^{N}_{t_{1}} \right] = \mathbb{E}\left[\mathbb{E}\left[\theta^{N,t_{2}}(\boldsymbol{x}^{N}(t_{2})) \Big\vert \, \mathcal{F}^{N}_{t_{2}} \right]\, \Big\vert \, \mathcal{F}^{N}_{t_{1}} \right] = 0,
    \end{equation*}
    so that $M^{N}_{0}$ is a $\{\mathcal{F}^{N}_{T}\}$-martingale. To complete the proof, observe that by the definition of $M^N_0$ in~\eqref{Paper03_martingale_seeds_WF_constant_environment}, and by~\eqref{Paper03_definition_zeta_N_constant_environment} and~\eqref{Paper03_definition_theta_aux_rv_construction_martingale_constant_environment}, the predictable bracket process $(\langle M^N_0 \rangle(t))_{t \geq 0}$ satisfies for all $T \geq 0$ and $N \in \mathbb{N}$ the following relation almost surely
    \begin{equation} \label{Paper03_predictable_process_bound_constant_environment}
    \begin{aligned}
        & \langle M^N_0 \rangle(T) \\ & \quad = \frac{1}{N} \sum_{t = 0}^{\lfloor NT \rfloor - 1} \frac{\sum_{i = 0}^{K} b_{i}x^N_{i}(t/N)}{(1 - x^N_{0}(t/N)) + \sum_{i = 0}^{K} b_{i}x^N_{i}(t/N)} \cdot \left(1 - \frac{\sum_{i = 0}^{K} b_{i}x^N_{i}(t/N)}{(1 - x^N_{0}(t/N)) + \sum_{i = 0}^{K} b_{i}x^N_{i}(t/N)}\right) \\ & \quad \leq \frac{T}{4}.
    \end{aligned}
    \end{equation}
    Also, note that by the definition of the collection of random variables $\left\{ \zeta^{N, \tau}_{j} \, \vert \, \tau \in \mathbb{N}_{0}, \, j \in \mathbb{N}  \right\}$ introduced before~\eqref{Paper03_definition_zeta_N_constant_environment}, the following relation holds almost surely for every $N \in \mathbb{N}$, $j \in [N]$ and $\tau \in \mathbb{N}_0$,
    \begin{equation*}
        \sup_{\boldsymbol{x} \in [0,1]^{K+1}} \Bigg\vert \zeta^{N,\tau}_j(\boldsymbol{x}) - \frac{\sum_{i=0}^K b_ix_i}{(1 - x_0) + \sum_{i=0}^K b_ix_i} \Bigg\vert \leq 1,
    \end{equation*}
    and therefore~\eqref{Paper03_martingale_seeds_WF_constant_environment} and~\eqref{Paper03_definition_theta_aux_rv_construction_martingale_constant_environment} imply that for any $T \geq 0$ and $N \in \mathbb{N}$,
    \begin{equation} \label{Paper03_martingale_bound_jumps_constant_environment}
        \sup_{t \in [0,T]} \vert M^N_0(t +1/N) - M^N_0(t) \vert \leq 1.
    \end{equation}
    Hence,~\eqref{Paper03_predictable_process_bound_constant_environment},~\eqref{Paper03_martingale_bound_jumps_constant_environment} and Lemma~\ref{Paper03_auxiliary_lemma_uniform_integrability_martingales} imply that the sequences $\Big((M^{N}_{0}(T))^2\Big)_{N \in \mathbb{N}}$ and $\Big(\left[M^{N}_{0}\right](T)\Big)_{N \in \mathbb{N}}$ are uniformly integrable, which completes the proof.
\end{proof}

We are now in a position to prove Theorem~\ref{Paper03_weak_convergence_diffusion_constant_environment}.

\begin{proof}[Proof of Theorem~\ref{Paper03_weak_convergence_diffusion_constant_environment}]
    From the dynamics of $\boldsymbol{x}^{N}$ given in~\eqref{Paper03_semimartingale_formulation_constant_environment_discrete} and~\eqref{Paper03_dynamics_seed_bank_WF_constant_environment}, and by Lemmas~\ref{Paper03_eigenvalue_condition_constant_environment} and~\ref{Paper03_martingales_WF_model_constant_env}, we conclude that the sequence of stochastic processes $(\boldsymbol{x}^{N}, M^{N}_{0})_{N \in \mathbb{N}}$ satisfies Assumptions~\ref{Paper03_assumption_manifold} and~\ref{Paper03_assumption_katzenberger_stochastic_process}. Therefore, Theorem~\ref{Paper03_major_tool_katzenberger} is applicable, which means the sequence $(\boldsymbol{x}^{N}, M^{N}_{0})_{N \in \mathbb{N}}$ is tight and any limiting process $(\boldsymbol{x}, M_{0})$ is a continuous semimartingale such that $\boldsymbol{x}(T) \in \Gamma$ for all $T \geq 0$ almost surely, where $\Gamma$ is the diagonal of the hypercube. In particular, by the uniform integrability of the sequence~$\left(\left[M^{N}_{0}\right](T)\right)_{N \in \mathbb{N}}$ stated in Lemma~\ref{Paper03_martingales_WF_model_constant_env}, any limiting process $M_{0}$ is also a martingale. To characterise its quadratic variation process, we first define, for all $N \in \mathbb{N}$ and $t \geq 0$, the process $(Q^{N}(t))_{N \in \mathbb{N}}$ by
    \begin{equation} \label{Paper03_quadratic_variation_prelimit_characterisation}
    \begin{aligned}
        Q^{N}(T) & \defeq (M^{N}_{0}(T))^{2} - \int_{0}^{T} \frac{\sum_{i = 0}^{K} b_{i}x^{N}_{i}}{(1 - x^{N}_{0}) + \sum_{i = 0}^{K} b_{i}x^{N}_{i}} \left(1 - \frac{\sum_{i = 0}^{K} b_{i}x^{N}_{i}}{(1 - x^{N}_{0}) + \sum_{i = 0}^{K} b_{i}x^{N}_{i}}\right) \, d(\lfloor Nt \rfloor/N) \\ & = (M^{N}_{0}(T))^{2} - \int_{0}^{T} \frac{(1 - x^{N}_{0})\left(\sum_{i = 0}^{K} b_{i}x^{N}_{i}\right)}{\left((1 - x^{N}_{0}) + \sum_{i = 0}^{K} b_{i}x^{N}_{i}\right)^{2}} \, d(\lfloor Nt \rfloor/N).
    \end{aligned}
    \end{equation}
    Then by~\eqref{Paper03_martingale_seeds_WF_constant_environment}, $(Q^N(t))_{t \geq 0}$ is a martingale. By Lemma~\ref{Paper03_martingales_WF_model_constant_env}, $\left((M^{N}_{0}(T))^{2}\right)_{N \in \mathbb{N}}$ is uniformly integrable, and therefore, taking $N \rightarrow \infty$ in~\eqref{Paper03_quadratic_variation_prelimit_characterisation} and recalling that on the diagonal $\Gamma$ of the hypercube, $x_{0} = \ldots = x_{K}$, we conclude that the process $(Q(t))_{t \geq 0}$ given by
    \begin{equation*}
        Q(T) \defeq M_{0}(T)^{2} - \int_{0}^{T} x_{0}(1-x_{0}) \, dt
    \end{equation*}
    is also a martingale, i.e.~$[M_{0}](T) = \int_{0}^{T} x_{0}(1-x_{0}) \, dt$, for all $T \geq 0$. Hence, there exists a standard Brownian motion $W_{0}$ adapted to the natural filtration of $(\boldsymbol{x}, M_{0})$ such that for all $T \geq 0$,
    \begin{equation*}
        M_{0}(T) = \int_{0}^{T} \sqrt{x_{0}(1-x_{0})} \, dW_0(t).
    \end{equation*}
    Recalling that $\boldsymbol{\Phi}: [0,1]^{K+1} \rightarrow \Gamma$ is the projection map associated to the ODE with vector field $\mathbf{F}$, we conclude 
from Theorem~\ref{Paper03_major_tool_katzenberger} that for all $T \geq 0$,
     \begin{equation} \label{Paper03_last_step_proof_SDE_WF_constant_environment}
    \begin{aligned}
        x_{0}(T) & = x_{0} + \int_{0}^{T} \frac{\partial \Phi_{0}}{\partial x_{0}} \sqrt{x_{0}(1 - x_{0})} \, dW_{0}(t) + \frac{1}{2}\int_{0}^{T} x_{0}(1-x_{0}) \frac{\partial^{2} \Phi_{0}}{\partial x_{0}^{2}} \, dt.
    \end{aligned}
    \end{equation}
    The proof is then complete by applying Lemma~\ref{Paper03_lemma_first_order_derivatives_projection_map} to~\eqref{Paper03_last_step_proof_SDE_WF_constant_environment}.
\end{proof}

We can now establish Theorem~\ref{Paper03_bound_fixation_probability_any_K}.

\begin{proof}[Proof of Theorem~\ref{Paper03_bound_fixation_probability_any_K}]
We start by establishing that the fixation probability is bounded by the map $(B,y) \mapsto \Psi(B,y)$ defined in~\eqref{Paper03_identity_fixation_probability_bound}. Since we do not have a closed expression for the second order derivative $\frac{\partial^{2} \Phi_{0}}{\partial x_{0}^{2}}$ which holds for any $K \in \mathbb{N}$, we will use property~$(ii)$ of Proposition~\ref{Paper03_bound_drift_general_K_proposition}, i.e.~that for any $\boldsymbol{x} \in \Gamma$,
\begin{equation*}
    \frac{\partial^{2} \Phi_{0}}{\partial x_{0}^{2}}(\boldsymbol{x}) \leq \frac{B(B(1-x_{0}) + 2)}{(B(1-x_{0}) + 1)^{3}}.
\end{equation*}
Therefore, instead of analysing directly~\eqref{Paper03_SDE_WF_constant_environment}, we consider the alternative SDE
\begin{equation} \label{Paper03_WF_diffusion_constant_environment_alternative}
    \tilde{x}_{0}(T)= \; {x}_{0}(0) + \frac{1}{2} \int_{0}^{T} \left(\tilde{x}_{0}(1 - \tilde{x}_{0})\frac{B(B(1-\tilde{x}_{0}) + 2)}{(B(1-\tilde{x}_{0}) + 1)^{3}}\right) \, dt + \int_{0}^{T} \frac{\sqrt{\tilde{x}_{0}(1 - \tilde{x}_{0})}}{(B(1 - \tilde{x}_{0}) + 1)} \, dW_{0}(t),
\end{equation}
starting from the same initial condition as $x_{0}$. Note that to bound the fixation probability of $(x_{0}(t))_{t \geq 0}$, it will suffice to bound the fixation probability of $(\tilde{x}_{0}(t))_{t \geq 0}$. For any $a \in [0,1]$, recall the definition of $T_a$ in~\eqref{Paper03_stopping_time_border}, and define the stopping time $$\widetilde{T}_{a} \defeq \inf \{t \geq 0 \, \vert \, \tilde{x}_{0}(t) = a\}.$$ Then, we have for any $y \in [0, 1)$,
\begin{equation} \label{Paper03_upper_bound_fixation_probability_step_i}
\begin{aligned}
    \mathbb{P}\left(T_{1} < T_{0} \, \Big\vert \, x_{0}(0) = \psi_B(y) \right) & \leq \mathbb{P}\left(\widetilde{T}_{1} < \widetilde{T}_{0} \, \Big\vert \, \tilde{x}_{0}(0) = \psi_B(y)\right).
\end{aligned}
\end{equation}
The right hand side is most easily calulated using the scale function for 
the diffusion.
%We shall apply the speed and scale theory of SDEs (see for instance~\cite[Section~8.1]{ethier2009markov}) to bound the right-hand side of~\eqref{Paper03_upper_bound_fixation_probability_step_i}.
Writing the SDE in~\eqref{Paper03_WF_diffusion_constant_environment_alternative} as
\begin{equation*}
    d\tilde{x}_{0} = \mu_{0}(\tilde{x}_{0}) \, dt + \eta_{0}(\tilde{x}_{0}) \, dW_{0}(t),
\end{equation*}
the scale function $S: [0,1] \rightarrow [0, +\infty)$ can be written
\begin{equation} \label{Paper03_definition_scale_function}
    S(v) = \int_{0}^{v} \exp\left(-2\int_{0}^{w} \frac{\mu_{0}(z)}{\eta_{0}^{2}(z)} \, dz\right) \, dw,
\end{equation}
where
%Observe that we can define the function $S$ for both $0$ and $1$ in our case since we have
\begin{equation*}
    -2\frac{\mu_{0}(z)}{\eta_{0}^{2}(z)} = -\frac{B(B(1-z) + 2)}{(B(1-z) + 1)}, \; \qquad \forall z \in [0,1].
\end{equation*}
Hence, for any $w \in [0,1]$,
\begin{equation*}
\begin{aligned}
    -2\int_{0}^{w} \frac{\mu_{0}(z)}{\eta_{0}^{2}(z)} \, dz & = - \int_{0}^{w} \frac{B(B(1-z) + 2)}{(B(1-z) + 1)} dz \\ & = - B\int_{0}^{w} \left(1 + \frac{1}{(B(1-z) + 1)} \right)  dz \\ & = -Bw - \log \left(\frac{B+1}{B(1 - w) + 1}\right).
\end{aligned}
\end{equation*}
%By applying~\eqref{Paper03_computation_speed_scale_i} to
Substituting in~\eqref{Paper03_definition_scale_function}, %we conclude that 
for any $v \in [0,1]$,
\begin{equation}
\label{Paper03_definition_scale_function_i}
\begin{aligned}
    S(v) = \frac{1}{(B+1)} \int_{0}^{v} e^{-B w} (B(1 - w) + 1) \, dw = \frac{1}{(B+1)}\left({1 - e^{-Bv} + ve^{-Bv}}\right).
\end{aligned}
\end{equation}
Thus, %by the speed and scale theory of diffusion 
(see e.g.~\cite[Equation~(4.17)]{ewens2004mathematical}), %we have that 
for any $y \in [0,1)$,
\begin{equation} \label{Paper03_upper_bound_fixation_probability_step_ii} \mathbb{P}\left(\widetilde{T}_{1} < \widetilde{T}_{0} \, \Big\vert \, \tilde{x}_{0}(0) = \psi_B(y) \right) = \frac{S(\psi_B(y)) - S(0)}{S(1) - S(0)} = \Psi(B,y),
\end{equation}
where the last equality follows 
from~\eqref{Paper03_definition_scale_function_i} and the definition of 
$\Psi(B,y)$ in~\eqref{Paper03_identity_fixation_probability_bound}. The 
bound on the fixation probability stated in 
Theorem~\ref{Paper03_bound_fixation_probability_any_K} then follows from 
applying~\eqref{Paper03_upper_bound_fixation_probability_step_ii} 
to~\eqref{Paper03_upper_bound_fixation_probability_step_i}.

To conclude our proof, it remains to establish the properties of the map $B \mapsto \Psi(B,y)$, for any $y \in (0,1)$. By the definition of $\psi_B$ in Lemma~\ref{Paper03_lemma_correspondence_value_branching_phase}, we have that for any $y \in [0,1)$,
\begin{equation*}
    \lim_{B \rightarrow 0^+} \psi_B(y) = y \quad \textrm{and} \quad \lim_{B \rightarrow \infty} \psi_B(y) = 0.
\end{equation*}
Since by~\eqref{Paper03_definition_scale_function_i}, the function $v \mapsto S(v)$ is continuous, 
we conclude that for any $y \in (0,1)$,
\begin{equation} \label{Paper03_bounding_fixation_probability_auxilliary_iii}
    \lim_{B \rightarrow 0^+} \Psi(B,y) = y \quad \textrm{and} \quad \lim_{B \rightarrow \infty} \Psi(B,y) = 0.
\end{equation}
It remains to establish that the map $B \mapsto \Psi(B,y)$ is strictly decreasing for any $y \in (0,1)$. By~\eqref{Paper03_identity_fixation_probability_bound}, we have for any $y \in (0,1)$,
\begin{equation} \label{Paper03_bounding_fixation_probability_auxilliary_iv}
\begin{aligned}
    \frac{\partial}{\partial B} \Psi(B,y) = \left(\psi_B(y) + B\frac{\partial}{\partial B} \psi_B(y) + \frac{\partial}{\partial B} \psi_B(y) + \psi_B^2(y) + B \psi_B(y) \frac{\partial}{\partial B} \psi_B(y)\right)e^{-B\psi_B(y)}.
\end{aligned}
\end{equation}
By applying Lemma~\ref{Paper03_lemma_correspondence_value_branching_phase} and~\eqref{Paper03_derivative_B_correspondence_branching_phase} to~\eqref{Paper03_bounding_fixation_probability_auxilliary_iv} and rearranging terms, we conclude that
\begin{equation} \label{Paper03_bounding_fixation_probability_auxilliary_v}
\begin{aligned}
     \frac{\partial}{\partial B} \Psi(B,y) & = \frac{e^{-B\psi_B(y)}}{2(B+1)^2} \log \left(\frac{B+1}{B + e^{-2y}}\right) \left[\frac{(1 - B)}{2(B+1)^3}\log\left(\frac{B+1}{B + e^{-2y}}\right)- 1 \right] \\ & \quad \quad - \frac{e^{-B\psi_B(y)}(1 - e^{-2y})}{2(B+1)^3(B+ e^{-2y})}\left[B + 1 + \frac{1}{2(B+1)^2} \log\left(\frac{B+1}{B+ e^{-2y}}\right)\right].
\end{aligned}
\end{equation}
In order to prove that the map $B \mapsto \Psi(B,y)$ is strictly decreasing for any $y \in (0,1)$, it will suffice to establish that
\begin{equation} \label{Paper03_bounding_fixation_probability_auxilliary_vi}
    \frac{(1 - B)}{2(B+1)^3}\log\left(\frac{B+1}{B + e^{-2y}}\right) \leq 1 \quad \forall (B,y) \in [0, \infty) \times (0,1).
\end{equation}
Since for any fixed $y \in (0, 1)$, the map
\begin{equation*}
    B \mapsto \frac{(1 - B)}{2(B+1)^3}\log\left(\frac{B+1}{B + e^{-2y}}\right),
\end{equation*}
is strictly decreasing in $B$, we conclude that for any $y \in (0,1)$ and $B \in [0, \infty)$,
\begin{equation*}
    \frac{(1 - B)}{2(B+1)^3}\log\left(\frac{B+1}{B + e^{-2y}}\right) \leq \frac{1}{2}\log(e^{-2y}) = y < 1,
\end{equation*}
and therefore~\eqref{Paper03_bounding_fixation_probability_auxilliary_vi} holds. By applying~\eqref{Paper03_bounding_fixation_probability_auxilliary_vi} to~\eqref{Paper03_bounding_fixation_probability_auxilliary_v}, we conclude that for any $y \in (0,1)$, and any $B \in [0, \infty)$, we have
\begin{equation} \label{Paper03_bounding_fixation_probability_auxilliary_vii}
    \frac{\partial}{\partial B} \Psi(B,y) < 0.
\end{equation}
Hence, Theorem~\ref{Paper03_bound_fixation_probability_any_K} follows from applying~\eqref{Paper03_upper_bound_fixation_probability_step_ii} to~\eqref{Paper03_upper_bound_fixation_probability_step_i}, and then from~\eqref{Paper03_bounding_fixation_probability_auxilliary_iii} and~\eqref{Paper03_bounding_fixation_probability_auxilliary_vii}.
\end{proof}

\subsection{Diffusion in slowly changing environment} \label{Paper03_subsection_diffusion_WF_model_slow_environment}

In this section, we prove the weak convergence of the Wright-Fisher model 
with seed bank with a slowly changing population size, i.e.~we prove 
Theorem~\ref{Paper03_weak_convergence_WF_model_slow_env}. Our arguments will 
be closely related to %the ones used in our proof of weak convergence of the Wright-Fisher model in constant environment 
those in 
Section~\ref{Paper03_subsection_diffusion_WF_model_constant_environment}. 
In contrast to the case of constant population size, we must now control two
sources of noise:
%when the population size changes, there are two noises that we must control: the noise 
that arising from the random process of sampling the genetic types in 
each generation, and the environmental noise. 

Recall 
that $(\Xi^{N}(t))_{t \in \mathbb{N}_{0}}$ records the change in 
the population size over time, and, 
from~\eqref{Paper03_scaled_process_WF_model}, that
$x^N(t)=X^N\left(\lfloor Nt\rfloor \right)/N$ and 
$\xi^N(t)=\Xi^N\left(\lfloor Nt\rfloor\right)$.
By Assumption~\ref{Paper03_diffusion_limit_slow_environment}, there exist functions $\alpha, \eta: [\xi_{\min}, \xi_{\max}]$ so that as $N \rightarrow \infty$, $\xi^{N}$ converges weakly to the diffusion
\begin{equation*}
    d\xi = \alpha(\xi) \,dt + \eta(\xi) \, dW_{\textrm{env}},
\end{equation*}
where $W_{\textrm{env}}$ is a standard Brownian motion. To introduce the proof, we define
\begin{equation} \label{Paper03_open_set_flow_slowly_changing_environment}
    \mathscr{U}^{\textrm{sl}} \defeq \left\{(\xi, x_{0}, \ldots, x_{K}) \in [\xi_{\min}, \xi_{\max}] \times \left(\mathbb{R}_{+}\right)^{K + 1}: \; \xi \geq x_{0} \textrm{ and } \xi_{\max} \geq x_i \; \forall \, i \in \llbracket K \rrbracket \right\}.
\end{equation}
By construction the process $(\boldsymbol{x}^{N}, \xi^{N})$ 
has sample paths in $\mathscr{U}^{\textrm{sl}}$, for all $N \in \mathbb{N}$. 
Following the same strategy as was used in the proof of 
Theorem~\ref{Paper03_weak_convergence_diffusion_constant_environment}, we shall
define
$\{0,1\}$-valued random variables that record, for each individual
in a new generation, whether or not it is a mutant. 
In contrast to the case of constant environment, the parameters of these 
Bernoulli-simile random variables will depend not only on 
$\boldsymbol{x}^{N}$, but also on $\xi^{N}$.

Formally, for each $N \in \mathbb{N}$, let $\left\{ \zeta^{N, \tau}_{j} \, 
\vert \, \tau \in \mathbb{N}_{0}, \, j \in \mathbb{N}  \right\}$ be a set 
of i.i.d.~random variables taking values in the space of measurable functions 
from $\mathscr{U}^{\textrm{sl}}$ to $\{0,1\}$ such that for all 
$(\boldsymbol{x}, \xi) \in \mathscr{U}^{\textrm{sl}}$,
\begin{equation} \label{Paper03_definition_zeta_N}
    \mathbb{E}\left[\zeta^{N, \tau}_{j}(\boldsymbol{x}, \xi)\right] = \frac{\sum_{i = 0}^{K}b_{i}x_{i}}{(\xi - x_{0}) + \sum_{i = 1}^{K}b_{i}x_{i}}.
\end{equation}
Note that the denominator in~\eqref{Paper03_definition_zeta_N} 
differs from that in~\eqref{Paper03_definition_zeta_N_constant_environment} 
because when the carrying capacity is not kept constant, the number of wild 
type individuals (divided by $N$) in the population is given by $\xi - x_{0}$ 
instead of $1 - x_{0}$. The other terms are the same because we are keeping 
track of the number of mutants in each generation, rather than the proportion. 
Then, for every $N \in \mathbb{N}$ and $t \in \mathbb{N}_{0}$, we can write 
$X^{N}_{0}(t + 1)$ as
\begin{equation} 
\label{Paper03_equation_mature_mutants_WF_first_part}
    X^{N}_{0}(t + 1) = \sum_{j = 1}^{\lfloor \Xi^{N}(t + 1)N \rfloor} \zeta^{N, \tau}_{j}\left(\frac{\boldsymbol{x}^{N}(t)}{N}, \Xi^{N}(t) \right).
\end{equation}
As in the proof of Theorem~\ref{Paper03_weak_convergence_diffusion_constant_environment}, in order to study the fluctuations arising from the genetic sampling, we introduce for each $N\in \mathbb{N}$, the set of i.i.d.~random variables $\left\{\theta^{N,\tau} \, \vert \, \tau \in \mathbb{N}_{0} \right\}$ taking values in the space of measurable functions from $\mathscr{U}$ to $\mathbb{R}$ given by
\begin{equation} \label{Paper03_definition_theta_aux_rv_construction_martingale}
    \theta^{N,\tau} (\boldsymbol{x}, \xi) \defeq \frac{1}{\sqrt{N}} \sum_{j = 1}^{\lfloor \xi N \rfloor} \left(\zeta^{N,\tau}_{j}(\boldsymbol{x}, \xi) - \frac{\sum_{i = 0}^{K} b_{i}x_{i}}{(\xi - x_{0}) + \sum_{i = 0}^{K} b_{i}x_{i}}\right)
\end{equation}

In order to apply Theorem~\ref{Paper03_major_tool_katzenberger}, we must describe the dynamics of $\boldsymbol{x}^{N}$ and $\xi^{N}$ in terms of semimartingales and a forcing flow. Starting with $\xi^{N}$, we have that for all $t \in \mathbb{N}$,
\begin{equation*}
    \Xi^{N}(t) = \Xi^{N}(0) + \sum_{\tau = 0}^{t-1} \frac{\alpha(\Xi^{N}(\tau))}{N} + \sum_{\tau = 0}^{t - 1} \left(\Xi^{N}(\tau + 1) - \Xi^{N}(\tau) - \frac{\alpha(\Xi^{N}(\tau))}{N}\right),
\end{equation*}
where $\alpha: [\xi_{\min}, \xi_{\max}] \rightarrow \mathbb{R}$ is the continuous function satisfying Assumption~\ref{Paper03_diffusion_limit_slow_environment} and that corresponds to the drift of the limiting environmental diffusion process~$\xi$. Hence, recalling~\eqref{Paper03_rescaling_processes_slowly_changing_environment}, we have for all $T \in \mathbb{R}_+$,
\begin{equation} \label{Paper03_dynamics_slow_env_WF}
\begin{aligned}
    \xi^{N}(T) & = \xi^{N}(0) + \int_{0}^{T} \alpha\left(\xi^{N}\left(\frac{t}{N}\right)\right) \, d\left({\lfloor Nt\rfloor}/{N}\right) + M_{\textrm{env}}^{N}(T)  \\ & = \xi^{N}(0) + \int_{0}^{T} \alpha\left(\xi^{N}\left(\frac{t}{N}\right)\right) \, d\left({\lfloor Nt\rfloor}/{N}\right) + M_{\textrm{env}}^{N}(T) + \int_{0}^{T} F^{\, \textrm{sl,env}}_{\textrm{env}}(\boldsymbol{x}^{N}, \xi^{N}) \, d\lfloor Nt\rfloor,
\end{aligned}
\end{equation}
where the process $M^{N}_{\textrm{env}}$ is given by, for all $T \in \mathbb{R}_{+}$,
\begin{equation} \label{Paper03_martingale_slow_environment}
    M_{\textrm{env}}^{N}(T) \defeq \sum_{\tau = 0}^{\lfloor NT \rfloor - 1} \left(\xi^{N}\left(\frac{t + 1}{N}\right) - \xi^{N}\left(\frac{t}{N}\right) - \frac{1}{N}{\alpha\left(\xi^{N}\left(\frac{t}{N}\right)\right)}\right),
\end{equation}
and $F^{\, \textrm{sl,env}}_{\textrm{env}} \equiv 0$ is the coordinate indicating the impact of the forcing flow field $\mathbf{F}^{{\, \textrm{sl,env}}}$, given in~\eqref{Paper03_flow_definition_WF_model_slowly_changing_environment}, on the environment. 
The forcing field does not affect the population size
and $M_{\textrm{env}}^{N}$ represents the noise arising from the environmental dynamics. We also emphasise that, as expected, the process $\boldsymbol{x}^{N}$ does not appear in the dynamics of $\xi^{N}$. This reflects our underlying assumption that the population size is not influenced by the genetic composition of the population in this model.

We now describe the dynamics of $x^{N}_{0}$. Recall the definition of the 
flow field $\mathbf{F}^{\textrm{sl,env}}:\mathscr{U} \rightarrow \mathbb{R}^{K+2}$ in~\eqref{Paper03_flow_definition_WF_model_slowly_changing_environment}. Following the same strategy as we used to derive~\eqref{Paper03_semimartingale_formulation_constant_environment_discrete}, we write $x^{N}_{0}$ in terms of the sums of the family of random variables $\zeta^{N,\tau}_{j}$, i.e.~we write, for any $T \in \mathbb{R}_+$,
\begin{equation} \label{Paper03_dynamics_x_0_discrete_slow_environment}
\begin{aligned}
    x^{N}_{0}(T) & = x^{N}_{0}(0) + \sum_{t = 0}^{\lfloor NT \rfloor -1} \; \Bigg(\Bigg(\frac{1}{N} \sum_{j = 1}^{\lfloor N\xi^{N}((t + 1)/N) \rfloor} \zeta^{N,t}_{j}(\boldsymbol{x}^{N}(t/N) , \xi^{N}(t /N))\Bigg) - x^{N}_{0}(t/N)\Bigg) \\
    & = x^{N}_{0}(0) + M^{N}_{0}(T) + \widetilde{M}^{N}_{\textrm{env}}(T) + {M}^{N}_{\textrm{aux}}(T) + \int_{0}^{T} \frac{\alpha(\xi^{N})\sum_{i = 0}^{K} b_{i}x^{N}_{i}}{(\xi^{N} - x^{N}_{0}) + \sum_{i = 0}^{K} b_{i}x^{N}_{i}} \, d\left(\lfloor Nt \rfloor /N \right) \\ & \quad \quad + \int_{0}^{T} F^{\, \textrm{sl,env}}_{0}(\boldsymbol{x}^{N}, \xi^{N}) \, d \lfloor N\tau \rfloor,
\end{aligned}
\end{equation}
where the processes $(M^{N}_{0}(t))_{t \in \mathbb{R}_{+}}$, $(\widetilde{M}^{N}_{\textrm{env}}(t))_{t \in \mathbb{R}_{+}}$ and $({M}^{N}_{\textrm{aux}}(t))_{t \in \mathbb{R}_{+}}$ are given by, for all $N \in \mathbb{N}$ and all $T \in \mathbb{R}_{+}$,
\begin{align}
    & M^{N}_{0}(T) \defeq \int_{0}^{\lfloor NT\rfloor-1} \, \frac{1}{\sqrt{N}} \, d\theta^{N,t}(\boldsymbol{x}^{N}(t/N), \xi^{N}(t/N)), \label{Paper03_martingale_seeds_WF} \\[2ex] & \widetilde{M}^{N}_{\textrm{env}}(T) \defeq \int_{0}^{\lfloor NT\rfloor-1} \Bigg(\frac{\sum_{i = 0}^{K} b_{i}x^{N}_{i}(t/N)}{(\xi^{N}(t/N) - x^{N}_{0}(t/N)) + \sum_{i = 0}^{K} b_{i}x^{N}_{i}(t/N)}\Bigg) \, dM^{N}_{\textrm{env}}(t/N), \label{Paper03_martingale_impact_env_on_seeds_WF} \\[2ex] & M^{N}_{\textrm{aux}}(T) \nonumber \\  & \quad  \defeq \frac{1}{N} \sum_{t = 0}^{\lfloor NT \rfloor - 1} \Bigg(\mathds{1}_{\left\{\xi^{N}(t/N) < \xi^{N}((t + 1)/N)\right\}} \sum_{j = \left\lfloor N\xi^{N}(t/N) \right\rfloor + 1}^{\left\lfloor N\xi^{N}((t+1)/N) \right\rfloor} \beta_{j}^{N,t}(\boldsymbol{x}^{N}(t/N), \xi^{N}(t/N)) \nonumber \\[2ex] & \quad \quad \quad \quad \quad \quad \quad \quad - \mathds{1}_{\left\{\xi^{N}((t+1)/N) < \xi^{N}(t/N)\right\}} \sum_{j = \left\lfloor N\xi^{N}((t+1)/N) \right\rfloor + 1}^{\left\lfloor N\xi^{N}(t/N) \right\rfloor} \beta_{j}^{N,t}(\boldsymbol{x}^{N}(t/N), \xi^{N}(t/N))  \Bigg). \label{Paper03_martingale_both_seeds_slow_env_WF}
\end{align}
In~\eqref{Paper03_martingale_impact_env_on_seeds_WF}, $M^{N}_{\textrm{env}}$ 
is as defined in~\eqref{Paper03_martingale_slow_environment}, and for all 
$N \in \mathbb{N}$, $t \in \mathbb{N}_{0}$, $j \in \mathbb{N}$ and 
$(\boldsymbol{x}, \xi) \in \mathscr{U}$, $\beta^{N,t}_j(\boldsymbol{x},\xi)$ 
in~\eqref{Paper03_martingale_both_seeds_slow_env_WF}
is given by
\begin{equation} \label{Paper03_definition_error_beta_slowly_changing_environment}
    \beta^{N,t}_{j}(\boldsymbol{x}, \xi) \defeq \zeta^{N,t}_{j}(\boldsymbol{x}, \xi) - \frac{\sum_{i = 0}^{K} b_{i}x_{i}}{(\xi - x_{0}) + \sum_{i = 0}^{K} b_{i}x_{i}}.
\end{equation}
Finally, for all $i \in \llbracket K \rrbracket$, since $X^{N}_{i}(t) \defeq X^{N}_{0}(t-i)$, we can write, for all $i \in \llbracket K \rrbracket$ and all $t \geq 0$,
\begin{equation} \label{Paper03_dynamics_seed_bank_WF}
    x^{N}_{i}(t) = x^{N}_{i}(0) + \int_{0}^{t} F^{\, \textrm{sl,env}}_{i}(\boldsymbol{x}^{N}, \xi^{N}) \, d\lfloor N\tau \rfloor.
\end{equation}

The three fluctuation terms $M_0^N$, $\widetilde{M}_{\textrm{env}}^N$, and 
$M^N_{\textrm{aux}}$ account, respectively, for the intrinsic randomness due
to sampling in the Wright-Fisher model, the fluctuations due to changes in 
population size, and fluctuations due to {\em both} changes in population
size and competition with wild type individuals.
%Heuristically, $M^{N}_{0}$ represents fluctuations in the number of mutant individuals which are due to the intrinsic randomness of the way the type of individual in a generation is chosen in our Wright-Fisher model, while $\widetilde{M}^{N}_{\textrm{env}}$ indicates fluctuations in the number of mutant individuals that occur due to fluctuations in the environment. Furthermore, $M^{N}_{\textrm{aux}}$ takes into account the fluctuations that occur both due to the environment and the competition between the wild-type and mutant individuals. 
Our next result shows that these processes capture the randomness of our Wright-Fisher model.

\begin{lemma} \label{Paper03_martingales_WF_model_slow_env}
     For all $N \in \mathbb{N}$, the processes $M^{N}_{0}$ and ${M}^{N}_{\textrm{aux}}$ are square integrable martingales. Moreover, there exist processes $(\varepsilon^{N}(t))_{t \geq 0}$ and $(\tilde{\varepsilon}^{N}(t))_{t \geq 0}$, such that $M^{N}_{\textrm{env}} + \varepsilon^{N}$ and $\widetilde{M}^{N}_{\textrm{env}} + \tilde{\varepsilon}^{N}$ are square integrable martingales for all $N \in \mathbb{N}$, and such that for all $T \in \mathbb{R}_+$, $\mathbb{E}\left[\vert\varepsilon^{N}(T)\vert^{p}\right] \rightarrow 0$ and $\mathbb{E}\left[\vert\tilde{\varepsilon}^{N}(T)\vert^{p}\right] \rightarrow 0$ as $N \rightarrow \infty$, for all $p \geq 1$. Furthermore, for all $T \in \mathbb{R}_+$, the sequences of random variables $\left(\left[M^{N}_{0}\right](T)\right)_{N \in \mathbb{N}}$, $\left(\left[M^{N}_{\textrm{aux}}\right](T)\right)_{N \in \mathbb{N}}$, $\left(\left[M^{N}_{\textrm{env}} + \varepsilon^{N}\right](T)\right)_{N \in \mathbb{N}}$, $\left(\left[\widetilde{M}^{N}_{\textrm{env}} + \tilde{\varepsilon}^{N}\right](T)\right)_{N \in \mathbb{N}}$, $\Big((M^{N}_{0}(T))^2\Big)_{N \in \mathbb{N}}$ and $\Big((M^{N}_{\textrm{env}}(T) + \varepsilon^{N}(T))^2\Big)_{N \in \mathbb{N}}$ are uniformly integrable.
\end{lemma}

\begin{proof}
    We shall study each of the processes in the statement of the lemma separately. 
    
    \medskip

    \noindent \underline{Step~$(i)$: Characterisation of $M^N_0$:}

    \medskip
    
    Note that by~\eqref{Paper03_definition_theta_aux_rv_construction_martingale}, for each $N \in \mathbb{N}$, $\left\{\theta^{N,\tau} \, \vert \, \tau \in \mathbb{N}_{0}\right\}$ is a set of i.i.d.~mean $0$ random variables, and therefore $M^{N}_{0}$ is a martingale. To establish the required uniform integrability, we first note that~\eqref{Paper03_martingale_seeds_WF},~\eqref{Paper03_definition_zeta_N} and~\eqref{Paper03_definition_theta_aux_rv_construction_martingale} imply that the following relation holds almost surely for every $N \in \mathbb{N}$ and $T \in [0, \infty)$:
    \begin{equation} \label{Paper03_bound_predict_variation_martingale_genetic_sampling_slow_environment}
    \begin{aligned}
        & \langle M^N_0\rangle (T) \\ & \quad = \frac{1}{N} \sum_{t = 0}^{\lfloor NT \rfloor -1} \left(\frac{\lfloor \xi^N N \rfloor}{N} \cdot \frac{\sum_{i = 0}^{K} b_{i}x^N_{i}}{(\xi^N - x^N_{0}) + \sum_{i = 0}^{K} b_{i}x^N_{i}} \cdot \left(1 - \frac{\sum_{i = 0}^{K} b_{i}x^N_{i}}{(\xi^N - x^N_{0}) + \sum_{i = 0}^{K} b_{i}x^N_{i}}\right)\right)(t/N) \\ & \quad \leq \frac{T\xi_{\max}}{4}.
    \end{aligned}
    \end{equation}
    Moreover, by the same argument as was used to 
derive~\eqref{Paper03_martingale_bound_jumps_constant_environment}, 
for every $N \in \mathbb{N}$ and $T \in [0, \infty)$, the following relation holds almost surely
    \begin{equation} \label{Paper03_bound_jump_martingale_genetic_sampling_slow_environment}
        \sup_{t \in [0,T]} \; \vert M^N_0(t+1/N) - M^N_0(t) \vert \leq 1.
    \end{equation}
 Combining~\eqref{Paper03_bound_predict_variation_martingale_genetic_sampling_slow_environment} with~\eqref{Paper03_bound_jump_martingale_genetic_sampling_slow_environment}, and applying Lemma~\ref{Paper03_auxiliary_lemma_uniform_integrability_martingales}, we conclude that the sequences of random variables $\Big((M^{N}_{0}(T))^2\Big)_{N \in \mathbb{N}}$ and $\Big(\left[M^{N}_{0}\right](T)\Big)_{N \in \mathbb{N}}$ are uniformly integrable.

    \medskip

    \noindent \underline{Step~$(ii)$: Characterisation of $M^{N}_{\textrm{aux}}$:}

    \medskip    
    
   Let $\{\mathcal{F}^{N}_{t}\}_{t \geq 0}$ be the natural filtration of $(\boldsymbol{x}^{N}(t), \xi^{N}(t))_{t \geq 0}$. Then, by the definition of $\zeta^{N,\tau}$ given in~\eqref{Paper03_definition_zeta_N} and the definition of $M^{N}_{\textrm{aux}}$ in~\eqref{Paper03_martingale_both_seeds_slow_env_WF}, we conclude that for all $t \geq 0$,
    \begin{equation*}
        \mathbb{E}\left[M^{N}_{\textrm{aux}}(t + 1/N) - M^{N}_{\textrm{aux}}(t) \, \vert \mathcal{F}^{N}_{t} \times \sigma(\xi^N(t +1/N)) \right] = 0.
    \end{equation*}
    Thus, by the tower property of conditional expectations, $M^{N}_{\textrm{aux}}$ is a martingale. Moreover, since the collection of random variables $\left\{\zeta^{N,t}_j: \, j \in \mathbb{N}_0\right\}$ is conditionally independent given $\mathcal{F}^N_t$, we have
    \begin{equation*}
    \begin{aligned}
        & \mathbb{E}\left[(M^{N}_{\textrm{aux}}(t+1/N) - M^{N}_{\textrm{aux}}(t))^{2} \, \vert \, \mathcal{F}^{N}_{t} \times \sigma(\xi^N(t +1/N)) \right] \\ & \quad \lesssim N^{-2}\sum_{j = \lfloor \xi_{\min}N \rfloor}^{\lfloor \xi_{\max}N \rfloor} \mathds{1}_{\left\{N\left(\xi^{N}(T) \wedge \xi^{N}\left(t + \frac{1}{N}\right)\right) \leq j \leq N\left(\xi^{N}(t) \vee \xi^{N}\left(t + \frac{1}{N}\right)\right)\right\}} \\ & \quad \quad \quad \cdot \mathbb{E}\Bigg[\Bigg(\zeta^{N,\tau}_{j}(\boldsymbol{x}^N(t),\xi^{N}(t)) - \frac{\sum_{i = 0}^{K} b_{i}x^{N}_{i}(t)}{(\xi^{N}(t) - x^{N}_{0}(t)) + \sum_{i = 0}^{K} b_{i}x^{N}_{i}(t)}\Bigg)^{2} \Bigg\vert \, \mathcal{F}^{N}_{t} \times \sigma(\xi^N(t +1/N))\Bigg] \\ & \quad \lesssim N^{-1}\left\vert \xi^{N}(t + 1/N) - \xi^{N}(t) \right\vert.
    \end{aligned}
    \end{equation*}
    Hence, by the tower property, for all $t \geq 0$, we have
    \begin{equation} \label{Paper03_boundedness_M_aux_WF}
    \begin{aligned}
        & \mathbb{E}\left[\left(M^{N}_{\textrm{aux}}(t+1/N) - M^{N}_{\textrm{aux}}(t)\right)^{2} \, \vert \, \mathcal{F}^N_t\right] \\ & \quad \lesssim \frac{1}{N}\mathbb{E}\left[\left\vert \xi^{N}(t + 1/N) - \xi^{N}(t) \right\vert \vert \, \mathcal{F}^N_t \right] \\ & \quad \leq \frac{1}{N}\mathbb{E}\left[\left( \xi^{N}(t + 1/N) - \xi^{N}(t) \right)^{2} \vert \, \mathcal{F}^N_t \right]^{1/2} \\ & \quad \leq {\vert \vert \eta \vert \vert_{\infty}^{1/2}}N^{-3/2} + \mathcal{O}\left(N^{-3/2 - \delta}\right)
    \end{aligned}
    \end{equation}
    for some $\delta > 0$, where we applied Cauchy-Schwarz to obtain the 
second and third inequalities, and 
Assumption~\ref{Paper03_diffusion_limit_slow_environment} on the dynamics 
of the environment for the last one.
    Since $M^N_{\textrm{aux}}$ is a purely discontinuous martingale, we conclude that the following estimate holds almost surely for every $N \in \mathbb{N}$ and all $T \geq 0$
    \begin{equation} \label{Paper03_bound_predictable_process_auxiliary_martingale_slow_env}
        \langle M^N_\textrm{aux} \rangle(T) = \sum_{t = 0}^{\lfloor NT \rfloor - 1} \mathbb{E}\left[(M^{N}_{\textrm{aux}}(t+1/N) - M^{N}_{\textrm{aux}}(t))^{2} \, \vert \, \mathcal{F}^N_t\right] \leq  T{\vert \vert \eta \vert \vert_{\infty}^{1/2}}N^{-1/2} + \mathcal{O}(N^{-1/2 - \delta}),
    \end{equation}
    where the last inequality follows from~\eqref{Paper03_boundedness_M_aux_WF}. Moreover, from the definition of $M^N_{\textrm{aux}}$ in~\eqref{Paper03_martingale_both_seeds_slow_env_WF} and from Assumption~\ref{Paper03_diffusion_limit_slow_environment}, we have that for every $N \in \mathbb{N}$, the following estimate holds almost surely
    \begin{equation} \label{Paper03_bound_jump_auxiliary_martingale_slow_env}
        \sup_{t \in [0, \infty)} \vert M^N_{\textrm{aux}}(t + 1/N) - M^N_{\textrm{aux}}(t-) \vert \leq \xi_{\max} - \xi_{\min}.
    \end{equation}
    Therefore,~\eqref{Paper03_bound_predictable_process_auxiliary_martingale_slow_env},~\eqref{Paper03_bound_jump_auxiliary_martingale_slow_env} and Lemma~\ref{Paper03_auxiliary_lemma_uniform_integrability_martingales} imply that $\Big(\left[M^N_{\textrm{aux}}\right](T)\Big)_{N \in \mathbb{N}}$ is uniformly integrable for every $T \geq 0$, as desired.

    \medskip

    \noindent \underline{Step~$(iii)$: Characterisation of $M^{N}_{\textrm{env}}$ and of $\widetilde{M}^{N}_{\textrm{env}}$:}

    \medskip
    
    Let $w^N_1,w^N_2,w^N_3: [\xi_{\min}, \xi_{\max}] \rightarrow \mathbb{R}$ be the maps satisfying Assumption~\ref{Paper03_diffusion_limit_slow_environment}, and for every $ N \in \mathbb{N}$, let the process $(\varepsilon^N(t))_{t \geq 0}$ be given by, for all $T \in [0, \infty)$,
    \begin{equation} \label{Paper03_correction_environment_martingale}
        \varepsilon^N(T) \defeq - \sum_{t = 0}^{\lfloor NT \rfloor - 1} w^N_1\left(\xi^N\left(\frac{t}{N}\right)\right).
    \end{equation}
    Then, by Assumption~\ref{Paper03_diffusion_limit_slow_environment}, $\left((M^N_{\textrm{env}} + \varepsilon^N)(t)\right)_{t \geq 0}$ is a martingale. Moreover, by applying the triangle inequality and Assumption~\ref{Paper03_diffusion_limit_slow_environment} to~\eqref{Paper03_correction_environment_martingale}, we conclude that there exists $\delta > 0$ such that for every $N \in \mathbb{N}$ and $T \in [0, \infty)$, the following relation holds almost surely
    \begin{equation} \label{Paper03_control_correction_martingale_slow_environment}
        \sup_{t \in [0, T]} \vert \varepsilon^N(t) \vert \lesssim T N^{-\delta}.
    \end{equation}
    Observe also that for every $N \in \mathbb{N}$, Assumption~\ref{Paper03_diffusion_limit_slow_environment} and~\eqref{Paper03_martingale_slow_environment} imply that almost surely
    \begin{equation} \label{Paper03_bound_jump_environmental_noise_slow}
        \sup_{t \in [0, \infty)} \, \vert (M^N_{\textrm{env}} + \varepsilon^N)(t + 1/N) -(M^N_{\textrm{env}} + \varepsilon^N)(t) \vert \leq \xi_{\max} - \xi_{\min} + \mathcal{O}(N^{-1}).
    \end{equation}
    Moreover,~\eqref{Paper03_martingale_slow_environment},~\eqref{Paper03_correction_environment_martingale} and Assumption~\ref{Paper03_diffusion_limit_slow_environment} imply that for every $N \in \mathbb{N}$, we have almost surely
    \begin{equation*}
        \sup_{t \in [0, \infty)} \, \mathbb{E}\Big[\Big( (M^N_{\textrm{env}} + \varepsilon^N)(t + 1/N) -(M^N_{\textrm{env}} + \varepsilon^N)(t) \Big)^2 \, \Big\vert \, \mathcal{F}^N_t\Big] \leq \vert \vert \eta \vert \vert_{\infty}N^{-1} + \mathcal{O}(N^{-1-\delta}),
    \end{equation*}
    and therefore
    \begin{equation} \label{Paper03_bound_predictable_bracket_process_environmental_noise_slow}
        \langle M^N_{\textrm{env}} \rangle(T) \leq T\vert \vert \eta \vert \vert_{\infty} + \mathcal{O}(N^{-\delta}) \; \textrm{almost surely}. 
    \end{equation}
    It follows from~\eqref{Paper03_bound_jump_environmental_noise_slow},~\eqref{Paper03_bound_predictable_bracket_process_environmental_noise_slow} and Lemma~\ref{Paper03_auxiliary_lemma_uniform_integrability_martingales} that, for all $T \in [0, \infty)$, both the sequences of random variables $\Big((M^{N}_{\textrm{env}}(T))^2\Big)_{N \in \mathbb{N}}$ and $\Big(\left[M^{N}_{\textrm{env}}\right](T)\Big)_{N \in \mathbb{N}}$ are uniformly integrable.
    
    The analogous results concerning the process $\widetilde{M}^{N}_{\textrm{env}}$, defined in~\eqref{Paper03_martingale_impact_env_on_seeds_WF}, hold due to the fact that the process
    \begin{equation*}
        \left(\frac{\sum_{i = 0}^{K} b_{i}x^N_{i}(t-)}{\xi^N(t-) - x^N_{0}(t-) + \sum_{i = 0}^{K} b_{i}x^N_{i}(t-)}\right)_{t \geq 0}
        \end{equation*}
        is predictable with respect to the filtration $\{\mathcal{F}^N_t\}_{t \geq 0}$ introduced before~\eqref{Paper03_source_ODE_slow_environment_env_coordinate}, and then by observing that the map
    \begin{equation} \label{Paper03_ratio_proportion_seeds}
        (\boldsymbol{x},\xi) \mapsto \frac{\sum_{i = 0}^{K} b_{i}x_{i}}{(\xi - x_{0}) + \sum_{i = 0}^{K} b_{i}x_{i}}
    \end{equation}
    is uniformly bounded by 1. The proof is then complete.
\end{proof}

We now characterise the evolution of $\xi^{N}$ and $\boldsymbol{x}^{N}$ in terms of a well-behaved semimartingale. For any $N \in \mathbb{N}$, we define the $K + 2$-dimensional process 
\[
\left(\boldsymbol{\Lambda}^{N}(t)\right)_{t \geq 0} = \left(\Lambda^{N}_{0}(t), \ldots, \Lambda^{N}_{K}(t), \Lambda^{N}_{\textrm{env}}(t)\right)_{t \geq 0}
\]
such that, for all $T \geq 0$,
\begin{equation} \label{Paper03_semimartingale_discrete_process_WF}
\begin{aligned}
    \Lambda^{N}_{\textrm{env}}(T) & \defeq \int_{0}^{T} \alpha(\xi^{N}) \, d\left({\lfloor Nt\rfloor}/{N}\right) + M_{\textrm{env}}^{N}(T), \\ \Lambda^{N}_{0}(T) & \defeq M^{N}_{0}(T) + \widetilde{M}^{N}_{\textrm{env}}(T) + {M}^{N}_{\textrm{aux}}(T) + \int_{0}^{T} \frac{\alpha(\xi^{N})\sum_{i = 0}^{K} b_{i}x^{N}_{i}}{\xi^{N} - (1 - b_{0})x^{N}_{0} + \sum_{i = 1}^{K} b_{i}x^{N}_{i}} \, d\left(\lfloor Nt \rfloor /N \right), \\
    \Lambda^{N}_{i}(T) & \defeq 0 \, \textrm{ for all } i \in \llbracket K \rrbracket.
\end{aligned}
\end{equation}
Note that, by the description of $(\boldsymbol{x}^{N}(t), \xi^N(t))_{t \geq 0}$, we can write for all $T \geq 0$,
\begin{equation} \label{Paper03_writing_slow_environment_in_terms_Katzenberger}
    (\boldsymbol{x}^{N}(T), \xi^N(T)) = (\boldsymbol{x}^{N}(0), \xi^{N}(0)) + \boldsymbol{\Lambda}^{N}(T) + \int_{0}^{T} \boldsymbol{F}^{\textrm{sl,env}}(\boldsymbol{x}^{N}, \xi^N) \, d\lfloor Nt \rfloor.
\end{equation}
Our next result is a corollary of Lemma~\ref{Paper03_martingales_WF_model_slow_env}. Recall that the total variation of a càdlàg process over the interval $[0,T]$ is denoted $\mathfrak{T}(\cdot,t)$.

\begin{corollary}
\label{Paper03_semimartingale_behaviour_WF_model}
    The sequence of processes $(\boldsymbol{\Lambda}^{N})_{N \in \mathbb{N}}$ is relatively compact in $\mathscr{D}\left([0, \infty), \mathbb{R}^{K+2}\right)$. Moreover, for any $T \geq 0$, the sequence of random variables $(\boldsymbol{\Lambda}^{N}(T))_{N \in \mathbb{N}}$ is uniformly integrable, and for each $T \geq 0$, the sequences of processes
    \begin{equation*}
    \begin{aligned}
        & \left(\mathfrak{T}\left(\left(\int_{0}^{t} \alpha(\xi^{N}) \, d\left({\lfloor Nt'\rfloor}/{N}\right)\right)_{t \geq 0}, T \right)\right)_{N \in \mathbb{N}}, \\ & \left(\mathfrak{T}\left(\left(\int_{0}^{t} \frac{\alpha(\xi^{N})\sum_{i = 0}^{K} b_{i}x^{N}_{i}}{\xi^{N} - (1 - b_{0})x^{N}_{0} + \sum_{i = 1}^{K} b_{i}x^{N}_{i}} \, d\left(\lfloor Nt' \rfloor /N\right)\right)_{t \geq 0}, T\right)\right)_{N \in \mathbb{N}},
    \end{aligned}
    \end{equation*}
    are both relatively compact in $\mathscr{D}([0, + \infty), \mathbb{R})$.
\end{corollary}

\begin{proof}
    By Lemma~\ref{Paper03_martingales_WF_model_slow_env}, the fact that both $\alpha$ and the map defined in~\eqref{Paper03_ratio_proportion_seeds} are bounded in $[\xi_{\min}, \xi_{\max}]$ and in $\mathscr{U}^{\textrm{sl}}$, and by standard results regarding the Skorokhod topology of $\mathscr{D}\left([0, \infty), \mathbb{R}^{K+2}\right)$, see e.g.~\cite[Theorem~3.8.6(c)]{ethier2009markov}, we conclude that $(\boldsymbol{\Lambda}^{N})_{N \in \mathbb{N}}$ is relatively compact in $\mathscr{D}\left([0, \infty), \mathbb{R}^{K+2}\right)$. Uniform integrability also follows from Lemma~\ref{Paper03_martingales_WF_model_slow_env} and the boundedness of $\alpha$ and the map given in~\eqref{Paper03_ratio_proportion_seeds}. Finally, for any $T \geq 0$, and any Lebesgue integrable function $g: [0,T] \rightarrow \mathbb{R}$, we have
    \begin{equation*}
        \mathfrak{T}\left(\left(\int_{0}^{t} g(t') \, dt'\right)_{t \geq 0}, \, T\right) = \int_{0}^{T} \vert g(t) \vert \, dt.
    \end{equation*}
    Hence, the relative compactness of the total variation processes follows from the almost surely boundedness of the integrands, as desired.
\end{proof}

Having verified 
Assumption~\ref{Paper03_assumption_katzenberger_stochastic_process}, we now 
turn to the forcing flow $\boldsymbol{F}^{\textrm{sl,env}}$ defined 
in~\eqref{Paper03_flow_definition_WF_model_slowly_changing_environment} and associated to the projection map $\boldsymbol{\Phi}^{\textrm{sl,env}}$.
Recall 
the map
$\boldsymbol{\Phi}$ from~\eqref{Paper03_definition_projection_map},
the manifold $\Gamma^{\; \textrm{sl,env}}$ 
from~\eqref{Paper03_attractor_manifold_slow_changing_environment}, and 
the notation 
$\boldsymbol{\rho}(\boldsymbol{x},\xi) \defeq \boldsymbol{x}/\xi$, 
for any~$(\boldsymbol{x},\xi) \in \Gamma^{\; \textrm{sl,env}}$.
To emphasise that, in contrast to the case of a constant population size,
$\boldsymbol{\Phi}$ is defined as a function of $\boldsymbol{\rho}(\boldsymbol{x},\xi)$
rather than $\boldsymbol{x}$,
we shall use the notation $\partial/\partial \rho_i$ to indicate 
differentiation of $\boldsymbol{\Phi}$ with respect to 
the $i$th component; it is
important to keep this in mind when we use 
Lemma~\ref{Paper03_lemma_first_order_derivatives_projection_map} in the proof of
our next result.

\begin{lemma} 
\label{Paper03_derivatives_flow_manifold_moran_constant_env}
    For any $(\boldsymbol{x}, \xi) \in \Gamma^{\; \textrm{sl,env}}$, we have
    \begin{equation} \label{Paper03_derivatives_slow_env_wrt_x_0}
    \begin{aligned}
        \frac{\partial \Phi^{\textrm{sl,env}}_{0}}{\partial x_{0}} (\boldsymbol{x}, \xi) = \frac{\xi}{B(\xi-x_{0}) + \xi} \quad \textrm{ and } \quad \frac{\partial^{2} \Phi_{0}^{\textrm{sl,env}}}{\partial x^{2}_{0}} (\boldsymbol{x},\xi) = \frac{1}{\xi} \frac{\partial^{2} \Phi_{0}}{\partial \rho^{2}_{0}}(\boldsymbol{\rho}(\boldsymbol{x}, \xi)).
        \end{aligned}
        \end{equation}
        Moreover, 
        \begin{equation} \label{Paper03_derivatives_slow_env_wrt_x_0_and_environment}
        \begin{aligned}
            \frac{\partial \Phi_{0}^{\textrm{sl,env}}}{\partial \xi} (\boldsymbol{x}, \xi) =  \frac{\partial^{2} \Phi_{0}^{\textrm{sl,env}}}{\partial \xi^{2}} (\boldsymbol{x},\xi) = 0 
\quad \textrm{ and } \quad \frac{\partial^{2} \Phi^{\textrm{sl,env}}_{0}}{\partial \xi \partial x_{0}} (\boldsymbol{x},\xi) = \frac{-B x_{0}}{(B(\xi - x_{0}) + \xi)^{2}}.
        \end{aligned}
        \end{equation}
        Furthermore,
\begin{equation}
\label{Paper03_derivatives_slow_env_environment_component_wrt_x_0_and_environment}
        \begin{aligned}
        \frac{\partial \Phi^{\textrm{sl,env}}_{\textrm{env}}}{\partial \xi} (\boldsymbol{x},\xi) & = 1, \;  \frac{\partial \Phi^{\textrm{sl,env}}_{\textrm{env}}}{\partial x_i} (\boldsymbol{x},\xi) = 0, \\ 
\textrm{and} \qquad \; \frac{\partial^2 \Phi^{\textrm{sl,env}}_{\textrm{env}}}{\partial \xi^2} (\boldsymbol{x},\xi) & = \frac{\partial^2 \Phi^{\textrm{sl,env}}_{\textrm{env}}}{\partial \xi \partial x_i}  (\boldsymbol{x},\xi) = \frac{\partial^2 \Phi^{\textrm{sl,env}}_{\textrm{env}}}{\partial x_i \partial x_j}(\boldsymbol{x},\xi) = 0 \qquad\; \forall \, i,j \in \llbracket K \rrbracket_0.
        \end{aligned}
        \end{equation}
\end{lemma}

\begin{proof}
    Observe that~\eqref{Paper03_derivatives_slow_env_environment_component_wrt_x_0_and_environment} follows directly from~\eqref{Paper03_bijection_projection_maps_DEs_environment_coordinate}. To prove~\eqref{Paper03_derivatives_slow_env_wrt_x_0}, we first observe that by applying the chain rule to~\eqref{Paper03_bijection_projection_maps_DEs}, we have
    \begin{equation} \label{Paper03_easy_identitfication_derivatives_slow_environment}
    \begin{aligned}
        \frac{\partial \Phi^{\textrm{sl,env}}_{0}}{\partial x_0}(\boldsymbol{x},\xi) = \frac{\partial \Phi_0}{\partial \rho_0}(\boldsymbol{\rho}(\boldsymbol{x},\xi)).
    \end{aligned}
    \end{equation}
    Hence, the first identity in~\eqref{Paper03_derivatives_slow_env_wrt_x_0} follows from applying 
Lemma~\ref{Paper03_lemma_first_order_derivatives_projection_map} 
to~\eqref{Paper03_easy_identitfication_derivatives_slow_environment}, while the second %identity in~\eqref{Paper03_derivatives_slow_env_wrt_x_0} 
follows from differentiating both sides of~\eqref{Paper03_easy_identitfication_derivatives_slow_environment} and applying the chain rule.

    It remains to establish~\eqref{Paper03_derivatives_slow_env_wrt_x_0_and_environment}. By differentiating both sides 
of~\eqref{Paper03_bijection_projection_maps_DEs} and applying the chain rule, we have
    \begin{equation}
    \label{Paper03_intermediate_step_coomputing_derivatives_wrt_slow_environment}
    \frac{\partial \Phi^{\textrm{sl,env}}_{0}}{\partial \xi}(\boldsymbol{x},\xi) = \Phi_0(\boldsymbol{\rho}(\boldsymbol{x}, \xi)) - \frac{1}{\xi} \sum_{i = 0}^{K} x_i \frac{\partial \Phi_0}{\partial \rho_i}(\boldsymbol{\rho}(\boldsymbol{x}, \xi)).
    \end{equation}
    Using the fact that on $\Gamma^{\; \textrm{sl,env}}$, 
\[\Phi_0(\boldsymbol{\rho}(\boldsymbol{x},\xi)) = \frac{x_0}{\xi} = \frac{x_i}{\xi} \qquad \forall \, i \in \llbracket K \rrbracket,
\]
applying Lemma~\ref{Paper03_lemma_first_order_derivatives_projection_map} to~\eqref{Paper03_intermediate_step_coomputing_derivatives_wrt_slow_environment}, and then observing that by~\eqref{Paper03_mean_age_germination}, $B = \sum_{i=0}^{K} \sum_{j = i}^{K} b_i$, we conclude that
    \begin{equation}
    \label{Paper03_intermediate_step_coomputing_derivatives_wrt_slow_environment_i}
    \frac{\partial \Phi^{\textrm{sl,env}}_0}{\partial \xi}(\boldsymbol{x}, \xi) = 0.
    \end{equation}
    Differentiating both sides of~\eqref{Paper03_intermediate_step_coomputing_derivatives_wrt_slow_environment} with respect to $\xi$, and then using that $\boldsymbol{\rho}(\boldsymbol{x}, \xi) \defeq \boldsymbol{x}/\xi$ and the chain rule, we conclude that
    \begin{equation}
    \label{Paper03_intermediate_step_coomputing_derivatives_wrt_slow_environment_ii}
    \begin{aligned}
     & \frac{\partial^2 \Phi^{\textrm{sl,env}}_{0}}{\partial \xi^2}(\boldsymbol{x},\xi) \\ & \quad  = - \frac{1}{\xi^2} \sum_{i = 0}^{K} x_i \frac{\partial \Phi_0}{\partial \rho_i}(\boldsymbol{\rho}(\boldsymbol{x}, \xi)) + \frac{1}{\xi^2} \sum_{i = 0}^{K} x_i \frac{\partial \Phi_0}{\partial \rho_i}(\boldsymbol{\rho}(\boldsymbol{x}, \xi)) + \frac{1}{\xi^2} \sum_{i = 0}^{K} \sum_{j = 0}^{K} x_i x_j \frac{\partial^2 \Phi_0}{\partial \rho_i \partial \rho_j}(\boldsymbol{\rho}(\boldsymbol{x}, \xi)) \\ & \quad = \frac{1}{\xi^2} \sum_{i = 0}^{K} \sum_{j = 0}^{K} x_i x_j \frac{\partial^2 \Phi_0}{\partial \rho_i \partial \rho_j}(\boldsymbol{\rho}(\boldsymbol{x}, \xi)).
    \end{aligned}
    \end{equation}
    To compute the right-hand side of this equation, %of~\eqref{Paper03_intermediate_step_coomputing_derivatives_wrt_slow_environment_ii},
we first observe that on $\Gamma^{\; \textrm{sl,env}}$, $x_0 = \cdots = x_K$, and therefore~\eqref{Paper03_intermediate_step_coomputing_derivatives_wrt_slow_environment_ii} yields
    \begin{equation} \label{Paper03_intermediate_step_coomputing_derivatives_wrt_slow_environment_iii}
        \frac{\partial^2 \Phi^{\textrm{sl,env}}_{0}}{\partial \xi^2}(\boldsymbol{x},\xi) =  \left(\frac{x_0}{\xi}\right)^2\sum_{i = 0}^{K} \sum_{j = 0}^{K} \frac{\partial^2 \Phi_0}{\partial \rho_i \partial \rho_j}(\boldsymbol{\rho}(\boldsymbol{x}, \xi)).
    \end{equation}
    Applying~\eqref{Paper03_real_formula_second_order_derivative},~\eqref{Paper03_real_formula_second_order_derivative_0_i} and~\eqref{Paper03_real_formula_second_order_derivative_i_j} to~\eqref{Paper03_intermediate_step_coomputing_derivatives_wrt_slow_environment_iii}, we have 
    \begin{equation}
    \label{Paper03_intermediate_step_coomputing_derivatives_wrt_slow_environment_iv}
        \frac{\partial^2 \Phi^{\textrm{sl,env}}_{0}}{\partial \xi^2}(\boldsymbol{x},\xi) = - \left(\frac{x_0}{\xi}\right)^2 \sum_{i = 0}^{K} \sum_{j= 0}^{K} \theta^{\Delta}_{ij},
    \end{equation}
    where $\boldsymbol{\Theta}^{\Delta} = \left(\theta^{\Delta}_{ij}\right)_{i,j \in \llbracket K \rrbracket_0}$ is defined after~\eqref{Paper03_difference_matrix_between_hessians}. By combining~\eqref{Paper03_null_space_theta},~\eqref{Paper03_right_eigenvector_Jacobian} and~\eqref{Paper03_intermediate_step_coomputing_derivatives_wrt_slow_environment_iv}, we conclude that
    \begin{equation}
    \label{Paper03_intermediate_step_coomputing_derivatives_wrt_slow_environment_v}
        \frac{\partial^2 \Phi^{\textrm{sl,env}}_{0}}{\partial \xi^2}(\boldsymbol{x},\xi) = 0.
    \end{equation}
    It remains to establish the last identity in~\eqref{Paper03_derivatives_slow_env_wrt_x_0_and_environment}. By differentiating both sides of~\eqref{Paper03_intermediate_step_coomputing_derivatives_wrt_slow_environment} with respect to $x_0$, and then using the definition of $\boldsymbol{\rho}$
% given before Lemma~\ref{Paper03_derivatives_flow_manifold_moran_constant_env}
and the chain rule, we have
    \begin{equation}
    \label{Paper03_intermediate_step_coomputing_derivatives_wrt_slow_environment_vi}
        \frac{\partial^{2} \Phi^{\textrm{sl,env}}_{0}}{\partial \xi \partial x_{0}} (\boldsymbol{x},\xi) = - \frac{1}{\xi^2}\sum_{i = 0}^{K} x_i \frac{\partial ^2 \Phi_0}{\partial \rho_0 \rho_i}(\boldsymbol{\rho}(\boldsymbol{x}, \xi)).  
    \end{equation}
    Substituting~\eqref{Paper03_real_formula_second_order_derivative} 
and~\eqref{Paper03_real_formula_second_order_derivative_0_i}, the fact 
that $\boldsymbol{\Theta}^{\Delta}\boldsymbol{u} = 0$, and the definition 
of $\boldsymbol{u}$ from~\eqref{Paper03_right_eigenvector_Jacobian}, we
conclude
from~\eqref{Paper03_intermediate_step_coomputing_derivatives_wrt_slow_environment_vi} that the last identity in~\eqref{Paper03_derivatives_slow_env_wrt_x_0_and_environment} holds, which completes the proof.
\end{proof}

We are now ready to prove our convergence result.

\begin{proof}[Proof of Theorem~\ref{Paper03_weak_convergence_WF_model_slow_env}]
    Our proof consists of an application of 
Theorem~\ref{Paper03_major_tool_katzenberger}. Recall the definition of 
the multidimensional semimartingale $\boldsymbol{\Lambda}^{N}$ given 
in~\eqref{Paper03_semimartingale_discrete_process_WF}, and
identity~\eqref{Paper03_writing_slow_environment_in_terms_Katzenberger}. 
By Theorem~\ref{Paper03_major_tool_katzenberger}, 
Lemma~\ref{Paper03_martingales_WF_model_slow_env}, and 
Corollary~\ref{Paper03_semimartingale_behaviour_WF_model}, the sequence 
of processes  $(\boldsymbol{x}^{N}, \xi^{N}, \boldsymbol{\Lambda}^{N})_{N \in \mathbb{N}}$ is relatively compact in~$\mathscr{D}\left([0, \infty), \mathscr{U}^{\textrm{sl}} \times \mathbb{R}^{K + 2}\right)$. Moreover, Theorem~\ref{Paper03_major_tool_katzenberger} also implies that any 
subsequential limit $(\boldsymbol{x}(t), \xi(t), \boldsymbol{\Lambda}(t))_{t \geq 0}$ is a diffusion with sample paths in $\Gamma^{\; \textrm{sl,env}} \times \mathbb{R}^{K+2}$.  From now on, let $(\boldsymbol{x}(t), \xi(t), \boldsymbol{\Lambda}(t))_{t \geq 0}$ be a subsequential limit. By Assumption~\ref{Paper03_diffusion_limit_slow_environment} on the dynamics of the environment, it is immediate that $\xi$ must be the (unique in law) solution to the SDE
    \begin{equation*}
        d \xi = \alpha(\xi) \, dt + \eta(\xi) \, dW_{\textrm{env}}(t).
    \end{equation*}
    To compute the drift and the noise terms of the SDE describing the dynamics of $\boldsymbol{x}$, we need to characterise the limiting process $\boldsymbol{\Lambda}$ in terms of $(\boldsymbol{x}(t), \xi(t))_{t \geq 0}$. Bearing this aim in mind, we shall characterise the limit points of each of the martingales $M^{N}_{0}$, $M^{N}_{\textrm{aux}}$ and $\widetilde{M}_{\textrm{env}}^{N}$ that appear in the dynamics of the discrete process (see Equations~\eqref{Paper03_martingale_seeds_WF},~\eqref{Paper03_martingale_impact_env_on_seeds_WF} and~\eqref{Paper03_martingale_both_seeds_slow_env_WF}). Starting with $M^{N}_{\textrm{aux}}$, we note that by estimate~\eqref{Paper03_boundedness_M_aux_WF}, we have for all $t \geq 0$,
    \begin{equation*}
        \lim_{N \rightarrow \infty} \mathbb{E}\left[\left[M^{N}_{\textrm{aux}}\right](T)\right] = 0,
    \end{equation*}
    and therefore, in the limit, $M^{N}_{\textrm{aux}}$ vanishes. Now, regarding $M^{N}_{0}$, we note that by Lemma~\ref{Paper03_martingales_WF_model_slow_env}, $M^{N}_{0}$ converges as $N \rightarrow \infty$ to a stochastic process $M_{0}$ which is also a martingale. Moreover, by~\eqref{Paper03_martingale_seeds_WF}, we have that for all $N \in \mathbb{N}$, the process $\left(Q^{N}(t)\right)_{t \geq 0}$ given by, for all $T \geq 0$
    \begin{equation} \label{Paper03_auxiliary_identity_martingale_WF_slow_environment_[]}
    \begin{aligned}
        & Q^{N}(T) \\ & \quad \defeq \left(M^{N}_{0}(T)\right)^{2} - \int_{0}^{T} \frac{\left(\sum_{i = 0}^{K}b_{i}x^{N}_{i}\right)\left(\xi^{N} - (1 - b_{0})x_{0} + \sum_{i = 1}^{K} b_{i}x^{N}_{i}\right) - \sum_{i = 0}^{K}b_{i}x^{N}_{i}}{\left(\xi^{N} - (1 - b_{0})x_{0} + \sum_{i = 1}^{K} b_{i}x^{N}_{i}\right)} \; d(\lfloor Nt \rfloor/N),
    \end{aligned}
    \end{equation}
    is a martingale. Then, by Lemma~\ref{Paper03_martingales_WF_model_slow_env}, we have uniform integrability for the sequence $\left((M^{N}_{0}(T))^{2}\right)_{N \in \mathbb{N}}$, and therefore, taking $N \rightarrow \infty$ in~\eqref{Paper03_auxiliary_identity_martingale_WF_slow_environment_[]} and recalling that on $\Gamma^{\textrm{sl,env}}$, $x_{0} = \ldots = x_{K}$, we conclude that the process $(Q(t))_{t \geq 0}$ given by
    \begin{equation*}
        Q(T) = M_{0}(T)^{2} - \int_{0}^{T} \xi_{s}\left(\frac{x_{0}}{\xi} - \frac{x_{0}^{2}}{\xi^{2}}\right) \, dt,
    \end{equation*}
    is also a martingale. This characterises the quadratic variation of $M_{0}$ and we deduce that there exists a standard Brownian motion $W_{0}$ such that
    \begin{equation} \label{Paper03_BM_representation_WF_noise_slowly_changing_environment}
        M_{0}(T) = \int_{0}^{T} \sqrt{\frac{x_{0}(\xi - x_{0})}{\xi}} \, dW_{0}(t).
    \end{equation}

By identities~\eqref{Paper03_martingale_slow_environment} 
and~\eqref{Paper03_martingale_impact_env_on_seeds_WF}, we can write 
$\widetilde{M}^{N}_{\textrm{env}}$ as a stochastic integral 
    \begin{equation*}
        \widetilde{M}^{N}_{\textrm{env}}(T) = \int_{0}^{T} \frac{\sum_{i = 0}^{K}b_{i}x^{N}_{i}}{\left((\xi^{N} - x^N_0) + \sum_{i = 0}^{K} b_{i}x^{N}_{i}\right)} dM^{N}_{\textrm{env}}(t).
    \end{equation*}
Using Lemma~\ref{Paper03_martingales_WF_model_slow_env}, the convergence 
result stated in Theorem~7.2 in~\cite{katzenberger1991solutions}, and 
the fact that on $\Gamma$, $x_{0} = \ldots = x_{K}$, we deduce that in the 
limit, $\widetilde{M}^{N}_{\textrm{env}}$ converges to a martingale 
$\widetilde{M}_{\textrm{env}}$ that can be written as 
%the stochastic integral with respect to the noise term regarding the fluctuations in the environment, i.e.
    \begin{equation*}
        \widetilde{M}_{\textrm{env}}(T) = \int_{0}^{T} \frac{x_{0}}{\xi} \, d\left(\int_{0}^{t} \eta(\xi) \, dW_{\textrm{env}}(t')\right) = \int_{0}^{T}  \frac{x_{0}}{\xi_{s}} \eta(\xi) \, dW_{\textrm{env}}(t).
    \end{equation*}
    We must now establish that the Brownian motions $W_{0}$ and 
$W_{\textrm{env}}$ are independent. By Knight's theorem, it is enough to 
establish that in the limit, the covariation process satisfies almost surely, 
for any $T \geq 0$,
    \begin{equation} \label{Paper03_necessary_sufficient_condition_independence}
       \left[M_0,M_{\textrm{env}}\right] \equiv 0.
    \end{equation}
    Let $\varepsilon^N$ be the process defined in~\eqref{Paper03_correction_environment_martingale}. By the definition of $M^N_0$ and~$M^N_{\textrm{env}}$ in~\eqref{Paper03_martingale_seeds_WF} and in~\eqref{Paper03_martingale_slow_environment}, respectively, we have for every $N \in \mathbb{N}$ and $T \geq 0$,
    \begin{equation} \label{Paper03_independence_noise_sources_slow_environment_i}
    \begin{aligned}
        & [M^N_0, M^N_{\textrm{env}} + \varepsilon^N](T) \\ & \quad = \frac{1}{\sqrt{N}}\sum_{t = 0}^{\lfloor NT \rfloor -1} \theta^{N,t}(\boldsymbol{x}^N(t/N), \xi^N(t/N)) \Bigg(\xi^{N}\left(\frac{t + 1}{N}\right) - \xi^{N}\left(\frac{t}{N}\right) - \frac{1}{N}{\alpha\left(\xi^{N}\left(\frac{t}{N}\right)\right)} \\ & \quad \quad \quad  \quad \quad \quad   \quad \quad \quad  \quad \quad \quad  \quad \quad \quad  \quad \quad \quad \quad \quad \quad - w^N_1\left(\xi^N\left(\frac{t}{N}\right)\right)\Bigg).
    \end{aligned}
    \end{equation}
    Since for any $t \geq 0$, $$\theta^{N,t}(\boldsymbol{x}^N(t/N), \xi^N(t/N)) \; \textrm{and} \; \Bigg(\xi^{N}\left(\frac{t + 1}{N}\right) - \xi^{N}\left(\frac{t}{N}\right) - \frac{1}{N}{\alpha\left(\xi^{N}\left(\frac{t}{N}\right)\right)} - w^N_1\left(\xi^N\left(\frac{t}{N}\right)\right)\Bigg)$$ are conditionally independent given $\mathcal{F}^N_t$, it follows that $\left(\left[M^N_0, M^N_{\textrm{env}} + \varepsilon^N\right](t)\right)_{t \geq 0}$ is a martingale, for every $N \in \mathbb{N}$. By Lemma~\ref{Paper03_martingales_WF_model_slow_env}, both $(M^N_0)_{N \in \mathbb{N}}$ and $(M^N_{\textrm{env}} + \varepsilon^N)_{N \in \mathbb{N}}$ satisfy Assumption~\ref{Paper03_assumption_katzenberger_stochastic_process}, and therefore Proposition~4.3 of~\cite{katzenberger1991solutions} implies that the sequence of càdlàg processes $\Big(\left(\left[M^N_0, M^N_{\textrm{env}} + \varepsilon^N\right](t)\right)_{t \geq 0}\Big)_{N \in \mathbb{N}}$ 
converges weakly in $\mathscr{D}([0, \infty), \mathbb{R})$ to $\Big([M_0, M_{\textrm{env}}](t)\Big)_{t \geq 0}$ as $N \rightarrow 0$. By the uniform integrability of the sequences $\left(\left[M^{N}_{0}\right](T)\right)_{N \in \mathbb{N}}$ and $\left(\left[M^{N}_{\textrm{env}} + \varepsilon^{N}\right](T)\right)_{N \in \mathbb{N}}$ stated for all $T \geq 0$ in Lemma~\ref{Paper03_martingales_WF_model_slow_env}, we conclude that $\Big([M_0, M_{\textrm{env}}](t)\Big)_{t \geq 0}$ is also a martingale. Since $M_0$ and $M_{\textrm{env}}$ are continuous, $\Big([M_0, M_{\textrm{env}}](t)\Big)_{t \geq 0}$ is a continuous finite variation process which is also a martingale, i.e.~\eqref{Paper03_necessary_sufficient_condition_independence} holds. This concludes the proof that $W_0$ and $W_{\textrm{env}}$ are independent Brownian motions. Therefore, from~\eqref{Paper03_semimartingale_discrete_process_WF}, the process $\Lambda_{0}$ is given by, for all $T \geq 0$,
    \begin{equation}
    \begin{aligned}
        \Lambda_{0}(T) & = M_{0}(T) + \widetilde{M}_{\textrm{env}}(T) + \int_{0}^{T} \frac{\alpha(\xi)x_{0}}{\xi} \, dt \\  & =  \int_{0}^{T} \sqrt{\frac{x_{0}(\xi - x_{0})}{\xi_{s}}} \, dW_{0}(t) + \int_{0}^{T}  \frac{\eta(\xi_{s})x_{0}}{\xi_{s}} \, dW_{\textrm{env}}({\tau}) + \int_{0}^{T} \frac{\alpha(\xi)x_{0}}{\xi} \, dt.
    \end{aligned}
    \end{equation}
    Finally, by Theorem~\ref{Paper03_major_tool_katzenberger}, we conclude that the limiting process $(x_0(t))_{t \geq 0}$ satisfies the SDE
    \begin{equation} \label{Paper03_final_step_slow_environment}
    \begin{aligned}
        & x_{0}(T) - x_0(0)\\ & \quad = \int_{0}^{T} \frac{\partial \Phi_{0}^{\textrm{sl,env}}}{\partial x_{0}} \sqrt{\frac{x_{0}(\xi - x_{0})}{\xi}} \, dW_{0}(t) + \int_{0}^{T}  \frac{\partial \Phi_{0}^\textrm{sl,env}}{\partial x_{0}} \frac{\eta(\xi)x_{0} }{\xi} \, dW_{\textrm{env}}(t) + \int_{0}^{T}  \frac{\partial \Phi_{0}^\textrm{sl,env}}{\partial x_{0}}\frac{\alpha(\xi)x_{0}}{\xi} \, dt \\ & \quad \quad + \frac{1}{2}\int_{0}^{T} \frac{\partial^{2} \Phi_{0}^\textrm{sl,env}}{\partial x_{0}^{2}} \frac{x_{0}(\xi - x_{0})}{\xi} \, dt + \frac{1}{2} \int_{0}^{t} \frac{\partial^{2} \Phi_{0}^\textrm{sl,env}}{\partial x_{0}^{2}} \frac{\eta^{2}(\xi)x^{2}_{0} }{\xi^{2}} \, dt + \int_{0}^{T} \frac{\partial^{2} \Phi_{0}^\textrm{sl,env}}{\partial x_{0} \partial \xi} \frac{\eta^{2}(\xi)x_{0} }{\xi} \, dt.
    \end{aligned}
    \end{equation}
   Applying Lemma~\ref{Paper03_derivatives_flow_manifold_moran_constant_env} to~\eqref{Paper03_final_step_slow_environment} completes the proof.
\end{proof}

We now proceed to the proof of 
Proposition~\ref{Paper03_interpretation_impact_dormancy_slow_fluctuations_population_size}, which gives a qualitative description of the impact of fluctuations in the population size on the evolution of the dormancy trait.

\begin{proof}[Proof of Proposition~\ref{Paper03_interpretation_impact_dormancy_slow_fluctuations_population_size}]
    We will divide the proof into steps corresponding to each of the assertions (i) - (iv).

    \medskip

    \noindent \underline{Proof of assertion~(i):}

    \medskip

    The fact that $g(\boldsymbol{b},0,\xi) \equiv 0$ follows directly from~\eqref{Paper03_integrand_understand_impact_slow_fluctuations_environment}, while the fact that $g(\boldsymbol{b},1,\xi) \equiv 0$ follows from applying Proposition~\ref{Paper03_bound_drift_general_K_proposition}(iii) to~\eqref{Paper03_integrand_understand_impact_slow_fluctuations_environment}.

    \medskip

    \noindent \underline{Proof of assertion~(ii):}

    \medskip

    Observe that for every $K \in \mathbb{N}$, all $\boldsymbol{b} \in \mathbb{S}_K^*$ and all $0 < \xi_{\min} < \xi_{\max}$, we have for all $\xi \in [\xi_{\min},\xi_{\max}]$
    \begin{equation*}
    \begin{aligned}
        \lim_{\rho_0 \rightarrow 0} \frac{B(1- \rho_{0})}{(B(1 - \rho_{0})+1)\xi}+\frac{\rho_{0}}{2}\left(\frac{\partial^{2} \Phi_{0}}{\partial \rho_{0}^{2}}\left(\boldsymbol{\rho}\right)  - \frac{2B}{(B(1 - \rho_{0})+1)}\right) = \frac{B}{(B+1)\xi} \geq \frac{B}{(B+1)\xi_{\max}} > 0.
    \end{aligned}
    \end{equation*}
    Hence, the result follows by the continuity of the map $g: \mathbb{S}^*_K \times [0,1] \times [\xi_{\min}, \xi_{\max}] \rightarrow \mathbb{R}$ defined in~\eqref{Paper03_integrand_understand_impact_slow_fluctuations_environment}.

    \medskip

    \noindent \underline{Proof of assertion~(iii):}

    \medskip

    By applying Proposition~\ref{Paper03_bound_drift_general_K_proposition}(ii) to~\eqref{Paper03_integrand_understand_impact_slow_fluctuations_environment}, we conclude that
    \begin{equation} \label{Paper03_intermediate_estimate_impact_slow_environmental_fluctuations_dormancy_i}
    \begin{aligned}
        g(\boldsymbol{b}, \rho_0,\xi) & \leq \rho_0\left[\frac{B(1- \rho_{0})}{(B(1 - \rho_{0})+1)\xi}+\frac{\rho_{0}}{2}\left(\frac{B(B(1-\rho_0) + 2)}{(B(1-\rho_0) + 1)^3}  - \frac{2B}{(B(1 - \rho_{0})+1)}\right)\right] \\ & = \frac{B\rho_0(1-\rho_0)}{(B(1 - \rho_{0})+1)^3\xi}\left[B^2(1 - \rho_0)(1 - \rho_0(1 + \xi)) + B \left(2 - \rho_0\left(2 + \frac{3\xi}{2}\right)\right) + 1\right].
    \end{aligned}
    \end{equation}
    Let
    \begin{equation} \label{Paper03_intermediate_estimate_impact_slow_environmental_fluctuations_dormancy_ii}
        \rho_c \defeq \frac{4 + \xi_{\min}}{4 + 3\xi_{\min}}.
    \end{equation}
    To establish assertion~(iii), we will first prove that for $B$ sufficiently large, the term on the right-hand side of~\eqref{Paper03_intermediate_estimate_impact_slow_environmental_fluctuations_dormancy_i} is negative for $\rho_0 \in (\rho_c,1)$. Since $\xi_{\min} > 0$, $\rho_c > 1/(1 + \xi_{\min}) \geq 1/(1 + \xi)$, for any $\xi \in [\xi_{\min}, \xi_{\max}]$, and therefore for $\rho_0 \in (\rho_c,1)$, we have
    \begin{equation} \label{Paper03_intermediate_estimate_impact_slow_environmental_fluctuations_dormancy_iii}
        B^2(1-\rho_0)(1 - \rho_0(1 + \xi)) < 0.
    \end{equation}
   Moreover, for $\rho_0 \in (\rho_c,1)$ and $\xi \in [\xi_{\min}, \xi_{\max}]$, we have
   \begin{equation} \label{Paper03_intermediate_estimate_impact_slow_environmental_fluctuations_dormancy_iv}
       2 - \rho_0\left(2 + \frac{3\xi}{2}\right) \leq - \frac{\xi_{\min}}{(4 + 3\xi_{\min})}\left(2 + \frac{3\xi}{2}\right) \leq - \frac{\xi_{\min}}{(4 + 3\xi_{\min})}\left(2 + \frac{3\xi_{\min}}{2}\right) = - \frac{\xi_{\min}}{2}.
   \end{equation}
   Let $B_c = 2/\xi_{\min}$. Then for $B > B_c$,~\eqref{Paper03_intermediate_estimate_impact_slow_environmental_fluctuations_dormancy_iv} implies that for $\rho \in (\rho_c,1)$, we have
   \begin{equation} \label{Paper03_intermediate_estimate_impact_slow_environmental_fluctuations_dormancy_v}
       B\left(2 - \rho_0\left(2 + \frac{3\xi}{2}\right)\right) + 1 < 0.
   \end{equation}
   Hence, by applying~\eqref{Paper03_intermediate_estimate_impact_slow_environmental_fluctuations_dormancy_iii} and~\eqref{Paper03_intermediate_estimate_impact_slow_environmental_fluctuations_dormancy_v} to~\eqref{Paper03_intermediate_estimate_impact_slow_environmental_fluctuations_dormancy_i}, we conclude that if $B > B_c$ and $\rho_0 \in (\rho_c,1)$, then $g(\boldsymbol{b}, \rho_0, \xi) < 0$ for all $\xi \in [\xi_{\min}, \xi_{\max}]$. Since the map $\xi \in (0, \infty) \mapsto B_c = 2/\xi$ is decreasing and satisfies the limit in~\eqref{Paper03_limit_population_size_impact_beneficial_events_dormancy_fluctuations}, the proof is complete.
\end{proof}

\subsection{Diffusion in fast-changing environment} \label{Paper03_subsection_diffusion_proofs_fast_environment}

In this section, we prove 
Theorem~\ref{Paper03_thm_convergence_diffusion_fast_changing_environment}. 
Our arguments will be closely related to the ones used in 
Sections~\ref{Paper03_subsection_diffusion_WF_model_constant_environment} 
and~\ref{Paper03_subsection_diffusion_WF_model_slow_environment}. As in the 
proof of Theorem~\ref{Paper03_weak_convergence_WF_model_slow_env}, we will 
construct Bernoulli-simile random variables that will 
record whether or not each individual in a new generation is a mutant. 
In the case of a fast changing environment, the outcome probability depends on 
the genetic composition of the previous $K$ generations, the history of the 
environment in the previous $K$ generations, and on the current environment. Formally, for each $N \in \mathbb{N}$, let $\left\{ \zeta^{N, \tau}_{j} \, \vert \, \tau \in \mathbb{N}_{0}, \, j \in [N]  \right\}$ be a set of i.i.d.~random variables taking values in the space of measurable functions from $[0,1]^{K+1} \times [-1,1]^{K+1}$ to $\{0,1\}$ such that for all $(\boldsymbol{x}, \boldsymbol{\upsilon}, \bar{\upsilon}) \in [0,1]^{K+1} \times [-1,1]^{K} \times [-1,1]$,
\begin{equation} \label{Paper03_definition_zeta_N_fast_environment}
    \mathbb{E}\left[\zeta^{N, \tau}_{j}(\boldsymbol{x}, \boldsymbol{\upsilon}, \bar{\upsilon})\right] = \frac{b_0(1 + \bar{\upsilon})x_0 +  \sum_{i = 1}^{K}b_{i}(1 + \upsilon_{i-1})x_{i}}{\left((1 - x_{0}) + b_0x_0\right)(1 + \bar{\upsilon}) + \sum_{i = 1}^{K}b_{i}(1 + \upsilon_{i-1})x_{i}}.
\end{equation}
By~\eqref{Paper03_binomial_WF_fast_fluctuating_environment}, we can construct our Wright-Fisher process in a fast-changing environment in such a way that for any $t \in (0, \infty)$ and every $N \in \mathbb{N}$ such that $\lfloor Nt \rfloor \geq 1$,
\begin{equation} \label{Paper03_writing_proportion_mutants_random_zetas}
    x^N_0(t) = \frac{1}{N} \sum_{j = 1}^{N} \zeta_j^{N,(\lfloor Nt \rfloor -1)}\left(\boldsymbol{x}^N\left(t - \tfrac{1}{N}\right),\boldsymbol{\upsilon}^N\left(t - \tfrac{1}{N}\right), \upsilon_0^N\left(t\right)\right).
\end{equation}

Our next result is an elementary moment computation, which will be used in our description of the moments of increments of $x^N_0$. For every 
$N \in \mathbb{N}$, we write $\pi^N$ for the distribution of $\upsilon^N_0(0)$
which, by Assumption~\ref{Paper03_assumption_fast_changing_environment}, 
takes the form
\begin{equation} \label{Paper03_distribution_scaled_fast_environment}
    \pi^N(\{-s_N\}) =  \pi^N(\{s_N\}) = \frac{1-\pi^N(\{0\})}{2} = p \in [0,1/2],
\end{equation}
where $(s_N)_{N \in \mathbb{N}}$ satisfies~\eqref{Paper03_limit_scaling_fluctuating_fitness_parameter_assumption}.

\begin{lemma} \label{Paper03_elementary_lemma_moments_fast_environment}
    For constants $f_1,f_2,f_3,f_4 \in [0,1]$ with $f_1 \geq f_2$ and $f_3 \geq f_4 > 0$, and for $\bar\upsilon$ a random variable with distribution $\pi^N$, we have
    \begin{equation} \label{Paper03_eq:elementary_lemma_first_moment_fast_environment}
        \mathbb{E}_{\pi^N}\left[\frac{f_1 + \bar\upsilon f_2}{f_3 + \bar\upsilon f_4}\right] = \frac{f_1}{f_3} + \frac{2ps_N^2f_4(f_1f_4 -f_2f_3)}{(f_3^2 -s_N^2f_4^2)f_3},
    \end{equation}
    and
    \begin{equation} \label{Paper03_eq:elementary_lemma_second_moment_fast_environment}
        \mathbb{E}_{\pi^N}\left[\left(\frac{f_1 + \bar\upsilon f_2}{f_3 + \bar\upsilon f_4} - \mathbb{E}_{\pi^N}\left[\frac{f_1 + \bar\upsilon f_2}{f_3 + \bar\upsilon f_4}\right]\right)^2\right] = \frac{2ps_N^2(f_1f_4 - f_2f_3)^2}{f_3^4} + \mathcal{O}(s_N^3).
    \end{equation}
    Moreover, for any $q \geq 3$, we have
     \begin{equation} \label{Paper03_eq:elementary_lemma_higher_moment_fast_environment}
        \mathbb{E}_{\pi^N}\left[\left\vert\frac{f_1 + \bar\upsilon f_2}{f_3 + \bar\upsilon f_4} - \mathbb{E}_{\pi^N}\left[\frac{f_1 + \bar\upsilon f_2}{f_3 + \bar\upsilon f_4}\right]\right\vert^q\right] = \mathcal{O}(s_N^q).
    \end{equation}
    Furthermore,
    \begin{equation} \label{Paper03_eq:elementary_lemma_mixed_moments_fast_environment}
        \mathbb{E}_{\pi^N}\left[\bar\upsilon \left(\frac{f_1+\bar\upsilon f_2}{f_3+\bar\upsilon f_4} - \mathbb{E}_{\pi^N}\left[\frac{f_1+\bar\upsilon f_2}{f_3+\bar\upsilon f_4}\right]\right)\right] = -2ps_N^2\frac{(f_1f_4 - f_2f_3)}{f_3^2} + \mathcal{O}(s_N^3).
    \end{equation}
\end{lemma}

The proof of Lemma~\ref{Paper03_elementary_lemma_moments_fast_environment} is 
elementary, and we postpone it to 
Section~\ref{Paper03_subsection_proofs_auxiliary_computations_fast_env}. 
We now show that 
Lemma~\ref{Paper03_lemma_computation_conditional_increments_fast_environment}, 
which describes the expected increments over a single generation in the 
environment and in the proportion of mutants, is a corollary of 
Lemma~\ref{Paper03_elementary_lemma_moments_fast_environment}.

\begin{proof}[Proof of Lemma~\ref{Paper03_lemma_computation_conditional_increments_fast_environment}]
    Identities~\eqref{Paper03_conditional_increment_x_i_fast_environment} and~\eqref{Paper03_conditional_increment_upsilon_i_fast_environment} follow directly from the definition of the Wright-Fisher model with seed bank in a
fast-changing environment introduced in Section~\ref{subsection:WF_fast_env_model}, and from Assumption~\ref{Paper03_assumption_fast_changing_environment}. It remains to establish~\eqref{Paper03_conditional_increment_x_0_fast_environment}. By~\eqref{Paper03_definition_zeta_N_fast_environment} and~\eqref{Paper03_writing_proportion_mutants_random_zetas}, we have for every $N \in \mathbb{N}$, $t \in [0, +\infty)$,
    \begin{equation} \label{Paper03_first_step_increment_prop_mutant_fast_environment}
    \begin{aligned}
        & \mathbb{E}\Bigg[x^N_0\left(t + \frac{1}{N}\right) - x^N_0(t) \Bigg\vert \, \mathcal{F}^N_t \times \sigma\left(\upsilon^N_0\left(t + \frac{1}{N}\right)\right)\Bigg] \\ & \quad = \frac{b_0(1 + \upsilon^N_0(t + 1/N))x_0^N(t) + \sum_{i = 1}^K b_i(1 + \upsilon^N_{i-1}(t))x^N_i(t)}{\left((1 - x^N_0(t)) + b_0x^N_0(t)\right)(1 + \upsilon^N_0(t + 1/N)) + \sum_{i = 1}^K b_i(1 + \upsilon^N_{i-1}(t))x^N_i(t)} - x_0^N(t).
    \end{aligned}
    \end{equation}
Applying the tower property 
and~\eqref{Paper03_probability_distribution_fast_environment_assumption}, 
and observing that conditional on $\mathcal{F}^N_t$, the random variable 
$\upsilon_0^N(t+1/N)$ 
in~\eqref{Paper03_first_step_increment_prop_mutant_fast_environment} has 
distribution given by~\eqref{Paper03_distribution_scaled_fast_environment}, we 
conclude that~\eqref{Paper03_conditional_increment_x_0_fast_environment} 
follows from~\eqref{Paper03_eq:elementary_lemma_first_moment_fast_environment}.
\end{proof}

It will be convenient to write the dynamics of the Wright-Fisher model in a 
fast-changing environment in such way that we keep track of both
%the two sources of noise in the dynamics of the number of mutants in the population, i.e.~to record 
the noise from the binomial sampling and that from the environment. To study the noise arising from the genetic sampling, we introduce for each $N\in \mathbb{N}$, the set of i.i.d.~random variables $\left\{\theta^{N,\tau}_0 \, \vert \, \tau \in \mathbb{N}_{0} \right\}$ taking values in the space of measurable functions from $[0,1]^{K+1} \times [-1,1]^{K+1}$ to $\mathbb{R}$ given by
\begin{equation} \label{Paper03_definition_theta_aux_rv_construction_martingale_fast_environment_genetic_drift}
    \theta^{N,\tau}_0 (\boldsymbol{x}, \boldsymbol{\upsilon}, \bar\upsilon) \defeq \frac{1}{\sqrt{N}} \sum_{j = 1}^{N} \left(\zeta^{N,\tau}_{j}(\boldsymbol{x}, \boldsymbol{\upsilon}, \bar\upsilon) - \frac{b_0(1 + \bar\upsilon)x_0 +  \sum_{i = 1}^{K}b_{i}(1 + \upsilon_{i-1})x_{i}}{\left((1 - x_0) +b_0x_0\right)(1 + \bar\upsilon) + \sum_{i = 1}^{K}b_{i}(1 + \upsilon_{i-1})x_{i}}\right).
\end{equation}
To study the noise term in the dynamics of mutant individuals that arises from the environmental fluctuations, we introduce for each $N \in \mathbb{N}$ and $(\boldsymbol{x}, \boldsymbol{\upsilon}, \bar\upsilon) \in [0,1]^{K+1} \times (-1,1)^{K+1}$, the term
\begin{equation} \label{Paper03_noise_discrete_WF_fast_environment_source_env_impact_mut}
\begin{aligned}
\theta^{N,\tau}_{\textrm{env}} (\boldsymbol{x}, \boldsymbol{\upsilon}, \bar\upsilon) & \defeq \frac{b_0(1 + \bar\upsilon)x_0 + \sum_{i = 1}^K b_i(1 + \upsilon_{i-1})x_i}{\left((1 - x_0) + b_0x_0\right)(1 + \bar{\upsilon}) + \sum_{i = 1}^K b_i(1 + \upsilon_{i-1})x_i} \\ & \quad \quad - \frac{2ps^2_N(1-x_0)((1-x_0)+b_0x_0)\Big(\sum_{i = 1}^{K} b_i(1 + \upsilon_{i-1})x_i\Big)}{\left(h(\boldsymbol{x}, \boldsymbol{\upsilon})^2 - s^2_N(1-(1-b_0)x_0)^2\right)h(\boldsymbol{x}, \boldsymbol{\upsilon})} \\ & \quad \quad - \frac{b_0x_0 + \sum_{i = 1}^K b_i(1 + \upsilon_{i-1})x_i}{(1 - x_0) + b_0x_0 + \sum_{i = 1}^K b_i(1 + \upsilon_{i-1})x_i},
\end{aligned}
\end{equation}
where the function $h: [0,1]^{K+1} \times (-1,1)^K \rightarrow \mathbb{R}$ is defined in~\eqref{Paper03_simplified_notation_fast_environment_drift}.

\begin{remark}
    Observe that, if $\bar\upsilon$ is a random variable with distribution $\pi^N$ defined in~\eqref{Paper03_distribution_scaled_fast_environment} and $(\boldsymbol{x},\boldsymbol{\upsilon}) \in [0,1]^{K+1} \times [-1,1]^K$, then, by applying identity~\eqref{Paper03_eq:elementary_lemma_first_moment_fast_environment} with
    \begin{equation} \label{Paper03_identitfication_terms_predictable_process_fast_environment}
    \begin{aligned}
        f_1 & = b_0x_0 + \sum_{i=1}^K b_i(1 + \upsilon_{i-1})x_i, \\ f_2 & = b_0x_0, \\ f_3 & = (1-x_0) + b_0x_0 + \sum_{i =1}^{K} b_i(1 + \upsilon_{i-1})x_i, \\ f_4 & = b_0x_0 + \sum_{i =1}^{K} b_i(1 + \upsilon_{i-1})x_i,
    \end{aligned}
    \end{equation}
    and by rearranging terms, we conclude that
    \begin{equation} \label{Paper03_proof_fluctuations_coming_from_fast_environment_martingale}
        \mathbb{E}_{\pi^N}\left[\theta^{N,\tau}_{\textrm{env}} (\boldsymbol{x}, \boldsymbol{\upsilon}, \bar\upsilon)\right] = 0.
    \end{equation}
\end{remark}

We now rewrite the dynamics of $(x^N_0(t))_{t \geq 0}$ in terms of semimartingales and a forcing field. Recall the definition of $F^{\textrm{fast}}_0: [0,1]^{K+1} \times (-1,1)^{K} \rightarrow \mathbb{R}$ in~\eqref{Paper03_flow_definition_WF_model_fast_changing_environment}. Using~\eqref{Paper03_writing_proportion_mutants_random_zetas},~\eqref{Paper03_definition_theta_aux_rv_construction_martingale_fast_environment_genetic_drift} and~\eqref{Paper03_noise_discrete_WF_fast_environment_source_env_impact_mut}, we have that for any $T \in [0, \infty)$ and every $N \in \mathbb{N}$,
\begin{equation} \label{Paper03_rewriting_dynamics_fast_environment_xx}
\begin{aligned}
    x^N_0(T) & = x^N_0(0) + M^N_0(T) + M^N_{\textrm{env}}(T) + \int_0^{T} F^{\textrm{fast}}_0(\boldsymbol{x}^N, \boldsymbol{\upsilon}^N) \, d\lfloor Nt \rfloor \\ & \quad \; + \int_{0}^{T} 2p(Ns^2_N) \left(\frac{(1-x^N_0)\left((1 - x^N_0) + b_0x^N_0\right)\left(\sum_{i=1}^K b_i(1+\upsilon^N_{i-1})x^N_i\right)}{(1 - x^N_0) + b_0x^N_0 + \sum_{i=1}^K b_i(1+\upsilon^N_{i-1})x^N_i}\right) \, d(\lfloor Nt \rfloor/N),
\end{aligned}
\end{equation}
where $(M^N_0(t))_{t \geq 0}$ and $(M^N_{\textrm{env}}(t))_{t \geq 0}$ are càdlàg processes given by, for all $T \geq 0$,
\begin{equation} \label{Paper03_discrete_martingales_noise}
\begin{aligned}
    M^N_0(T) & \defeq  \frac{1}{\sqrt{N}}\sum_{t = 0}^{\lfloor NT \rfloor - 1} \theta^{N,t}_0\left(\boldsymbol{x}^N\left(\frac{t}{N}\right),\boldsymbol{\upsilon}^N\left(\frac{t}{N}\right),\upsilon^N_0\left(\frac{t+1}{N}\right)\right), \\ M^N_{\textrm{env}}(T) & \defeq \sum_{t = 0}^{\lfloor NT \rfloor - 1} \theta^{N,t}_\textrm{env}\left(\boldsymbol{x}^N\left(\frac{t}{N}\right),\boldsymbol{\upsilon}^N\left(\frac{t}{N}\right),\upsilon^N_0\left(\frac{t+1}{N}\right)\right).
\end{aligned}
\end{equation}

Our next result concerns the processes $(M^N_0(t))_{t \geq 0}$ and $(M^N_{\textrm{env}}(t))_{t \geq 0}$.

\begin{lemma}
\label{Paper03_martingales_WF_model_fast_env}
     For every $N \in \mathbb{N}$, the processes $M^{N}_{0}$ and ${M}^{N}_{\textrm{env}}$ are square integrable martingales. Moreover, for all $T \in [0, \infty)$, the sequences of random variables $\left(\left[M^{N}_{0}\right](T)\right)_{N \in \mathbb{N}}$, $\left(\left[M^{N}_{\textrm{env}}\right](T)\right)_{N \in \mathbb{N}}$, $\Big((M^{N}_{0}(T))^2\Big)_{N \in \mathbb{N}}$ and $\Big((M^{N}_{env}(T))^2\Big)_{N \in \mathbb{N}}$ are uniformly integrable. Furthermore, for every~$N \in \mathbb{N}$ and any~$T \geq 0$,
     \begin{equation} \label{Paper03_orthogonality_noises_fast_environment}
         \left\langle M^N_0, M^N_\textrm{env} \right\rangle (T) = 0.
     \end{equation}
\end{lemma}

\begin{proof}
    We shall study the processes $M^{N}_{0}$ and ${M}^{N}_{\textrm{env}}$ separately. 
    
    \medskip

    \noindent \underline{Step~$(i)$: Characterisation of $M^N_0$:}

    \medskip
    
    By the tower property of conditional expectations, and then by~\eqref{Paper03_discrete_martingales_noise},~\eqref{Paper03_definition_theta_aux_rv_construction_martingale_fast_environment_genetic_drift}, and~\eqref{Paper03_definition_zeta_N_fast_environment}, we have that for all $t \in [0, \infty)$,
    \begin{equation}
    \begin{aligned}
        & \mathbb{E}\left[M^N_0(t + 1/N) - M^N(t) \, \vert \mathcal{F}^N_t\right] \\ & \quad = \mathbb{E}\left[\mathbb{E}\left[M^N_0(t + 1/N) - M^N(t) \, \Big\vert \mathcal{F}^N_t \times \sigma\left(\upsilon^N_0(t + 1/N)\right)\right] \, \Big\vert  \mathcal{F}^N_t \right] \\ & \quad = 0,
    \end{aligned}
    \end{equation}
    and therefore $(M^N_0(t))_{t \geq 0}$ is a martingale. For the uniform integrability, we first observe that~\eqref{Paper03_discrete_martingales_noise},~\eqref{Paper03_definition_theta_aux_rv_construction_martingale_fast_environment_genetic_drift}, and~\eqref{Paper03_definition_zeta_N_fast_environment} imply that for every $N \in \mathbb{N}$, almost surely
    \begin{equation} \label{Paper03_UI_jumps_bound_fast_environment_genetic_drift_i}
    \sup_{t \in [0, \infty)} \; \vert M^N_0(t + 1/N) - M^N_0(t) \vert \leq 1.
    \end{equation}
    Moreover, using the binomial sampling and rearranging terms, for all $t \geq 0$ and $\bar\upsilon \in \{-s_N,0,s_N\}$, we have
    \begin{equation} \label{Paper03_UI_predictable_bracket_bound_fast_environment_genetic_drift_i}
    \begin{aligned}
        & \mathbb{E}\left[\Big(M^N_0(t + 1/N) - M^N_0(t)\Big)^2 \, \Big\vert \, \sigma\left(\mathcal{F}^N_t, \left\{\upsilon^N_0\left(t + \frac{1}{N}\right) = \bar\upsilon\right\} \right)\right] \\ & \quad = \frac{1}{N}\frac{(1+\bar\upsilon)(1-x^N_0(t))\left(b_0(1+\bar\upsilon)x^N_0(t) + \sum_{i=1}^{K} b_i(1 + \upsilon_{i-1}^N(t))x^N_i(t)\right)}{\left[\Big((1 + x^N_0(t)) + b_0x^N_0(t)\Big)(1 + \bar\upsilon) + \sum_{i = 1}^{K} b_i(1 + \upsilon^N_{i-1}(t))x^N_i(t)\right]^2} \\ & \quad = \frac{1}{N}\frac{(1-x^N_0(t))\left(b_0x^N_0(t) + \sum_{i=1}^{K} b_i(1 + \upsilon_{i-1}^N(t))x^N_i(t)\right)}{\left[\Big((1 + x^N_0(t)) + b_0x^N_0(t)\Big) + \sum_{i = 1}^{K} b_i(1 + \upsilon^N_{i-1}(t))x^N_i(t)\right]^2} + \frac{\bar\upsilon}{N}\epsilon(\boldsymbol{x}^N, \boldsymbol{\upsilon}^N,\bar\upsilon),
    \end{aligned}
    \end{equation}
    where $\epsilon: [0,1]^{K+1} \times (-1,1)^{K+1} \rightarrow \mathbb{R}$ satisfies
    \begin{equation}
        \sup_{(\boldsymbol{x},\boldsymbol{\upsilon},\bar\upsilon) \in [0,1]^{K+1} \times (-1,1)^{K+1}} \vert \epsilon(\boldsymbol{x},\boldsymbol{\upsilon},\bar\upsilon) \vert < \infty.
    \end{equation}
    Taking the conditional expectation given $\mathcal{F}^N_t$ on both sides of~\eqref{Paper03_UI_predictable_bracket_bound_fast_environment_genetic_drift_i} and then using the tower property and~\eqref{Paper03_limit_scaling_fluctuating_fitness_parameter_assumption}, we conclude that for every $N \in \mathbb{N}$ and any $T \in [0, \infty)$,
    \begin{equation} \label{Paper03_UI_predictable_bracket_bound_fast_environment_genetic_drift_iii}
    \begin{aligned}
        \left\langle M^N_0 \right\rangle(T) & = \frac{1}{N} \sum_{t = 0}^{\lfloor NT \rfloor - 1} \frac{(1 - x^N_0(t/N))\left(b_0x^N_0(t/N) + \sum_{i = 1}^{K}b_{i}(1 + \upsilon^N_{i-1}(t/N))x^N_{i}(t/N)\right)}{\left[\left(1 - x^N_{0}(t/N)\right) + b_0x^N_0(t/N) + \sum_{i = 1}^{K} b_{i} (1 + \upsilon^N_{i-1}(t/N)) x^N_{i}(t/N)\right]^2} \\ & \quad \quad + \mathcal{O}(N^{-1/2}).
    \end{aligned}
    \end{equation}
    In particular, for all $T \geq 0$, we have that almost surely
    \begin{equation} \label{Paper03_UI_predictable_bracket_bound_fast_environment_genetic_drift_iv}
        \limsup_{N \rightarrow \infty} \;   \left\langle M^N_0 \right\rangle(T) \leq \frac{T}{4}.
    \end{equation}
    Using~\eqref{Paper03_UI_jumps_bound_fast_environment_genetic_drift_i} and~\eqref{Paper03_UI_predictable_bracket_bound_fast_environment_genetic_drift_iv}, we can use Lemma~\ref{Paper03_auxiliary_lemma_uniform_integrability_martingales} to conclude that  for all $T \in [0, \infty)$, the sequences of random variables $\left(\left[M^{N}_{0}\right](T)\right)_{N \in \mathbb{N}}$ and $\Big(\left(M^{N}_{0}(T)\right)^2\Big)_{N \in \mathbb{N}}$ are uniformly integrable, as desired.

    \medskip

    \noindent \underline{Step~$(ii)$: Characterisation of $M^N_\textrm{env}$:}

    \medskip

    The fact that $M^N_{\textrm{env}}$ is a martingale follows directly from~\eqref{Paper03_proof_fluctuations_coming_from_fast_environment_martingale}. Moreover, by~\eqref{Paper03_noise_discrete_WF_fast_environment_source_env_impact_mut}, using~\eqref{Paper03_eq:elementary_lemma_second_moment_fast_environment} and~\eqref{Paper03_identitfication_terms_predictable_process_fast_environment} and rearranging terms, we conclude that
for every $N \in \mathbb{N}$, there is a predictable finite variation process $(\tilde{\epsilon}^{N}(t))_{t \geq 0}$ such that
    \begin{equation} \label{Paper03_bound_predictable_bracket+process_fast_environment_env_martingale}
    \begin{aligned}
        \left\langle M^N_{\textrm{env}} \right\rangle (T) & = 2p(Ns_N^2) \int_{0}^{T}  \frac{\Big(1 - x^N_0\Big)^2\Big(\sum_{i = 1}^K b_i(1 + \upsilon^N_{i-1})x^N_i\Big)^2}{\Big[(1-x^N_0) + b_0x^N_0 + \sum_{i = 1}^K b_i(1 + \upsilon^N_{i-1})x^N_i\Big]^2} \, d(\lfloor Nt \rfloor/N) + s_N \tilde{\epsilon}^N(T),
    \end{aligned}
    \end{equation}
    and for every $T > 0$, there exists $C_T^{(1)}>0$ such that for every $N \in \mathbb{N}$, %the following estimate holds 
almost surely
    \begin{equation} \label{Paper03_bound_error_term_predictable_bracket_process_martingale_fast_fluctuations}
        \sup_{t \leq T} \; \vert \tilde{\epsilon}^N(t) \vert \leq C_T^{(1)}.
    \end{equation}
    Moreover, Assumption~\ref{Paper03_assumption_fast_changing_environment} and~\eqref{Paper03_noise_discrete_WF_fast_environment_source_env_impact_mut} imply that there exists $C^{(2)}_{T}>0$ such that almost surely
    \begin{equation} \label{Paper03_bound_jumps_fast_environment_env_martingale}
        \sup_{N \in \mathbb{N}} \; \sup_{t \in [0, \infty)} \; \left\vert M^N_{\textrm{env}} (t + 1/N) - M^N_{\textrm{env}} (t)\right\vert \leq C_T^{(2)}.
    \end{equation}
    By~\eqref{Paper03_bound_predictable_bracket+process_fast_environment_env_martingale},~\eqref{Paper03_bound_error_term_predictable_bracket_process_martingale_fast_fluctuations},~\eqref{Paper03_bound_jumps_fast_environment_env_martingale} and by~\eqref{Paper03_limit_scaling_fluctuating_fitness_parameter_assumption}, we can use Lemma~\ref{Paper03_auxiliary_lemma_uniform_integrability_martingales} to conclude that for all $T \in [0, \infty)$, the sequences of random variables $\left(\left[M^{N}_{\textrm{env}}\right](T)\right)_{N \in \mathbb{N}}$ and $\Big(\left(M^{N}_{env}(T)\right)^2\Big)_{N \in \mathbb{N}}$ are uniformly integrable.

     \medskip

    \noindent \underline{Step~$(iii)$: Orthogonality of~$M^N_0$ and~$M^N_{\textrm{env}}$:}

    \medskip
    
    It remains to establish~\eqref{Paper03_orthogonality_noises_fast_environment}. By~\eqref{Paper03_discrete_martingales_noise}, it will suffice to establish that for every $N \in \mathbb{N}$ and all $t \geq 0$,
    \begin{equation} \label{Paper03_proving_orthogonality_fast_env_step_i}
    \begin{aligned}
        \mathbb{E}\left[ \theta^{N,t}_0\left(\boldsymbol{x}^N\left(\frac{t}{N}\right),\boldsymbol{\upsilon}^N\left(\frac{t}{N}\right),\upsilon^N_0\left(\frac{t+1}{N}\right)\right) \theta^{N,t}_\textrm{env}\left(\boldsymbol{x}^N\left(\frac{t}{N}\right),\boldsymbol{\upsilon}^N\left(\frac{t}{N}\right),\upsilon^N_0\left(\frac{t+1}{N}\right)\right) \Bigg\vert \mathcal{F}^N_{t/N}\right] = 0.
    \end{aligned}
    \end{equation}
    By~\eqref{Paper03_definition_theta_aux_rv_construction_martingale_fast_environment_genetic_drift} and~\eqref{Paper03_noise_discrete_WF_fast_environment_source_env_impact_mut},
    \begin{equation} \label{Paper03_proving_orthogonality_fast_env_step_ii}
    \begin{aligned}
         & \mathbb{E}\Bigg[ \theta^{N,t}_0\left(\boldsymbol{x}^N\left(\frac{t}{N}\right),\boldsymbol{\upsilon}^N\left(\frac{t}{N}\right),\upsilon^N_0\left(\frac{t+1}{N}\right)\right) \theta^{N,t}_\textrm{env}\left(\boldsymbol{x}^N\left(\frac{t}{N}\right),\boldsymbol{\upsilon}^N\left(\frac{t}{N}\right),\upsilon^N_0\left(\frac{t+1}{N}\right)\right) \\ & \quad \quad \quad \quad \quad \quad \quad \quad \quad \quad \quad \quad \quad \quad \quad \quad \quad \quad \quad \quad \quad \quad \quad \quad \quad \quad \Bigg\vert \mathcal{F}^N_{t/N} \times \sigma\left(\upsilon^N_0\left(\frac{t+1}{N}\right)\right)\Bigg] \\ & \quad = \theta^{N,t}_\textrm{env}\left(\boldsymbol{x}^N\left(\frac{t}{N}\right),\boldsymbol{\upsilon}^N\left(\frac{t}{N}\right),\upsilon^N_0\left(\frac{t+1}{N}\right)\right) \\ & \quad \quad \quad \cdot \mathbb{E}\Bigg[ \theta^{N,t}_0\left(\boldsymbol{x}^N\left(\frac{t}{N}\right),\boldsymbol{\upsilon}^N\left(\frac{t}{N}\right),\upsilon^N_0\left(\frac{t+1}{N}\right)\right) \Bigg\vert \mathcal{F}^N_{t/N} \times \sigma\left(\upsilon^N_0\left(\frac{t+1}{N}\right)\right)\Bigg] \\ & \quad = 0,
    \end{aligned}
    \end{equation}
    where the last equality follows from~\eqref{Paper03_definition_zeta_N_fast_environment} and~\eqref{Paper03_definition_theta_aux_rv_construction_martingale_fast_environment_genetic_drift}. Identity~\eqref{Paper03_proving_orthogonality_fast_env_step_i} then follows from~\eqref{Paper03_proving_orthogonality_fast_env_step_ii} and the tower property, which completes the proof.
\end{proof}

We will also need to describe the dynamics of $\boldsymbol{\upsilon}^N$ in terms of a semimartingale and a large drift. Recall the definition of $\boldsymbol{F}^{\, \textrm{fast}}: [0,1]^{K+1} \times [-1,1]^{K} \rightarrow \mathbb{R}^{2K+1}$ in~\eqref{Paper03_flow_definition_WF_model_fast_changing_environment}. For $i \in \llbracket K-1 \rrbracket$, we have
\begin{equation*}
    \upsilon_{i}^N(T) = \int_0^T F^{\, \textrm{fast}}_{\textrm{env},i}(\boldsymbol{x}^N, \boldsymbol{\upsilon}^N) \, d\lfloor Nt \rfloor.
\end{equation*}
Moreover, we have
\begin{equation} \label{Paper03_dyn_fast_env_first_env_coordinate}
    \upsilon^N_0(T) = \widetilde{M}^N_{\textrm{env}}(T) + \int_0^T F^{\, \textrm{fast}}_{\textrm{env},0}(\boldsymbol{x}^N, \boldsymbol{\upsilon}^N) \, d\lfloor Nt \rfloor,
\end{equation}
where
\begin{equation} \label{Paper03_martingale_fluctuations_env_dyn_fast_environment}
    \widetilde{M}^N_{\textrm{env}}(T) = \sum_{t = 0}^{\lfloor NT \rfloor - 1} \upsilon^N_0\left(t + \frac{1}{N}\right).
\end{equation}

Next, we characterise the process $(\widetilde{M}^N_{\textrm{env}}(t))_{t \geq 0}$.

\begin{lemma} \label{Paper03_lem_noise_env_specific_fast_env}
    For every $N \in \mathbb{N}$, $  (\widetilde{M}^N_{\textrm{env}}(t))_{t \geq 0}$ is a square-integrable martingale, and for all $T \in [0, \infty)$, the sequences of random variables $\left(\left[\widetilde{M}^N_{\textrm{env}}\right](T)\right)_{N \in \mathbb{N}}$ and $\Big(\left(\widetilde{M}^N_{\textrm{env}}(T)\right)^2\Big)_{N \in \mathbb{N}}$ are uniformly integrable. Moreover, for every~$N \in \mathbb{N}$ and any~$T \geq 0$,
     \begin{equation} \label{Paper03_orthogonality_noises_fast_environment_second_version}
         \left\langle M^N_0, \widetilde{M}^N_\textrm{env} \right\rangle (T) = 0.
     \end{equation}
     Finally, for every $N \in \mathbb{N}$, there exists a finite variation process $(\epsilon^N(t))_{t \geq 0}$ such that for every $T \in [0, \infty)$, there exists $C_T > 0$ for which for every $N \in \mathbb{N}$, almost surely
     \begin{equation*}
         \sup_{t \leq T} \vert \epsilon^N(t) \vert \leq C_T,
     \end{equation*}
     and for every $N \in \mathbb{N}$ and $T \geq 0$,
     \begin{equation} \label{Paper03_noises_fast_environment_complete_dependence_first_version}
     \begin{aligned}
         & \left\langle M^N_\textrm{env}, \widetilde{M}^N_\textrm{env} \right\rangle (T) \\ & \quad = - {2p(Ns^2_N)} \int_0^{T} \frac{\left(1 - x^N_0\right)\left(\sum_{i = 1}^K b_i(1+\upsilon^N_{i-1})x^N_i\right)}{\Big[(1-x^N_0) + b_0x^N_0 + \sum_{i = 1}^K b_i(1 + \upsilon^N_{i-1})x^N_i\Big]^2} \, d(\lfloor Nt \rfloor/N) + s_N \epsilon^N(T).
    \end{aligned}
     \end{equation}
\end{lemma}

\begin{proof}
   By~\eqref{Paper03_martingale_fluctuations_env_dyn_fast_environment},~\eqref{Paper03_scaling_fast_environment} and Assumption~\ref{Paper03_assumption_fast_changing_environment}, $  (\widetilde{M}^N_{\textrm{env}}(t))_{t \geq 0}$ is a sum of mean zero i.i.d.~random variables, and therefore $  (\widetilde{M}^N_{\textrm{env}}(t))_{t \geq 0}$ is a martingale. By Assumption~\ref{Paper03_assumption_fast_changing_environment} and~\eqref{Paper03_martingale_fluctuations_env_dyn_fast_environment}, we also have
   \begin{equation} \label{Paper03_predictable_process_fast_environment_noise_2}
       \left\langle \widetilde{M}^N_{\textrm{env}} \right\rangle(T) = 2p(Ns^2_N)\frac{\lfloor NT \rfloor}{N}.
   \end{equation}
   Also, by~\eqref{Paper03_limit_scaling_fluctuating_fitness_parameter_assumption} and~\eqref{Paper03_scaling_fast_environment}, we have that almost surely
   \begin{equation} \label{Paper03_bounded_size_jump_noise_2_fast_environment}
       \sup_{N \in \mathbb{N}} \; \sup_{t \leq T} \vert \widetilde{M}^N_{\textrm{env}}(t) - \widetilde{M}^N_{\textrm{env}}(t-) \vert \leq 1.
   \end{equation}
   Using~\eqref{Paper03_limit_scaling_fluctuating_fitness_parameter_assumption},~\eqref{Paper03_predictable_process_fast_environment_noise_2},~\eqref{Paper03_bounded_size_jump_noise_2_fast_environment}, and then applying Lemma~\ref{Paper03_auxiliary_lemma_uniform_integrability_martingales}, we conclude that the sequences of random variables $\left(\left[\widetilde{M}^N_{\textrm{env}}\right](T)\right)_{N \in \mathbb{N}}$ and $\Big(\left(\widetilde{M}^N_{\textrm{env}}(T)\right)^2\Big)_{N \in \mathbb{N}}$ are uniformly integrable.

   It remains to compute the predictable bracket covariation between $\widetilde{M}^N_\textrm{env}$ and the martingales~$M^N_0$ and~$M^{N}_{\textrm{env}}$ introduced in~\eqref{Paper03_discrete_martingales_noise}. Observe that by~\eqref{Paper03_discrete_martingales_noise} and~\eqref{Paper03_martingale_fluctuations_env_dyn_fast_environment},~\eqref{Paper03_orthogonality_noises_fast_environment_second_version} will follow after establishing that for every $N \in \mathbb{N}$ and all $t \geq 0$,
   \begin{equation} \label{Paper03_pre_step_predictable_covariation_step_i_fast_env}
         \mathbb{E}\left[ \upsilon^N_0\left(\frac{t+1}{N}\right) \cdot \theta^{N,t}_0\left(\boldsymbol{x}^N\left(\frac{t}{N}\right),\boldsymbol{\upsilon}^N\left(\frac{t}{N}\right),\upsilon^N_0\left(\frac{t+1}{N}\right)\right) \Bigg\vert \mathcal{F}^N_{t/N}\right] = 0.
   \end{equation}
   To show that~\eqref{Paper03_pre_step_predictable_covariation_step_i_fast_env} holds, we observe that by~\eqref{Paper03_discrete_martingales_noise},~\eqref{Paper03_definition_theta_aux_rv_construction_martingale_fast_environment_genetic_drift} and~\eqref{Paper03_definition_zeta_N_fast_environment}, we have
   \begin{equation*}
    \begin{aligned}
        & \mathbb{E}\left[ \upsilon^N_0\left(\frac{t+1}{N}\right) \cdot \theta^{N,t}_0\left(\boldsymbol{x}^N\left(\frac{t}{N}\right),\boldsymbol{\upsilon}^N\left(\frac{t}{N}\right),\upsilon^N_0\left(\frac{t+1}{N}\right)\right) \Bigg\vert \mathcal{F}^N_{t/N} \times \sigma\left(\upsilon^N_0\left(\frac{t+1}{N}\right)\right)\right] \\ & \quad = \upsilon^N_0\left(\frac{t+1}{N}\right) \mathbb{E}\left[\theta^{N,t}_0\left(\boldsymbol{x}^N\left(\frac{t}{N}\right),\boldsymbol{\upsilon}^N\left(\frac{t}{N}\right),\upsilon^N_0\left(\frac{t+1}{N}\right)\right) \Bigg\vert \mathcal{F}^N_{t/N} \times \sigma\left(\upsilon^N_0\left(\frac{t+1}{N}\right)\right)\right] \\ & \quad = 0,
    \end{aligned}
   \end{equation*}
   and therefore~\eqref{Paper03_pre_step_predictable_covariation_step_i_fast_env} follows from the tower property of conditional expectation.

   It remains to establish~\eqref{Paper03_noises_fast_environment_complete_dependence_first_version}. Observe that by~\eqref{Paper03_noise_discrete_WF_fast_environment_source_env_impact_mut} and by using identity~\eqref{Paper03_eq:elementary_lemma_mixed_moments_fast_environment} from Lemma~\ref{Paper03_elementary_lemma_moments_fast_environment} with~\eqref{Paper03_identitfication_terms_predictable_process_fast_environment}, we have that for every $N \in \mathbb{N}$ and $T \in (0,\infty)$,
   \begin{equation} \label{Paper03_expected_increment_covariation_noises_env_fast_env}
    \begin{aligned}
        & \mathbb{E}\left[ \upsilon^N_0\left(\frac{t+1}{N}\right) \cdot \theta^{N,t}_\textrm{env}\left(\boldsymbol{x}^N\left(\frac{t}{N}\right),\boldsymbol{\upsilon}^N\left(\frac{t}{N}\right),\upsilon^N_0\left(\frac{t+1}{N}\right)\right) \Bigg\vert \mathcal{F}^N_{t/N}\right] \\ & \quad = -2ps^2_N\frac{\left(1 - x^N_0(t/N)\right)\left(\sum_{i = 1}^K b_i(1+\upsilon^N_{i-1}(t/N))x^N_i(t/N)\right)}{\Big[(1-x^N_0(t/N)) + b_0x^N_0(t/N) + \sum_{i = 1}^K b_i(1 + \upsilon^N_{i-1}(t/N)x^N_i(t/N)\Big]^2} + \mathcal{O}(s_N^3).
    \end{aligned}
   \end{equation}
   Identity~\eqref{Paper03_noises_fast_environment_complete_dependence_first_version} then follows from~\eqref{Paper03_expected_increment_covariation_noises_env_fast_env} and~\eqref{Paper03_limit_scaling_fluctuating_fitness_parameter_assumption}, which completes the proof.
\end{proof}

We will also need to characterise the finite-variation process that appears on the right-hand side of~\eqref{Paper03_rewriting_dynamics_fast_environment_xx}. 
To simplify notation, for $N \in \mathbb{N}$ and $T \in [0,\infty)$, let
\begin{equation} \label{Paper03_eq:fv_process_fast_environment}
    A^{N}_0(T) \defeq \int_{0}^{T} 2p(Ns^2_N) \left(\frac{(1-x^N_0)\left((1 - x^N_0) + b_0x^N_0\right)\left(\sum_{i=1}^K b_i(1+\upsilon^N_{i-1})x^N_i\right)}{(1 - x^N_0) + b_0x^N_0 + \sum_{i=1}^K b_i(1+\upsilon^N_{i-1})x^N_i}\right) \, d(\lfloor Nt \rfloor/N).
\end{equation}
The next result is an immediate consequence of Assumption~\ref{Paper03_assumption_fast_changing_environment}. Recall the definition of total variation introduced before Assumption~\ref{Paper03_assumption_katzenberger_stochastic_process}.

\begin{lemma} \label{Paper03_lem_FV_proc_fast_env}
    For any $T \in [0, \infty)$, the following conditions hold:
    \begin{enumerate}[(i)]
        \item The sequence $\left(\mathfrak{T}(A^N_0,T)\right)_{N \in \mathbb{N}}$ is uniformly integrable.
        \item For any $\varepsilon > 0$, we have
        \begin{equation*}
            \lim_{\delta \rightarrow 0^+} \, \limsup_{N \rightarrow \infty} \mathbb{P}\left(\sup_{t \leq T} \left(\mathfrak{T}(A^N_0,t + \delta) - \mathfrak{T}(A^N_0,t)\right) > \varepsilon \right) = 0.
        \end{equation*}
    \end{enumerate}
\end{lemma}

\begin{proof}
    Observe that the integrand on the right-hand side of~\eqref{Paper03_eq:fv_process_fast_environment} is non-negative and therefore for every $N \in \mathbb{N}$ and any $T \in [0, \infty)$, $ \mathfrak{T}(A^N_0,T) = A^N_0(T)$. Then, the lemma follows from the observation that by~\eqref{Paper03_limit_scaling_fluctuating_fitness_parameter_assumption}, 
    \begin{equation*}
        \sup_{N \in \mathbb{N}} \; \sup_{(\boldsymbol{x},\boldsymbol{\upsilon}) \in [0,1]^{K+1} \times [-1,1]^K} 2p(Ns^2_N) \left(\frac{(1-x_0)\left((1 - x_0) + b_0x_0\right)\left(\sum_{i=1}^K b_i(1+\upsilon_{i-1})x_i\right)}{(1 - x_0) + b_0x_0 + \sum_{i=1}^K b_i(1+\upsilon_{i-1})x_i}\right) < \infty.
    \end{equation*}
\end{proof}

In order to apply the Katzenberger's framework to the analysis of the Wright-Fisher model with a fast-changing environment, we will also need to characterise the flow $\boldsymbol{F}^{\, \textrm{fast}}: [0,1]^{K+1} \times (-1,1)^{K} \rightarrow \mathbb{R}^{2K+1}$ defined in~\eqref{Paper03_flow_definition_WF_model_fast_changing_environment} and its associated projection map $\boldsymbol{\Phi}^{\, \textrm{fast}}$ introduced after~\eqref{Paper03_attractor_manifold_fast_environment}. This will be our next result. Recall the definition of $\Gamma^{\, \textrm{fast}}$ in~\eqref{Paper03_attractor_manifold_fast_environment}, the definition of the flow $\boldsymbol{F}: [0,1]^{K+1} \rightarrow \mathbb{R}^{K+1}$ in~\eqref{Paper03_flow_definition_WF_model} and its associated projection map~$\boldsymbol{\Phi}$ in~\eqref{Paper03_definition_projection_map}, and of the set $D(1)$ introduced before Assumption~\ref{Paper03_assumption_manifold}. 
Recall that we use $\boldsymbol{0}$ to denote the $K$-dimensional null vector.

\begin{lemma} \label{Paper03_characterisation_flow_associated_fast_changing_environment}
    For any $(\boldsymbol{x},\boldsymbol{0}) \in \Gamma^{\, \textrm{fast}}$, the Jacobian of $\boldsymbol{F}^{\, \textrm{fast}}$ evaluated at~$(\boldsymbol{x},\boldsymbol{0})$, i.e.~the matrix $\mathcal{J}\boldsymbol{F}^{\, \textrm{fast}}_{(\boldsymbol{x},\boldsymbol{0})}$, has exactly one eigenvalue $0$ and $2K$ eigenvalues in~$D(1)$. Moreover, for any $\boldsymbol{x} = (x, \cdots,x) \in \Gamma$, the following relations hold
    \begin{equation} \label{Paper03_first_order_derivatives_fast_environment}
        \frac{\partial \Phi^{\, \textrm{fast}}_0}{\partial x_0} \left(\boldsymbol{x}, \boldsymbol{0}\right) = \frac{1}{B(1-x) + 1} \quad \textrm{and} \quad \frac{\partial \Phi^{\, \textrm{fast}}_0}{\partial \upsilon_0} \left(\boldsymbol{x}, \boldsymbol{0}\right) = \frac{(1-b_0)x(1-x)}{B(1-x)+1}.
    \end{equation}
    Furthermore,
    \begin{equation} \label{Paper03_second_order_derivatives_fast_environment_easy_part}
         \frac{\partial^2 \Phi^{\, \textrm{fast}}_0}{\partial x_0^2} \left(\boldsymbol{x}, \boldsymbol{0}\right) = \frac{\partial^2 \Phi_0}{\partial x_0^2} (\boldsymbol{x}).
    \end{equation}
\end{lemma}

The proof of Lemma~\ref{Paper03_characterisation_flow_associated_fast_changing_environment} is based on formulae~\eqref{Paper03_formula_Parsons_Rogers_first_order_derivatives} and~\eqref{Paper03_formula_Parsons_Rogers_second_order_derivatives}, and on Theorem~\ref{Paper03_prop_uniqueness_quadratic_term_nonlinear_manifold}. Since we used a similar argument in Section~\ref{Paper03_derivatives_specific_ode} when we studied the properties of the flow arising from the Wright-Fisher model with seed bank in a constant environment, we will postpone the proof of Lemma~\ref{Paper03_characterisation_flow_associated_fast_changing_environment} to Section~\ref{Paper03_sec_comp_lemmas_fast_env}.

We are now ready to establish Theorem~\ref{Paper03_thm_convergence_diffusion_fast_changing_environment}.

\begin{proof}[Proof of Theorem~\ref{Paper03_thm_convergence_diffusion_fast_changing_environment}]
    Our proof consists of an application of 
Theorem~\ref{Paper03_major_tool_katzenberger}. As in the proof of 
Theorem~\ref{Paper03_weak_convergence_WF_model_slow_env}, in order to 
keep our notation consistent with that of 
Theorem~\ref{Paper03_major_tool_katzenberger}, we introduce the 
$(2K+1)$-dimensional process
    \begin{equation*}
        \left(\boldsymbol{\Lambda}^N(t)\right)_{t \geq 0} \defeq \left(M^N_0(t) + M^N_\textrm{env}(t) + A^N_0(t), 0, \cdots, 0, \widetilde{M}^N_{\textrm{env}}(t), 0, \cdots, 0\right)_{t \geq 0},
    \end{equation*}
    where~$M^N_0$,~$M^{N}_{\textrm{env}}$,~$A^N_0$ and~$\widetilde{M}^{N}_{\textrm{env}}$ are defined in~\eqref{Paper03_discrete_martingales_noise},~\eqref{Paper03_eq:fv_process_fast_environment} and~\eqref{Paper03_martingale_fluctuations_env_dyn_fast_environment}. By~\eqref{Paper03_flow_definition_WF_model_fast_changing_environment},~\eqref{Paper03_rewriting_dynamics_fast_environment_xx} and~\eqref{Paper03_dyn_fast_env_first_env_coordinate}, we can write the dynamics of $(\boldsymbol{x}^N(t),\upsilon^N(t))_{t \geq 0}$ for every $N \in \mathbb{N}$ and for any $T \in [0, \infty)$ as
    \begin{equation*}
        (\boldsymbol{x}^N(T),\boldsymbol{\upsilon}^N(T)) = (\boldsymbol{x}^N(0),\boldsymbol{\upsilon}^N(0)) + \boldsymbol{\Lambda}^N(T) + \int_{0}^T \boldsymbol{F}^{\, \textrm{fast}}  (\boldsymbol{x}^N,\boldsymbol{\upsilon}^N) \, d\lfloor Nt \rfloor.
    \end{equation*}
    To improve the readability, we split the proof into steps.

    \medskip

    \noindent \textit{Step~(i): tightness of $\left(\boldsymbol{x}^{N}, \boldsymbol{\upsilon}^{N}, \boldsymbol{\Lambda}^{N}\right)_{N \in \mathbb{N}}$}

    \medskip
    
    By Theorem~\ref{Paper03_major_tool_katzenberger}, and Lemmas~\ref{Paper03_martingales_WF_model_fast_env},~\ref{Paper03_lem_noise_env_specific_fast_env},~\ref{Paper03_lem_FV_proc_fast_env} and~\ref{Paper03_characterisation_flow_associated_fast_changing_environment}, the sequence of processes  $\left(\boldsymbol{x}^{N}, \boldsymbol{\upsilon}^{N}, \boldsymbol{\Lambda}^{N}\right)_{N \in \mathbb{N}}$ is relatively compact in~$\mathscr{D}\left([0, \infty), [0,1]^{K+1} \times [-1,1]^K \times \mathbb{R}^{2K + 1}\right)$. From now on, let $(\boldsymbol{x}(t), \boldsymbol{\upsilon}(t), \boldsymbol{\Lambda}(t))_{t \geq 0}$ be a subsequential limit. 
Theorem~\ref{Paper03_major_tool_katzenberger} implies that 
$(\boldsymbol{x}(t), \boldsymbol{\upsilon}(t), 
\boldsymbol{\Lambda}(t))_{t \geq 0}$ is a diffusion with sample paths in 
$\Gamma^{\, \textrm{fast}} = \Gamma \times \{0\}^K$, where $\Gamma$ is the 
diagonal of the $(K+1)$-dimensional hypercube. Since, in the limit, 
coordinates describing the environment vanish, and 
$x_0 \equiv x_1 \equiv \cdots \equiv x_K$, it will suffice to describe the 
dynamics of $x_0$.
    
    \medskip

    \noindent \textit{Step~(ii): characterisation of $(\boldsymbol{\Lambda}(t))_{t \geq 0}$}

    \medskip
    
    To compute the drift and the noise terms of the corresponding SDE, we 
need to characterise the limiting process $\boldsymbol{\Lambda}$ in terms 
of $x_0$. With this aim in mind, we shall characterise the limit points of each of the martingales $M^{N}_{0}$, $M^{N}_{\textrm{env}}$ and $\widetilde{M}_{\textrm{env}}^{N}$ that appear in the dynamics of the discrete process (see Equations~\eqref{Paper03_martingale_seeds_WF},~\eqref{Paper03_martingale_impact_env_on_seeds_WF} and~\eqref{Paper03_martingale_both_seeds_slow_env_WF}).
    
    Starting with $M^{N}_{0}$, by using the uniform integrability of the 
sequences $\left(\left(M^{N}_{0}(T)\right)^2\right)_{N \in \mathbb{N}}$ and~$\left(\left[M^{N}_{0}\right](T)\right)_{N \in \mathbb{N}}$ stated in Lemma~\ref{Paper03_martingales_WF_model_fast_env}, we conclude that $M^{N}_{0}$ 
converges weakly to a process $(M_0(t))_{t \geq 0}$ which is a continuous martingale with respect to the filtration generated by $(\boldsymbol{x}(t), \boldsymbol{\Lambda}(t))_{t \geq 0}$, and whose quadratic variation is the (weak) limit of $([M^N_0])_{N \in \mathbb{N}}$. Hence, taking the limit as $N \rightarrow \infty$ on both sides of~\eqref{Paper03_UI_predictable_bracket_bound_fast_environment_genetic_drift_iii} and using the continuity of $M_0$, we conclude that the quadratic variation of the martingale $M_0$ is given by, for all $T \in [0, \infty)$,
    \begin{equation*} 
        [M_0](T) = \langle M_0 \rangle(T) = \int_{0}^T x_0 (1 - x_0) \, dt.
    \end{equation*}
    Therefore, there exists a standard Brownian motion $W_0$ such that for all $T \in [0, \infty)$,
    \begin{equation} \label{Paper03_WF_noise_fast_env_lim}
    M_0(T) = \int_{0}^{T} \sqrt{x_0(1-x_0)} \, dW_0(t).
    \end{equation}

    Repeating the same argument for $M^{N}_{\textrm{env}}$, using~\eqref{Paper03_bound_predictable_bracket+process_fast_environment_env_martingale} instead of~\eqref{Paper03_UI_predictable_bracket_bound_fast_environment_genetic_drift_iii}, and then applying~\eqref{Paper03_limit_scaling_fluctuating_fitness_parameter_assumption}, we conclude that as $N \rightarrow \infty$, the sequence of processes $(M^N_{\textrm{env}})_{N \in \mathbb{N}}$ 
converges weakly to a continuous martingale $M_{\textrm{env}}$ whose quadratic variation is given by, for all $T \geq 0$,
    \begin{equation} \label{Paper03_quad_proc_env_noise_x_0_fast}
        [M_{\textrm{env}}](T) = \langle M_{\textrm{env}} \rangle(T) = 2ps^2(1-b_0)^2\int_{0}^{T}x_0^2(1-x_0)^2 \, dt.
    \end{equation}
    Therefore, there exists a standard Brownian motion $W_{\textrm{env}}$ such that for all $T \in [0, \infty)$,
    \begin{equation} \label{Paper03_first_env_noise_fast_env_lim}
        M_{\textrm{env}}(T) =(2p)^{1/2}s(1-b_0) \int_{0}^T x_0(1-x_0) \, dW_{\textrm{env}}(t).
    \end{equation}
    Moreover, by taking the limit as $N \rightarrow \infty$ in~\eqref{Paper03_orthogonality_noises_fast_environment} and using the uniform integrability conditions stated in Lemma~\ref{Paper03_martingales_WF_model_fast_env}, we conclude that the covariation of $M_0$ and $M_{\textrm{env}}$ is identically equal to $0$ and therefore, by Knight's theorem, the Brownian motions~$W_0$ and~$W_{\textrm{env}}$ are independent.

    Applying the argument that we used for $M^N_0$ to the analysis of~$\widetilde{M}^N_{\textrm{env}}$, using Lemma~\ref{Paper03_lem_noise_env_specific_fast_env} and identity~\eqref{Paper03_predictable_process_fast_environment_noise_2}, we conclude that $\left(\widetilde{M}^N_{\textrm{env}}\right)_{N \in \mathbb{N}}$ 
converges weakly to a continuous martingale $\widetilde{M}_{\textrm{env}}$ whose quadratic variation is given by
    \begin{equation} \label{Paper03_quad_proc_env_noise_upsilon_0_fast}
        [\widetilde{M}_{\textrm{env}}](T) = \langle \widetilde{M}_{\textrm{env}} \rangle(T) = 2ps^2T.
    \end{equation}
    By taking the limit as $N \rightarrow \infty$ on both sides of~\eqref{Paper03_noises_fast_environment_complete_dependence_first_version}, using the uniform integrability conditions stated in Lemmas~\ref{Paper03_martingales_WF_model_fast_env} and~\ref{Paper03_lem_noise_env_specific_fast_env}, and then recalling~\eqref{Paper03_quad_proc_env_noise_x_0_fast} and~\eqref{Paper03_quad_proc_env_noise_upsilon_0_fast}, we conclude that the covariation of~$M_{\textrm{env}}$ and~$\widetilde{M}_{\textrm{env}}$ satisfies, for all $T \in [0, \infty)$,
    \begin{equation} \label{Paper03_covariance_lim_fast_env_different_env_noises}
        [M_{\textrm{env}},\widetilde{M}_{\textrm{env}}](T) = - \left([M_{\textrm{env}}](T)[\widetilde{M}_{\textrm{env}}](T)\right)^{1/2} = - 2ps^2(1-b_0) \int_0^T x_0(1-x_0) \, dt.
    \end{equation}
    Therefore, by Knight's theorem and~\eqref{Paper03_quad_proc_env_noise_upsilon_0_fast}, we have for all $T \in [0, \infty)$
    \begin{equation} \label{Paper03_second_env_noise_fast_env_lim}
        \widetilde{M}_{\textrm{env}}(T) = -(2p)^{1/2}sW_{\textrm{env}}(T),
    \end{equation}
    where $W_{\textrm{env}}$ is the standard Brownian motion that appears on the right-hand side of~\eqref{Paper03_first_env_noise_fast_env_lim}.

    To compute the drift of the limiting SDE, we also need to characterise the limit of the sequence of finite variation processes $(A^N_0)_{N \in \mathbb{N}}$ defined in~\eqref{Paper03_eq:fv_process_fast_environment}. By Lemma~\ref{Paper03_lem_FV_proc_fast_env} and by Theorem~\ref{Paper03_major_tool_katzenberger}, the sequence $(A^N_0)_{N \in \mathbb{N}}$ 
converges weakly to a finite variation process $A_0$ as $N \rightarrow \infty$. By taking the limit as $N \rightarrow \infty$ on both sides of~\eqref{Paper03_eq:fv_process_fast_environment}, and then using dominated convergence and~\eqref{Paper03_limit_scaling_fluctuating_fitness_parameter_assumption}, we conclude that for all $T \in [0, \infty)$
    \begin{equation} \label{Paper03_FV_process_fast_env_lim}
        A_0(T) = 2ps^2(1-b_0)\int_0^T x_0(1-x_0)\Big((1-x_0) + b_0x_0\Big) \, dt.
    \end{equation}

    \medskip

    \noindent \textit{Step~(iii): the limiting SDE}

    \medskip

    We are now ready to state the limiting SDE for $x_0$. Applying Theorem~\ref{Paper03_major_tool_katzenberger} to the sequence of processes $\left(\boldsymbol{x}^{N}, \boldsymbol{\upsilon}^{N}, \boldsymbol{\Lambda}^{N}\right)_{N \in \mathbb{N}}$, we conclude that in the limit, $(x_0(t))_{t \geq 0}$ satisfies the following SDE
    \begin{equation} \label{Paper03_pre_computation_limiting_SDE_fast_env}
    \begin{aligned}
    dx_0 & = \frac{\partial \Phi^{\, \textrm{fast}}_0}{\partial x_0} dA_0(t) + \frac{\partial \Phi^{\, \textrm{fast}}_0}{\partial x_0} dM_0(t) + \frac{\partial \Phi^{\, \textrm{fast}}_0}{\partial x_0} dM_\textrm{env}(t) + \frac{\partial \Phi^{\, \textrm{fast}}_0}{\partial \upsilon_0} d\widetilde{M}_\textrm{env}(t) + \frac{1}{2} \frac{\partial^2 \Phi^{\, \textrm{fast}}_0}{\partial x_0^2} d[M_0](t) \\ & \quad \; + \frac{1}{2} \frac{\partial^2 \Phi^{\, \textrm{fast}}_0}{\partial x_0^2} d[M_\textrm{env}](t) + \frac{1}{2} \frac{\partial^2 \Phi^{\, \textrm{fast}}_0}{\partial \upsilon_0^2} d[\widetilde{M}_\textrm{env}](t) + \frac{\partial^2 \Phi^{\, \textrm{fast}}_0}{\partial x_0 \partial \upsilon_0} d[M_\textrm{env}, \widetilde{M}_{\textrm{env}}](t),
    \end{aligned}
    \end{equation}
    where to compute the the terms involving second-order derivatives we used the fact that for all $T \in [0, \infty)$,
\[
[M_0,M_\textrm{env}](T) = [M_0,\widetilde{M}_\textrm{env}](T) = 0.
\]
    Then,~\eqref{Paper03_limiting_SDE_fast_env} follows from applying identities~\eqref{Paper03_first_order_derivatives_fast_environment},~\eqref{Paper03_second_order_derivatives_fast_environment_easy_part},~\eqref{Paper03_WF_noise_fast_env_lim},~\eqref{Paper03_first_env_noise_fast_env_lim},~\eqref{Paper03_covariance_lim_fast_env_different_env_noises},~\eqref{Paper03_second_env_noise_fast_env_lim} and~\eqref{Paper03_FV_process_fast_env_lim} to~\eqref{Paper03_pre_computation_limiting_SDE_fast_env} and rearranging terms. This completes the proof.
\end{proof}

\section{Proof of technical and auxiliary results} 
\label{Paper03_section_proof_auxilliary_lemmas}

\subsection{Explicit solution for ODE with $K=1$} \label{Paper03_explicit_solution_small_K}

In this section we apply the method of characteristics to study the flow in~\eqref{Paper03_flow_definition_WF_model} for $K=1$. Recall that $\Gamma$ is the diagonal of $[0,1]^2$. We will establish the following result.

\begin{lemma} \label{Paper03_lemma_example_explicit_solution_ODE_K_1}
    Let $\boldsymbol{\Phi}: [0,1]^2 \rightarrow \Gamma$ be the map defined in~\eqref{Paper03_definition_projection_map}, associated to the flow $\boldsymbol{F}: [0,1]^{2} \rightarrow \mathbb{R}^2$ in~\eqref{Paper03_flow_definition_WF_model} with $K = 1$. Then, for any $\boldsymbol{x} = (x_0,x_1) \in [0,1) \times [0,1]$, we have
    \begin{equation} \label{Paper03_implicit_definition_derivatives_K_1}
        \frac{(1 - b_0) x_1 + b_0}{(1 - x_0)} - (1 - b_0) \log(1 - x_0) = \frac{(1 - b_0) \Phi_0(\boldsymbol{x}) + b_0}{( 1 - \Phi_0(\boldsymbol{x})) } - (1 - b_0) \log ( 1 - \Phi_0(\boldsymbol{x})).
    \end{equation}    
\end{lemma}

\begin{proof}
For $K = 1$, the ODE associated to the flow defined in~\eqref{Paper03_flow_definition_WF_model} is reduced to
\begin{equation} \label{Paper03_simplified_ODE_K_1}
\begin{aligned}
    \frac{d}{dt}x_0(t) & = - \frac{(1-b_0)(1 -x_0)(x_0 - x_1)}{1 - (1 - b_0)(x_0-x_1)} \quad \textrm{and} \quad \frac{d}{dt}x_1(t) = x_0 - x_1,
\end{aligned}
\end{equation}
since by~\eqref{Paper03_assumption_probability_distribution_bi}, we have $b_1 = 1 - b_0$. Setting then
\begin{equation} \label{Paper03_simplified_ODE_K_1_changing_variables}
    y_0 \defeq 1 - x_0 \quad \textrm{and} \quad y_1 \defeq (1 - b_0)x_1 + b_0,
\end{equation}
we have by~\eqref{Paper03_simplified_ODE_K_1} that
\begin{equation} \label{Paper03_simplified_ODE_K_1_changing_variables_corresponding_ODE}
    \frac{d}{dt} y_0(t) = y_0 \frac{(1-b_0)(x_0 - x_1)}{1 - (1 - b_0)(x_0-x_1)} \quad \textrm{and} \quad \frac{d}{dt} y_1(t) = (1 - b_0)(x_0-x_1).
\end{equation}
Observe that $y_0$ represents the proportion of wild type individuals in the current generation, and $y_1$ represents the proportion of seeds produced one generation in the past that will contribute to the genetic pool of the current generation. Then, by the chain rule and~\eqref{Paper03_simplified_ODE_K_1_changing_variables_corresponding_ODE}, we have
\begin{equation} \label{Paper03_simplified_ODE_K_1_changing_variables_characteristic_curve}
\begin{aligned}
    \frac{d}{d y_0} y_1 & = \frac{1 - (1 - b_0)(x_0 - x_1)}{y_0} \\ & = \frac{y_1}{y_0} + (1 - b_0),
\end{aligned}
\end{equation}
where for the last equality we used the fact that by~\eqref{Paper03_simplified_ODE_K_1_changing_variables},
\begin{equation*}
    (1 - b_0)(x_0 - x_1) = 1 - y_1 - (1 - b_0)y_0.
\end{equation*}
To explicitly solve~\eqref{Paper03_simplified_ODE_K_1_changing_variables_characteristic_curve}, we observe that again by the chain rule
\begin{equation*}
    \frac{d}{d(\log y_0)} y_1 = y_1 + (1-b_0)e^{\log y_0},
\end{equation*}
which admits as its unique solution
\begin{equation*}
\begin{aligned}
    y_1 & = (1-b_0)(\log y_0)e^{\log y_0} + Ce^{\log y_0} \\ & = (1 - b_0) y_0 \log y_0 + Cy_0,
\end{aligned}
\end{equation*}
where $C$ is a parameter which is determined by the initial condition of the ODE in~\eqref{Paper03_simplified_ODE_K_1_changing_variables_characteristic_curve}. Then, by applying~\eqref{Paper03_simplified_ODE_K_1_changing_variables}, and letting the initial condition be $\boldsymbol{x} = (x_0, x_1) \in [0, 1) \times [0,1]$, we conclude that for all $t \geq 0$, we have
\begin{equation} \label{Paper03_almost_last_step_computation_ODE_K_1}
     \frac{(1 - b_0) x_1 + b_0}{(1 - x_0)} - (1 - b_0) \log(1 - x_0) = \frac{(1 - b_0) x_1(t) + b_0}{( 1 - x_0(t)) } - (1 - b_0) \log ( 1 - x_0(t)).
\end{equation}
By Lemma~\ref{Paper03_eigenvalue_condition_constant_environment}, the projection map $\boldsymbol{\Phi}: [0,1]^2 \rightarrow \Gamma$ defined in~\eqref{Paper03_definition_projection_map} is well defined, and moreover we have $\Phi_0 = \Phi_1$. Therefore, taking the limit as $t \rightarrow \infty$ on the right-hand side of~\eqref{Paper03_almost_last_step_computation_ODE_K_1} completes the proof.
\end{proof}

One can then establish~\eqref{Paper03_drift_K_1}  by differentiating both sides of~\eqref{Paper03_implicit_definition_derivatives_K_1} with respect to $x_0$ and $x_1$. We were not able to apply the method of characteristics for  the case $K \geq 2$. In particular, we did not find
the analogue of~\eqref{Paper03_simplified_ODE_K_1_changing_variables_characteristic_curve} in that case. 
%We explain in Sections~\ref{Paper03_subsection_computing_derivatives} and~\ref{Paper03_general_result_derivatives_manifold} how we can compute the first and second order derivatives without using the method of characteristics. Our approach combines Parsons and Rogers' method~\cite{parsons2017dimension} with a new formula for the quadratic approximation of the nonlinear stable manifold near an attractor manifold for an ODE.

\subsection{Proof of Theorems~\ref{Paper03_prop_uniqueness_quadratic_term_nonlinear_manifold} and~\ref{Paper03_explicit_formula_solution_lyapunov_equation}} \label{Paper03_section_proofs_dynamical_systems}

In this section, we will prove Theorems~\ref{Paper03_prop_uniqueness_quadratic_term_nonlinear_manifold} and~\ref{Paper03_explicit_formula_solution_lyapunov_equation}, i.e.~we will establish an explicit formula for the matrix $\boldsymbol{\Theta}$ which is a quadratic approximation for the stable nonlinear manifold (see~\eqref{Paper03_Parsons_Rogers_description_flow}), by studying the dynamical system near the manifold $\Gamma$. Recall that $\Gamma$, defined in~\eqref{Paper03_definition_manifold}, is the invariant attractor manifold of ODE~\eqref{Paper03_deterministic_flow}. Suppose Assumptions~\ref{Paper03_assumption_manifold} and~\ref{Paper03_geometric_properties_Gamma} are both satisfied. For any $\boldsymbol{x} \in \Gamma$, let $\boldsymbol{u} \equiv \boldsymbol{u}(\boldsymbol{x})$ be a right eigenvector of $\mathcal{J}\mathbf{F}_{\boldsymbol{x}}$ associated to the eigenvalue~$0$, and let $\boldsymbol{v} \equiv \boldsymbol{v}(\boldsymbol{x})$ be a left eigenvector of $\mathcal{J}\mathbf{F}_{\boldsymbol{x}}$ associated to the eigenvalue~$0$, such that $\langle \boldsymbol{u}, \boldsymbol{v} \rangle = 1$. For any $\boldsymbol{x} \in \Gamma$, let $E_{c} \equiv E_{c}(\boldsymbol{x})$ and $E_{s} \equiv E_{s}(\boldsymbol{x})$ denote, respectively, the linear centre and stable manifolds. Observe that $E_{c} = \spanvector(\boldsymbol{u})$, and the right eigenvectors of $\mathcal{J}\mathbf{F}_{\boldsymbol{x}}$ associated to the eigenvalues of $\mathcal{J}\mathbf{F}_{\boldsymbol{x}}$ with negative real part are a basis for $E_{s}$. Recall that $\Gamma$ is a non-linear centre manifold, and let $\mathfrak{W}_{s} \equiv \mathfrak{W}_{s}(\boldsymbol{x})$ be the non-linear stable manifold tangent to $E_{s}$ at $\boldsymbol{x}$. Also, recall the definition of the projection matrices $\mathcal{P}_c$ and $\mathcal{P}_s$ in~\eqref{Paper03_projection_linear_center_manifold} and~\eqref{Paper03_projection_linear_stable_manifold}, respectively. By Lemma~5.2 in~\cite{kuznetsov1998elements}, $\mathcal{P}_{c}$ and $\mathcal{P}_{s}$ are well defined and satisfy the relations $\mathcal{P}^{2}_{c} = \mathcal{P}_{c}$, $\mathcal{P}^{2}_{s} = \mathcal{P}_{s}$ and $\mathcal{P}_{c}\mathcal{P}_{s} = \mathcal{P}_{s}\mathcal{P}_{c} = 0$.

We begin by observing that for any $\boldsymbol{x} \in \Gamma$, the left eigenvector $\boldsymbol{v}(\boldsymbol{x})$ is orthogonal to the linear stable manifold $E_{s}(\boldsymbol{x})$. Indeed, for any $\boldsymbol{y} \in E_{s}$, there exists $\tilde{\boldsymbol{y}} \in E_{s}$ such that $\mathcal{J}\mathbf{F}_{\boldsymbol{x}} \tilde{\boldsymbol{y}} = \boldsymbol{y}$, since $E_{s}$ is the space spanned by the (generalised) eigenvectors associated to the eigenvalues of $\mathcal{J}\mathbf{F}_{\boldsymbol{x}}$ with negative real part. Therefore, for any $\boldsymbol{y} \in E_{s}$, there exists $\tilde{\boldsymbol{y}} \in E_{s}$ such that
\begin{equation*}
    \langle \boldsymbol{y}, \boldsymbol{v} \rangle = \langle \tilde{\boldsymbol{y}}, \mathcal{J}\mathbf{F}^{*}_{\boldsymbol{x}}\boldsymbol{v} \rangle = 0.
\end{equation*}
Moreover $\mathcal{P}_{s}^{*}\boldsymbol{v} = 0$, since
\begin{equation*}
    \langle \mathcal{P}_{s}^{*}\boldsymbol{v}, \mathcal{P}_{s}^{*}\boldsymbol{v} \rangle =  \langle \mathcal{P}_{s}\left(\mathcal{P}_{s}^{*}\boldsymbol{v}\right), \boldsymbol{v} \rangle = 0,
\end{equation*}
where the last equality is derived by observing that $\mathcal{P}_{s}\left(\mathcal{P}_{s}^{*}\boldsymbol{v}\right) \in E_{s}$, and $\boldsymbol{v}$ is orthogonal to $E_{s}$.
Following~\cite{kuznetsov1998elements}, we will write the dynamics of the flow given by ODE~\eqref{Paper03_deterministic_flow} in terms of its projections on $E_{c}$ and $E_{s}$.

To simplify notation, let $(\boldsymbol{y}(t))_{t\geq 0}$ be the flow curve of the ODE started from $\boldsymbol{y}(0) \in \mathscr{U}$. Let $\varepsilon > 0$ be such that whenever $\boldsymbol{y}(0) \in \mathcal{B}(\boldsymbol{x}, \varepsilon) \cap \mathfrak{W}_{s}$, we have $\boldsymbol{y}(t) \in \mathcal{B}(\boldsymbol{x}, \varepsilon) \cap \mathfrak{W}_{s}$ for all $t \geq 0$. By the smoothness of $\mathbf{F}$ and assuming $\varepsilon$ sufficiently small, we can consider the Taylor expansion of $\mathbf{F}$, and write that for any flow $(\boldsymbol{y}(t))_{t \geq 0} \subset \mathcal{B}(\boldsymbol{x}, \varepsilon) \cap \mathfrak{W}_{s}$, 
\begin{equation} \label{Paper03_taylor_expansion_F}
    \frac{d \boldsymbol{y}(t)}{dt} = \mathcal{J}\mathbf{F}_{\boldsymbol{x}}(\boldsymbol{y} - \boldsymbol{x}) + \frac{1}{2}\mathcal{H}_{\boldsymbol{x}}(\boldsymbol{y} - \boldsymbol{x}) + o(\vert \vert \boldsymbol{y} - \boldsymbol{x} \vert \vert_{2}^{2}), 
\end{equation}
where $\mathcal{H}_{\boldsymbol{x}} = \left(\mathcal{H}_{\boldsymbol{x}, i}\right)_{i = 0}^{K}: \mathbb{R}^{K+1} \rightarrow \mathbb{R}^{K + 1}$ is given for all $i \in \llbracket K \rrbracket_{0}$ by
\begin{equation} \label{Paper03_definition_multidimensional_tensor_hessian}
    \mathcal{H}_{\boldsymbol{x}, i}(\boldsymbol{y} - \boldsymbol{x}) \defeq \Big\langle \boldsymbol{y} - \boldsymbol{x}, \Hess_{\boldsymbol{x}}(F_{i})(\boldsymbol{y} - \boldsymbol{x})\Big\rangle.
\end{equation}
We shall now compute the impact of the flow dynamics on the projections of $(\boldsymbol{y}(t))_{t \geq 0}$. We start by considering the evolution of the projection onto the linear centre manifold $E_{c}$.
\begin{equation} \label{Paper03_taylor_representation_flow_center_projection}
\begin{aligned}
    \frac{d \left(\mathcal{P}_{c}(\boldsymbol{y}(t) - \boldsymbol{x})\right)}{dt} & = \left\langle \mathcal{J}\mathbf{F}_{\boldsymbol{x}}(\boldsymbol{y} - \boldsymbol{x}) + \frac{1}{2}\mathcal{H}_{\boldsymbol{x}}(\boldsymbol{y} - \boldsymbol{x}) + o(\vert \vert \boldsymbol{y} - \boldsymbol{x}\vert \vert_{2}^{2}), \boldsymbol{v}\right\rangle\boldsymbol{u} \\ & = \Big\langle \boldsymbol{y} - \boldsymbol{x}, \mathcal{J}\mathbf{F}_{\boldsymbol{x}}^{*}\boldsymbol{v}\Big\rangle\boldsymbol{u} +  \frac{1}{2} \Big\langle \mathcal{H}_{\boldsymbol{x}}(\boldsymbol{x} - \boldsymbol{x}), \boldsymbol{v} \Big\rangle \boldsymbol{u} + o(\vert \vert \boldsymbol{y} - \boldsymbol{x}\vert \vert_{2}^{2}) \\ & = \frac{1}{2}\Big\langle \mathcal{H}_{\boldsymbol{x}}(\boldsymbol{y} - \boldsymbol{x}), \boldsymbol{v} \Big\rangle \boldsymbol{u} + o(\vert \vert \boldsymbol{y} - \boldsymbol{x}\vert \vert_{2}^{2}),
\end{aligned}
\end{equation}
where the last equality comes from the fact that by construction $\boldsymbol{v}$ is a left eigenvector of $\mathcal{J}\mathbf{F}_{\boldsymbol{x}}$ associated to the eigenvalue $0$. To compute the derivative of the projection onto the linear stable manifold, we combine the definition of $\mathcal{P}_{s}$ in~\eqref{Paper03_projection_linear_stable_manifold} with identities~\eqref{Paper03_taylor_expansion_F} and~\eqref{Paper03_taylor_representation_flow_center_projection}, to obtain
\begin{equation} \label{Paper03_taylor_representation_flow_linear_stable_projection}
     \frac{d \left(\mathcal{P}_{s}(\boldsymbol{y}(t) - \boldsymbol{x})\right)}{dt} = \mathcal{J}\mathbf{F}_{\boldsymbol{x}}(\boldsymbol{y} - \boldsymbol{x}) + \frac{1}{2}\mathcal{H}_{\boldsymbol{x}}(\boldsymbol{y} - \boldsymbol{x}) - \frac{1}{2}\Big\langle \mathcal{H}_{\boldsymbol{x}}(\boldsymbol{y} - \boldsymbol{x}), \boldsymbol{v} \Big\rangle \boldsymbol{u} + o(\vert \vert \boldsymbol{y} - \boldsymbol{x}\vert \vert_{2}^{2}).
\end{equation}
To simplify notation, for any $\boldsymbol{y} \in \mathbb{R}^{K + 1}$, we shall write $\boldsymbol{y}_{s} = \mathcal{P}_{s}\left(\boldsymbol{y} - \boldsymbol{x}\right)$. By the Centre Manifold Theorem, there is a map $V: E_{s} \rightarrow \mathbb{R}$ such that $\boldsymbol{y} \in \mathfrak{W}_{s} \cap \mathcal{B}(\boldsymbol{x}, \varepsilon)$ if and only if
\begin{equation} \label{Paper03_map_defining_nonlinear_manifold}
    V(\boldsymbol{y}_{s}) \boldsymbol{u} = \mathcal{P}_{c}(\boldsymbol{y} - \boldsymbol{x}).
\end{equation}
Then, for any $\boldsymbol{y} \in \mathfrak{W}_{s} \cap \mathcal{B}(\boldsymbol{x}, \varepsilon)$, we have
\begin{equation} \label{Paper03_rewriting_nonlinear_stable_manifold}
    \boldsymbol{y} - \boldsymbol{x} = \boldsymbol{y}_{s} + V(\boldsymbol{y}_{s})\boldsymbol{u}.
\end{equation}
Note that since $\boldsymbol{x} \in \mathfrak{W}_{s}$, we must have $V(0) = 0$. Moreover, as $\mathfrak{W}_{s}$ is tangent to $E_{s}$ at $\boldsymbol{x}$, we have $\nabla V (0) = 0$. Hence, by Taylor expanding $V$ around $0$, we find that for any $\boldsymbol{y}_{s} \in E_{s}$ with $\vert \vert \boldsymbol{y}_{s}\vert \vert_{2}$ sufficiently small,
\begin{equation} \label{Paper03_Talylor_expansion_map_defining_nonlinear_manifold}
V(\boldsymbol{y}_{s}) = \frac{1}{2}\left \langle \boldsymbol{y}_{s}, \boldsymbol{\Theta} \boldsymbol{y}_{s}  \right\rangle + o(\vert \vert \boldsymbol{y}_{s} \vert \vert_{2}^{2}),
\end{equation}
where $\boldsymbol{\Theta}$ is a symmetric square matrix of order $K+1$. Comparing~\eqref{Paper03_Talylor_expansion_map_defining_nonlinear_manifold} to~\eqref{Paper03_Parsons_Rogers_description_flow}, the matrix $\boldsymbol{\Theta}$ is the one that we aim to compute. Since the domain of $V$ is the stable linear manifold $E_{s}$, we may assume without loss of generality that
\begin{equation} \label{Paper03_product_matrices_projections}
\mathcal{P}_{s}^{*}\boldsymbol{\Theta} = \boldsymbol{\Theta} \mathcal{P}_{s} = \boldsymbol{\Theta}.
\end{equation}
Furthermore, by~\eqref{Paper03_map_defining_nonlinear_manifold}, the invariance of the non-linear manifold $\mathfrak{W}_{s}$ under the flow implies that if $\boldsymbol{y}(0) \in \mathfrak{W}_{s} \cap \mathcal{B}(\boldsymbol{x}, \varepsilon)$, then for all $t \geq 0$,
\begin{equation} \label{Paper03_product_matrices_projections_auxi}
    V(\boldsymbol{y}_{s}(t)) \boldsymbol{u} = \mathcal{P}_{c}(\boldsymbol{y}(t) - \boldsymbol{x}).
\end{equation}
Differentiating both sides of~\eqref{Paper03_product_matrices_projections_auxi} with respect to time, we conclude by the chain rule that
\begin{equation} \label{Paper03_product_matrices_projections_auxii}
    \left\langle \nabla V(\boldsymbol{y}_{s}(t)), \frac{d \boldsymbol{y}_{s}(t)}{dt} \right\rangle \boldsymbol{u} = \frac{d \left(\mathcal{P}_{c}(\boldsymbol{y}(t) - \boldsymbol{x})\right)}{dt}.
\end{equation}
By applying~\eqref{Paper03_taylor_representation_flow_center_projection},~\eqref{Paper03_taylor_representation_flow_linear_stable_projection},~\eqref{Paper03_rewriting_nonlinear_stable_manifold} and~\eqref{Paper03_Talylor_expansion_map_defining_nonlinear_manifold} to~\eqref{Paper03_product_matrices_projections_auxii}, we conclude that
\begin{equation*}
\begin{aligned}
    \Big\langle \boldsymbol{\Theta} \boldsymbol{y}_{s}, \mathcal{J}\mathbf{F}_{\boldsymbol{x}}\Big(\boldsymbol{y}_{s} + V(\boldsymbol{y}_{s})\boldsymbol{u}\Big)\Big\rangle + \mathcal{O}(\vert \vert \boldsymbol{y}_{s} \vert \vert_{2}^{3}) = \frac{1}{2}\Big\langle \mathcal{H}_{\boldsymbol{x}}(\boldsymbol{y}_{s} + V(\boldsymbol{y}_{s})\boldsymbol{u}), \boldsymbol{v}\Big\rangle + o(\vert \vert \boldsymbol{y}_{s} \vert \vert_{2}^{2}).
\end{aligned}
\end{equation*}
Recalling that $\boldsymbol{u} \in \ker\left(\mathcal{J}\mathbf{F}_{\boldsymbol{x}}\right)$, identity~\eqref{Paper03_definition_multidimensional_tensor_hessian}, and equating the quadratic terms, we conclude that for any $\boldsymbol{y}_{s} \in E_{s}$, we must have
\begin{equation} \label{Paper03_fundamental_step_computation_vartheta}
    \Big \langle \boldsymbol{\Theta} \boldsymbol{y}_{s}, \mathcal{J}\mathbf{F}_{\boldsymbol{x}} \boldsymbol{y}_{s}\Big \rangle = \frac{1}{2} \Big \langle \mathcal{H}_{\boldsymbol{x}}(\boldsymbol{y}_{s}), \boldsymbol{v} \Big \rangle.
\end{equation}
We now show that one can obtain $\boldsymbol{\Theta}$ from a linear system.

\begin{lemma} \label{Paper03_computation_quadratic_term_nonlinear_manifold_first_step}
    The matrix $\boldsymbol{\Theta}$ satisfies~\eqref{Paper03_lyapunov_equation}.
\end{lemma}

\begin{proof}
    We start by recalling~\eqref{Paper03_definition_multidimensional_tensor_hessian}, and rewriting the right-hand side of~\eqref{Paper03_fundamental_step_computation_vartheta} as
    \begin{equation*}
    \begin{aligned}
    \frac{1}{2} \Big \langle \mathcal{H}_{\boldsymbol{x}}(\boldsymbol{y}_{s}), \boldsymbol{v} \Big \rangle & = \frac{1}{2} \sum_{i = 0}^{K} v_{i} \Big\langle \boldsymbol{y}_{s}, \Hess_{\boldsymbol{x}}(F_{i})\boldsymbol{y}_{s}\Big\rangle \\ & = \left\langle \boldsymbol{y}_{s}, \left(\frac{1}{2} \sum_{i = 0}^{K} v_{i} \Hess_{\boldsymbol{x}}(F_{i})\right) \boldsymbol{y}_{s}\right \rangle \\ & = \left\langle \mathcal{P}_{s}\boldsymbol{y}_{s}, \left(\frac{1}{2} \sum_{i = 0}^{K} v_{i} \Hess_{\boldsymbol{x}}(F_{i})\right) \mathcal{P}_{s}\boldsymbol{y}_{s}\right \rangle \\ & = \left\langle \boldsymbol{y}_{s}, \mathcal{P}^{*}_{s}\left(\frac{1}{2} \sum_{i = 0}^{K} v_{i} \Hess_{\boldsymbol{x}}(F_{i})\right) \mathcal{P}_{s}\boldsymbol{y}_{s}\right \rangle,
    \end{aligned}
    \end{equation*}
    where we used the fact that $\boldsymbol{y}_{s} = \mathcal{P}_{s} \boldsymbol{y}_{s}$ since $\boldsymbol{y}_{s} \in E_{s}$ to derive the third and fourth equalities. Hence, by~\eqref{Paper03_fundamental_step_computation_vartheta} and since $\boldsymbol{\Theta}$ is symmetric, we conclude that for all $\boldsymbol{y}_{s} \in E_{s}$,
    \begin{equation} \label{Paper03_intermediate_step_identity_compute_quadratic_term_nonlinear_manifold}
    \left\langle \boldsymbol{y}_{s},  \left(\boldsymbol{\Theta} \mathcal{J} F_{\boldsymbol{x}} - \mathcal{P}_{s}^{*} \left(\frac{1}{2} \sum_{i = 0}^{K} v_{i} \Hess_{\boldsymbol{x}}(F_{i})\right) \mathcal{P}_{s} \right) \boldsymbol{y}_{s} \right\rangle = 0.
    \end{equation}
    We shall verify that $\boldsymbol{\Theta} \mathcal{J} F_{\boldsymbol{x}} - \mathcal{P}_{s}^{*} \left(\frac{1}{2} \sum_{i = 0}^{K} v_{i} \Hess_{\boldsymbol{x}}(F_{i})\right) \mathcal{P}_{s}$ is skew-symmetric. Recall that a square matrix $\mathbf{M}$ of order $K+1$ is skew-symmetric if and only if $\mathbf{M} = - \mathbf{M}^{*}$, or, equivalently, if and only if $\langle \boldsymbol{y}, \mathbf{M}\boldsymbol{y} \rangle = 0$, for all $\boldsymbol{y} \in \mathbb{R}^{K+1}$. By the definition of the projections $\mathcal{P}_{c}$ and $\mathcal{P}_{s}$, for any $\boldsymbol{y} \in \mathbb{R}^{K+1}$, there exists $r \in \mathbb{R}$ and $\boldsymbol{y}_{s} \in E_{s}$ such that $\boldsymbol{y} = \boldsymbol{y}_{s} + r\boldsymbol{u}$. Hence, by the bilinearity of the inner product,
    \begin{equation*}
    \begin{aligned}
        & \left\langle \boldsymbol{y},  \left(\boldsymbol{\Theta} \mathcal{J} F_{\boldsymbol{x}} - \mathcal{P}_{s}^{*} \left(\frac{1}{2} \sum_{i = 0}^{K} v_{i} \Hess_{\boldsymbol{x}}(F_{i})\right) \mathcal{P}_{s} \right) \boldsymbol{y} \right\rangle \\ & \quad = \left\langle (\boldsymbol{y}_{s} + r\boldsymbol{u}),  \left(\boldsymbol{\Theta} \mathcal{J} F_{\boldsymbol{x}} - \mathcal{P}_{s}^{*} \left(\frac{1}{2} \sum_{i = 0}^{K} v_{i} \Hess_{\boldsymbol{x}}(F_{i})\right) \mathcal{P}_{s} \right) (\boldsymbol{y}_{s} + r \boldsymbol{u})\right\rangle \\ & \quad = \left\langle \boldsymbol{y}_{s},  \left(\boldsymbol{\Theta} \mathcal{J} F_{\boldsymbol{x}} - \mathcal{P}_{s}^{*} \left(\frac{1}{2} \sum_{i = 0}^{K} v_{i} \Hess_{\boldsymbol{x}}(F_{i})\right) \mathcal{P}_{s} \right) \boldsymbol{y}_{s} \right\rangle \\ & \quad \quad + r\left\langle \boldsymbol{y}_{s},  \left(\boldsymbol{\Theta} \mathcal{J} F_{\boldsymbol{x}} - \mathcal{P}_{s}^{*} \left(\frac{1}{2} \sum_{i = 0}^{K} v_{i} \Hess_{\boldsymbol{x}}(F_{i})\right) \mathcal{P}_{s} \right) \boldsymbol{u} \right\rangle \\ & \quad \quad + r\left\langle \boldsymbol{u},  \left(\boldsymbol{\Theta} \mathcal{J} F_{\boldsymbol{x}} - \mathcal{P}_{s}^{*} \left(\frac{1}{2} \sum_{i = 0}^{K} v_{i} \Hess_{\boldsymbol{x}}(F_{i})\right) \mathcal{P}_{s} \right) \boldsymbol{y}_{s} \right\rangle \\ & \quad \quad + r^{2}\left\langle \boldsymbol{u},  \left(\boldsymbol{\Theta} \mathcal{J} F_{\boldsymbol{x}} - \mathcal{P}_{s}^{*} \left(\frac{1}{2} \sum_{i = 0}^{K} v_{i} \Hess_{\boldsymbol{x}}(F_{i})\right) \mathcal{P}_{s} \right) \boldsymbol{u} \right\rangle. 
    \end{aligned}
    \end{equation*}    Applying~\eqref{Paper03_intermediate_step_identity_compute_quadratic_term_nonlinear_manifold} and the fact that $\boldsymbol{u} \in \ker\left(\mathcal{J}\mathbf{F}_{\boldsymbol{x}}\right) \cap \ker \left(\mathcal{P}_{s}\right)$, we conclude that
    \begin{equation} \label{Paper03_skew_symmetric_property}
    \begin{aligned}
        & \left\langle \boldsymbol{y},  \left(\boldsymbol{\Theta} \mathcal{J} F_{\boldsymbol{x}} - \mathcal{P}_{s}^{*} \left(\frac{1}{2} \sum_{i = 0}^{K} v_{i} \Hess_{\boldsymbol{x}}(F_{i})\right) \mathcal{P}_{s} \right) \boldsymbol{y} \right\rangle \\ & \quad = r\left\langle \boldsymbol{u},  \left(\boldsymbol{\Theta} \mathcal{J} F_{\boldsymbol{x}} - \mathcal{P}_{s}^{*} \left(\frac{1}{2} \sum_{i = 0}^{K} v_{i} \Hess_{\boldsymbol{x}}(F_{i})\right) \mathcal{P}_{s} \right) \boldsymbol{x}_{s} \right\rangle \\ & \quad = r\left\langle \boldsymbol{u},  \mathcal{P}_{s}^{*}\left(\boldsymbol{\Theta} \mathcal{J} F_{\boldsymbol{x}} - \left(\frac{1}{2} \sum_{i = 0}^{K} v_{i} \Hess_{\boldsymbol{x}}(F_{i})\right) \mathcal{P}_{s} \right) \boldsymbol{x}_{s} \right\rangle \\ & \quad = r\left\langle \mathcal{P}_{s}\boldsymbol{u}, \left(\boldsymbol{\Theta} \mathcal{J} F_{\boldsymbol{x}} - \left(\frac{1}{2} \sum_{i = 0}^{K} v_{i} \Hess_{\boldsymbol{x}}(F_{i})\right) \mathcal{P}_{s} \right) \boldsymbol{x}_{s} \right\rangle \\ & \quad = 0,
    \end{aligned}
    \end{equation}
    where we applied the identity $\boldsymbol{\Theta} = \mathcal{P}^{*}_{s} \boldsymbol{\Theta}$ to derive the second equality. Since~\eqref{Paper03_skew_symmetric_property} holds for all $\boldsymbol{y} \in \mathbb{R}^{K+1}$, we conclude that $\boldsymbol{\Theta} \mathcal{J} F_{\boldsymbol{x}} - \mathcal{P}_{s}^{*} \left(\frac{1}{2} \sum_{i = 0}^{K} v_{i} \Hess_{\boldsymbol{x}}(F_{i})\right) \mathcal{P}_{s}$ is skew-symmetric. Therefore,
    \begin{equation} \label{Paper03_skew_symmetric_propertyi}
        \boldsymbol{\Theta} \mathcal{J} F_{\boldsymbol{x}} - \mathcal{P}_{s}^{*} \left(\frac{1}{2} \sum_{i = 0}^{K} v_{i} \Hess_{\boldsymbol{x}}(F_{i})\right) \mathcal{P}_{s} = - \left( \boldsymbol{\Theta} \mathcal{J} F_{\boldsymbol{x}} - \mathcal{P}_{s}^{*} \left(\frac{1}{2} \sum_{i = 0}^{K} v_{i} \Hess_{\boldsymbol{x}}(F_{i})\right) \mathcal{P}_{s}\right)^{*}.
    \end{equation}
     Recalling that $\boldsymbol{\Theta}$ and $\Hess_{\boldsymbol{x}}(F_{i})$ are symmetric for all $i \in \llbracket K \rrbracket_{0}$ and by rearranging the terms in~\eqref{Paper03_skew_symmetric_propertyi}, our claim holds.
\end{proof}

By Lemma~\ref{Paper03_computation_quadratic_term_nonlinear_manifold_first_step}, it will suffice to solve Equation~\eqref{Paper03_lyapunov_equation} to compute $\boldsymbol{\Theta}$. Existence and uniqueness of solutions depend in general on the eigenvalues of the Jacobian matrix $\mathcal{J}\mathbf{F}_{\boldsymbol{x}}$~\cite[Theorem~VII.2.1]{bhatia2013matrix}. If all the eigenvalues of $\mathcal{J}\mathbf{F}_{\boldsymbol{x}}$ had negative real part, then we would classify Equation~\eqref{Paper03_lyapunov_equation} as a stable Lyapunov equation, and existence and uniqueness of solutions would be immediate. Since $\mathcal{J}\mathbf{F}_{\boldsymbol{x}}$ has one eigenvalue $0$, and all the other eigenvalues with negative real part, we say this is a singular Lyapunov equation which is semistable (see Section~4 in~\cite{bhat1999lyapunov} as a reference on this theme). In this case, it is not clear \emph{a priori} that existence and uniqueness hold. This will be our next result.

\begin{lemma} \label{Paper03_computation_symmetric_matrix_easier_linear_system}
    There exists a unique square matrix $\boldsymbol{\Theta}$ of order~$K+1$ which satisfies both Equations~\eqref{Paper03_lyapunov_equation} and~\eqref{Paper03_product_matrices_projections}.
\end{lemma}

\begin{proof}
    Note that a square matrix $\boldsymbol{\Theta}$ satisfies~\eqref{Paper03_product_matrices_projections} if, and only if, it is an element of the real vector space
    \begin{equation*}
        \mathscr{V} \defeq \left\{\varphi \in \mathscr{L}\left(\mathbb{R}^{K+1}, \mathbb{R}^{K+1}\right): \; E_{c} \subset \ker(\varphi) \quad \textrm{and} \quad \mathcal{P}^{*}_{s} \varphi = \varphi\right\},
    \end{equation*}
    where for any vector spaces $\mathcal{X}$ and $\mathcal{Y}$, $\mathscr{L}(\mathcal{X}, \mathcal{Y})$ is the set of $\mathcal{Y}$-valued linear maps defined on $\mathcal{X}$. We now define the linear operators $\Lambda, \widetilde{\Lambda}: \mathscr{V} \rightarrow \mathscr{V}$ by, for all $\varphi \in \mathscr{V}$,
    \begin{equation*}
        \Lambda(\varphi) = \mathcal{J}\mathbf{F}_{\boldsymbol{x}}^{*}\varphi \quad \textrm{and} \quad \widetilde{\Lambda}(\varphi) = \varphi \mathcal{J}\mathbf{F}_{\boldsymbol{x}}.
    \end{equation*}
    Since $E_{c} = \ker\left(\mathcal{J}\mathbf{F}_{\boldsymbol{x}}\right)$, we have $\mathcal{J}\mathbf{F}_{\boldsymbol{x}} \mathcal{P}_{s} = \mathcal{J}\mathbf{F}_{\boldsymbol{x}}$, and therefore the operators $\Lambda$ and $\widetilde{\Lambda}$ are well defined. Let $L = \Lambda + \widetilde{\Lambda}$. To prove our claim, it will be enough to establish that $L$ is invertible, i.e.~that $0 \not\in \spec(L)$.

    Observing that $\Lambda$ and $\widetilde{\Lambda}$ commute, i.e.~that for any $\varphi \in \mathscr{V}$, we have
    \[
\Lambda \widetilde{\Lambda} \varphi = \Lambda(\varphi \mathcal{J}\mathbf{F}_{\boldsymbol{x}}) = \mathcal{J}\mathbf{F}_{\boldsymbol{x}}^*\varphi\mathcal{J}\mathbf{F}_{\boldsymbol{x}} = \widetilde{\Lambda}(\mathcal{J}\mathbf{F}_{\boldsymbol{x}}^*\varphi) =  \widetilde{\Lambda} \Lambda  \varphi,
\]
    we conclude that $\spec(L) \subseteq \spec(\Lambda) + \spec(\widetilde{\Lambda})$ (see for instance~\cite[Theorem~2.4.9]{horn2012matrix} for a proof of this fact). Hence, it will suffice to prove that $0 \not\in \spec(\Lambda) + \spec(\widetilde{\Lambda})$. To verify this, observe that we can extend $\mathscr{V}$ to a complex vector space by taking
    \begin{equation*}
        \widehat{\mathscr{V}} \defeq \left\{\varphi \in \mathscr{L}\left(\mathbb{C}^{K+1}, \mathbb{C}^{K+1}\right): \; E_{c} \subset \ker(\varphi) \quad \textrm{and} \quad \mathcal{P}^{*}_{s} \varphi = \varphi\right\}.
    \end{equation*}
    Then $\Lambda, \widetilde{\Lambda} \in \mathscr{L}(\widehat{\mathscr{V}}, \widehat{\mathscr{V}})$. Let $\lambda$ be an eigenvalue of $\Lambda$, and let $\varsigma \neq 0$ be an associated eigenvector. Then for all $\boldsymbol{y} \in \mathbb{C}^{K+1}$, we have
    \begin{equation*}
        \Lambda \varsigma \boldsymbol{y} = \mathcal{J}\mathbf{F}_{\boldsymbol{x}}^{*} (\varsigma \boldsymbol{y}) = \lambda (\varsigma \boldsymbol{y}).
    \end{equation*}
    Hence, $\lambda$ is an eigenvalue of $\mathcal{J}\mathbf{F}_{\boldsymbol{x}}^{*}$, which means by Assumption~\ref{Paper03_assumption_manifold} that either $\lambda = 0$ or $\Re(\lambda) < 0$. We now claim that only the latter alternative is possible. Indeed, if $\lambda = 0$, we would have that for all $\boldsymbol{y} \in \mathbb{C}^{K+1}$ there should exist $r \in \mathbb{C}$ such that $\varsigma \boldsymbol{y} = r\boldsymbol{v}$. Since $\mathcal{P}^{*}_{s} \varsigma = \varsigma$, this would imply that for all $\boldsymbol{y} \in \mathbb{C}^{K+1}$,
    \begin{equation*}
        \varsigma \boldsymbol{y} = \mathcal{P}^{*}_{s}  \varsigma \boldsymbol{y} = r \mathcal{P}^{*}_{s} \boldsymbol{v} = 0,
    \end{equation*}
    which is a contradiction since $\varsigma \neq 0$. Now, let $\tilde{\lambda}$ be an eigenvalue of $\widetilde{\Lambda}$, and let $\tilde{\varsigma}$ be an associated eigenvector. Then for all $\boldsymbol{y} \in \mathbb{C}^{K+1}$ we have
    \begin{equation*}
        \widetilde{\Lambda} \tilde{\varsigma} \boldsymbol{y} = \tilde{\lambda} \tilde{\varsigma} \boldsymbol{y} = \tilde{\varsigma} \mathcal{J}\mathbf{F}_{\boldsymbol{x}} \boldsymbol{y}.
    \end{equation*}
    Suppose that $\tilde{\lambda} \not\in \spec\left(\mathcal{J}\mathbf{F}_{\boldsymbol{x}}\right) \setminus \{0\}$. Then, from the previous identity, we conclude that for any eigenvector $\boldsymbol{y}_{s}$ of $\mathcal{J}\mathbf{F}_{\boldsymbol{x}}$ associated with an eigenvalue $\lambda$ with negative real part, we must have
    \begin{equation*}
        (\tilde{\lambda} - \lambda)\tilde{\varsigma} \boldsymbol{y}_{s} = 0,
    \end{equation*}
    which implies that $\tilde{\varsigma} \boldsymbol{y}_{s} = 0$. Since by construction $\tilde{\varsigma} \boldsymbol{u} = 0$, we conclude that $\tilde{\varsigma} = 0$, which is a contradiction. Therefore, $\spec(\widetilde{\Lambda}) \subseteq \spec\left(\mathcal{J}\mathbf{F}_{\boldsymbol{x}}\right) \setminus \{0\}$. Finally, by Assumption~\ref{Paper03_assumption_manifold} and our previous arguments, this means that $\lambda \in \spec(\Lambda) \cup \spec(\widetilde{\Lambda}) \Rightarrow \Re(\lambda) < 0 $. In particular, $0 \not\in \spec(L) = \spec(\Lambda + \widetilde{\Lambda}) \subset \spec(\Lambda) + \spec(\widetilde{\Lambda})$, and therefore the operator $L$ is invertible, as desired.
\end{proof}

Lemma~\ref{Paper03_computation_symmetric_matrix_easier_linear_system} shows that we can obtain the matrix $\boldsymbol{\Theta}$ by solving a system of linear equations. Theorem~\ref{Paper03_prop_uniqueness_quadratic_term_nonlinear_manifold} can be then understood as a corollary of Lemma~\ref{Paper03_computation_symmetric_matrix_easier_linear_system}.

\begin{proof}[Proof of Theorem~\ref{Paper03_prop_uniqueness_quadratic_term_nonlinear_manifold}]
   Theorem~\ref{Paper03_prop_uniqueness_quadratic_term_nonlinear_manifold} follows from Lemma~\ref{Paper03_computation_symmetric_matrix_easier_linear_system} and the observation that a symmetric matrix $\boldsymbol{\Theta}$ of order $K+1$ satisfies both Equations~\eqref{Paper03_product_matrices_projections} and~\eqref{Paper03_lyapunov_equation} if, and only if,
   it satisfies both $\boldsymbol{\Theta} \boldsymbol{u} = 0$ and~\eqref{Paper03_lyapunov_equation}. In fact, if $\boldsymbol{\Theta}$ satisfies~\eqref{Paper03_product_matrices_projections}, then $\boldsymbol{\Theta} \boldsymbol{u} = \boldsymbol{\Theta} \mathcal{P}_{s}\boldsymbol{u} = 0$. On the other hand, if $\boldsymbol{\Theta} \boldsymbol{u} = 0$, then for all $\boldsymbol{y} = \boldsymbol{y}_{s} + r\boldsymbol{u} \in \mathbf{R}^{K + 1}$, we have
    \begin{equation*}
         \boldsymbol{\Theta} \mathcal{P}_{s} \boldsymbol{y} = \boldsymbol{\Theta} \mathcal{P}_{s} (\boldsymbol{y}_{s} + r\boldsymbol{u}) = \boldsymbol{\Theta} \boldsymbol{y}_{s} = \boldsymbol{\Theta} (\boldsymbol{y}_{s} + r\boldsymbol{u}) = \boldsymbol{\Theta} \boldsymbol{y}.
    \end{equation*}
    Therefore, $\boldsymbol{\Theta} \boldsymbol{u} = 0$ implies that $\boldsymbol{\Theta} = \boldsymbol{\Theta} \mathcal{P}_{s}$. Since $\boldsymbol{\Theta}$ is symmetric, this implies that $\boldsymbol{\Theta} = \mathcal{P}_{s}^{*}\boldsymbol{\Theta} $, i.e.~that identity~\eqref{Paper03_product_matrices_projections} is satisfied, as desired.
\end{proof}

We are now in a position to prove Theorem~\ref{Paper03_explicit_formula_solution_lyapunov_equation}.

\begin{proof}[Proof of Theorem~\ref{Paper03_explicit_formula_solution_lyapunov_equation}]
    To prove our desired claim, we first show that
the integral expression for $\boldsymbol{\Theta}$ is well defined, for which it will suffice to establish that there exist $C_{1}, C_{2} > 0$ such that for all $t \geq 0$,
    \begin{equation} \label{Paper03_bound_norm_exponential_lyapunov}
        \max \Big\{\Big\vert \Big\vert \exp({\mathcal{J}\mathbf{F}^{*}_{\boldsymbol{x}}t})\mathcal{P}_{s}^{*}\Big\vert \Big\vert, \,  \Big\vert \Big\vert \mathcal{P}_{s}\exp({\mathcal{J}\mathbf{F}_{\boldsymbol{x}}t}) \Big\vert \Big\vert\Big\} \leq C_{1}e^{- C_{2}t},
    \end{equation}
    where $\vert \vert \cdot \vert \vert$ denotes the operator norm. In order to establish~\eqref{Paper03_bound_norm_exponential_lyapunov}, observe that for any $\boldsymbol{y} = \boldsymbol{y}_{s} + \langle \boldsymbol{v}, \boldsymbol{y}\rangle \boldsymbol{u}$, where $\boldsymbol{y}_{s} \in E_{s}$, we have
    \begin{equation*}
    \begin{aligned}
        \Big\vert \Big\vert \mathcal{P}_{s}\exp({\mathcal{J}\mathbf{F}_{\boldsymbol{x}}t}) (\boldsymbol{y}_{s} + \langle \boldsymbol{v}, \boldsymbol{y}\rangle \boldsymbol{u}) \Big\vert \Big\vert & \leq \Big\vert \Big\vert \mathcal{P}_{s}\exp({\mathcal{J}\mathbf{F}_{\boldsymbol{x}}t})\boldsymbol{y}_{s} \Big\vert \Big\vert + \vert \langle \boldsymbol{v}, \boldsymbol{y}\rangle \vert \cdot \Big\vert \Big\vert \mathcal{P}_{s}\exp({\mathcal{J}\mathbf{F}_{\boldsymbol{x}}t})\boldsymbol{u} \Big\vert \Big\vert \\ & \leq \vert \vert \mathcal{P}_{s} \vert \vert \cdot \Big\vert \Big\vert \exp({\mathcal{J}\mathbf{F}_{\boldsymbol{x}}t})\boldsymbol{y}_{s} \Big\vert \Big\vert + \vert \langle \boldsymbol{v}, \boldsymbol{y}\rangle \vert \cdot \vert \vert \mathcal{P}_{s}\boldsymbol{u} \vert\vert \\ & \leq e^{-Ct}\vert \vert \mathcal{P}_{s} \vert \vert \cdot \vert \vert \boldsymbol{y}_{s}\vert \vert,
    \end{aligned}
    \end{equation*}
    where in the second inequality we used the fact that since $\boldsymbol{u} \in \ker\left(\mathcal{J}\mathbf{F}_{\boldsymbol{x}}\right)$, $\exp({\mathcal{J}\mathbf{F}_{\boldsymbol{x}}t})\boldsymbol{u} \equiv \boldsymbol{u}$ for all $t \geq 0$, and in the final one we used the fact that on $E_{s}$, all the eigenvalues of $\mathcal{J}\mathbf{F}_{\boldsymbol{x}}$ have negative real part (see~\cite{perko2013differential}). Observing that, since $\mathcal{P}^{*}_{s}\boldsymbol{v} = 0$, $\mathcal{P}^{*}_{s}$ is a projection matrix on the space spanned by the eigenvectors of $\mathcal{J}\mathbf{F}_{\boldsymbol{x}}^{*}$ associated to the eigenvalues with negative real part, we can also bound $\Big\vert \Big\vert \exp({\mathcal{J}\mathbf{F}^{*}_{\boldsymbol{x}}t})\mathcal{P}_{s}^{*}\Big\vert \Big\vert$. Hence,~\eqref{Paper03_bound_norm_exponential_lyapunov} is proved, and therefore the expression
for $\boldsymbol{\Theta}$ is well defined. %Let $\boldsymbol{\Theta}$ be as in the statement of the theorem. 
Symmetry of the Hessian implies that $\boldsymbol{\Theta}$ is a 
symmetric matrix. Moreover, 
$\boldsymbol{\Theta}\boldsymbol{u} = 0$, since
    \begin{equation*}
        \mathcal{P}_{s}\exp({\mathcal{J}\mathbf{F}_{\boldsymbol{x}}t})\boldsymbol{u} = \mathcal{P}_{s}\boldsymbol{u} = 0.
    \end{equation*}
By Theorem~\ref{Paper03_computation_symmetric_matrix_easier_linear_system}, it remains to check that $\boldsymbol{\Theta}$ also solves Equation~\eqref{Paper03_lyapunov_equation}. Observe that
    \begin{equation*}
    \begin{aligned}
        \mathcal{J}\mathbf{F}_{\boldsymbol{x}}^{*}\boldsymbol{\Theta} + \boldsymbol{\Theta}\mathcal{J}\mathbf{F}_{\boldsymbol{x}} & = - \int_{0}^{+\infty} \frac{d}{dt} \left(\exp({\mathcal{J}\mathbf{F}^{*}_{\boldsymbol{x}}t})\mathcal{P}_{s}^{*} \left(\sum_{i = 0}^{K} v_{i} \Hess_{\boldsymbol{x}}(F_{i})\right) \mathcal{P}_{s}\exp({\mathcal{J}\mathbf{F}_{\boldsymbol{x}}t}) \right) \, dt \\ & = \mathcal{P}_{s}^{*} \left(\sum_{i = 0}^{K} v_{i} \Hess_{\boldsymbol{x}}(F_{i})\right) \mathcal{P}_{s},
    \end{aligned}
    \end{equation*}
    since, by~\eqref{Paper03_bound_norm_exponential_lyapunov}, 
    \begin{equation*}
        \lim_{t \rightarrow +\infty} \exp({\mathcal{J}\mathbf{F}^{*}_{\boldsymbol{x}}t})\mathcal{P}_{s}^{*} \left(\sum_{i = 0}^{K} v_{i} \Hess_{\boldsymbol{x}}(F_{i})\right) \mathcal{P}_{s}\exp({\mathcal{J}\mathbf{F}_{\boldsymbol{x}}t}) = 0.
    \end{equation*}
    Therefore, the formula holds. 

Suppose now that $\mathcal{P}_{s}^{*} \left(\sum_{i = 0}^{K} v_{i} \Hess_{\boldsymbol{x}}(F_{i})\right) \mathcal{P}_{s} \succeq 0$. Fix $\boldsymbol{y} \in \mathbb{R}^{K+1}$, and consider the ODE
    \begin{equation*}
    \begin{aligned}
        \frac{d}{dt} \boldsymbol{y}(t) & = \mathcal{J}\mathbf{F}_{\boldsymbol{x}} \boldsymbol{y}(t), \\ \boldsymbol{y}(0) & = \boldsymbol{y},
    \end{aligned}
    \end{equation*}
    with solution $\boldsymbol{y}(t) = \exp({\mathcal{J}\mathbf{F}_{\boldsymbol{x}}t}) \boldsymbol{y}$, for all $t \geq 0$. By the integral representation of $\boldsymbol{\Theta}$, we have
    \begin{equation} \label{Paper03_lyapunov_exponential_formula_i}
    \begin{aligned}
        \langle \boldsymbol{y}, \boldsymbol{\Theta}\boldsymbol{y} \rangle & = - \int_{0}^{+\infty} \boldsymbol{y}^{*}\exp\Big({\mathcal{J}\mathbf{F}_{\boldsymbol{x}}^{*}t}\Big)\mathcal{P}_{s}^{*} \left(\sum_{i = 0}^{K} v_{i} \Hess_{\boldsymbol{x}}(F_{i})\right) \mathcal{P}_{s}\exp\Big({\mathcal{J}\mathbf{F}_{\boldsymbol{x}}t}\Big)\boldsymbol{y} \, dt \\ & = - \int_{0}^{+\infty} \left\langle \boldsymbol{y}(t), \, \mathcal{P}_{s}^{*} \left(\sum_{i = 0}^{K} v_{i} \Hess_{\boldsymbol{x}}(F_{i})\right) \mathcal{P}_{s}\boldsymbol{y}(t)\right\rangle \; dt \\ & \leq 0,
    \end{aligned}
    \end{equation}
    since we assumed $\mathcal{P}_{s}^{*} \left(\sum_{i = 0}^{K} v_{i} \Hess_{\boldsymbol{x}}(F_{i})\right) \mathcal{P}_{s}$ is positive semidefinite. Since~\eqref{Paper03_lyapunov_exponential_formula_i} holds for any $\boldsymbol{y} \in \mathbb{R}^{K+1}$, we conclude that $\boldsymbol{\Theta}$ is negative semidefinite, which completes the proof.
\end{proof}

\subsection{Proof of Lemma~\ref{Paper03_auxiliary_lemma_uniform_integrability_martingales}} \label{Paper03_section_proof_lemma_uniform_integrability_martingales}

We will need the following version of the Burkholder-Davis-Gundy (BDG) inequality for càdlàg martingales; the inequality in this form is stated on page~527 of~\cite{mueller1995stochastic}, and in~(53) in~\cite{durrett2016genealogies}, and it is a consequence of~\cite[Theorem~21.1]{burkholder1973distribution}. For a square-integrable càdlàg martingale $(M(t))_{t \geq 0}$ with $M(0) = 0$, for $T \geq 0$ and $p \geq 2$,
    \begin{equation} \label{Paper03_BDG_inequality_cadlag_version}
        \mathbb{E}\left[\sup_{t \leq T} \left\vert M(t) \right\vert^{p} \right] \lesssim_{p} \mathbb{E}\left[\left(\left\langle M\right\rangle(T)\right)^{p/2} + \sup_{t \leq T} \left\vert M(t) - M(t-) \right\vert^{p}\right],
    \end{equation}
where $(\langle M\rangle (t))_{t \geq 0}$ denotes the unique predictable càdlàg process for which $\left(M^2(t) - \langle M\rangle (t)\right)_{t \geq 0}$ is a martingale.
% (see~\cite{bass2004stochastic} for a review of the difference between quadratic variation and the predictable bracket process of càdlàg martingales with jumps).

We are ready to prove Lemma~\ref{Paper03_auxiliary_lemma_uniform_integrability_martingales}.

\begin{proof}[Proof of Lemma~\ref{Paper03_auxiliary_lemma_uniform_integrability_martingales}]
    The fact that the sequence of random variables $\Big((M^N(T))^2\Big)_{N \in \mathbb{N}}$ is uniformly integrable follows directly from the conditions of the lemma and an application of the BDG inequality in~\eqref{Paper03_BDG_inequality_cadlag_version} with~$p=4$. By the standard BDG inequality and~\eqref{Paper03_BDG_inequality_cadlag_version}, we conclude that
    \begin{equation*}
        \mathbb{E}\Big[\Big[M^N\Big]^{p/2}(T)\Big] \lesssim_p \mathbb{E}\left[\left(\left\langle M\right\rangle(T)\right)^{p/2} + \sup_{t \leq T} \left\vert M(t) - M(t-) \right\vert^{p}\right],
    \end{equation*}
    and therefore, by the conditions~(i) and~(ii) of this lemma, the sequence of random variables $\Big(\left[M^N\right](T)\Big)_{N \in \mathbb{N}}$ is also uniformly integrable, which completes the proof.
\end{proof}

\subsection{Proof of Lemma~\ref{Paper03_elementary_lemma_moments_fast_environment}} \label{Paper03_subsection_proofs_auxiliary_computations_fast_env}

\begin{proof}[Proof of Lemma~\ref{Paper03_elementary_lemma_moments_fast_environment}] Observe that by an algebraic manipulation of the variables, we have
\begin{equation} \label{Paper03_simplification_elementary_fraction_fast_environment_step_i}
    \frac{f_1 + \bar\upsilon f_2}{f_3 + \bar\upsilon f_4} = \frac{f_1}{f_3} - \frac{\bar\upsilon(f_1f_4 - f_2f_3)}{f_3^2} + \frac{\bar\upsilon^2 f_4(f_1f_4 - f_2f_3)}{(f_3 + \bar\upsilon f_4)f_3^2}.
\end{equation}
    By taking expectations with respect to $\bar\upsilon$ on both sides of~\eqref{Paper03_simplification_elementary_fraction_fast_environment_step_i}, and then using~\eqref{Paper03_distribution_scaled_fast_environment} and the fact that
    \begin{equation*}
        \frac{1}{(f_3 + s_Nf_4)} + \frac{1}{(f_3 - s_Nf_4)} = \frac{2f_3}{(f_3^2 -s_N^2f_4^2)},
    \end{equation*}
    we conclude that~\eqref{Paper03_eq:elementary_lemma_first_moment_fast_environment} holds.

    In order to establish~\eqref{Paper03_eq:elementary_lemma_second_moment_fast_environment} and~\eqref{Paper03_eq:elementary_lemma_higher_moment_fast_environment}, we observe that by combining~\eqref{Paper03_simplification_elementary_fraction_fast_environment_step_i} with~\eqref{Paper03_eq:elementary_lemma_first_moment_fast_environment}, we have for every $q \in \mathbb{N}$,
    \begin{equation} \label{Paper03_elementary_fraction_fast_environment_step_ii}
    \begin{aligned}
        & \left(\frac{f_1+\bar\upsilon f_2}{f_3 + \bar\upsilon f_4} -\mathbb{E}_{\pi^N}\left[\frac{f_1+\bar\upsilon f_2}{f_3 + \bar\upsilon f_4}\right]\right)^q \\ & \quad = \left(-\frac{\bar\upsilon(f_1f_4 - f_2f_3)}{f_3^2} + \frac{\bar\upsilon^2f_4(f_1f_4 - f_2f_3)}{f_3^2(f_3+\bar\upsilon f_4)} - \frac{2ps_N^2f_4(f_1f_4 -f_2f_3)}{(f_3^2 -s_N^2f_4^2)f_3}\right)^q \\ & \quad = (-1)^q\frac{\bar\upsilon^q (f_1f_4 -f_2f_3)^q}{f_3^{2q}} \\ & \quad \quad \quad + (f_1f_4-f_2f_3)^q\sum_{\substack{l_1 + l_2 + l_3 = q: \\ l_2 + l_3 > 0}} \bar\upsilon^{l_1+2l_2}s_N^{2l_3}\left(\frac{(-1)^{l_1+l_3}(2p)^{l_3}q!f_4^{l_2+l_3}}{l_1!l_2!l_3! f_3^{q + l_1+l_2}(f_3 + \bar\upsilon)^{l_2}(f_3^2 -s_N^2f_4^2)^{l_3}}\right).
    \end{aligned}
    \end{equation}
    Identities~\eqref{Paper03_eq:elementary_lemma_second_moment_fast_environment} and~\eqref{Paper03_eq:elementary_lemma_higher_moment_fast_environment} then follow by taking expectations on both sides of~\eqref{Paper03_elementary_fraction_fast_environment_step_ii} and using~\eqref{Paper03_distribution_scaled_fast_environment}. 
    
    It remains to establish~\eqref{Paper03_eq:elementary_lemma_mixed_moments_fast_environment}. By using~\eqref{Paper03_distribution_scaled_fast_environment} and then~\eqref{Paper03_simplification_elementary_fraction_fast_environment_step_i}, we have
    \begin{equation*}
    \begin{aligned}
        \mathbb{E}_{\pi^N}\left[\bar\upsilon \left(\frac{f_1+\bar\upsilon f_2}{f_3+\bar\upsilon f_4} - \mathbb{E}_{\pi^N}\left[\frac{f_1+\bar\upsilon f_2}{f_3+\bar\upsilon f_4}\right]\right)\right] & = \mathbb{E}_{\pi^N}\left[\frac{\bar\upsilon f_1}{f_3} - \frac{\bar\upsilon^2(f_1f_4 - f_2f_3)}{f_3^2} + \frac{\bar\upsilon^3 f_4(f_1f_4 - f_2f_3)}{(f_3 + \bar\upsilon f_4)f_3^2}\right] \\ & = -2ps_N^2\frac{(f_1f_4 - f_2f_3)}{f_3^2} + \mathcal{O}(s_N^3),
    \end{aligned}
    \end{equation*}
    where the last equality also holds due to~\eqref{Paper03_distribution_scaled_fast_environment}, which completes the proof.
\end{proof}

\subsection{Proof of Lemmas~\ref{Paper03_lem:explicit_drifdt_fast_env} and~\ref{Paper03_characterisation_flow_associated_fast_changing_environment}} \label{Paper03_sec_comp_lemmas_fast_env}

We will start by establishing Lemma~\ref{Paper03_characterisation_flow_associated_fast_changing_environment}.  Recall the definition of the flow $\boldsymbol{F}^{\, \textrm{fast}}: [0,1]^{K+1} \times (-1,1)^{K} \rightarrow \mathbb{R}^{2K+1}$ in~\eqref{Paper03_flow_definition_WF_model_fast_changing_environment} and its associated projection map $\boldsymbol{\Phi}^{\, \textrm{fast}}$ introduced after~\eqref{Paper03_attractor_manifold_fast_environment}. Recall the definition of $\Gamma^{\, \textrm{fast}}$ in~\eqref{Paper03_attractor_manifold_fast_environment}, the definition of $\boldsymbol{F}: [0,1]^{K+1} \rightarrow \mathbb{R}^{K+1}$ in~\eqref{Paper03_flow_definition_WF_model} and its associated projection map~$\boldsymbol{\Phi}$ in~\eqref{Paper03_definition_projection_map}. Also, recall the definition of $D(1)$ introduced before Assumption~\ref{Paper03_assumption_manifold}.

Our strategy will be to use~\eqref{Paper03_formula_Parsons_Rogers_first_order_derivatives} and~\eqref{Paper03_formula_Parsons_Rogers_second_order_derivatives} to compute the derivatives of $\Phi^{\, \textrm{fast}}_{0}$ in $\Gamma^{\, \textrm{fast}}$. In the remainder of this section, let $(\boldsymbol{x}, \boldsymbol{0}) = (x, \cdots, x,0, \cdots,0) \in \Gamma^{\, \textrm{fast}}$. Observe that the Jacobian of $\boldsymbol{F}^{\, \textrm{fast}}$ evaluated at $(\boldsymbol{x}, \boldsymbol{0}) \in \Gamma^{\, \textrm{fast}}$ is given by
\begin{equation} \label{Paper03_Jacobian_flow_fast_environment}
\begin{aligned}
    & \mathcal{J}\mathbf{F}^{\, \textrm{fast}}_{(\boldsymbol{x}, \boldsymbol{0})} \\ & \, = \left(\begin{array}{cccccccc}
           - (1 - b_{0})(1-x) & b_{1}(1-x) & \cdots & b_{K}(1-x) & b_{1}x(1-x) & b_2x(1-x)& \cdots & b_Kx(1-x) \\
           
          1 & -1 & \cdots & 0 & 0 & 0 & \cdots & 0 
          \\ \vdots & \vdots & \ddots & \vdots & \vdots & \vdots & \ddots & \vdots \\ 0 & 0 & \cdots & -1 & 0 & 0 & \cdots & 0 
          \\ 0 & 0 & \cdots & 0 & -1 & 0 & \cdots & 0 
          \\ 0 & 0 & \cdots & 0 & 1 & -1 & \cdots & 0
          \\ \vdots & \vdots & \ddots & \vdots & \vdots & \vdots & \ddots & \vdots 
          \\ 0 & 0 & \cdots & 0 & 0 & 0 & \cdots & -1
        \end{array}\right).
\end{aligned}
\end{equation}
Next, we characterise the eigenvalues of $\mathcal{J}\mathbf{F}^{\, \textrm{fast}}$ in $\Gamma^{\, \textrm{fast}}$, and show that Assumption~\ref{Paper03_assumption_manifold} is satisfied in the case of a fast-changing environment.

\begin{lemma} 
\label{Paper03_eigenvalue_condition_fast_environment}
For every $(\boldsymbol{x}, \boldsymbol{0}) \in \Gamma^{\, \textrm{fast}}$, the Jacobian matrix of the flow 
$\mathbf{F}^{\, \textrm{fast}}$ evaluated at $(\boldsymbol{x}, \boldsymbol{0})$ has exactly one eigenvalue equal to $0$ and $2K$ eigenvalues in the open disk 
$D(1) = \{\lambda \in \mathbb{C}: \; \vert \lambda + 1 \vert < 1 \}$.
\end{lemma}

\begin{proof}
    We will follow closely the proof of Lemma~\ref{Paper03_eigenvalue_condition_constant_environment}. Observe that $\lambda \in \mathbb{C}$ is an eigenvalue of 
$\mathcal{J}\mathbf{F}^{\, \textrm{fast}}_{(\boldsymbol{x}, \boldsymbol{0})}$ if and only $\lambda + 1$ is an 
eigenvalue of $\mathcal{J}\mathbf{F}^{\, \textrm{fast}}_{(\boldsymbol{x}, \boldsymbol{0})} + \mathbf{I}_{2K+1}$, 
where $\mathbf{I}_{2K+1}$ is the identity matrix of order $2K + 1$. By using~\eqref{Paper03_Jacobian_flow_fast_environment} and observing that the bottom-right $K \times K$ principal submatrix of $\mathcal{J}\mathbf{F}^{\, \textrm{fast}}_{(\boldsymbol{x}, \boldsymbol{0})}$ is lower triangular, and then by directly computing the characteristic polynomial $\tilde{p}^*$ of 
$\mathcal{J}\mathbf{F}^{\, \textrm{fast}}_{(\boldsymbol{x}, \boldsymbol{0})} + \mathbf{I}_{2K+1}$, we conclude that the eigenvalues of $\mathcal{J}\mathbf{F}^{\, \textrm{fast}}_{(\boldsymbol{x}, \boldsymbol{0})}$ are the solutions of the polynomial equation
   \begin{equation} \label{Paper03_characteristic_polynomial_Jac_fast_env}
    \tilde{p}^*(\lambda + 1) 
= (-1)^{K}(\lambda + 1)^K \tilde{p}(\lambda + 1),
   \end{equation}
where $\tilde{p}$ is the polynomial defined in~\eqref{Paper03_characteristic_polynomial_J_+_I_const_env}. We proved in 
Lemma~\ref{Paper03_eigenvalue_condition_constant_environment} that 
$\lambda = 0$ is a root of algebraic multiplicity one of $\tilde{p}(\lambda + 1) = 0$, all the other roots of $\tilde{p}(\lambda + 1) = 0$ have negative real part and are elements of the disk $D(1)$. Hence,~\eqref{Paper03_characteristic_polynomial_Jac_fast_env} implies that $\mathcal{J}\mathbf{F}^{\, \textrm{fast}}_{(\boldsymbol{x}, \boldsymbol{0})}$ has exactly one eigenvalue $0$ and $2K$ eigenvalues in $D(1)$, which completes the proof.
\end{proof}

Lemma~\ref{Paper03_eigenvalue_condition_fast_environment} implies that $\Gamma^{\, \textrm{fast}}$ is the attractor manifold for the flow in~\eqref{Paper03_flow_definition_WF_model_fast_changing_environment}, and therefore the projection map $\boldsymbol{\Phi}^{\, \textrm{fast}}:[0,1]^{K+1} \times [-1,1]^K$ introduced after~\eqref{Paper03_attractor_manifold_fast_environment} is well defined. Let ${\boldsymbol{u}}^{\, \textrm{fast}} \equiv {\boldsymbol{u}}^{\, \textrm{fast}}(\boldsymbol{x}, \boldsymbol{0}) \in \mathbb{R}^{2K+1}$ be a right eigenvector of $\mathcal{J}\mathbf{F}^{\, \textrm{fast}}_{(\boldsymbol{x}, \boldsymbol{0})}$ associated with the eigenvalue $0$. To simplify computations, we set
\begin{equation} \label{Paper03_right_eigenvector_fast_env}
    {\boldsymbol{u}}^{\, \textrm{fast}} = \left(\begin{array}{c}
    1 \\ 
    1 \\ 
    \vdots \\ 
    1 \\ 
    0 \\ 
    0 \\ 
    \vdots \\ 
    0
    \end{array}\right).
\end{equation}
Let ${\boldsymbol{v}}^{\, \textrm{fast}} \equiv {\boldsymbol{v}}^{\, \textrm{fast}}(\boldsymbol{x}, \boldsymbol{0}) \in \mathbb{R}^{2K+1}$ denote the left eigenvector of $\mathcal{J}\mathbf{F}^{\, \textrm{fast}}_{(\boldsymbol{x}, \boldsymbol{0})}$ associated with the eigenvalue $0$ for which $\langle {\boldsymbol{v}}^{\, \textrm{fast}}, {\boldsymbol{u}}^{\, \textrm{fast}} \rangle = 1$. This gives
\begin{equation} \label{Paper03_left_eigenvector_fast_env}
    {\boldsymbol{v}}^{\, \textrm{fast}} = \frac{1}{(B(1-x)+1)}\left(\begin{array}{c}
    1 \\ 
    (1-b_0)(1-x) \\ 
    \vdots \\ 
    b_K(1-x) \\ 
    (1-b_0)x(1-x) \\ 
    (\sum_{i = 2}^{K} b_i)x(1-x) \\ 
    \vdots \\ 
    b_Kx(1-x)
    \end{array}\right),
\end{equation}
where $B$ is defined in~\eqref{Paper03_mean_age_germination}. Observe that the first $K$ entries of ${\boldsymbol{u}}^{\, \textrm{fast}}$ and ${\boldsymbol{v}}^{\, \textrm{fast}}$ 
coincide with those of the right~\eqref{Paper03_right_eigenvector_Jacobian} 
and left~\eqref{Paper03_left_eigenvector_jacobian} eigenvectors, respectively, 
associated with the eigenvalue $0$ of the Jacobian of the flow $\boldsymbol{F}$ 
arising in the Wright-Fisher dynamics under a constant environment. We also parametrise the manifold $\Gamma^{\, \textrm{fast}}$ using the map ${\boldsymbol{\gamma}}^{\, \textrm{fast}}: [0,1] \rightarrow \Gamma^{\, \textrm{fast}}$ given by
\begin{equation} \label{Paper03_parametrisation_manifold_fast_env}
    {\boldsymbol{\gamma}}^{\, \textrm{fast}}(x) \defeq (x, x, \cdots, x, 0, 0, \cdots, 0)^*.
\end{equation}
Observe that the map~${\boldsymbol{\gamma}}^{\, \textrm{fast}}$ satisfies
\begin{equation} \label{Paper03_parametrisation_manifold_fast_env_derivatives}
    \frac{d}{dx} {\boldsymbol{\gamma}}^{\, \textrm{fast}}(x) = (1, 1, \cdots, 1, 0, 0, \cdots, 0)^* \quad \textrm{and} \quad   \frac{d^2}{dx^2} {\boldsymbol{\gamma}}^{\, \textrm{fast}}(x) = 0.
\end{equation}
Next, we compute the first order derivatives of the map $\Phi_0^{\, \textrm{fast}}$ in $\Gamma^{\, \textrm{fast}}$.

\begin{lemma} \label{Paper03_lem_first_order_derivatives_fast_env}
    For any $K \in \mathbb{N}$, $\boldsymbol{b} = (b_i)_{i = 0}^{K}$ satisfying Assumption~\ref{Paper03_assumption_probability_distribution_bi}, and any $(\boldsymbol{x},\boldsymbol{0}) \in \Gamma^{\, \textrm{fast}}$, Equation~\eqref{Paper03_first_order_derivatives_fast_environment} holds.
\end{lemma}

\begin{proof}
    The result follows from using~\eqref{Paper03_formula_Parsons_Rogers_first_order_derivatives} with~\eqref{Paper03_left_eigenvector_fast_env},~\eqref{Paper03_parametrisation_manifold_fast_env} and~\eqref{Paper03_parametrisation_manifold_fast_env_derivatives}.
\end{proof}

We are now in a position to establish Lemma~\ref{Paper03_characterisation_flow_associated_fast_changing_environment}.

\begin{proof}[Proof of Lemma~\ref{Paper03_characterisation_flow_associated_fast_changing_environment}]
    By Lemmas~\ref{Paper03_eigenvalue_condition_fast_environment} and~\ref{Paper03_lem_first_order_derivatives_fast_env}, it remains only to verify that~\eqref{Paper03_second_order_derivatives_fast_environment_easy_part} holds. To do so we apply Theorem~\ref{Paper03_prop_uniqueness_quadratic_term_nonlinear_manifold}.

    \medskip
    \noindent\textit{Step~(i): Projection onto the linear stable manifold.}

    \medskip
    Using~\eqref{Paper03_right_eigenvector_fast_env} and~\eqref{Paper03_left_eigenvector_fast_env} in the definition~\eqref{Paper03_projection_linear_stable_manifold}, we compute the projection matrix
    \begin{equation} \label{Paper03_projection_matrix_fast_env_rewritten}
        \mathcal{P}_s^{\, \textrm{fast}}
        =
        \begin{pmatrix}
            \mathcal{P}_s & \boldsymbol{R} \\
            \boldsymbol{0}_{K \times (K+1)} & \boldsymbol{I}_K
        \end{pmatrix},
        \qquad
        (\mathcal{P}_s^{\, \textrm{fast}})^*
        =
        \begin{pmatrix}
            \mathcal{P}_s^{*} & \boldsymbol{0}_{(K+1)\times K} \\
            \boldsymbol{R}^* & \boldsymbol{I}_K
        \end{pmatrix},
    \end{equation}
    where $\mathcal{P}_s$ is defined in~\eqref{Paper03_computation_projection_matrix_stable_linear_manifold}.
    The matrix $\boldsymbol{R}$ has the form
    \begin{equation*}
        \boldsymbol{R}
        =
        x(1-x)
        \begin{pmatrix}
            -(1-b_0) & -\!\sum_{i=2}^K b_i & \cdots & -b_K \\
            -(1-b_0) & -\!\sum_{i=2}^K b_i & \cdots & -b_K \\
            \vdots & \vdots & \ddots & \vdots \\
            -(1-b_0) & -\!\sum_{i=2}^K b_i & \cdots & -b_K 
        \end{pmatrix}.
    \end{equation*}

    \medskip
    \noindent\textit{Step~(ii): Hessian of the fast-environment flow.}

    \medskip
    From~\eqref{Paper03_flow_definition_WF_model_fast_changing_environment}, the Hessians of $F^{\, \textrm{fast}}_i$ vanish along $\Gamma^{\, \textrm{fast}}$ for all $i\in\llbracket K\rrbracket$, and similarly $\Hess(F^{\, \textrm{fast}}_{\textrm{env},i}) \equiv 0$ for $i\in\llbracket K-1\rrbracket$. Therefore, at $(\boldsymbol{x},\boldsymbol{0})\in\Gamma^{\,\textrm{fast}}$, the only non-zero Hessian block is
    \begin{equation} \label{Paper03_hessian_fast_env_rewritten}
        \Hess_{(\boldsymbol{x},\boldsymbol{0})}\!\left(F^{\,\textrm{fast}}_0\right)
        =
        \begin{pmatrix}
            \Hess_{\boldsymbol{x}}(F_0) & \boldsymbol{H}_{\boldsymbol{x},\boldsymbol{\upsilon}} \\
            \boldsymbol{H}_{\boldsymbol{x},\boldsymbol{\upsilon}}^{*} & \boldsymbol{H}_{\boldsymbol{\upsilon},\boldsymbol{\upsilon}}
        \end{pmatrix},
    \end{equation}
    where the entries of $\Hess_{\boldsymbol{x}}(F_0)$ are given by~\eqref{Paper03_computation_hessian_manifold}, while $\boldsymbol{H}_{\boldsymbol{x},\boldsymbol{\upsilon}}$ and $\boldsymbol{H}_{\boldsymbol{\upsilon},\boldsymbol{\upsilon}}$ are given by
    \begin{equation*}\begin{aligned} & \boldsymbol{H}_{\boldsymbol{x},\boldsymbol{\upsilon}} \\ & \; = \left(\begin{array}{cccc} 2b_1(1-b_0)x(1-x) -b_1x & 2b_2(1-b_0)x(1-x) -b_2x & \cdots & 2b_K(1-b_0)x(1-x) - b_Kx \\ b_1(1-x) -2b_1^2x(1-x) &-2b_1b_2x(1-x) & \cdots & -2b_1b_Kx(1-x) \\ -2b_1b_2x(1-x) & b_2(1-x) -2b_2^2x(1-x) & \cdots & -2b_2b_Kx(1-x) \\ \vdots & \vdots & \ddots & \vdots \\ -2b_1b_Kx(1-x) & -2b_2b_Kx(1-x) & \cdots & b_K(1-x)-2b_K^2x(1-x) \end{array}\right), \end{aligned}
    \end{equation*}
    and
    \begin{equation*} \begin{aligned} \boldsymbol{H}_{\boldsymbol{\upsilon,\boldsymbol{\upsilon}}} =-2x^2(1-x)\left(\begin{array}{cccc} b_1^2 & b_1b_2 & \cdots & b_1b_K \\ b_1b_2 & b_2^2 & \cdots & b_2b_K\\ \vdots & \vdots & \ddots & \vdots \\ b_1b_K & b_2b_K & \cdots & b_K^2 \end{array}\right). \end{aligned}
    \end{equation*}

    \medskip
    
    \noindent\textit{Step~(iii): Quadratic term of the non-linear stable manifold.}

    \medskip
    
    Let
    \[
        \boldsymbol{\Theta}^{\,\textrm{fast}}
        =
        \begin{pmatrix}
            \boldsymbol{\Theta}^{\textrm{fast}}_{\boldsymbol{x},\boldsymbol{x}}
            & \boldsymbol{\Theta}^{\textrm{fast}}_{\boldsymbol{x},\boldsymbol{\upsilon}} \\
            (\boldsymbol{\Theta}^{\textrm{fast}}_{\boldsymbol{x},\boldsymbol{\upsilon}})^*
            & \boldsymbol{\Theta}^{\textrm{fast}}_{\boldsymbol{\upsilon},\boldsymbol{\upsilon}}
        \end{pmatrix}
    \]
    denote the quadratic approximation to the stable manifold of the flow $\boldsymbol{F}^{\textrm{fast}}$.
    Substituting~\eqref{Paper03_Jacobian_flow_fast_environment} and the matrices computed above into the Lyapunov system~\eqref{Paper03_lyapunov_equation} together with the constraint~\eqref{Paper03_restriction_theta_u_in_nullspace}, we find that the block
    $\boldsymbol{\Theta}^{\,\textrm{fast}}_{\boldsymbol{x},\boldsymbol{x}}$ satisfies
    \[
        \mathcal{J}\boldsymbol{F}_{\boldsymbol{x}} \,
        \boldsymbol{\Theta}^{\,\textrm{fast}}_{\boldsymbol{x},\boldsymbol{x}}
        +
        \boldsymbol{\Theta}^{\,\textrm{fast}}_{\boldsymbol{x},\boldsymbol{x}}
        \mathcal{J}\boldsymbol{F}_{\boldsymbol{x}}
        =
        v_0\,\mathcal{P}_s^{*}
        \Hess_{\boldsymbol{x}}(F_0)
        \mathcal{P}_s,
        \qquad
        \boldsymbol{\Theta}^{\,\textrm{fast}}_{\boldsymbol{x},\boldsymbol{x}}\,\boldsymbol{u}
        = 0,
    \]
    where $\boldsymbol{F}$, $\boldsymbol{u}$, $v_0$ and $\mathcal{P}_s$
    refer to the constant-environment system
    (see~\eqref{Paper03_flow_definition_WF_model},~\eqref{Paper03_right_eigenvector_Jacobian},~\eqref{Paper03_left_eigenvector_jacobian} and~\eqref{Paper03_computation_projection_matrix_stable_linear_manifold}).
    By Theorem~\ref{Paper03_prop_uniqueness_quadratic_term_nonlinear_manifold}, $\boldsymbol{\Theta}^{\,\textrm{fast}}_{\boldsymbol{x},\boldsymbol{x}}$ is therefore exactly the curvature matrix for the constant-environment Wright–Fisher flow.

    \medskip
    \noindent\textit{Step~(iv): Identifying the second-order derivatives.}

    \medskip
    
    Finally, Lemmas~\ref{Paper03_lemma_first_order_derivatives_projection_map} and~\ref{Paper03_lem_first_order_derivatives_fast_env}
    show that, along $\Gamma$,
    \[
        \frac{\partial \Phi^{\,\textrm{fast}}}{\partial x_0}(\boldsymbol{x},\boldsymbol{0})
        =
        \frac{\partial \Phi}{\partial x_0}(\boldsymbol{x}).
    \]
    Combining this identity with~\eqref{Paper03_formula_Parsons_Rogers_second_order_derivatives} and Step~(iii) yields~\eqref{Paper03_second_order_derivatives_fast_environment_easy_part}.
    This completes the proof.
\end{proof}

We are now ready to compute the second order derivatives for the case $K=1$.

\begin{proof}[Proof of Lemma~\ref{Paper03_lem:explicit_drifdt_fast_env}]
    For $K =1$, by substituting the matrices~$\mathcal{P}_s^{\, \textrm{fast}}$ and~$\Hess_{(\boldsymbol{x},\boldsymbol{0})}\!\left(F^{\,\textrm{fast}}_0\right)$  defined in~\eqref{Paper03_projection_matrix_fast_env_rewritten} and~\eqref{Paper03_hessian_fast_env_rewritten} into the Lyapunov system~\eqref{Paper03_lyapunov_equation} together with the constraint~\eqref{Paper03_restriction_theta_u_in_nullspace}, we conclude that the curvature matrix $\boldsymbol{\Theta}^{\, \textrm{fast}}$ satisfies
    \begin{equation} \label{Paper03_mixed_curv_fast_env}
        \theta^{\, \textrm{fast}}_{0, \upsilon_0}=\frac{(1-b_0)^2(1-x)^2\Big[(1-b_0)^2x^2-(1-b_0)(4-b_0)x+(2-b_0)\Big]}{((1-b_0)(1-x)+1)^3((1-b_0)(1-x)+2)},
    \end{equation}
    and
     \begin{equation} \label{Paper03_env_env_curv_fast_env}
        \theta^{\, \textrm{fast}}_{\upsilon_0, \upsilon_0}=-\frac{(1-b_0)^2x(1-x)^2\Big[(1-b_0)^2x^2-(1-b_0)(5-b_0)x+(4-b_0)\Big]}{((1-b_0)(1-x)+1)^3((1-b_0)(1-x)+2)}.
    \end{equation}
    The lemma follows from substituting~\eqref{Paper03_first_order_derivatives_fast_environment},~\eqref{Paper03_left_eigenvector_fast_env},~\eqref{Paper03_mixed_curv_fast_env} and~\eqref{Paper03_env_env_curv_fast_env} into~\eqref{Paper03_formula_Parsons_Rogers_second_order_derivatives}. Since the proof follows standard algebraic manipulations, we omit the details.
\end{proof}

\subsubsection*{Acknowledgements}

The authors thank Diala Abu Awad and Amandine Véber for posing the question that motivated this paper and for many helpful discussions. They are also grateful to Fernando Antoneli, Sofia Backlund, Nick Barton, Matthias Birkner and Matt Roberts for valuable comments and suggestions. While this work was being carried out, JLdOM was supported by a scholarship from the EPSRC Centre for Doctoral Training in Statistical Applied Mathematics at Bath (SAMBa), under the project EP/S022945/1.

\newpage

\bibliographystyle{abbrv}
\bibliography{bibli_seed_banks}

\end{document}